\documentclass[journal]{IEEEtran}
\usepackage{easyReview} 
\usepackage{amssymb}
\usepackage{amsmath}
\usepackage{amsfonts}
\usepackage{graphicx}
\usepackage{color}
\usepackage{caption}
\usepackage{multirow}
\usepackage{diagbox}
\usepackage{lineno}
\usepackage{subcaption}
\usepackage{mathrsfs}
\usepackage{amscd}
\usepackage{algorithm}
\usepackage{algorithmic}
\usepackage{url}
\usepackage{enumitem}
\usepackage{booktabs}
\usepackage{bm}
\newtheorem{theorem}{Theorem}[section]

\newtheorem{definition}{Definition}[section]

\newtheorem{lemma}{Lemma}[section]

\newtheorem{remark}{Remark}[section]

\newtheorem{assumption}{Assumption}[section]

\newenvironment{proof}[1][Proof]{\noindent\textbf{#1.} }{\ \hfill\rule{0.3em}{0.5em}}

\newcommand*{\R}{\mathbb{R}}

\newcommand*{\N}{\mathbb{N}}

\def\beq#1{\begin{equation}\label{#1}}
	\def\eeq{\end{equation}}
\def\bep{\begin{proof}}
	
	\def\ep{\end{proof}}

\def\M{\mathcal M}

\def\A{{\mathcal A}}
\def\B{\mathcal B}

\def\D{\mathcal D}

\def\ignore#1{}
\def\X{\mathcal X}
\def\Y{\mathcal Y}
\def\Z{\mathcal Z}

\def\W{\mathcal W}

\def\U{\mathcal U}
\def\V{\mathcal V}

\def\Q{\mathcal Q}

\def\rank{\operatorname {rank}}

\input epsf

\ifCLASSINFOpdf
\else
\fi
\begin{document}

%\linenumbers          添加行号
\title{Efficient Recovery of Low Rank Tensor via Triple Nonconvex Nonsmooth Rank Minimization}
%A fast proximal iteratively reweighted nuclear norm algorithm for nonconvex low-rank matrix minimization problems

\author{Quan~Yu
%\thanks{This work was supported by NSFC under Grant 11871369 and the Tianjin Research Innovation Project for Postgraduate Students under Grant 2020YJSS140. (Corresponding author: xzzhang@tju.edu.cn.)}% <-this % stops a space
\thanks{Quan Yu is with School of Mathematics, Hunan University, Hunan 410082, P.R. China. (e-mail:quanyu@tju.edu.cn).}
}


\markboth{}%Journal of \LaTeX\ Class Files,~Vol.~14, No.~8, August~2015}%
{Shell \MakeLowercase{\textit{et al.}}: Bare Demo of IEEEtran.cls for IEEE Journals}



\maketitle

\begin{abstract}
A tensor nuclear norm (TNN) based method for solving the tensor recovery problem was recently proposed, and it has achieved state-of-the-art performance. However, it may fail to produce a highly accurate solution since it tends to treats each frontal slice and each rank component of each frontal slice equally. In order to get a recovery with high accuracy, we propose a general and flexible rank relaxation function named double weighted nonconvex nonsmooth rank (DWNNR) relaxation function for efficiently solving the third order tensor recovery problem. The DWNNR relaxation function can be derived from the triple nonconvex nonsmooth rank (TNNR) relaxation function by setting the weight vector to be the hypergradient value of some concave function, thereby adaptively selecting the weight vector.
To accelerate the proposed model, we develop the general inertial smoothing proximal gradient method. Furthermore, we prove that any limit point of the generated subsequence is a critical point. Combining the Kurdyka–Lojasiewicz (KL) property with some milder assumptions, we further give its global convergence guarantee.
Experimental results on a practical tensor completion problem with both synthetic and real data, the results of which demonstrate the efficiency and superior performance of the proposed algorithm.

\end{abstract}


% \fi
\IEEEpeerreviewmaketitle

\section{Introduction}
\IEEEPARstart{L}{ow} rank tensor recovery problem has gotten a lot of attention during the last decade. Furthermore, low rank tensor can be recovered efficiently using tensor (matrix) factorization \cite{LAAW20,HDXY15,XWW+18,YZH20,YZ22a,ZLLZ18} and tensor rank minimization methods \cite{HTZ+17,LPW19,QBNZ21,JNZH20,ZHZ+20a,QBNZ21b,ZBN20,ZZXC16a}, respectively. In this paper, we consider the tensor rank minimization problem.
However, unlike the matrix rank, there is no unique definition of the tensor rank. To exploit the low-rankness of tensor, various tensor decompositions and corresponding tensor ranks are proposed, such as CANDECOMP/PARAFAC (CP) decomposition \cite{CC70,Kie00,LBF12,HXZZ20}, Tucker decomposition \cite{Tuc66} and tensor singular value decomposition (t-SVD) \cite{ZLLZ18,ZEA14}. 

Among them, the t-SVD based method decomposes a tensor into the product of two orthogonal tensors and one f-diagonal tensor (see Section \uppercase\expandafter{\romannumeral2} for details). With the help of the t-SVD framework, the tensor multi-rank and tubal rank were proposed by Kilmer et al. \cite{KBHH13}. Then, Semerci et al. \cite{SHKM14} developed a new tensor nuclear norm (TNN). As the t-SVD is based on an operator theoretic interpretation of third order tensors as linear operators on the space of oriented matrices, the tubal rank and multi-rank of the tensor describe the inherent low rank structure of the tensor without the loss of information inherent in matricization \cite{KBHH13,KM11}. However, two types of prior knowledge are underutilized during the definition of TNN for further exploiting the low-rankness. To begin with, in the tensor's Fourier transform domain, the low-frequency slices mostly carry the tensor's profile information, whilst the high-frequency slices primarily carry the tensor's detail and noise information. Second, larger single values in each frequency slice mostly include information on clean data, whereas lower singular values primarily contain information with noise.% \replace{Thirdly, tubal nuclear norm will produce biased estimators.}{replacement}

In recent years, many nonconvex nonsmooth rank relaxation functions have been proposed in order to take advantage of the prior knowledge contained in tensor data. Examples of them include partial sum of tubal nuclear norm \cite{JHZD20}, weighted tensor nuclear norm \cite{SYL18}, tensor truncated nuclear norm \cite{XQLJ18}, weighted t-Schatten-$ p $ norm \cite{LZT20}, and so on. However, these methods suffer from three disadvantages. First, these methods do not make good use of the two types of tensor prior information mentioned above at the same time; second, none of these methods establish global convergence due to the absence of convexity; finally, none of these methods design an efficient acceleration algorithm.

Fortunately, inspired by \cite{ZGQ19}, we can develop an iteratively double reweighted algorithm scheme, which aim to solve the TNNR minimization problem for obtaining nearly unbiased solution. Motivated by the general accelerated technique in \cite{WL19}, we propose a general inertial smoothing proximal gradient method for the TNNR minimization problem with both local and global convergence guarantees. We highlight the main contributions of this paper as follows:
\begin{itemize}
	\item[(1)] As the tensor rank substitute, we offer a broad and flexible double reweighted relaxation function and induce the double reweighted minimization problem, which is actually derived from the TNNR minimization problem. The TNNR minimization problem can adaptively assign weight values for nonconvex nonsmooth rank relaxation functions.
		
	\item[(2)] We propose an accelerated method for the TNNR minimization problem.
	
	\item[(3)] Under some milder assumptions, we achieve the local convergence guarantee and then use the KL property to provide the global convergence guarantee. Experimental results verify the advantages of our method.
\end{itemize}
The outline of this paper is given as follows. We recall the basic tensor notations in Section \uppercase\expandafter{\romannumeral2}. 
In Section \uppercase\expandafter{\romannumeral3}, we give the main results, including the proposed model, algorithm and the convergence analysis of algorithm. Extensive simulation results are reported in Section \uppercase\expandafter{\romannumeral4}.

\section{NOTATIONS AND PRELIMINARIES}

This section recalls some basic knowledge on tensors. We first give the basic notations and then present the tubal rank, t-SVD and TNN.

\subsection{Notations}

For a positive integer $n$,
$[n]:=\{1,2,\ldots, n\}$. Scalars, vectors and matrices are denoted as lowercase letters ($a,b,c,\ldots$), boldface lowercase letters ($\bm{a} ,\bm{b},\bm{c},\ldots$) and uppercase letters ($A,B,C,\ldots $), respectively.
Third order tensors are denoted as ($\mathcal{A},\mathcal{B},\mathcal{C},\ldots$). For a third order tensor $\mathcal{A} \in \R^{n_1\times n_2\times n_3}$, we use the Matlab notations $  \mathcal{A}(:, :, k)  $ to denote its $ k $-th frontal slice, denoted by $ A^{(k)} $ for all $k\in { [n_3]}$. The inner product of two tensors  $ \mathcal{A},\,\mathcal{B} \in {\R^{{n_1} \times {n_2} \times {n_3}}}$ is the sum of products of their entries, i.e.
$$\left\langle {\mathcal{A},\mathcal{B}} \right\rangle  = \sum\limits_{i = 1}^{{n_1}} {\sum\limits_{j = 1}^{{n_2}} {\sum\limits_{k = 1}^{{n_3}} {{\mathcal{A}_{ijk}}{\mathcal{B}_{ijk}}} } }. $$
The Frobenius norm is ${\left\| \mathcal{A} \right\|} = \sqrt {\left\langle {\mathcal{A},\mathcal{A}} \right\rangle } $. 
%$ {\left\| \mathcal{A} \right\|_{\infty}} $  represents the maximum absolute value of $\mathcal{A}  $.
For a matrix $A$, $ \left\|A\right\|_2  $ represents the largest singular value of $A$. $ \mathscr{D}_r^i\left(x_i\right) $ denotes a diagonal matrix generated by vector $ \bm{x}=\left(x_1,x_2,\ldots,x_r\right) $.

\subsection{$T$-product, tubal rank, t-SVD and TNN}
Discrete Fourier Transformation
(DFT) plays a key role in tensor-tensor product (t-product). For $\A\in \mathbb{R}^{n_1 \times n_2 \times n_3}$, let ${{\bar \A}} \in {{\mathbb C}^{{n_1} \times {n_2} \times {n_3}}}$ be the result of
Discrete Fourier transformation (DFT) of ${{ \A}} \in {{\mathbb R}^{{n_1} \times {n_2} \times {n_3}}}$
along the 3rd dimension. Specifically, let  $F_{n_3}=[f_1,\dots, f_{n_3}]\in \mathbb C^{n_3\times n_3}$, where
$$f_i=\left[ \omega^{0\times (i-1)}; \omega^{1\times (i-1)};\dots; \omega^{(n_3-1)\times (i-1)}\right] \in \mathbb C^{n_3}$$
with $\omega=e^{-\frac{2\pi \mathfrak{b}}{n_3}}$ and $\mathfrak{b}=\sqrt{-1}$. Then $ \bar \A(i,j,:)=F_{n_3}\A(i,j,:) $,
which can be computed by Matlab command ``$\bar \A=fft(\A,[\; ],3)$''. Furthermore, $\A$ can be computed by $\bar \A$ with the inverse DFT $ \A=ifft({\bar \A},[\; ],3) $.

\begin{lemma}\label{lem:v}\cite{RR04}
	Given any real vector $\bm{v} \in \mathbb{R}^{n_3}$, the associated $\bar{\bm{v}}=F_{n_3} \bm{v} \in \mathbb{C}^{n_3}$ satisfies
	$$
	\bar{v}_{1} \in \mathbb{R} \text { and } \operatorname{conj} \left(\bar{v}_{i}\right)=\bar{v}_{n_3-i+2},\; i=2, \ldots,\left\lfloor\frac{n_3+1}{2}\right\rfloor.
	$$
\end{lemma}
By using Lemma \ref{lem:v}, the frontal slices of $ \bar\A $ have the following properties:
\begin{equation}\label{conj}
	\left\{\begin{array}{l}
		\bar{{A}}^{(1)} \in \mathbb{R}^{n_{1} \times n_{2}}, \\
		\operatorname{conj} \left(\bar{{A}}^{(i)}\right)=\bar{{A}}^{\left(n_{3}-i+2\right)},\; i=2, \ldots,\left\lfloor\frac{n_{3}+1}{2}\right\rfloor.
	\end{array}\right.	
\end{equation}
For $\A\in \mathbb{R}^{n_1\times n_2\times n_3}$, we define matrix ${{\bar A}} \in {{\mathbb C}^{{n_1}{n_3} \times {n_2}{n_3}}}$ as
\begin{equation}\label{bdiag}
	{{\bar A}} = bdia{g}(\bar {{\mathcal{A}}} ) \hfill \\
	= \left[ {\begin{array}{*{20}{c}}
			{\bar A^{(1)}}&{}&{}&{} \\
			{}&{\bar A^{(2)}}&{}&{} \\
			{}&{}& \ddots &{} \\
			{}&{}&{}&{\bar A^{({n_3})}}
	\end{array}} \right]. \end{equation}
Here, $ bdiag(\cdot) $ is an operator which maps the tensor $ {{\bar {\mathcal A}}} $ to the block diagonal matrix $ \bar A$. The block circulant matrix $bcirc({\A}) \in {{\mathbb R}^{{n_1}{n_3} \times {n_2}{n_3}}}$ of $\A$ is defined  as
$$bcirc({\A}) = \left[ {\begin{array}{*{20}{c}}
		{A^{(1)}}&{A^{(n_3)}}& \cdots &{A^{(2)}}\\
		{A^{(2)}}&{A^{(1)}}& \cdots &{A^{(3)}}\\
		\vdots & \vdots & \ddots & \vdots \\
		{A^{(n_3)}}&{A^{({n_3} - 1)}}& \cdots &{A^{(1)}}
\end{array}} \right].$$
Based on these notations, the $T$-product is presented as follows.
\begin{definition}\label{def:T-pro}\textbf{(T-product)} \cite{KM11}
	For $\A\in \mathbb{R}^{n_1\times r\times n_3}$ and $\B\in \mathbb R^{r\times n_2\times n_3}$, define
	$$\A\ast\B:=fold\left(bcirc(\A)\ \cdot unfold(\B)\right) \in \mathbb{R}^{n_1\times n_2\times n_3}.$$
	Here
	$$ unfold(\B) = \left[B^{(1)};B^{(2)}; \ldots ;B^{(n_3)}\right],$$
	and its inverse operator ``fold" is defined by $$fold(unfold(\B)) = \B.$$
\end{definition}

We will now present the definition of tubal rank. Before then, we need to introduce some other concepts. 
\begin{definition}\textbf{(F-diagonal tensor)} \cite{KM11}
	If each of a tensor's frontal slices is a diagonal matrix, the tensor is denoted $ f $-diagonal.
\end{definition}
\begin{definition}\textbf{(Conjugate transpose)} \cite{KM11}
	The conjugate transpose of a tensor $\mathcal{A} \in \mathbb{R}^{n_{1} \times n_{2} \times n_{3}}$, denoted as $\mathcal{A}^{*}$, is the tensor obtained by conjugate transposing each of the frontal slices and then reversing the order of transposed frontal slices 2 through $n_{3}$.
\end{definition}
\begin{definition}\textbf{(Identity tensor)} \cite{KM11}
	The identity tensor $ \mathcal{I} \in \mathbb{R}^{n \times n \times n_{3}} $ is a tensor with the identity matrix as its first frontal slice and all other frontal slices being zeros.
\end{definition}
\begin{definition}\textbf{(Tensor inverse)} \cite{KM11}
	A tensor $ \A\in \mathbb{R}^{n \times n \times n_{3}} $ has an inverse $ \B $ provided that
	$ \A*\B=\mathcal{I} $ and $ \B*\A=\mathcal{I} $.
\end{definition}
\begin{definition}\textbf{(Orthogonal tensor)} \cite{KM11}
	A tensor $\mathcal{P} \in$ $\mathbb{R}^{n \times n \times n_{3}}$ is orthogonal if it fulfills the condition $\mathcal{P}^{*} * \mathcal{P}=\mathcal{P} * \mathcal{P}^{*}=\mathcal{I}.$
\end{definition}
\begin{theorem}\textbf{(T-SVD)} \cite{KM11}
	A tensor $\mathcal{A} \in \mathbb{R}^{n_{1} \times n_{2} \times n_{3}}$ can be factored as
	$$
	\mathcal{A}=\mathcal{U} * \mathcal{S} * \mathcal{V}^{*},
	$$
	where $\mathcal{U} \in \mathbb{R}^{n_{1} \times n_{1} \times n_{3}}$ and $\mathcal{V} \in \mathbb{R}^{n_{2} \times n_{2} \times n_{3}}$ are orthogonal tensors, and $\mathcal{S} \in \mathbb{R}^{n_{1} \times n_{2} \times n_{3}}$ is a $ f $-diagonal tensor.
\end{theorem}

Tensor multi-rank, tubal rank and tubal nuclear norm are now introduced.
\begin{definition}\label{def:tubal rank}{\bfseries (Tensor multi-rank and tubal rank)} \cite{KBHH13}
	For tensor $\A \in {{\mathbb R}^{{n_1} \times {n_2} \times {n_3}}}$, let $r_k=\rank\left(\bar A^{(k)}\right)$ for all $k\in {[n_3]}$.
	Then multi-rank of $\A$ is defined as $\rank_{m}(\A)=(r_1,\ldots,r_{n_3})$. The tensor tubal
	rank is defined as $ \rank_t(\A)=\max\left\lbrace r_k|k\in[n_3]\right\rbrace  $.
\end{definition}

\begin{definition}\label{def:nuclear norm}{\bfseries (Tubal nuclear norm)} \cite{LFC20} 
	The tubal nuclear norm $\|\mathcal{A}\|_{*}$ of a tensor $\mathcal{A} \in \mathbb{R}^{n_{1} \times n_{2} \times n_{3}}$ is defined as the sum of the singular values of all frontal slices of $\bar{\mathcal{A}}$, i.e., $\|\mathcal{A}\|_{*}=\frac{1}{n_{3}} \sum_{i=1}^{n_{3}}\left\|\bar{{A}}^{(i)}\right\|_{*}$.
\end{definition}

\section{Triple Nonconvex Nonsmooth Rank Minimization}

\subsection{Problem Formulation}
Matrix rank minimization problem can be expressed as
\begin{equation}\label{matrix}
	\min\limits_{X\in \mathbb{R}^{n_1 \times n_2}}\lambda\operatorname{rank}(X)+g(X),
\end{equation}
where $ \lambda $ is a positive parameter and $ g: \mathbb{R}^{n_1 \times n_2} \rightarrow[0,+\infty) $ is a differential loss function, which may be nonconvex.
Due to the discontinuous and nonconvex nature of the rank function, the above problem is generally NP-hard. 
Numerous relaxation functions, including convex nuclear norms \cite{PLYT18} and nonconvex relaxations \cite{YZ22,GYZC22,ZQZ20,SWKT19}, have been widely used to replace the rank function of the matrix. In \cite{ZGQ19}, the double singular values function, a continuous relaxation of the rank function, was adopted in matrix rank minimization problem with some advantages. The weighted singular values function of matrix $ X\in \mathbb{R}^{n_1 \times n_2} $ is defined as
$$
\rho_\beta(\sigma(X))=\sum_{i=1}^{r}\beta_i\rho\left(\sigma_{i}\right), \quad r=\min \left\lbrace n_1, n_2\right\rbrace,
$$
where $\rho(\cdot): \mathbb{R}^{+} \rightarrow \mathbb{R}^{+}$ is a differentiable concave function on $[0,+\infty]$, $ \beta=\left(\beta_1,\beta_2,\ldots,\beta_r\right)  $ is a weighting vector with $ 0\le\beta_1\le\beta_2\le\ldots\le\beta_r $, and $\sigma(X)=\left(\sigma_{1}, \sigma_{2}, \ldots, \sigma_{r}\right)$ is a singular values vector with $\sigma_{1} \geq \sigma_{2} \geq \ldots \geq \sigma_{r} \geq 0$.

Motivated by these, we consider the double weighted singular values function of tensor $ \X\in \mathbb{R}^{n_1 \times n_2 \times n_3} $:
\begin{equation}\label{eq:norm}
\left\|\X\right\|_{\rho_{\alpha}^\beta}=\sum_{k=1}^{n_3}\alpha_k\rho_\beta\left(\sigma^k\right)=\sum_{k=1}^{n_3}\sum_{i=1}^{r}\alpha_k\beta_i^k\rho\left(\sigma_i^k\right),	
\end{equation}
where $ \alpha_k>0 $ for $ k \in [n_3] $, $ \sigma^k=\sigma\left(\bar X^{(k)}\right) $ and $ \sigma_i^k=\sigma_i\left(\bar X^{(k)} \right)  $.
When acting on each singular value of the low rank tensor, $ \left\|\X\right\|_{\rho_{\alpha}^\beta} $ is a flexible rank relaxation function with different choices of $ \alpha $, $ \beta $ and $ \rho\left(\cdot\right) $, see Table \ref{tab:weighted}. 

\begin{table*}[htbp]
\centering
\caption{Flexible rank relaxation function with different choices of $ \alpha $, $ \beta $ and $ \rho\left(\cdot\right) $.}\label{tab:weighted}

\begin{tabular}{c|c|c|c}
	\hline
$ \alpha $	& $ \beta $ & $ \rho\left(\cdot\right) $ & $ \left\|\X\right\|_{\rho_{\alpha}^\beta} $ \\
	\hline
$ \alpha_k=1/n_3 $	& $ \beta_i^k=1 $ &$ \rho\left(\sigma_i^k\right)=\sigma_i^k $ & TNN \cite{LFC20} \\
	\hline
$ \alpha_k =1  $	& $ \beta_i^k=\left\{ \begin{array}{l}
	0,\;i = 1, \ldots ,N,\\
	1,\;i = N + 1, \ldots ,r.
\end{array} \right. $ &$ \rho\left(\sigma_i^k\right)=\sigma_i^k $ & PSTNN \cite{JHZD20} \\
	\hline
$ \alpha_k=1/n_3 $	& $\beta_i^k=\frac{1}{\sigma_i^k+\varepsilon}$ & $ \rho\left(\sigma_i^k\right)=\sigma_i^k $ & Weighted Tensor Nuclear Norm \cite{SYL18} \\
	\hline
$ \alpha_k=\left\{ \begin{array}{l}
	1,\;k = 1,\\
	0,\;i = 2, \ldots ,r.
\end{array} \right.$	& $ \beta_i^k=\left\{ \begin{array}{l}
0,\;i = 1, \ldots ,N,\\
1,\;i = N + 1, \ldots ,r.
\end{array} \right. $ & $ \rho\left(\sigma_i^k\right)=\sigma_i^k $ & Tensor Truncated Nuclear Norm \cite{XQLJ18}\\
	\hline
$ \alpha_k=1/n_3 $	&    $\beta_i^k=\frac{c}{\sqrt[1/2p]{\max\left(0,(\sigma_i^k)^2-(\sigma_r^k)^2\right)}+\varepsilon}$ & $ \rho\left(\sigma_i^k\right)=\left(\sigma_i^k\right)^p $ & Weighted t-Schatten-$ p $ Norm \cite{LZT20} \\
	\hline
\end{tabular}
\end{table*}
\iffalse
For example, $ \left\|\X\right\|_{\rho_{\alpha}^\beta} $ becomes tubal nuclear norm \cite{LFC20} when $ \alpha_k\beta_i^k=\frac{1}{n_3} $ and $ \rho\left(\cdot\right) $ is the absolute function.
\fi

\subsection{TNNR Minimization Problem}
Based on problem \eqref{matrix} and \eqref{eq:norm}, we introduce a general rank relaxation minimization problem:
\begin{equation}\label{tensor}
	\min\limits_{\X\in \mathbb{R}^{n_1 \times n_2\times n_3}}F_\alpha^\beta\left(\X\right):=\lambda\left\|\X\right\|_{\rho_{\alpha}^\beta}+f(\X),
\end{equation}
where $ f: \mathbb{R}^{n_1 \times n_2 \times n_3} \rightarrow[0,+\infty) $ is a differential loss function, which may be nonconvex. Combining problem \eqref{tensor},
reweighted strategies \cite{XGL16,LTYL14} and supergradient concepts \cite{Bor01}, the following TNNR minimization problem comes into existence:
\begin{equation}\label{Q:TNNR}
	\min\limits_{\X}F\left(\X\right):= \lambda\sum_{k=1}^{n_3}\rho_2\left(\sum_{i=1}^{r} \rho_1\left(\rho\left(\sigma_i^k\right)\right)\right)+f(\X).
\end{equation}
Without loss of generality, we choose $ \rho\left(\cdot\right)=\rho_1\left(\cdot\right)=\rho_2\left(\cdot\right) $. 
Next, we present the relationship between \eqref{tensor} and \eqref{Q:TNNR}. Without specific explanation, Assumption \ref{ass:rho} is assumed throughout the paper.
\begin{assumption}\label{ass:rho}
	The penalty function $\rho(\cdot): \mathbb{R}^{+} \rightarrow \mathbb{R}^{+}$ is a differentiable concave function with a $ L_g $-Lipschitz continuous gradient, i.e., for any $ t,\,s $, $$ \left|\rho^{\prime}\left(s\right)-\rho^{\prime}\left(t\right)\right| \leq L_{g}\left|s-t\right|, $$
	and $ \rho^{\prime}\left(t\right)>0 $ for any $ t\in\left[ 0,+\infty\right) $.
\end{assumption}

\begin{lemma}\cite{ZGQ19}
If the function $ \rho(\cdot) $ satisfies Assumption \ref{ass:rho}, then for any $ t $ and $ s $, we have:
\begin{itemize}\label{lem:gradient}
	\item $\rho(t) \leq \rho(s)+\rho'(s)(t-s)$;
	\item $\rho\left(\rho\left(t\right)\right) \leq \rho\left(\rho\left(s\right)\right)+ \rho'\left(\rho\left(s\right)\right)\left(\rho\left(t\right)-\rho\left(s\right)\right).$
\end{itemize}
\end{lemma}

From Lemma \ref{lem:gradient}, we know that when we set $ \alpha_k =  \rho'\left(\sum_{i=1}^{r} \rho\left(\rho\left(\sigma_i^k\right)\right)\right) $ and $ \beta_i^k=\rho'\left(\rho\left(\sigma_i^k\right)\right) $ in $ \left\|\X\right\|_{\rho_{\alpha}^\beta} $, we have
\begin{equation*}
	\min F_\alpha^\beta\left(\X\right) \ge \min F\left(\X\right).
\end{equation*}
From the above formula, we establish the relationship between problem \eqref{tensor} and problem \eqref{Q:TNNR}.

\begin{lemma}\cite{ZGQ19}\label{lem:SVF}
If $\sigma_{1}^k \geq \sigma_{2}^k \geq \ldots \geq \sigma_{r}^k \geq 0$ for $ k \in [n_3] $, then we have
$$
0 \leq \beta_{1}^k \leq \beta_{2}^k \leq \ldots \leq \beta_{r}^k.
$$
\end{lemma}
According to Lemma \ref{lem:SVF}, the insignificant singular values function have larger weights, otherwise inverse. Therefore, $ \left\|\X\right\|_{\rho_{\alpha}^\beta} $ can be made to relax the tensor multi-tubal rank\footnote{Similar conclusions can be extended to Tucker rank and tensor train (TT) rank in parallel.}.

\subsection{General Inertial Smoothing Proximal Gradient Algorithm}
Motivated by the efficiency of inertial method, we consider the following general inertial method of $ F_\alpha^\beta\left(\cdot\right) $, i.e.,
\begin{equation}\label{X}
	\left\{\begin{array}{l}
		\Y^{t}=\X^{t}+\theta_{1}^t\left(\X^{t}-\X^{t-1}\right), \\
		\Z^{t}=\X^{t}+\theta_2^{t}\left(\X^{t}-\X^{t-1}\right), \\
		\X^{t+1}=\mathop{\operatorname{argmin}}\limits_{\X} Q\left(\X,\Y^{t},\Z^{t}\right), 
	\end{array}\right.
\end{equation}
where $ Q\left(\X,\Y^{t},\Z^{t}\right) =\lambda\left\|\X\right\|_{\rho_{\alpha^t}^{\beta^t}}+\left\langle \X-\Y^{t}, \nabla f\left(\Z^{t}\right)\right\rangle+\frac{\mu^t}{2}\left\|\X-\Y^t\right\|^{2} $, $ \theta_1^t $, $ \theta_2^t $, $ \mu^t $ are different parameters and certain conditions are required for them.
To solve \eqref{X}, we introduce the following Theorem.
\begin{theorem}\label{thm:tSVD}
	Let $\rho(\cdot): \mathbb{R}^{+} \rightarrow \mathbb{R}^{+}$ be a function such that the proximal operator denoted by $\operatorname{Prox}_{\rho}(\cdot)$ is monotone. For any $\eta>0$, let $\Y=\U*{\mathcal S}*\V^*$ be the t-SVD of $\Y\in\mathbb{R}^{n_1 \times n_2\times n_3}$, all weighting values satisfy $ 0 \leq \beta_{1}^k \leq \beta_{2}^k \leq \ldots \leq \beta_{r}^k $ and $ \alpha_k>0 $ for $ k \in [n_3] $. Then, a minimizer to 
	\begin{equation*}
		\mathop{\operatorname{argmin}}\limits_{\X} \eta\left\|\X\right\|_{\rho_{\alpha}^\beta}+\frac{1}{2}\left\|\X-\Y\right\|^{2}
	\end{equation*}
is given by	
	$$ \X^\star=\U*\mathcal{S^\star} *\V^*, $$
where $ \mathcal{S^\star} $ satisfies $\bar\delta_{i,k}^\star=\mathcal{\bar S^\star}(i,i,k) \geq \mathcal{\bar S^\star}(j,j,k)=\bar\delta_{j,k}^\star$ for $1 \leq i \leq j \leq r,\,k \in [n_3]$, and then $ \bar\delta_{i,k}^\star $ is obtained by solving the problem as follows:
\begin{equation*}
	\bar\delta_{i,k}^\star \in \operatorname{Prox}_{\rho}\left(\sigma_i^k\right)=\mathop{\operatorname{argmin}}\limits_{\bar\delta_{i,k}\geq 0} \frac{\eta\alpha_k\beta_i^k }{n_3}\rho\left(\bar\delta_{i,k}\right)+\frac{1}{2}\left(\bar\delta_{i,k}-\sigma_{i}^k\right)^{2},
\end{equation*}
where $ \sigma_{i}^k=\mathcal{\bar S}(i,i,k) $.
\end{theorem}
Theorem \ref{thm:tSVD} can be easily proved by \cite[Proposition 1]{ZGQ19}. Thus, it is not repeated here.
\iffalse
Especially, according to Theorem 5, we can easily obtain that the mode-k LogTNN-based t-SVT can make larger singular values shrunk less than small singular values. This implies that 3DlogTNN is able to better preserve the major information of the target HSI.
\fi

After updating $ \X^{t+1} $, we need to compute the weighting $ \alpha_k^{t+1} $ and $ \beta_i^{k,t+1} $ by
\begin{equation}\label{weight}
	\alpha_k^{t+1} =  \rho'\left(\sum_{i=1}^{r} \rho\left(\rho\left(\sigma_i^{k,t+1}\right)\right)\right),\,
	\beta_i^{k,t+1}= \rho'\left(\rho\left(\sigma_i^{k,t+1}\right)\right), 
\end{equation}
where $ \sigma_i^{k,t+1}=\sigma_i\left(\bar X^{(k,t+1)}\right)  $ and $ \bar X^{(k,t+1)}= \left(\bar X^{t+1}\right)^{(k)}$.

The general inertial smoothing proximal gradient method for solving problem \eqref{Q:TNNR} is presented as follows.

\begin{algorithm}
	\caption{Solving problem \eqref{Q:TNNR} by TNNR}
	\label{alg}
	\begin{algorithmic}[]
		\REQUIRE Choosing the parameters $ \theta_1^t $, $ \theta_2^t $, $ \mu^t $. \\
	\!\!\!\!\!\!\!\!\!\! \textbf{Initialize:} $t=0,\,\X^{t}$, and $ \alpha_k^t,\,\beta_i^{k,t} $ for $ k\in[n_3],\,i\in[r] $.
		\WHILE{not converge}
		\STATE \textbf{Step~ 1.} Update $ \X^{t+1} $ by \eqref{X}.\\
		\STATE \textbf{Step~ 2.} Update $ \alpha_k^{t+1},\,\beta_i^{k,t+1} $ by \eqref{weight}.\\
		\STATE Let $ t:=t+1 $ and go to \textbf{Step~ 1}.
		\ENDWHILE
		\ENSURE $\X^{t+1}$.
	\end{algorithmic}
\end{algorithm}

\subsection{Convergence Analysis}
In this subsection, we shall prove the convergence for the proposed TNNR.
The following assumptions are needed to analyze the convergence of the proposed model:
\begin{assumption}\label{ass:smooth}
The loss function $f(\cdot)$ is continuously differentiable with the Lipschitz continuous gradient $\nabla f(\cdot)$, i.e., there exists a Lipschitz constant $L_{f}>0$ for any $\X_{1}, \X_{2} \in \mathbb{R}^{n_1\times n_2\times n_3 }$, such as
$$
\left\|\nabla f\left(\X_{1}\right)-\nabla f\left(\X_{2}\right)\right\| \leq L_{f}\left\|\X_{1}-\X_{2}\right\| .
$$
\end{assumption}
\begin{assumption}\label{ass:bounded}
$F(\cdot)$ is coercive and bounded from below, that is,
$$
\lim _{\|\X\| \rightarrow+\infty} F(\X)=+\infty \text { and } \liminf _{\|\X\| \rightarrow+\infty} F(\X)>-\infty.
$$
\end{assumption}
\begin{assumption}\label{ass:pra}
	The parameters $ \theta_1^t $, $ \theta_2^t $, $ \mu^t $ in the TNNR algorithm satisfy the following conditions: for any $0<\varepsilon \ll 1$, $\theta_1^t \in \left[0,(1-\varepsilon)/2\right)$, $\theta_2^t \in \left[0,1/2\right]$, $ \mu^t $ is nonincreasing and satisfies 
	$$
	\mu^{0} \ge \mu^{t} \ge \max \left\{\frac{\theta_2^tL_f}{\theta_1^t}, \frac{\left(1-\theta_2^t\right) L_f}{1-\theta_1^t-\theta_1^{t+1}-\varepsilon}\right\}.
	$$	
\end{assumption}
\begin{remark}
It is easy to prove for any $t \in \mathbb{N}$, $\mu^{t} \ge L_f$. Otherwise, we have
$$
\left\{\begin{array}{l}
	\frac{\theta_2^tL_f}{\theta_1^t}<L_f, \\
	\frac{\left(1-\theta_2^t\right) L_f}{1-\theta_1^t-\theta_1^{t+1}-\varepsilon}<L_f,
\end{array} \Rightarrow \theta_1^t+\theta_1^{t+1}+\varepsilon<\mu^{t}<\theta_1^t.\right.
$$
There is no $\mu^{t}$ satisfying the above inequalities. Then, we get $\mu^{t} \ge L_f$.	
\end{remark}

Now we establish Theorem \ref{thm:CC} to prove that Algorithm \ref{alg} can converge to a stationary point of our optimization problem.
\begin{theorem}\label{thm:CC}
	Suppose that Assumption \ref{ass:rho}-\ref{ass:pra} holds, and let $ \left\lbrace \X^t \right\rbrace $  be the sequence generated by TNNR. Then, we have that
	\begin{itemize}
		\item[(i)]	
		the sequence $\left\{\X^{t}\right\}$ is bounded, and has at least one accumulation point, i.e., there exists at least a tensor $\X^\star$ and a subsequence $\left\{\X^{t_{l}}\right\} \subseteq\left\{\X^{t}\right\}$ such that $\lim\limits_{l\rightarrow+\infty} \X^{t_{l}}=\X^\star$;
		\item[(ii)] the sequence $ \left\lbrace F\left(\X^{t}\right)\right\rbrace $ is convergent;
		\item[(iii)] any accumulation point $ \X^\star=\lim_{l\rightarrow \infty}\X^{t_{l}} $	of $\left\lbrace\X^{t}\right\rbrace_{t \in \mathbb{N}}$ is a stationary point of $ F\left(\X\right) $ and it further holds that
		$$ F\left(\X^{t_{l}+1}\right) \rightarrow F\left(\X^\star\right), \quad \text {as}\; l \rightarrow +\infty.
		$$
	\end{itemize}
\end{theorem}

The conclusions in Theorem \ref{thm:CC} just establish the convergence of the subsequence of $\left\{\X^{t}\right\}$. With some slightly additional assumptions, we also can prove the convergence the whole sequence $\left\{\X^{t}\right\}$. Now we assume that after a finite number of steps the sequence $\left\{\delta_{t}=\mu^{t}\theta_1^{t}/2\right\}$ is constant. % and consider the convergence of the sequence $\left\{\X^{t}\right\}$ starting from this iteration, i.e., we denote this iteration as $\X^{0}$. 
Below we present the convergence results for the whole sequence $ \left\lbrace \X^t \right\rbrace  $ generated by TNNR.
\begin{theorem}\label{thm:whole}
	Suppose that Assumption \ref{ass:rho}-\ref{ass:pra} holds, $ \delta_{t}\equiv \delta $ for all $ t\in \N $ and the function
		\begin{equation*}
		 H(\X, \Y):=F(\X)+\delta\|\X-\Y\|^{2}
	\end{equation*}
	is a KL function, and let $\left\lbrace \X^t \right\rbrace$ be the sequence generated by TNNR. Then, the sequence $\left\lbrace \X^t \right\rbrace$ has finite length, i.e., $\sum_{t=0}^{\infty}\left\|\X^{t+1}-\X^{t}\right\|<+\infty$, and $\left\lbrace \X^t \right\rbrace$ globally converges to a critical point $ \X^\star $ of the minimization problem \eqref{Q:TNNR}.
\end{theorem}

%-------------------------------------------------------------------------------------------

\section{Numerical Experiments}

In this section, we conduct some experiments on both synthetic and real world data to compare the performance of TNNR to show their validity. In particular, we apply it to solve the problem \eqref{Q:TNNR} with $ f\left(\X\right)=\left\|P_\Omega\left(\X-\M\right) \right\|^2   $, that is,
\begin{equation}
	\min\limits_{\X} \lambda\sum_{k=1}^{n_3}\rho\left(\sum_{i=1}^{r} \rho\left(\rho\left(\sigma_i^k\right)\right)\right)+\left\|P_\Omega\left(\X-\M\right) \right\|^2,
\end{equation}
where $ \rho(x)=x^{2/3} $, $ \M $ is a known tensor, $ \Omega $ is an index set which locates the observed data, $ P_\Omega $ is a linear operator that extracts the entries in $ \Omega $ and fills the entries not in $ \Omega $ with zeros.

We employ the peak signal-to-noise rate (PSNR) \cite{WBSS04}, the structural similarity (SSIM) \cite{WBSS04}, the feature similarity (FSIM) \cite{ZZMZ11} and the recovery computation time to measure the quality of the recovered results. 
We compare TNNR for the tensor completion problem with four existing methods, including TNN \cite{ZEA14}, PSTNN \cite{JHZD20}, T-TNN \cite{XQLJ18} and HaLRTC \cite{LMWY13}. 
All methods are implemented on the platform of Windows 11 and Matlab (R2020b) with an Intel(R) Core(TM) i5-12500H CPU at 2.50GHz and 16 GB RAM. The parameter choice of the compared methods rely on the authors' suggestions of the published papers or the default parameters of the released codes to obtain the best performance.

%In all experiments, the termination precision is set to be $ 1e\text{-}4 $ and the maximum iteration steps is set to be $ 150 $.
%\textbf{Experimental Settings:} If not specified, regularization parameter $ \lambda_A $ is set as $ 0.04 $, $ \lambda_H $ and $ \lambda_V $ are set as $ 0.05 $, $ \lambda_D $ is set as $ 0.03 $. $ \gamma $ is set as $ 1e\text{-}4 $ and $ \theta_2 $ is set as $ 1 $.
\subsection{Synthetic Data}
In this subsection, we aim to recover a random tensor $ \M\in\R^{n_1\times n_2 \times n_3} $ of tubal rank $ r $ from known entries $\left\{\M_{ijk}\right\}_{(i,j,k) \in \Omega}$. In detail, we first use Matlab command $ randn(n_1,r,n_3) $ and $ randn(r,n_2, n_3) $ to produce two tensors $ \A\in\R^{n_1\times r \times n_3} $ and $ \B\in\R^{r\times n_2 \times n_3} $. Then, let $\mathcal{M}=\mathcal{A} * \mathcal{B}$. Finally, we sample a subset with sampling ratio $ SR $ uniformly at random, where $ SR=\left|\Omega\right|/\left(n_1n_2n_3\right)  $. In our experiment, we set $ r= 10,\,n_1=n_2=n_3=100 $, $ SR=0.8 $. When $ \left\|\X^\star-\M\right\|/\left\|\M\right\|<1e-3  $, the iteration process terminates.
For each simulation, the result is obtained via 100 Monte Carlo runs with different realizations of $ \M $ and $ \Omega $.

To demonstrate the effectiveness of the TNNR with various extrapolations, we test the performance of the method with various choices of $ \theta_1 $ and $ \theta_2 $.
 For each instance, let $ \lambda=5 $, $ \mu $ be the minimal stepsize, i.e., $ \mu=\max\left\lbrace \theta_2/\theta_1, \left(1-\theta_2\right)/\left(0.99-2\theta_1\right)\right\rbrace  $  by the condition showed in Algorithm \ref{alg}. Accordingly, we test six scenarios with $\theta_1,\,\theta_2\in\{0,0.1,0.2,0.3,0.4,0.49\}$. We list the results of the CPU computing time in Table \ref{tab:extrapolation}.
 From the results, we find that almost all instances with inertia are more effective than the original method.

% Table generated by Excel2LaTeX from sheet 'Synthetic'
\begin{table}[htbp]
	\centering
	\caption{NUMERICAL RESULTS FOR SYNTHETIC DATA WITH TWO DIFFERENT EXTRAPOLATIONS.}\setlength{\tabcolsep}{1mm}{
	\begin{tabular}{ccccccc}
		\hline
		Time (s) & $\theta_1=0$ & $\theta_1=0.1$ & $\theta_1=0.2$ & $\theta_1=0.3$ & $\theta_1=0.4$ & $\theta_1=0.49$ \\
		\hline
		$\theta_2=0$ & 20.21  & 21.26  & 21.48  & 21.87  & 22.30  & 22.39  \\
		\hline
		$\theta_2=0.1$ & \textbf{19.95 } & \textbf{20.15 } & 20.34  & 20.35  & 20.55  & 20.58  \\
		\hline
		$\theta_2=0.2$ & \textbf{18.15 } & \textbf{18.64 } & \textbf{18.98 } & \textbf{18.82 } & \textbf{18.63 } & \textbf{18.86 } \\
		\hline
		$\theta_2=0.3$ & \textbf{16.37 } & \textbf{16.42 } & \textbf{16.43 } & \textbf{16.31 } & \textbf{16.53 } & \textbf{16.66 } \\
		\hline
		$\theta_2=0.4$ & \textbf{14.35 } & \textbf{14.43 } & \textbf{14.13 } & \textbf{14.14 } & \textbf{14.04 } & \textbf{14.00 } \\
		\hline
		$\theta_2=0.49$ & \textbf{12.83 } & \textbf{12.82 } & \textbf{12.46 } & \textbf{12.38 } & \textbf{12.35 } & \textbf{12.12 } \\
		\hline
	\end{tabular}}%
	\label{tab:extrapolation}%
\end{table}%

Therefore, in the following experiments we set $ \theta_1=0.49 $, $ \theta_2=0.49 $ and $ \mu=\max\left\lbrace \theta_2/\theta_1, \left(1-\theta_2\right)/\left(0.99-2\theta_1\right)\right\rbrace  $ in TNNR.

\subsection{Color Image Inpainting}
In this subsection, we use the Berkeley Segmentation database\footnote{https://www2.eecs.berkeley.edu/Research/Projects/CS/vision/bsds/.} of size $ 321\times 481 \times 3 $ to evaluate our proposed method TNNR for color image inpainting. In our test, four images are randomly selected from this database, including ``Airplane", ``Tiger", ``Flower" and ``Fish". 
The data of images are normalized in the range $ \left[0,1\right] $. %The sampling rates (SR) are set as $ 40\% $, $ 50\% $ and $ 60\% $. 
%We set $ \theta_1^t=0.4 $, $ \theta_2^t=0 $ and $ \mu=\max\left\lbrace \theta_2/\theta_1, \left(1-\theta_2\right)/\left(0.99-2\theta_1\right)\right\rbrace  $ in TNNR.

Table \ref{tab:ColorImage} reports the results of quantitative metrics for four experiments using different methods when $ SR =[40\%, 50\%, 60\%] $. Several visual examples for missing rates $ SR =30\% $ are presented in Figure \ref{fig:colorimage}.
From these results, we can find that representative approaches based on the Tucker, i.e., HaLRTC, perform relatively poor in terms of recovery quality. Another finding is that the performance of TNN based convex is inferior to the PSTNN and T-TNN based nonconvex. More significantly, the proposed TNNR, which is induced by triple nonconvex nonsmooth, performs best among all the compared approaches whether in the PSNR, SSIM, FSIM values or visual quality. From time consumption, our method uses similar running time as HaLRTC and is the second fastest method. It is about two times faster than the third fastest method PSTNN and at least fifteen times faster than the slowest method T-TNN.

% Table generated by Excel2LaTeX from sheet 'Color Image'
\begin{table*}[htbp]
	\caption{COLOR IMAGE INPAINTING PERFORMANCE COMPARISON: PSNR, SSIM, FSIM AND RUNNING TIME. THE BEST AND THE SECOND BEST PERFORMING METHODS IN EACH IMAGE ARE HIGHLIGHTED IN RED AND BOLD, RESPECTIVELY}\label{tab:ColorImage}
	\begin{tabular}{c|c|cccc|cccc|cccc}
		\hline
		\multirow{2}{*}{Color Image} & \multirow{2}{*}{Methods} & \multicolumn{4}{c|}{$ SR=40\% $}    & \multicolumn{4}{c|}{$ SR=50\% $}    & \multicolumn{4}{c}{$ SR=60\% $} \\
		\cline{3-14}          &       & PSNR  & SSIM  & FSIM  & Time  & PSNR  & SSIM  & FSIM  & Time  & PSNR  & SSIM  & FSIM  & Time \\
			\hline
		\multirow{5}{*}{Airplane} & TNNR  & \textcolor[rgb]{ 1,  0,  0}{31.921 } & \textcolor[rgb]{ 1,  0,  0}{0.842 } & \textcolor[rgb]{ 1,  0,  0}{0.935 } & \textbf{5.095 } & \textcolor[rgb]{ 1,  0,  0}{35.616 } & \textcolor[rgb]{ 1,  0,  0}{0.925 } & \textcolor[rgb]{ 1,  0,  0}{0.969 } & \textcolor[rgb]{ 1,  0,  0}{4.621 } & \textcolor[rgb]{ 1,  0,  0}{39.905 } & \textcolor[rgb]{ 1,  0,  0}{0.971 } & \textcolor[rgb]{ 1,  0,  0}{0.988 } & \textbf{4.179 }\\
		& TNN   & 30.424  & 0.833  & 0.924  & 42.838  & 33.588  & 0.912  & 0.960  & 82.985  & 37.389  & 0.962  & 0.982  & 92.441  \\
		& PSTNN & 30.563  & 0.839  & 0.926  & 6.630  & 33.696  & 0.915  & 0.961  & 14.197  & 37.492  & 0.964  & 0.983  & 14.372  \\
		& T-TNN & \textbf{30.905 } & \textbf{0.841 } & \textbf{0.931 } & 94.604  & \textbf{34.147 } & \textbf{0.921 } & \textbf{0.965 } & 145.284  & \textbf{37.913 } & \textbf{0.967 } & \textbf{0.985 } & 142.214  \\
		& HaLRTC & 28.950  & 0.814  & 0.907  & \textcolor[rgb]{ 1,  0,  0}{2.296 } & 30.898  & 0.874  & 0.940  & \textbf{5.293 } & 33.199  & 0.922  & 0.964  & \textcolor[rgb]{ 1,  0,  0}{3.999 } \\
		\hline
        \multirow{5}{*}{Tiger} & TNNR  & \textcolor[rgb]{ 1,  0,  0}{29.676 } & \textcolor[rgb]{ 1,  0,  0}{0.860 } & \textcolor[rgb]{ 1,  0,  0}{0.939 } & \textbf{5.038 } & \textcolor[rgb]{ 1,  0,  0}{32.831 } & \textcolor[rgb]{ 1,  0,  0}{0.926 } & \textcolor[rgb]{ 1,  0,  0}{0.968 } & \textbf{4.921 } & \textcolor[rgb]{ 1,  0,  0}{36.670 } & \textcolor[rgb]{ 1,  0,  0}{0.966 } & \textcolor[rgb]{ 1,  0,  0}{0.986 } & \textbf{4.579 } \\
        & TNN   & 28.375  & 0.851  & 0.930  & 46.804  & 31.222  & 0.916  & 0.959  & 81.420  & 34.345  & 0.958  & 0.979  & 75.853  \\
        & PSTNN & 28.623  & 0.854  & 0.933  & 7.221  & 31.434  & 0.917  & 0.961  & 14.312  & 34.526  & 0.958  & 0.980  & 14.815  \\
        & T-TNN & \textbf{28.722 } & \textbf{0.857 } & \textbf{0.934 } & 200.596  & \textbf{31.695 } & \textbf{0.920 } & \textbf{0.963 } & 188.537  & \textbf{34.830 } & \textbf{0.960 } & \textbf{0.981 } & 115.059  \\
        & HaLRTC & 27.118  & 0.830  & 0.914  & \textcolor[rgb]{ 1,  0,  0}{2.898 } & 29.298  & 0.889  & 0.944  & \textcolor[rgb]{ 1,  0,  0}{4.418 } & 31.447  & 0.930  & 0.965  & \textcolor[rgb]{ 1,  0,  0}{2.270 } \\
        \hline
		\multirow{5}{*}{Flower} & TNNR  & \textcolor[rgb]{ 1,  0,  0}{31.630 } & \textcolor[rgb]{ 1,  0,  0}{0.860 } & \textcolor[rgb]{ 1,  0,  0}{0.956 } & \textbf{4.813 } & \textcolor[rgb]{ 1,  0,  0}{35.215 } & \textcolor[rgb]{ 1,  0,  0}{0.934 } & \textcolor[rgb]{ 1,  0,  0}{0.981 } & \textbf{4.526 } & \textcolor[rgb]{ 1,  0,  0}{38.832 } & \textcolor[rgb]{ 1,  0,  0}{0.971 } & \textcolor[rgb]{ 1,  0,  0}{0.992 } & \textcolor[rgb]{ 1,  0,  0}{4.141 } \\
		& TNN   & 29.599  & 0.832  & 0.947  & 42.082  & 32.728  & 0.909  & 0.972  & 42.080  & 35.995  & 0.955  & 0.986  & 85.822  \\
		& PSTNN & 30.054  & 0.840  & 0.950  & 6.520  & 33.084  & 0.914  & 0.973  & 6.691  & 36.279  & 0.957  & 0.987  & 14.754  \\
		& T-TNN & \textbf{30.436 } & \textbf{0.857 } & \textbf{0.954 } & 77.385  & \textbf{33.468 } & \textbf{0.923 } & \textbf{0.975 } & 62.725  & \textbf{36.675 } & \textbf{0.962 } & \textbf{0.988 } & 100.470  \\
		& HaLRTC & 29.228  & 0.851  & 0.948  & \textcolor[rgb]{ 1,  0,  0}{2.836 } & 31.691  & 0.909  & 0.969  & \textcolor[rgb]{ 1,  0,  0}{2.513 } & 34.107  & 0.946  & 0.982  & \textbf{5.343 } \\
		\hline
		\multirow{5}{*}{Fish} & TNNR  & \textcolor[rgb]{ 1,  0,  0}{33.563 } & \textcolor[rgb]{ 1,  0,  0}{0.919 } & \textcolor[rgb]{ 1,  0,  0}{0.969 } & \textbf{5.402 } & \textcolor[rgb]{ 1,  0,  0}{37.556 } & \textcolor[rgb]{ 1,  0,  0}{0.965 } & \textcolor[rgb]{ 1,  0,  0}{0.987 } & \textbf{4.932 } & \textcolor[rgb]{ 1,  0,  0}{41.498 } & \textcolor[rgb]{ 1,  0,  0}{0.985 } & \textcolor[rgb]{ 1,  0,  0}{0.995 } & \textbf{4.320 } \\
		& TNN   & 31.417  & 0.898  & 0.958  & 94.509  & 34.747  & 0.947  & 0.978  & 42.238  & 38.291  & 0.975  & 0.990  & 69.916  \\
		& PSTNN & 31.629  & 0.902  & 0.960  & 14.403  & 34.865  & 0.948  & 0.979  & 6.622  & 38.373  & 0.976  & 0.990  & 6.490  \\
		& T-TNN & \textbf{32.025 } & \textbf{0.908 } & \textbf{0.963 } & 92.739  & \textbf{35.288 } & \textbf{0.952 } & \textbf{0.982 } & 59.725  & \textbf{38.834 } & \textbf{0.978 } & \textbf{0.992 } & 59.962  \\
		& HaLRTC & 30.747  & 0.900  & 0.955  & \textcolor[rgb]{ 1,  0,  0}{2.523 } & 33.430  & 0.942  & 0.974  & \textcolor[rgb]{ 1,  0,  0}{2.405 } & 36.216  & 0.968  & 0.987  & \textcolor[rgb]{ 1,  0,  0}{2.180 }\\
		\hline
	\end{tabular}%
\end{table*}%

\begin{figure*}[htbp]
	\centering
	\begin{subfigure}[b]{1\linewidth}
		\begin{subfigure}[b]{0.138\linewidth}
			\centering
			\includegraphics[width=\linewidth]{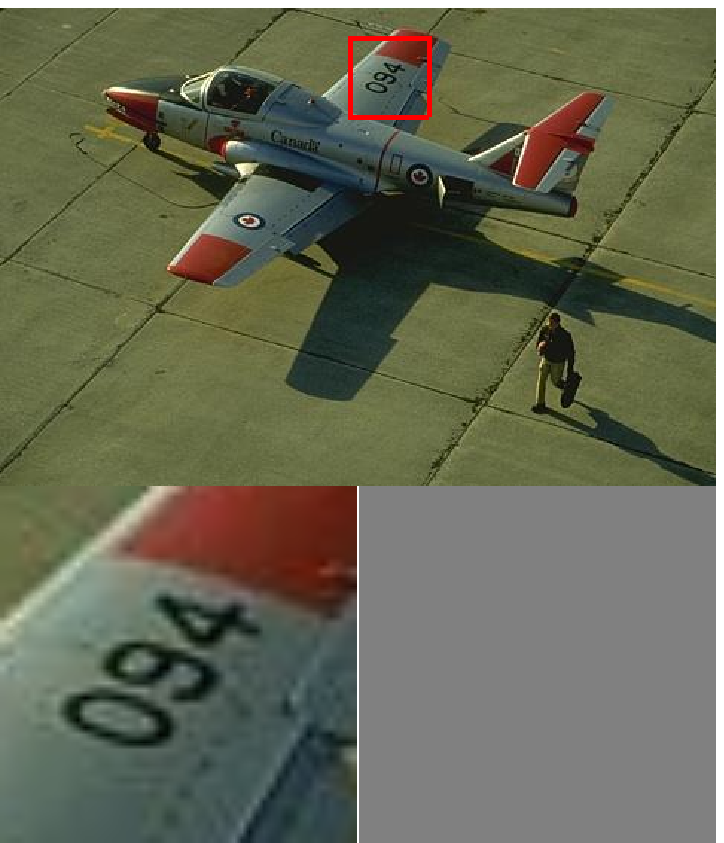}\vspace{0pt}
			\includegraphics[width=\linewidth]{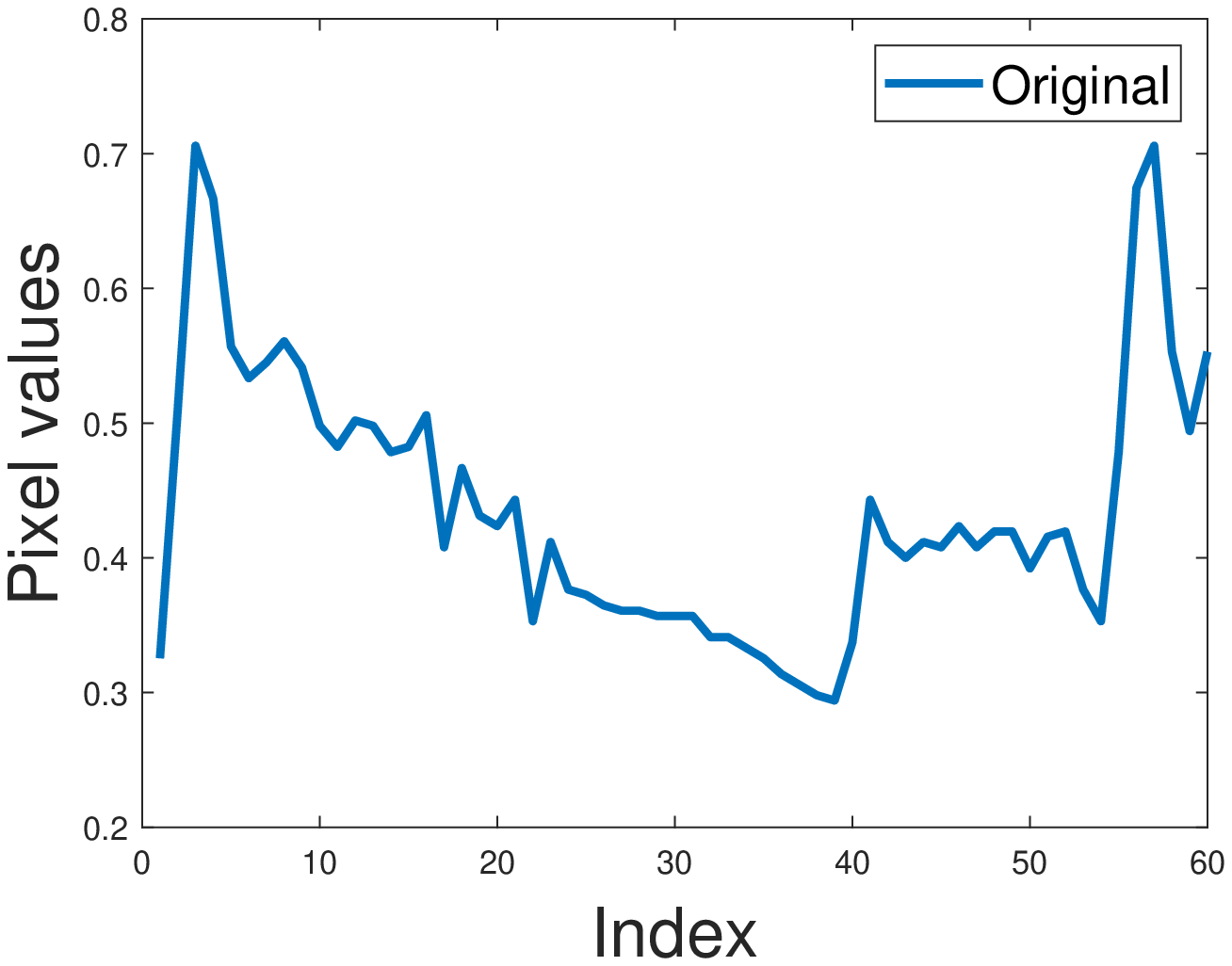}\vspace{0pt}
			\includegraphics[width=\linewidth]{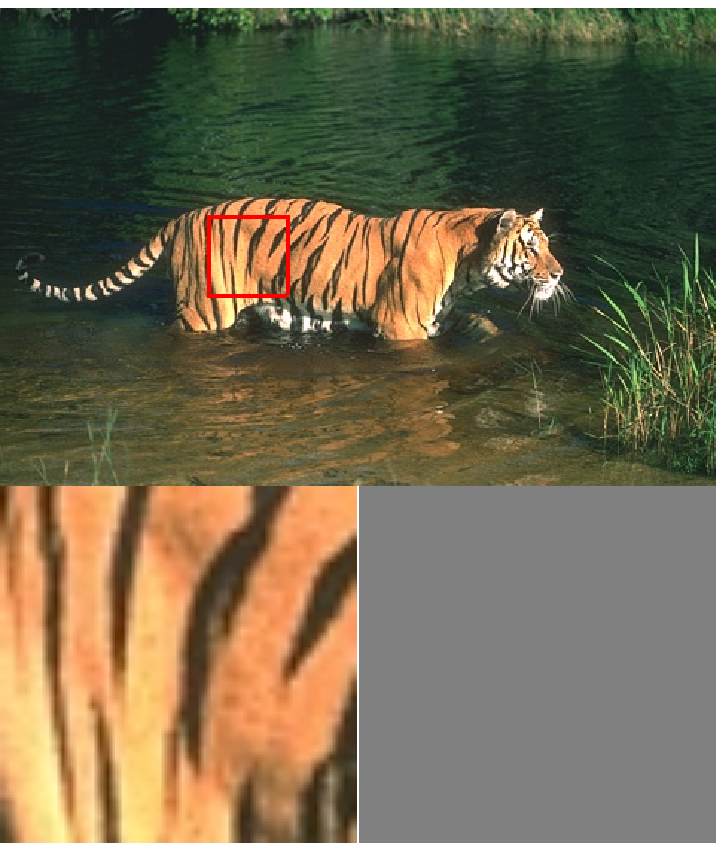}\vspace{0pt}
			\includegraphics[width=\linewidth]{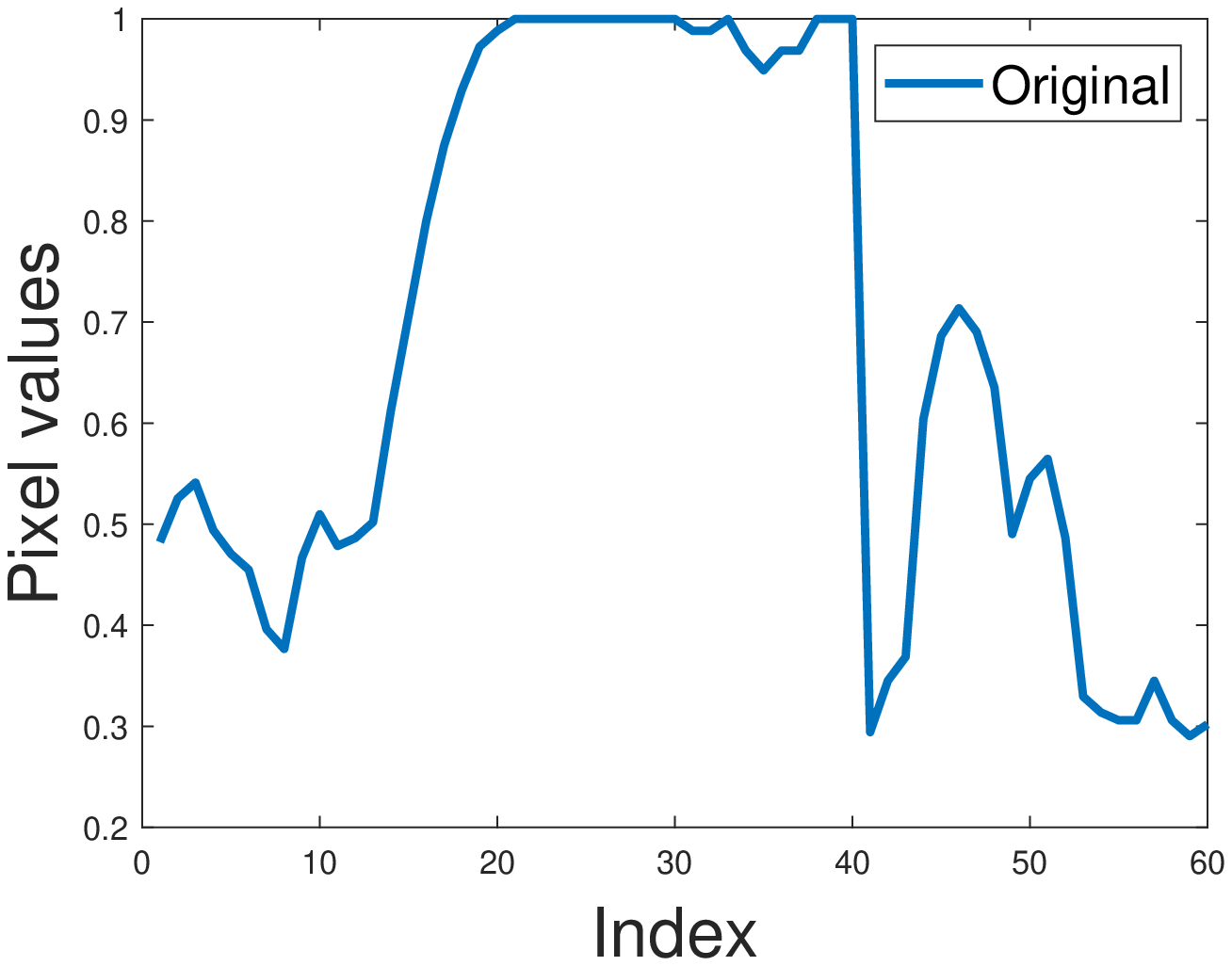}\vspace{0pt}
			\includegraphics[width=\linewidth]{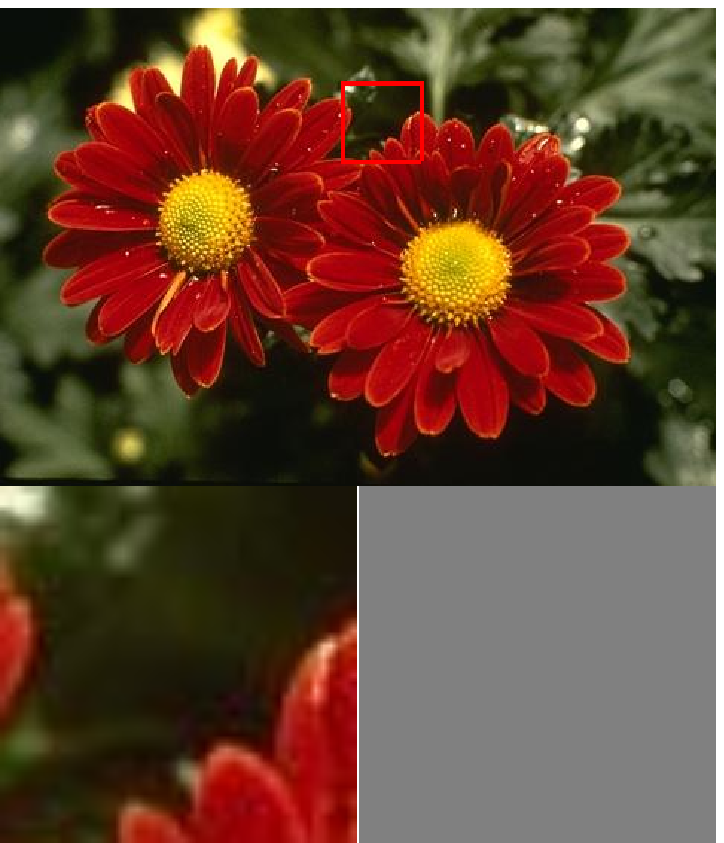}\vspace{0pt}
			\includegraphics[width=\linewidth]{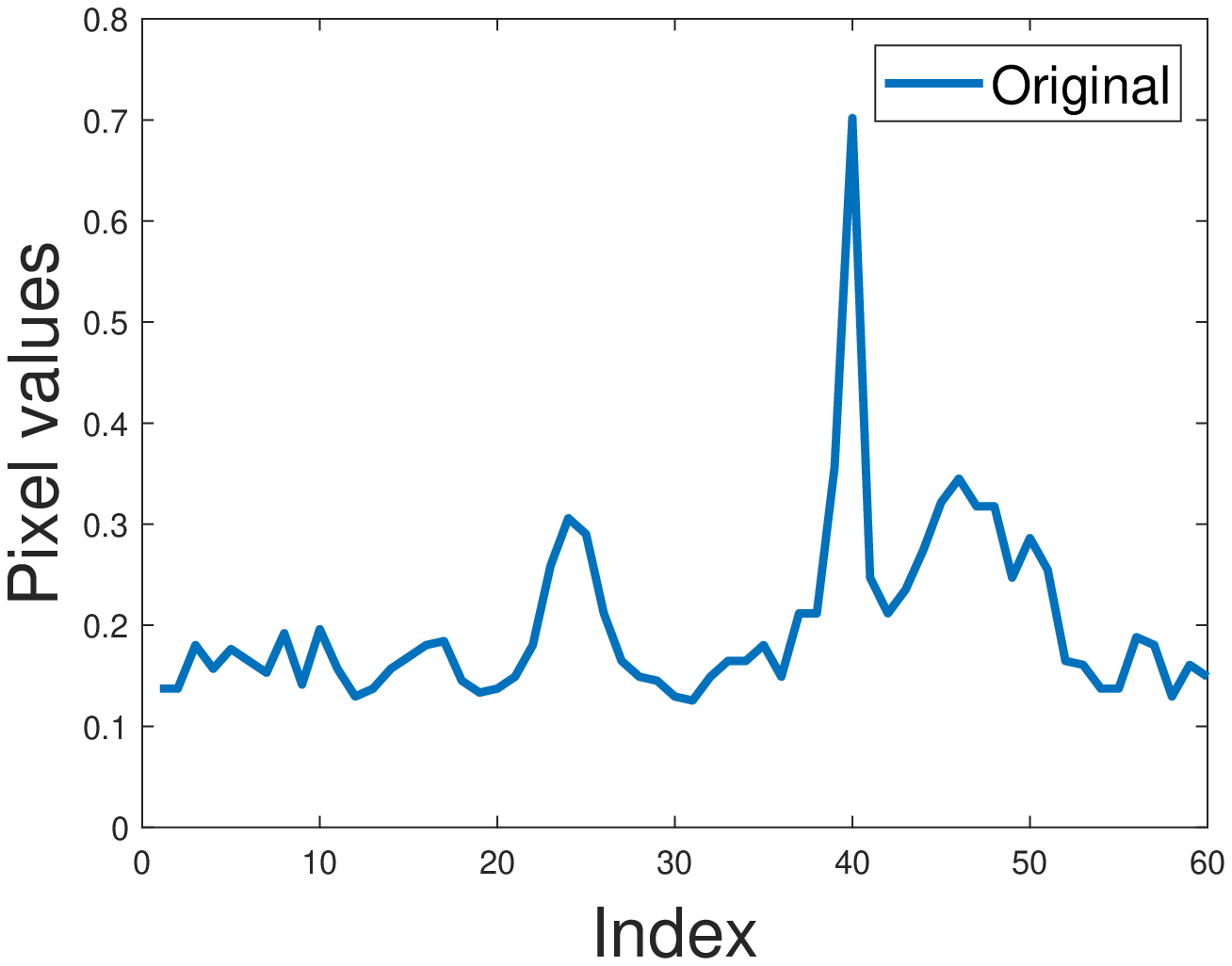}\vspace{0pt}
			\includegraphics[width=\linewidth]{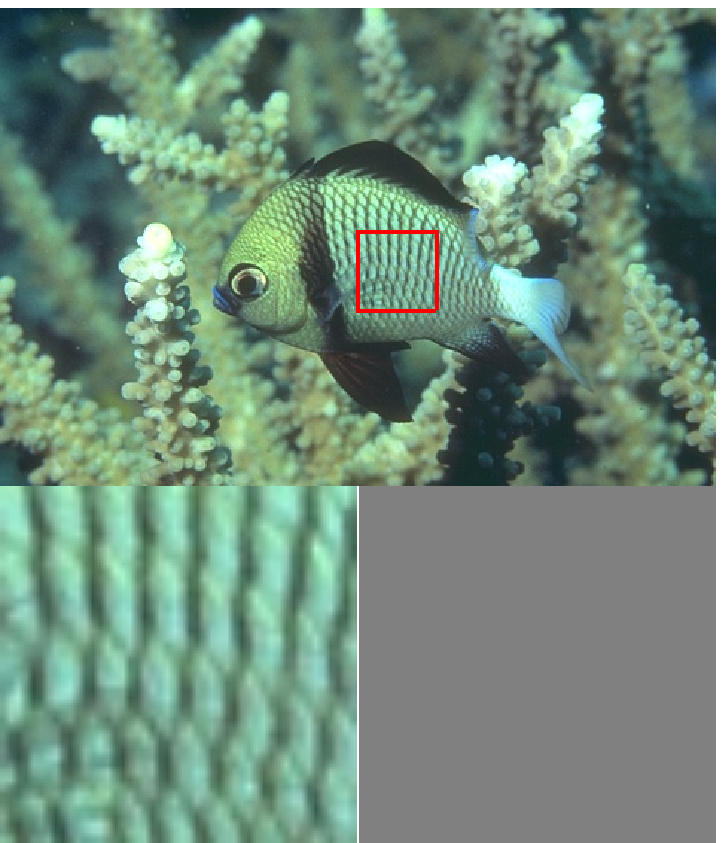}\vspace{0pt}
			\includegraphics[width=\linewidth]{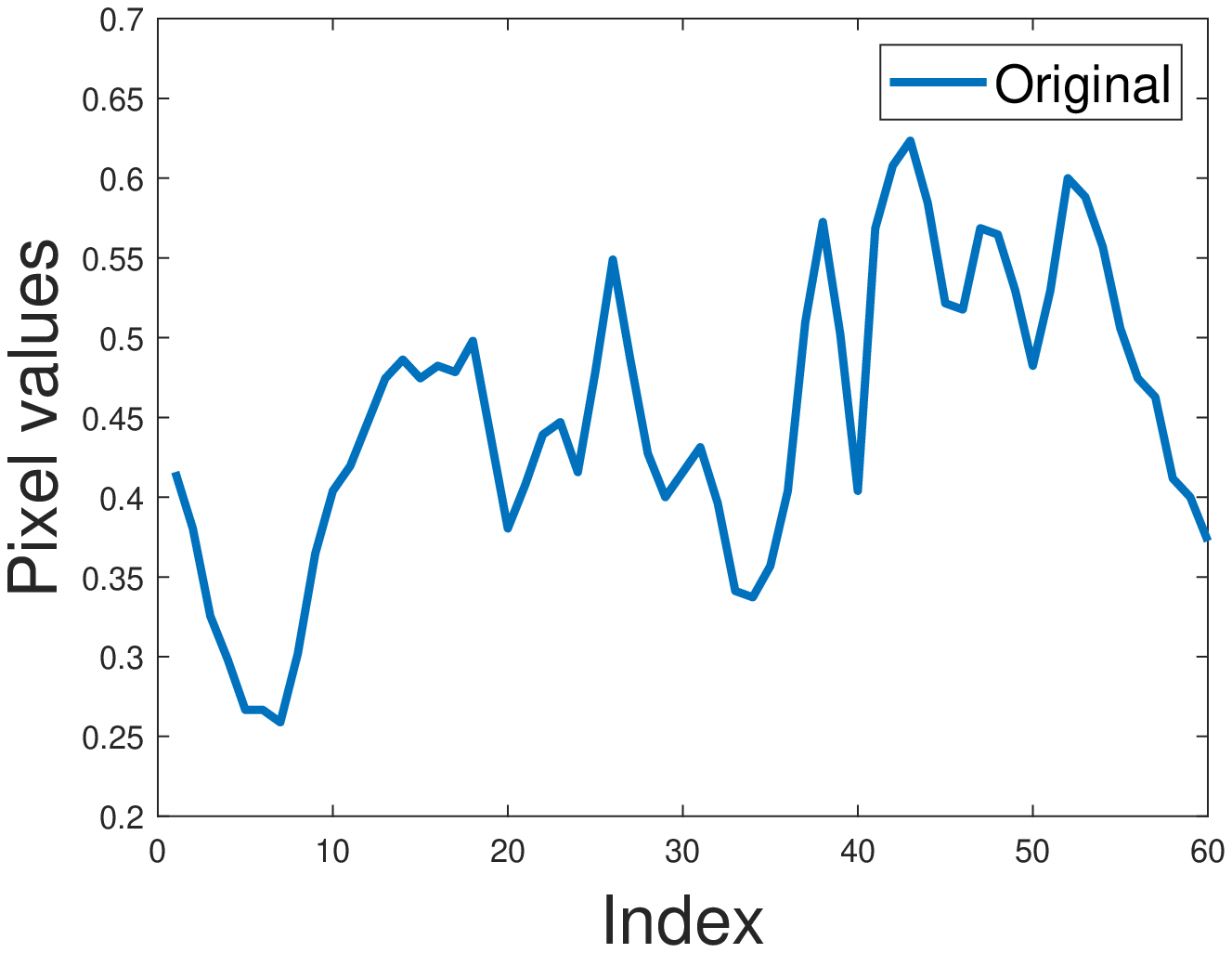}
			\caption{Original}
		\end{subfigure}   	
		\begin{subfigure}[b]{0.138\linewidth}
			\centering
			\includegraphics[width=\linewidth]{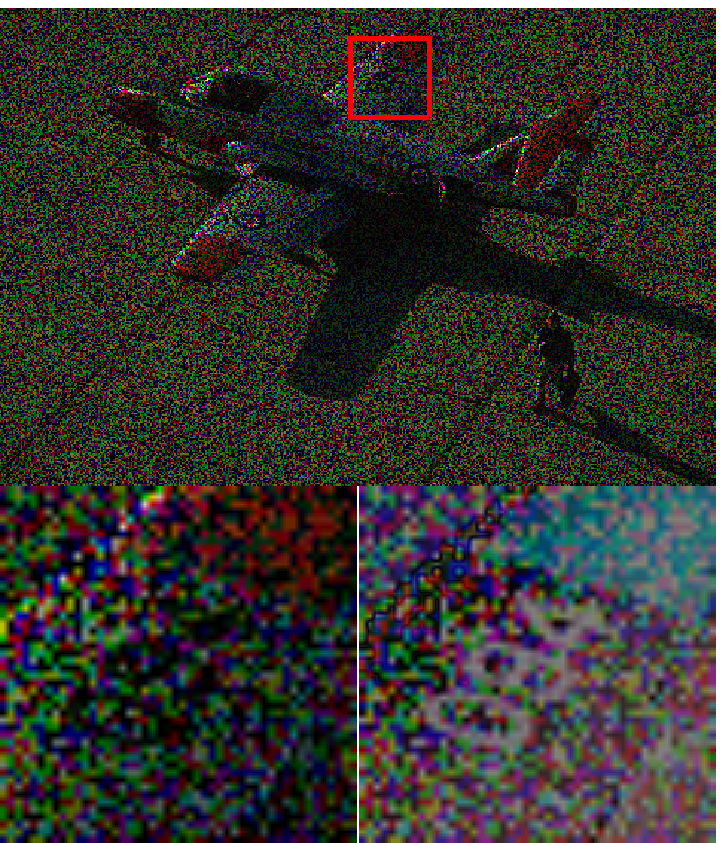}\vspace{0pt}
			\includegraphics[width=\linewidth]{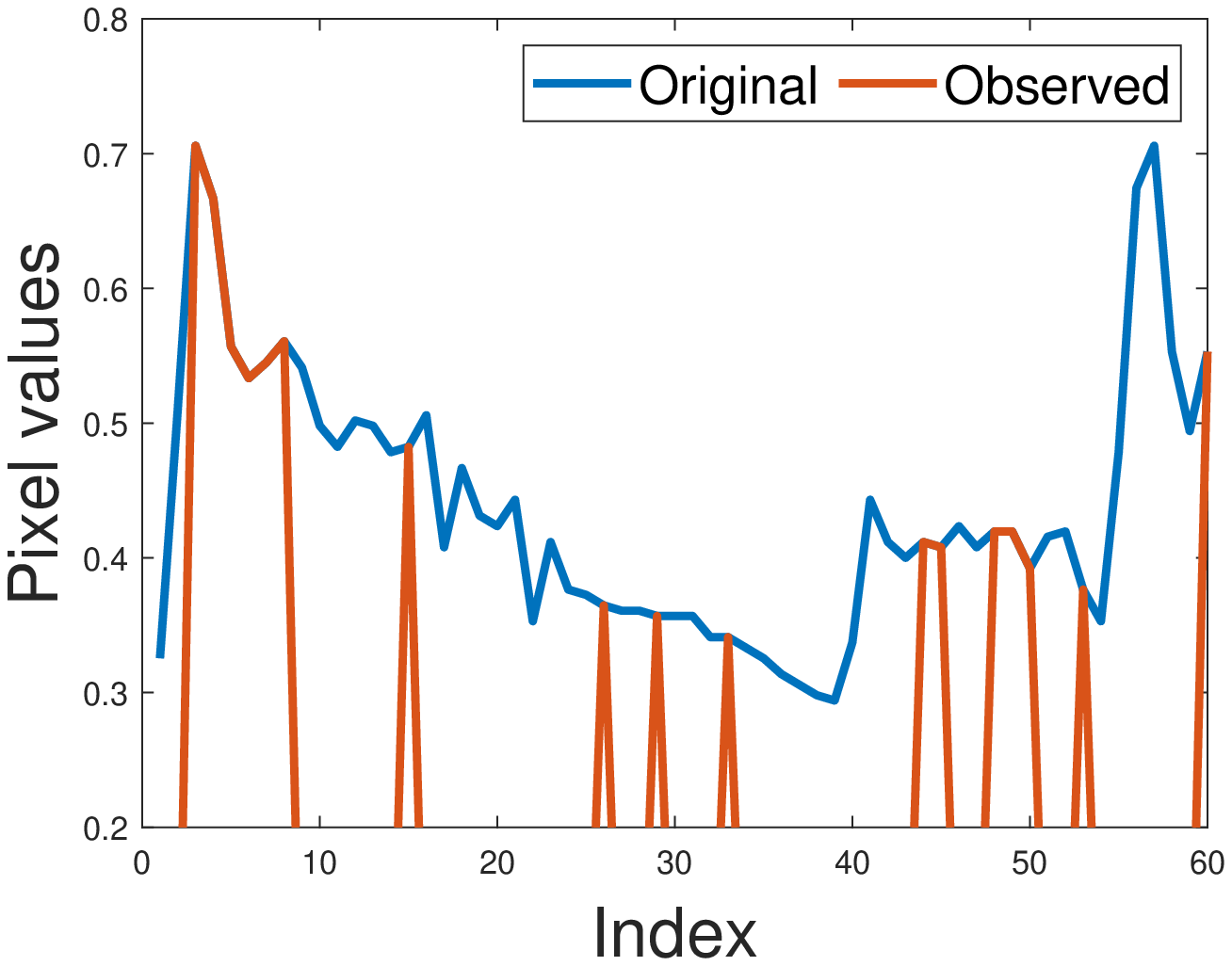}\vspace{0pt}
			\includegraphics[width=\linewidth]{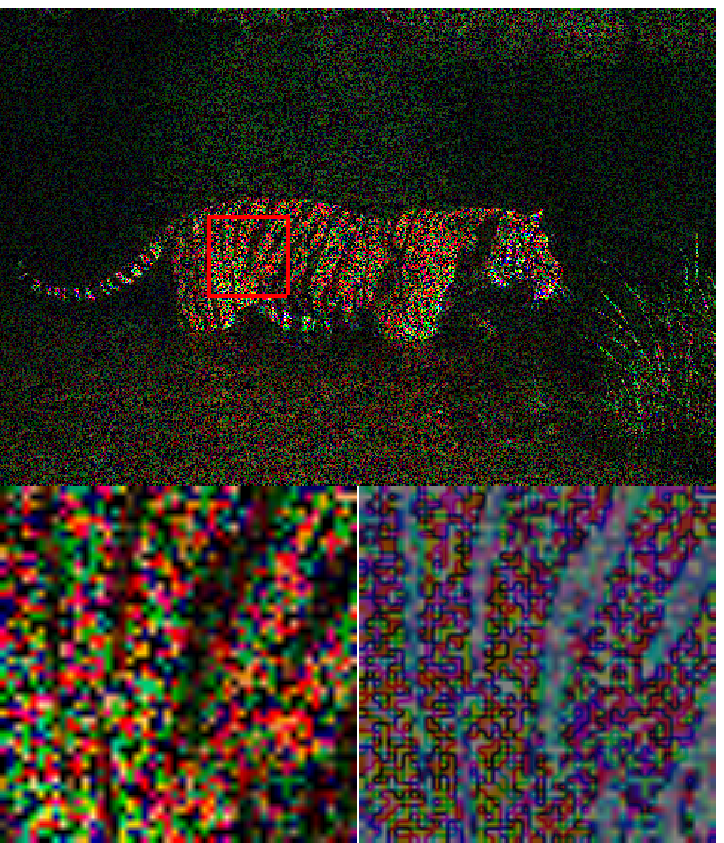}\vspace{0pt}
			\includegraphics[width=\linewidth]{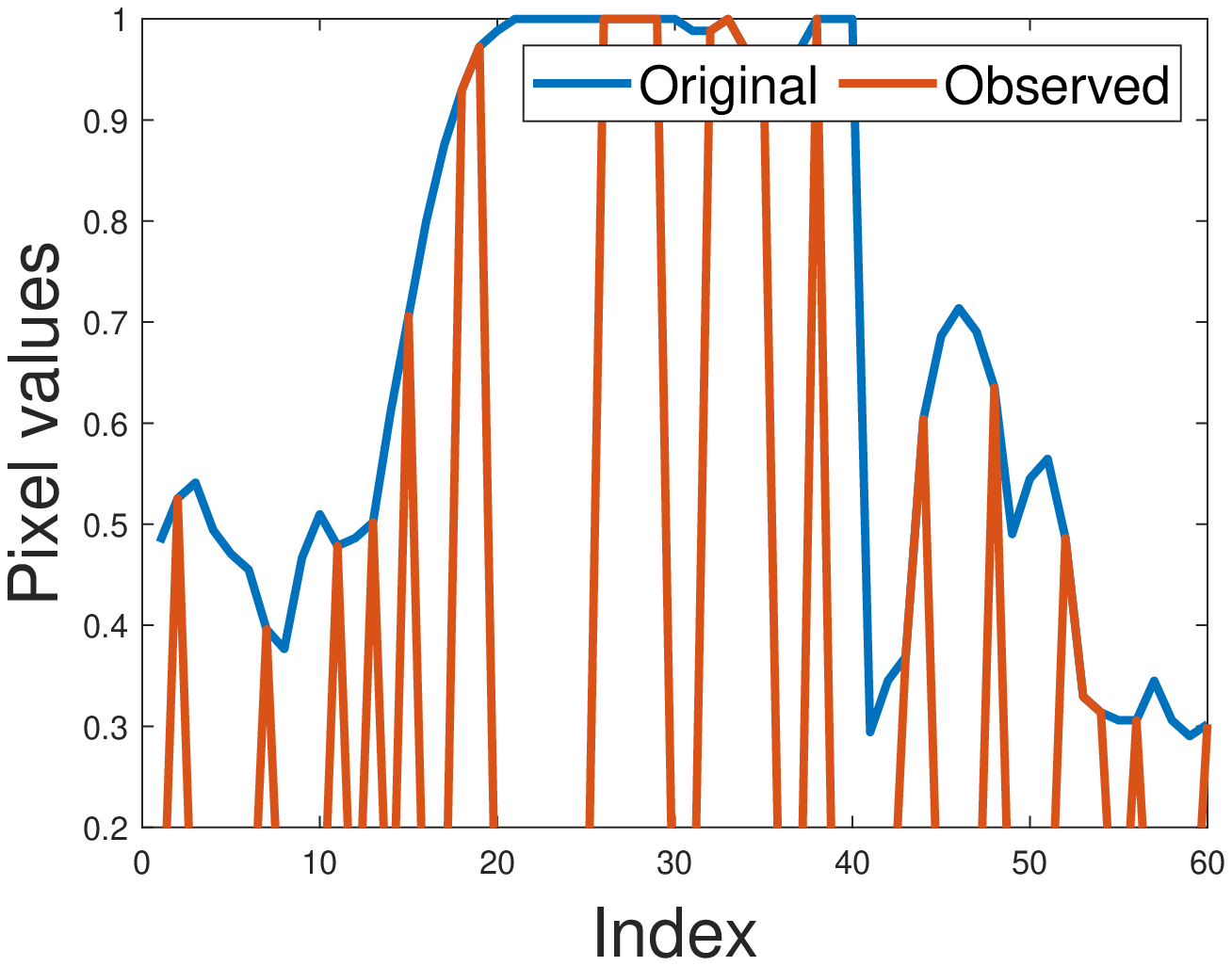}\vspace{0pt}
			\includegraphics[width=\linewidth]{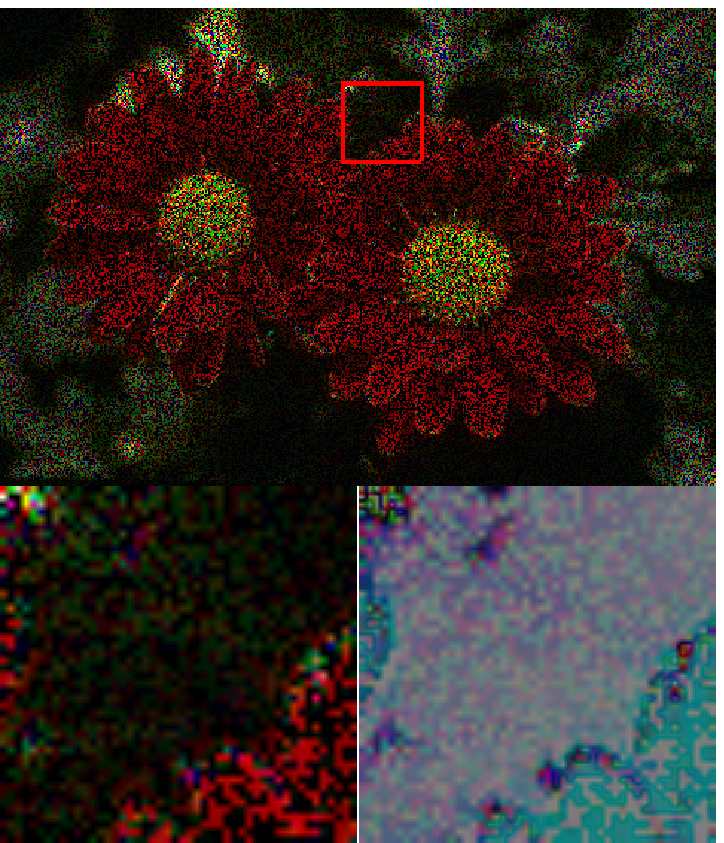}\vspace{0pt}
			\includegraphics[width=\linewidth]{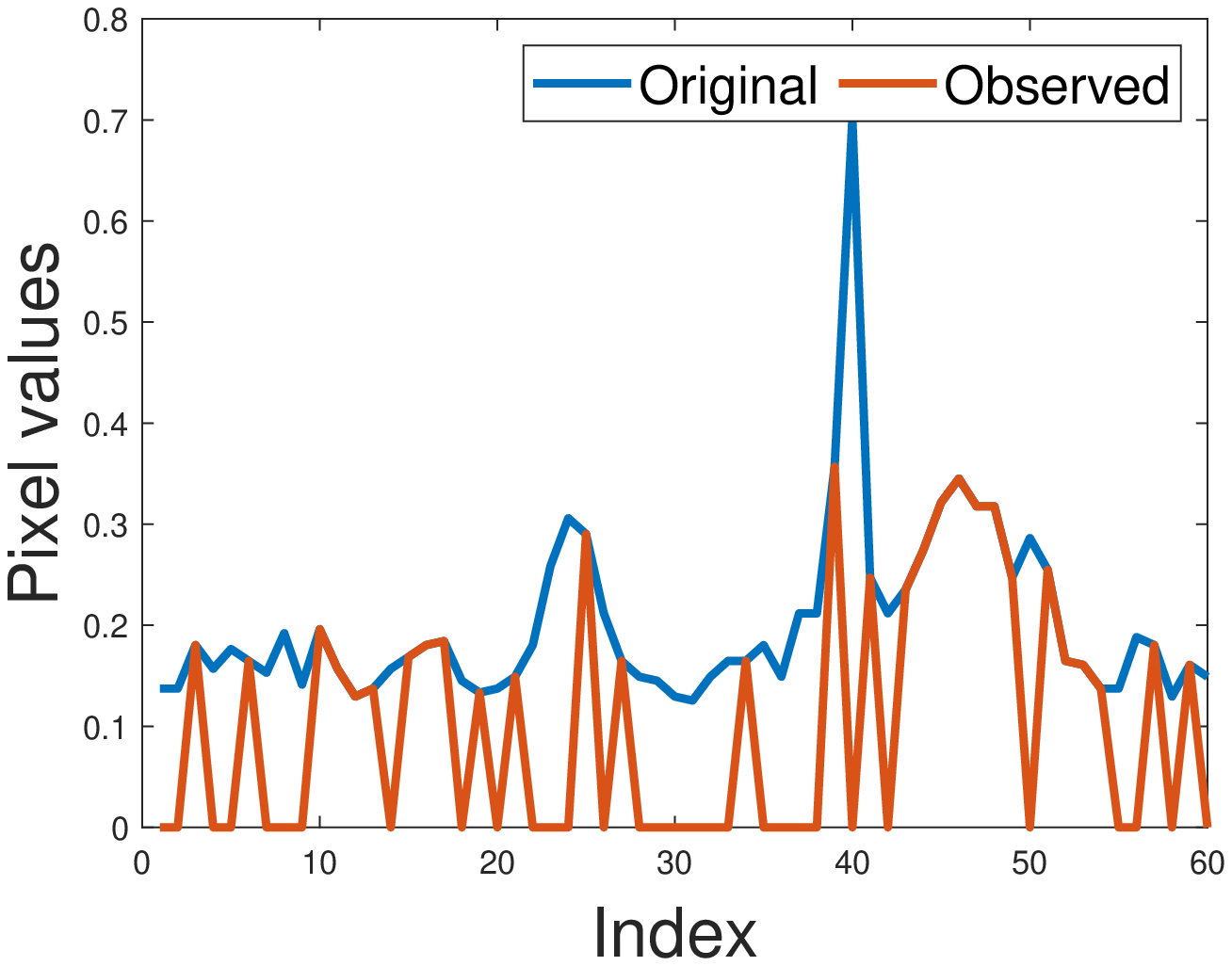}\vspace{0pt}
			\includegraphics[width=\linewidth]{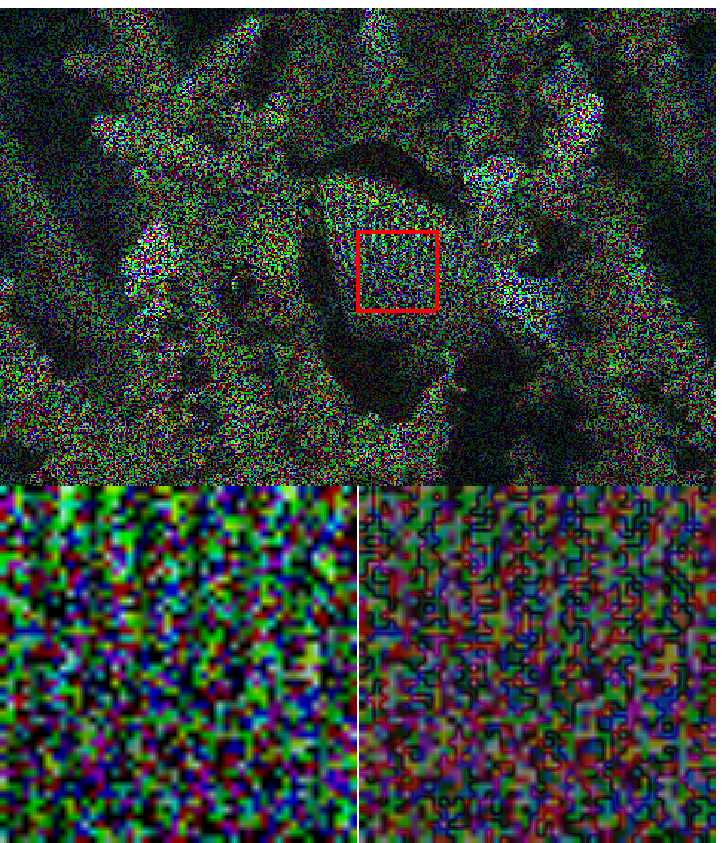}\vspace{0pt}
			\includegraphics[width=\linewidth]{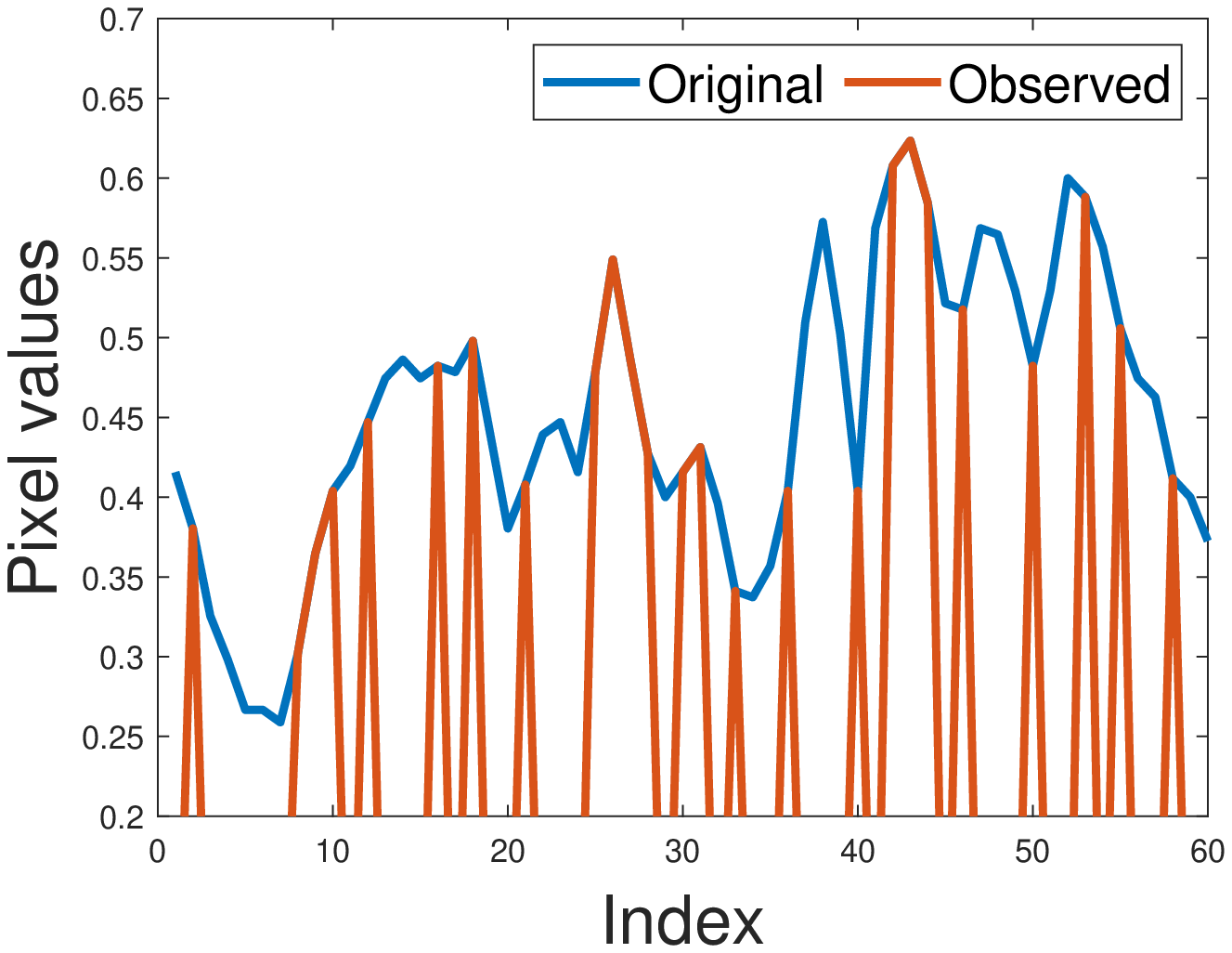}
			\caption{Observed}
		\end{subfigure}
		\begin{subfigure}[b]{0.138\linewidth}
			\centering
			\includegraphics[width=\linewidth]{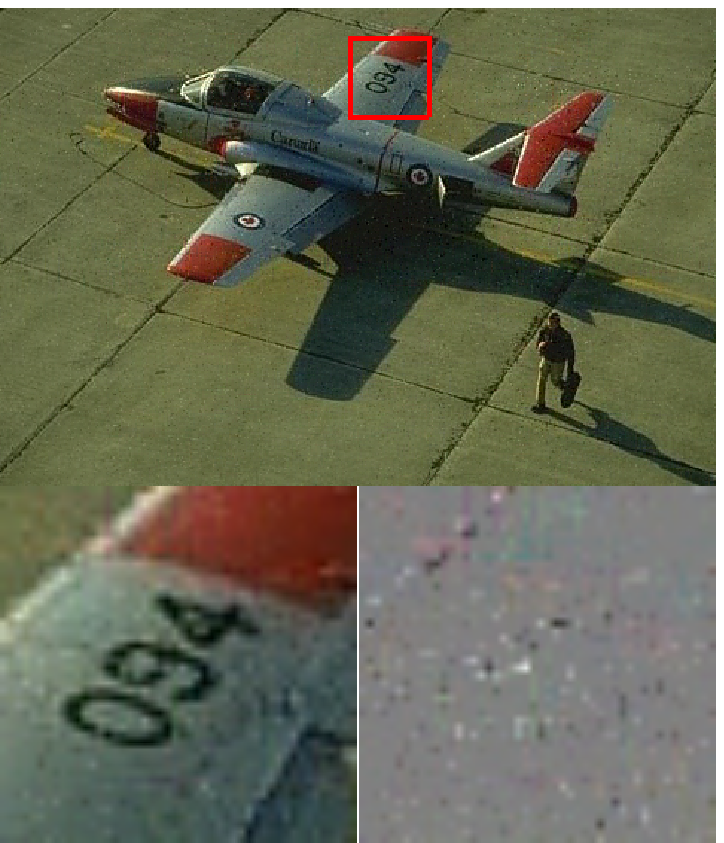}\vspace{0pt}
			\includegraphics[width=\linewidth]{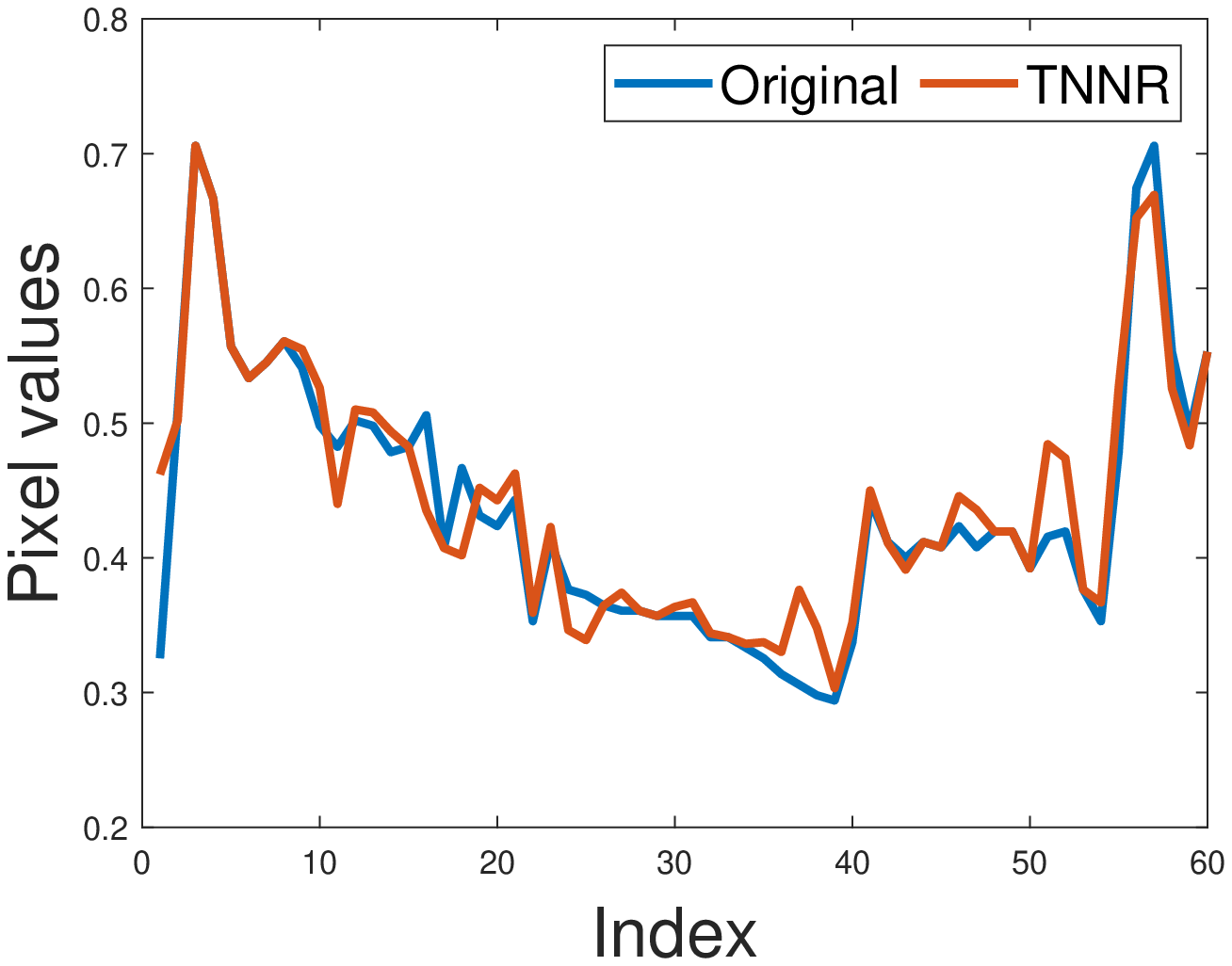}\vspace{0pt}
			\includegraphics[width=\linewidth]{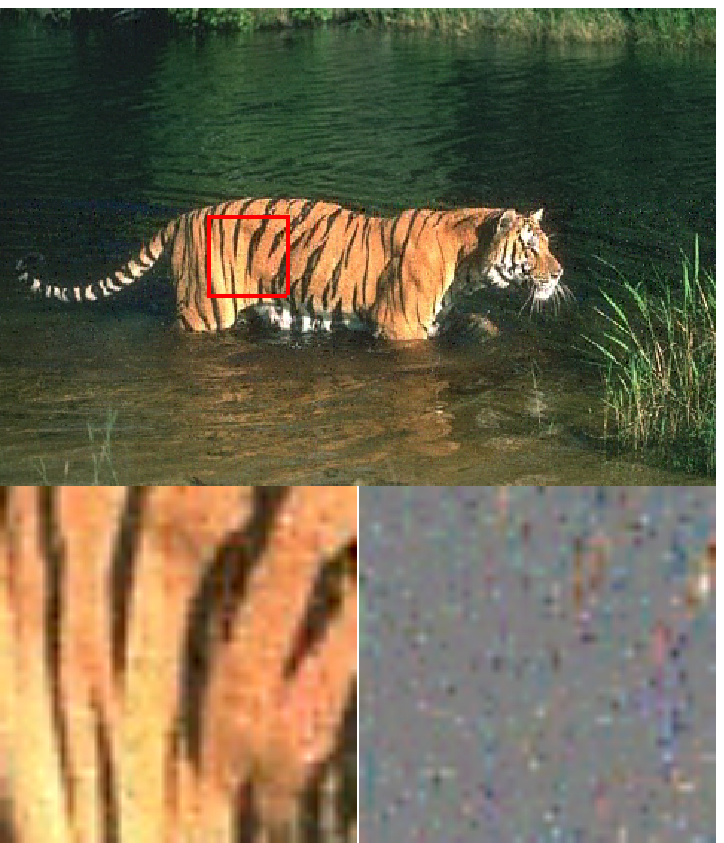}\vspace{0pt}
			\includegraphics[width=\linewidth]{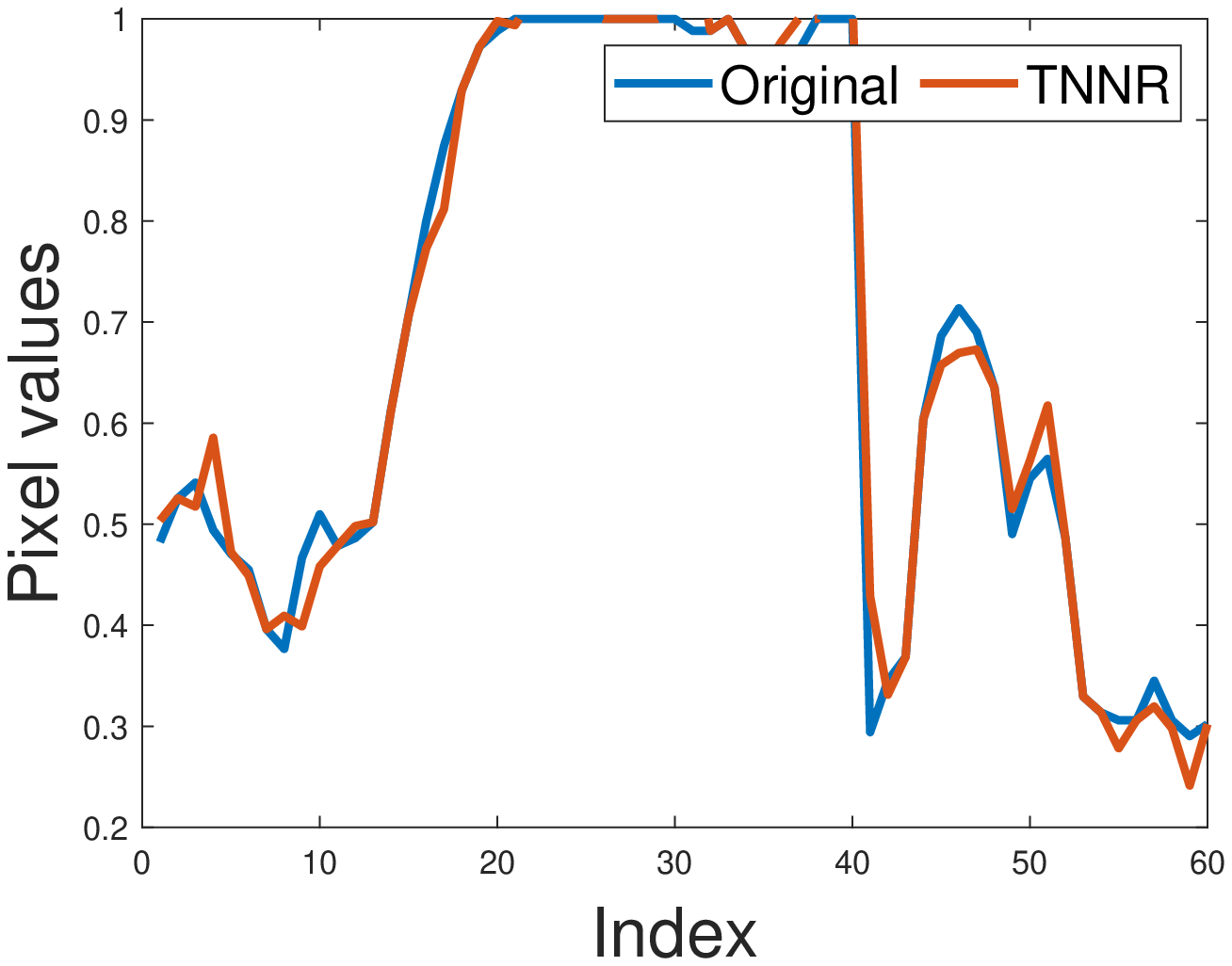}\vspace{0pt}
			\includegraphics[width=\linewidth]{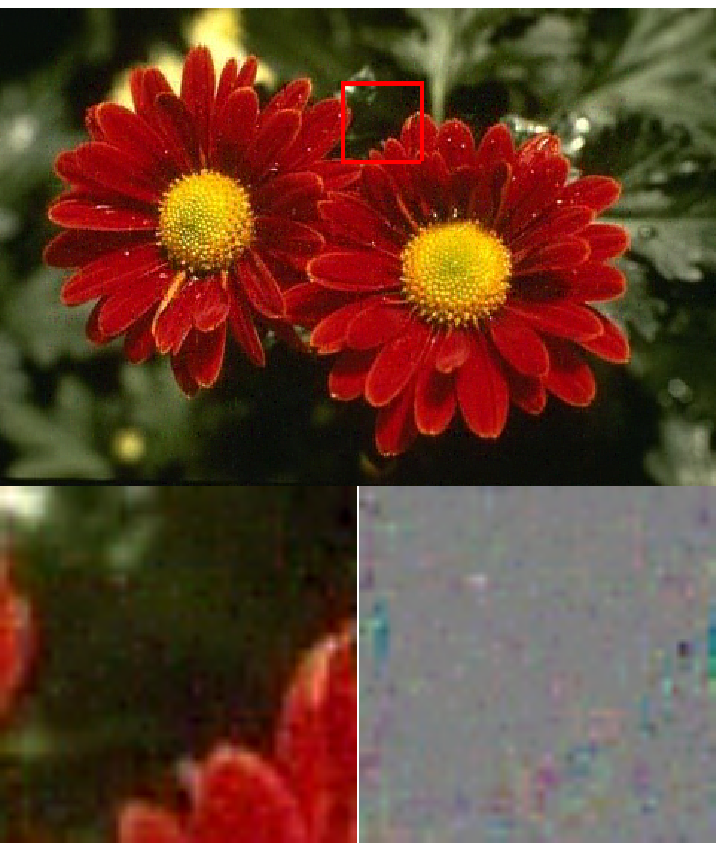}\vspace{0pt}
			\includegraphics[width=\linewidth]{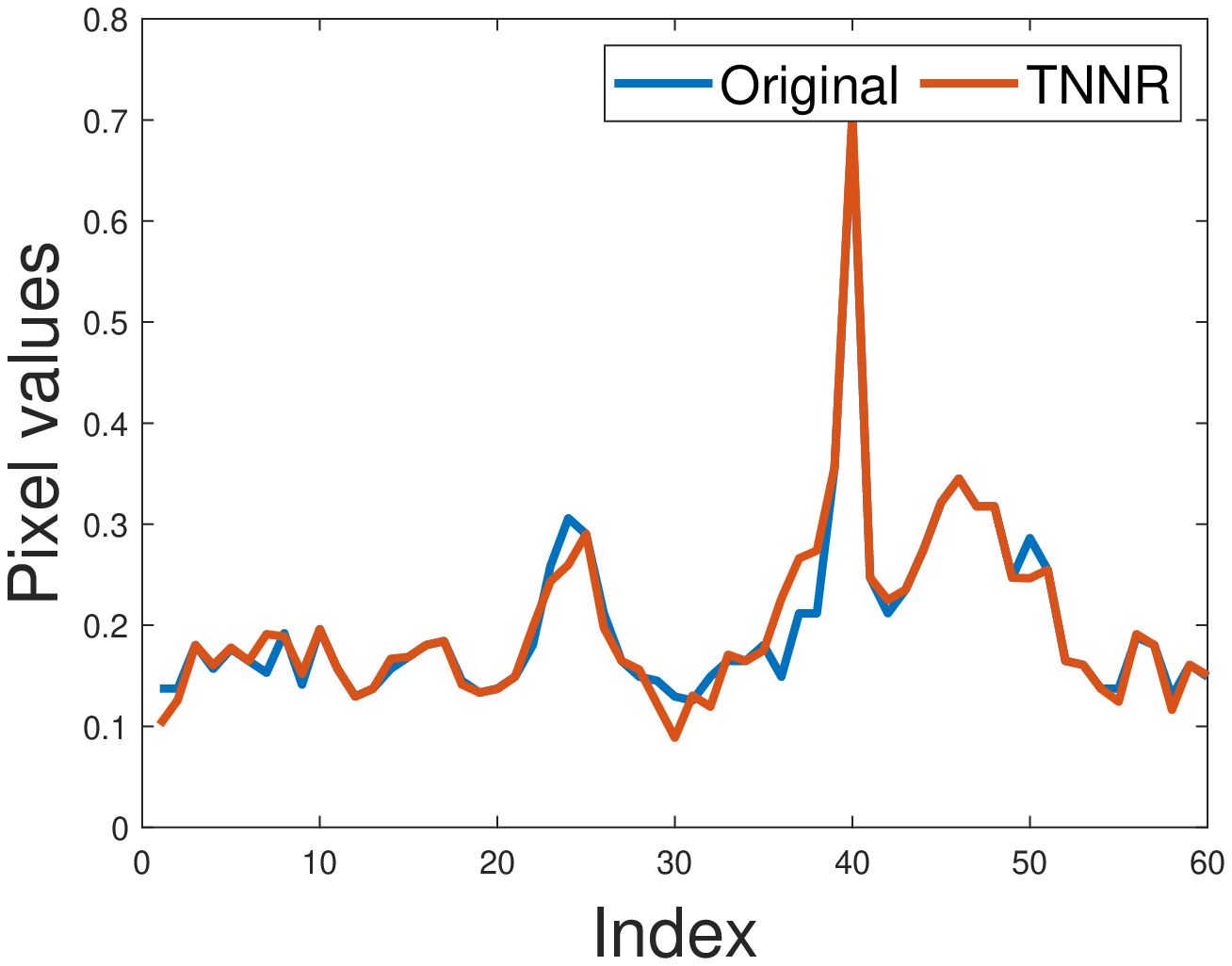}\vspace{0pt}
			\includegraphics[width=\linewidth]{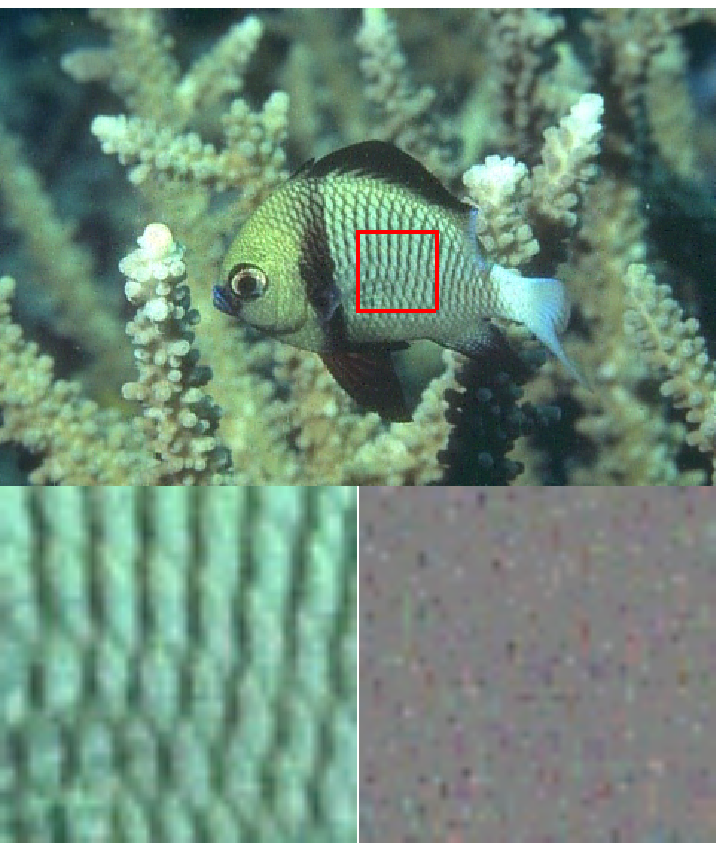}\vspace{0pt}
			\includegraphics[width=\linewidth]{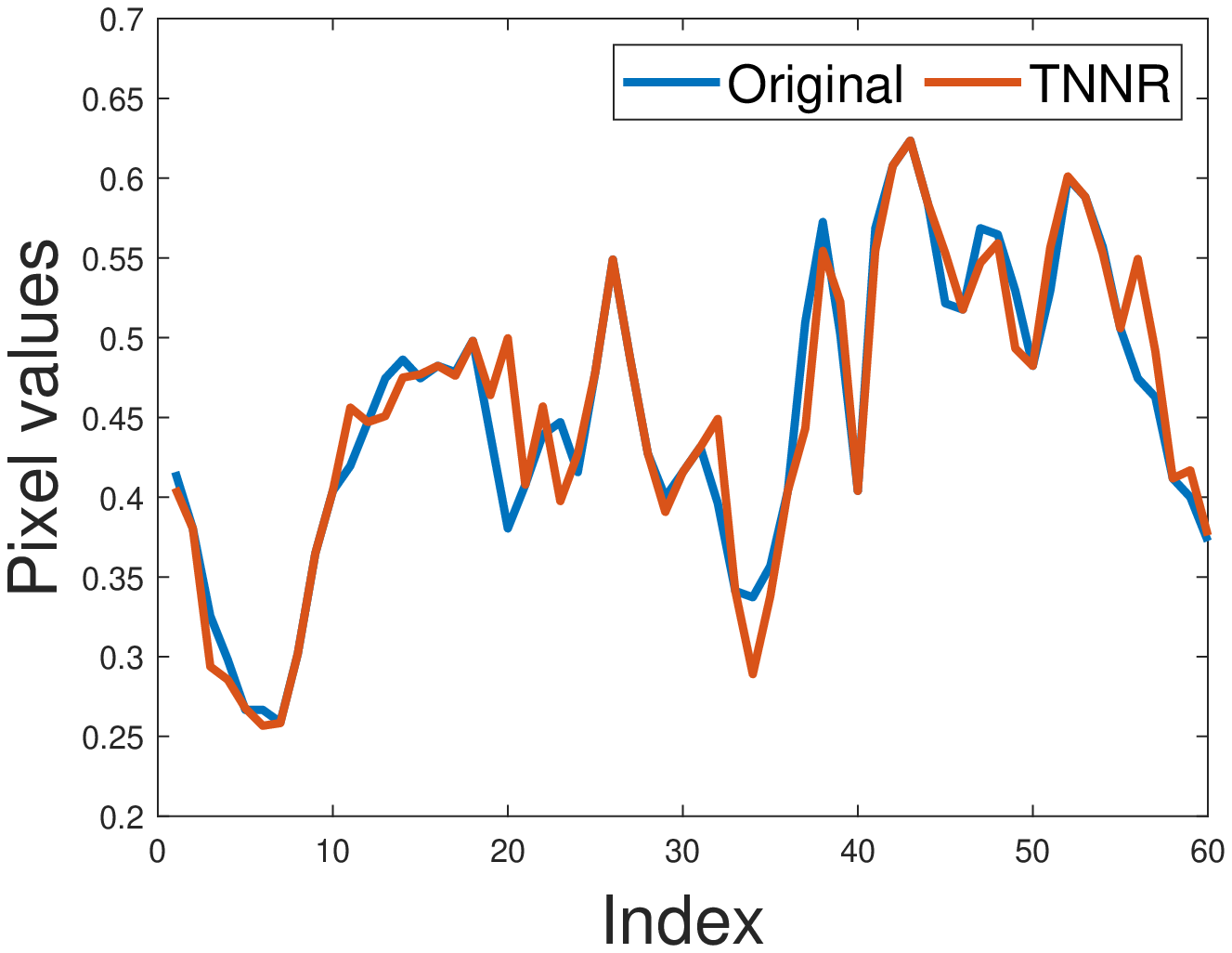}
			\caption{TNNR}
		\end{subfigure}
		\begin{subfigure}[b]{0.138\linewidth}
			\centering		
			\includegraphics[width=\linewidth]{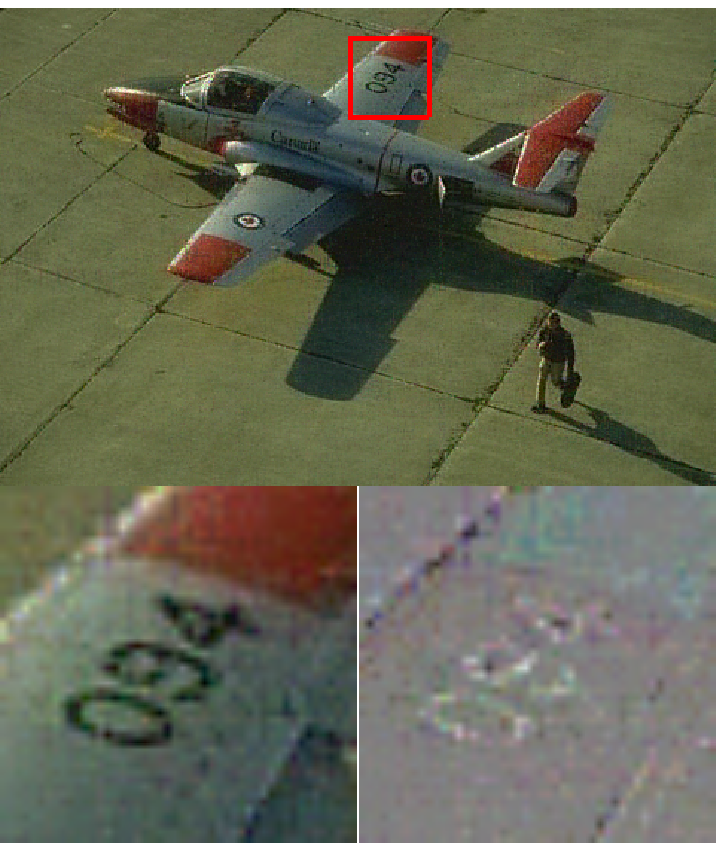}\vspace{0pt}
			\includegraphics[width=\linewidth]{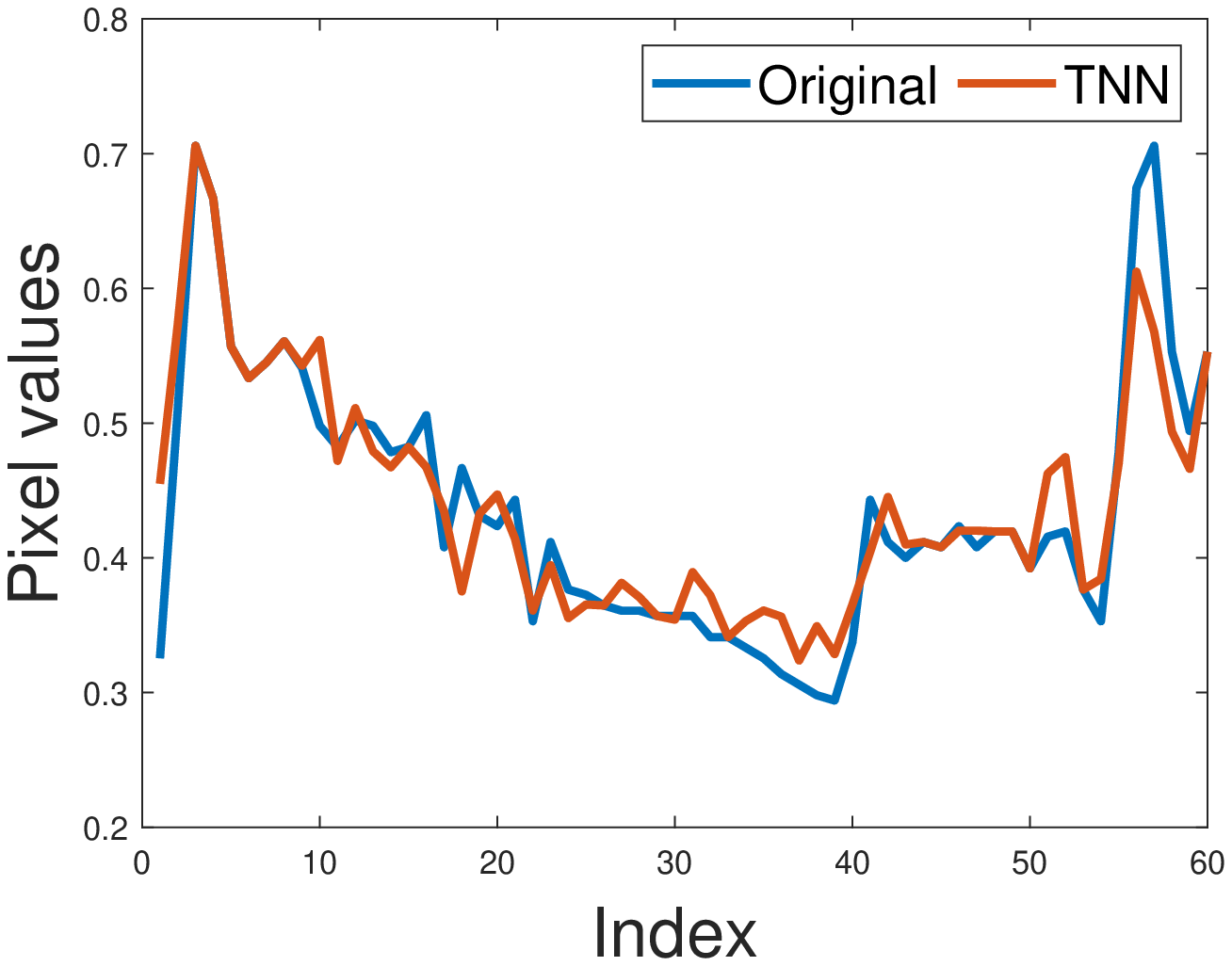}\vspace{0pt}
			\includegraphics[width=\linewidth]{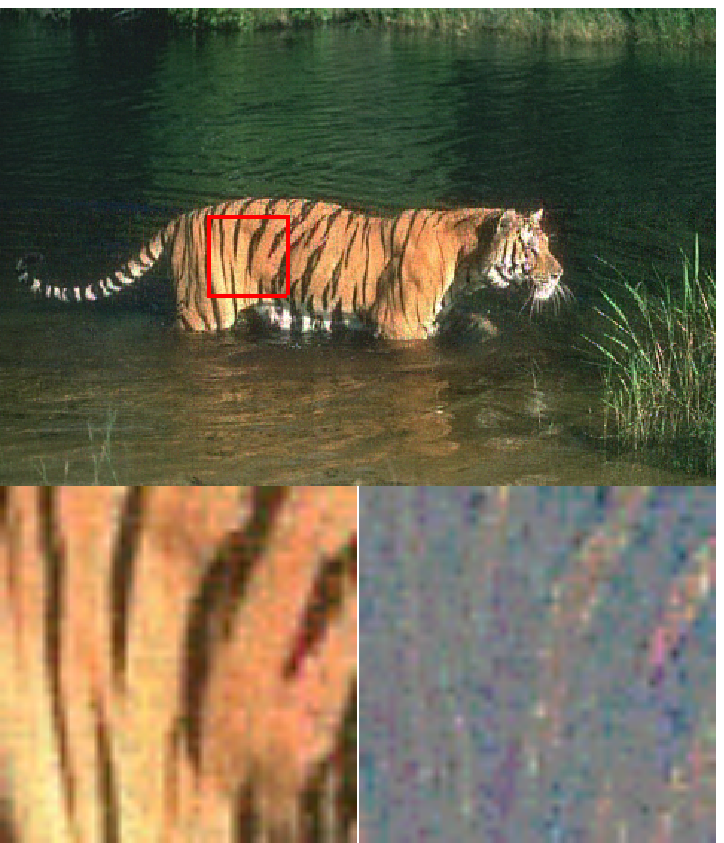}\vspace{0pt}
			\includegraphics[width=\linewidth]{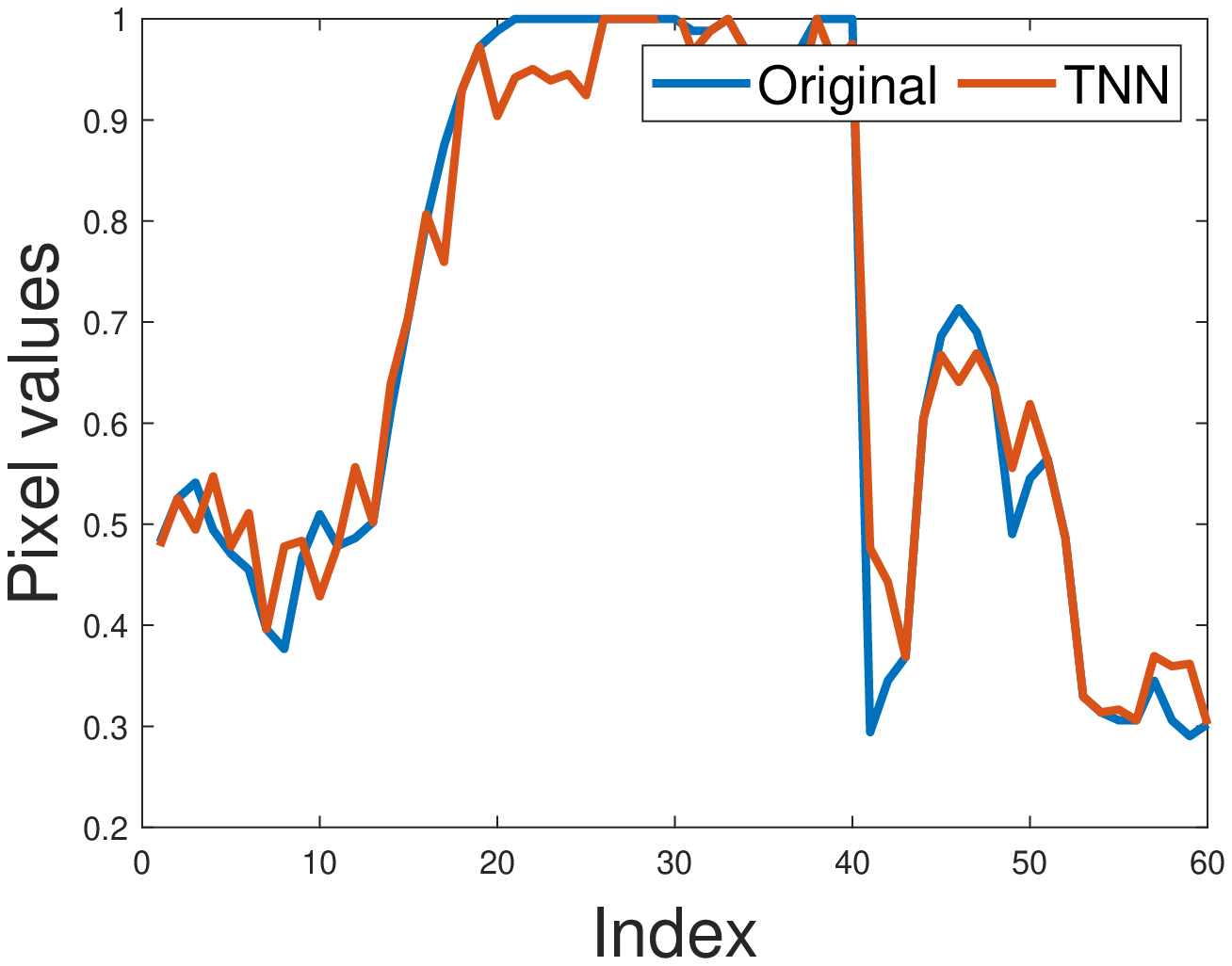}\vspace{0pt}
			\includegraphics[width=\linewidth]{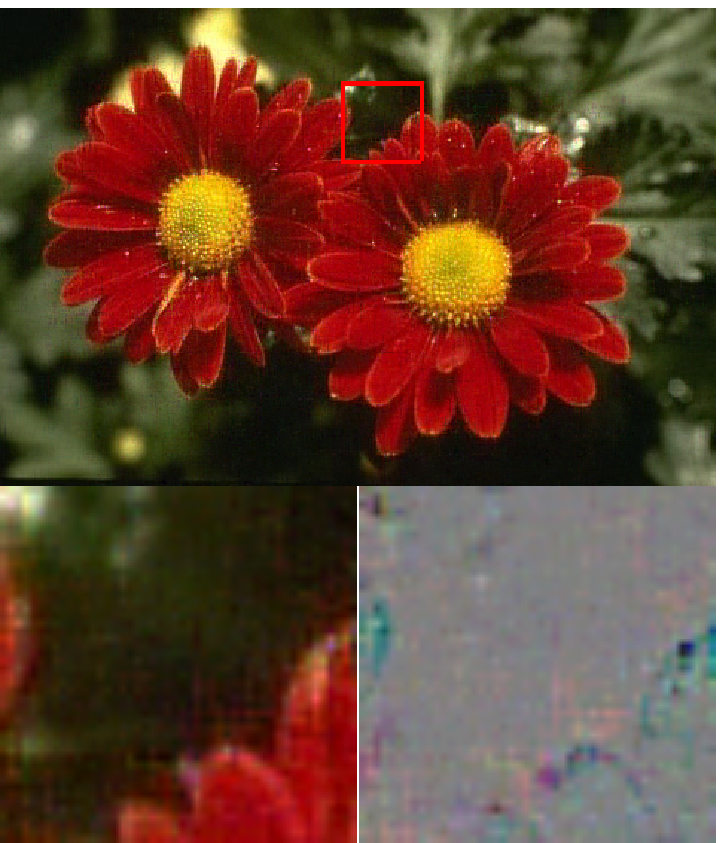}\vspace{0pt}
			\includegraphics[width=\linewidth]{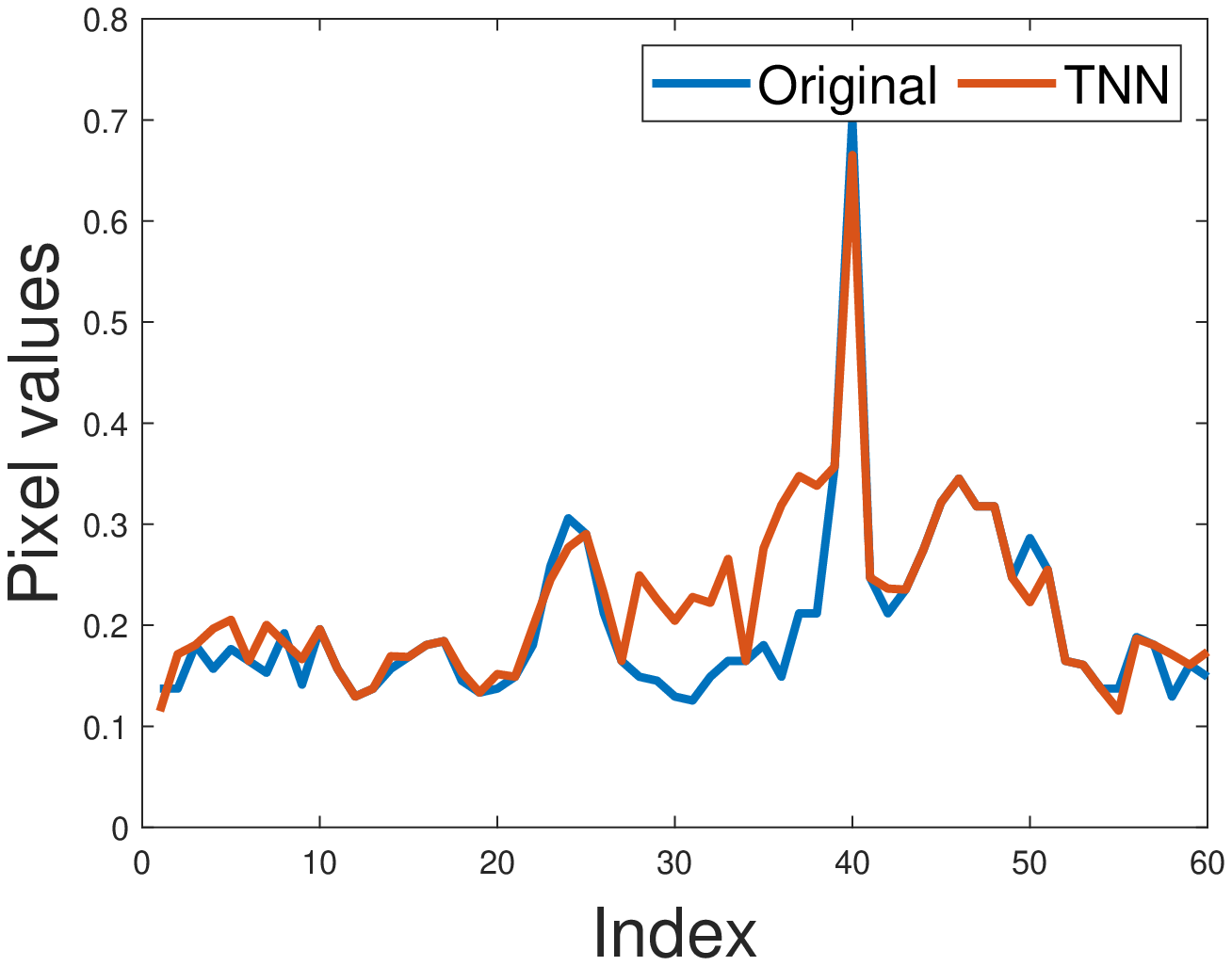}\vspace{0pt}
			\includegraphics[width=\linewidth]{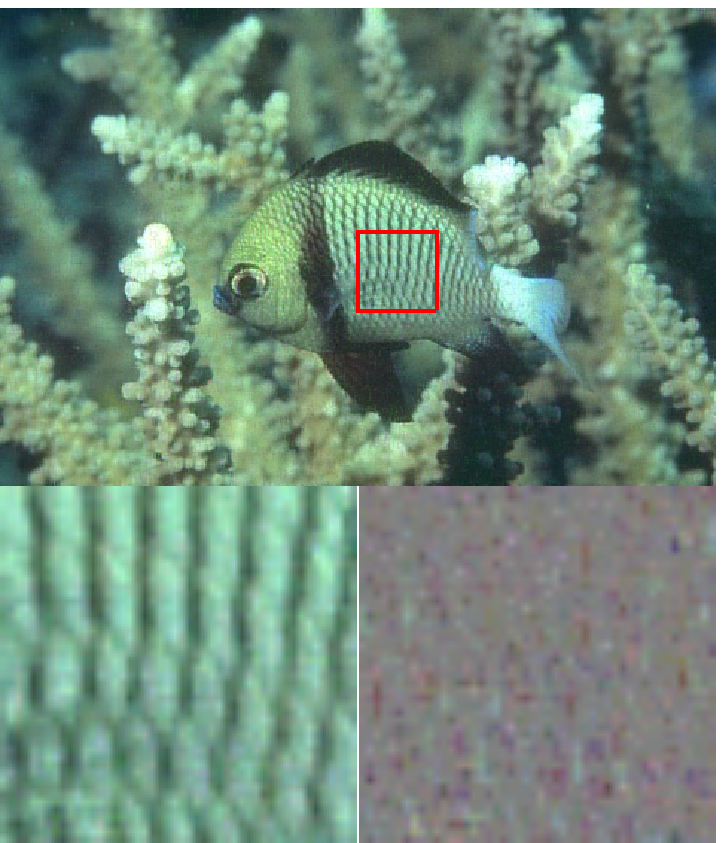}\vspace{0pt}
			\includegraphics[width=\linewidth]{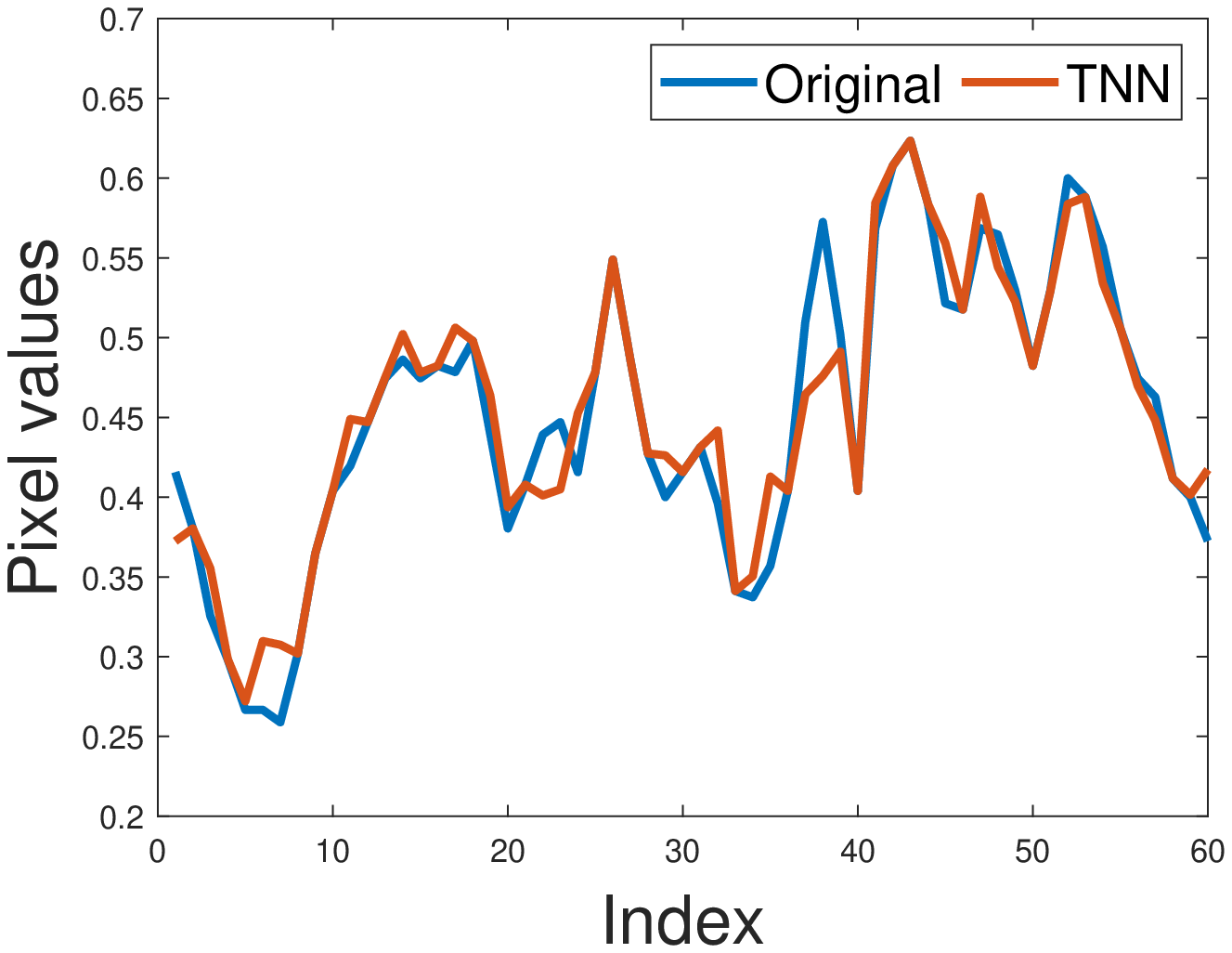}
			\caption{TNN}
		\end{subfigure}
		\begin{subfigure}[b]{0.138\linewidth}
			\centering			
			\includegraphics[width=\linewidth]{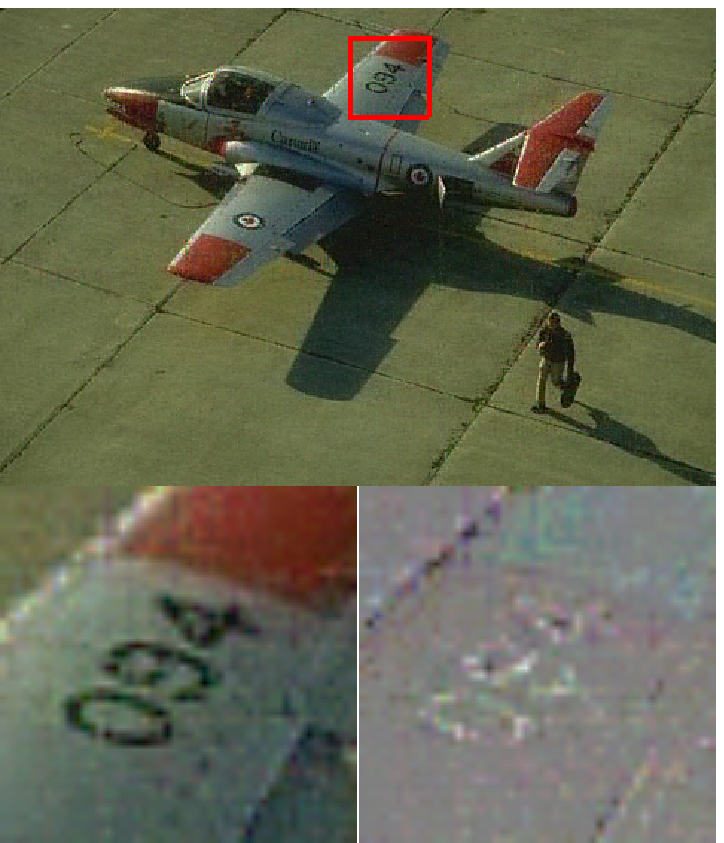}\vspace{0pt}
			\includegraphics[width=\linewidth]{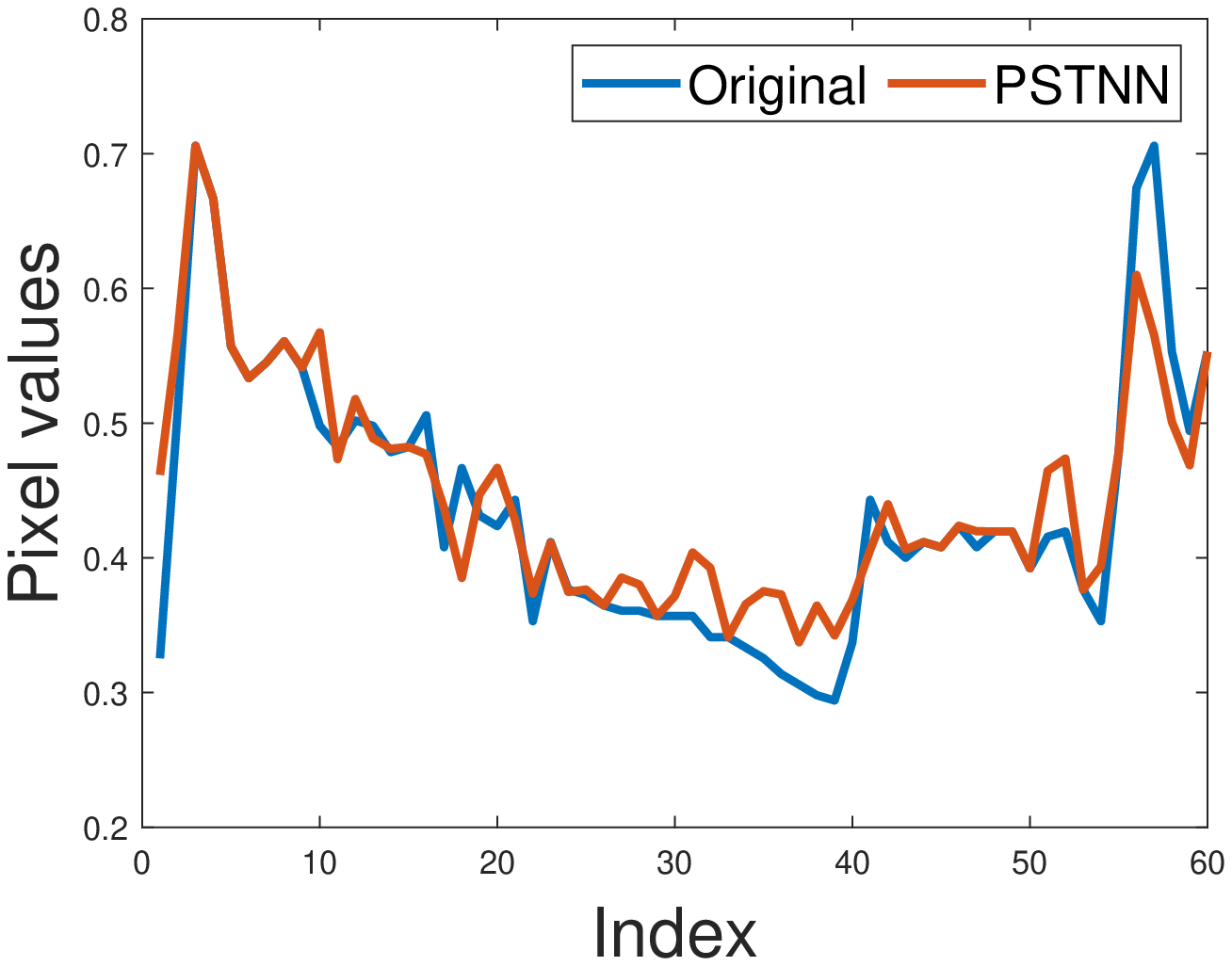}\vspace{0pt}
			\includegraphics[width=\linewidth]{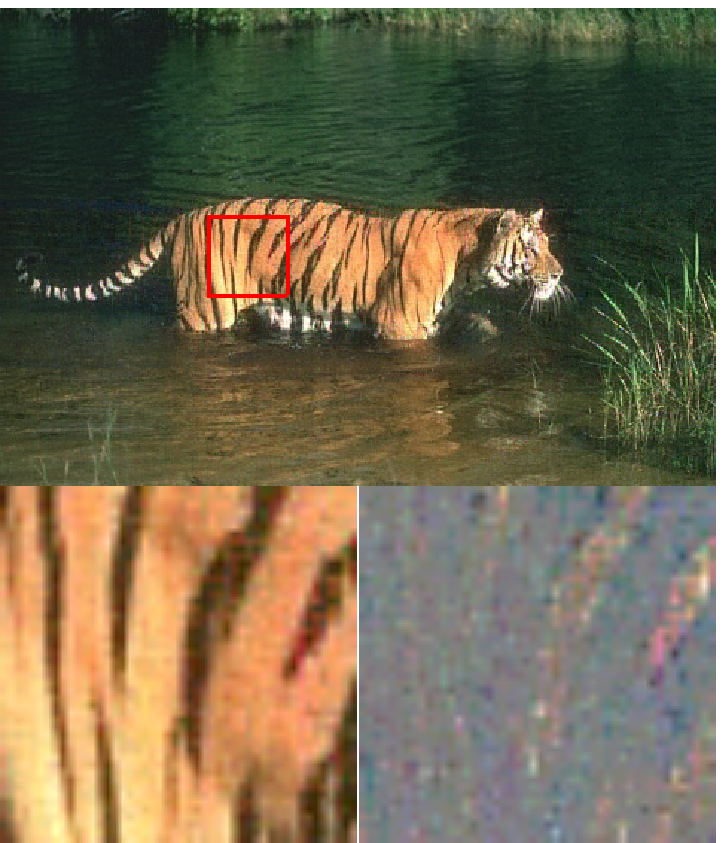}\vspace{0pt}
			\includegraphics[width=\linewidth]{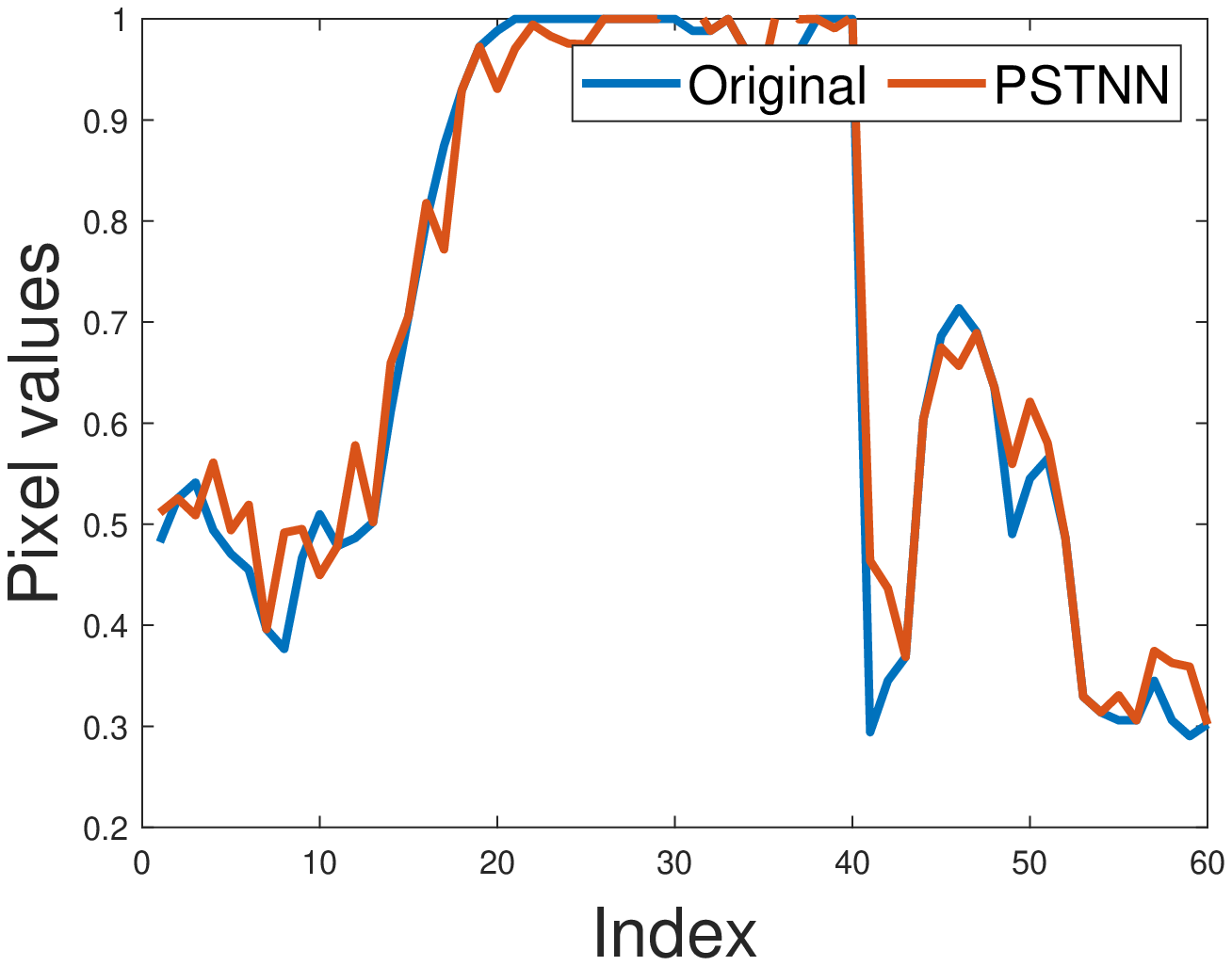}\vspace{0pt}
			\includegraphics[width=\linewidth]{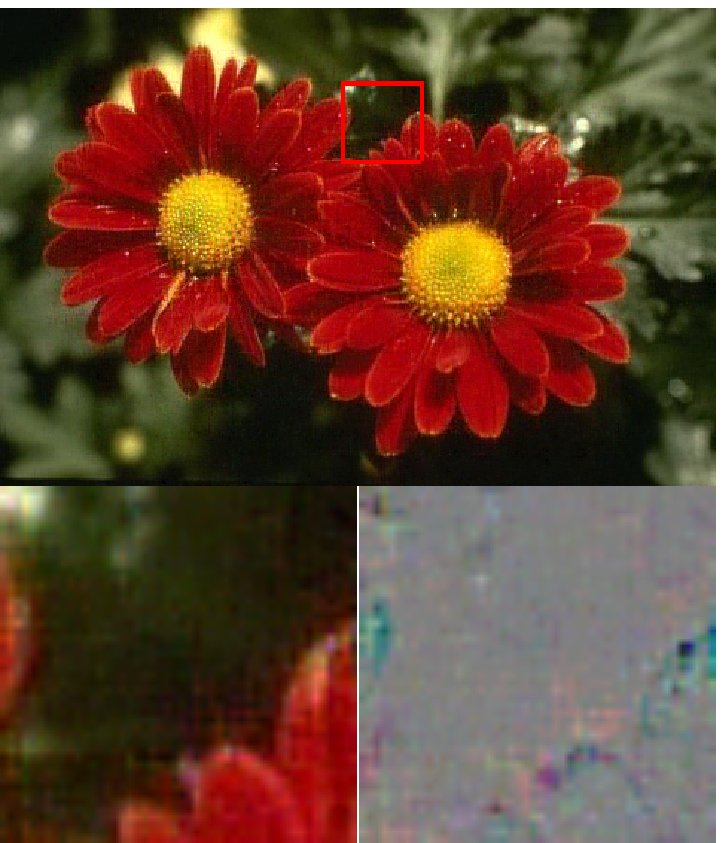}\vspace{0pt}
			\includegraphics[width=\linewidth]{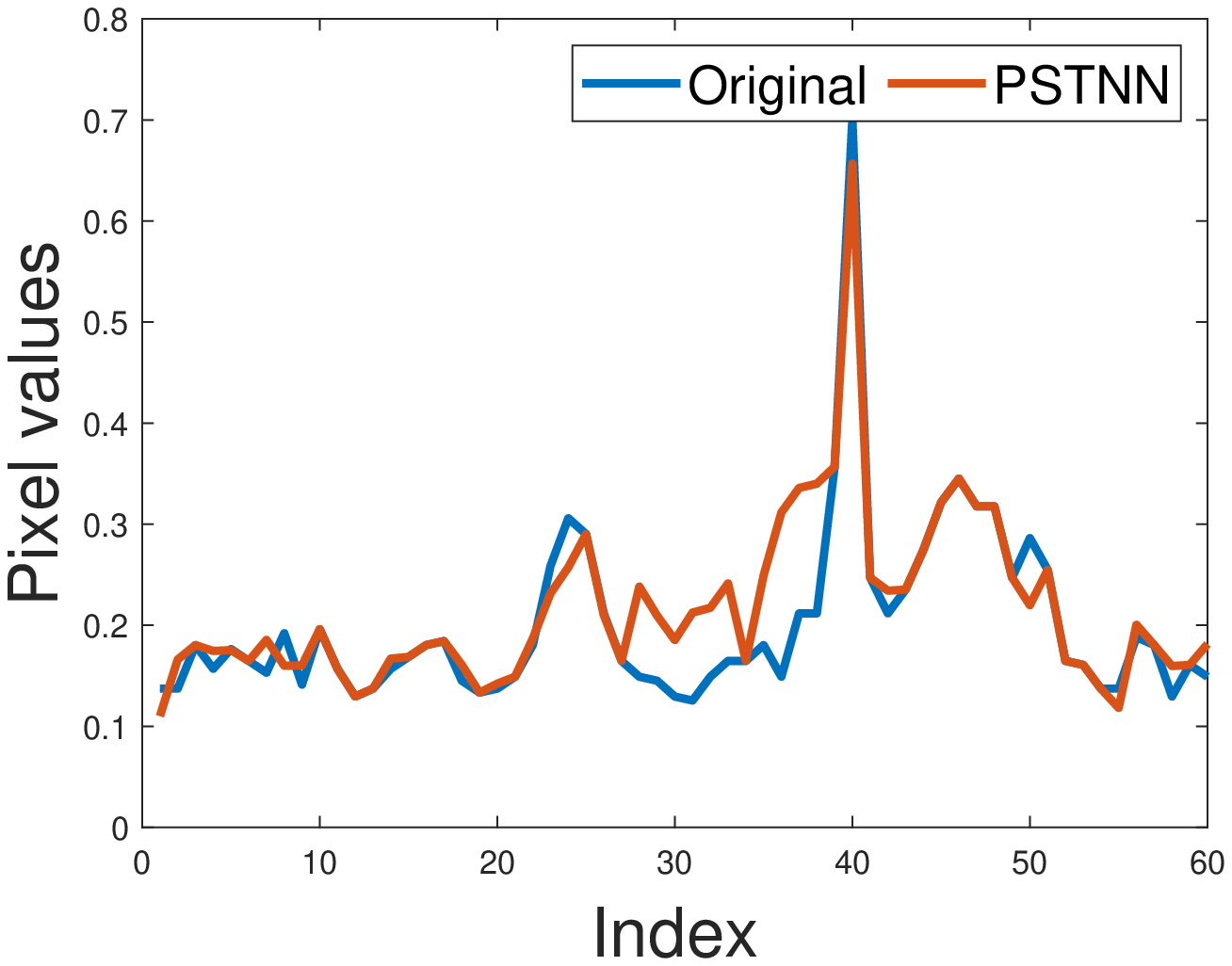}\vspace{0pt}
			\includegraphics[width=\linewidth]{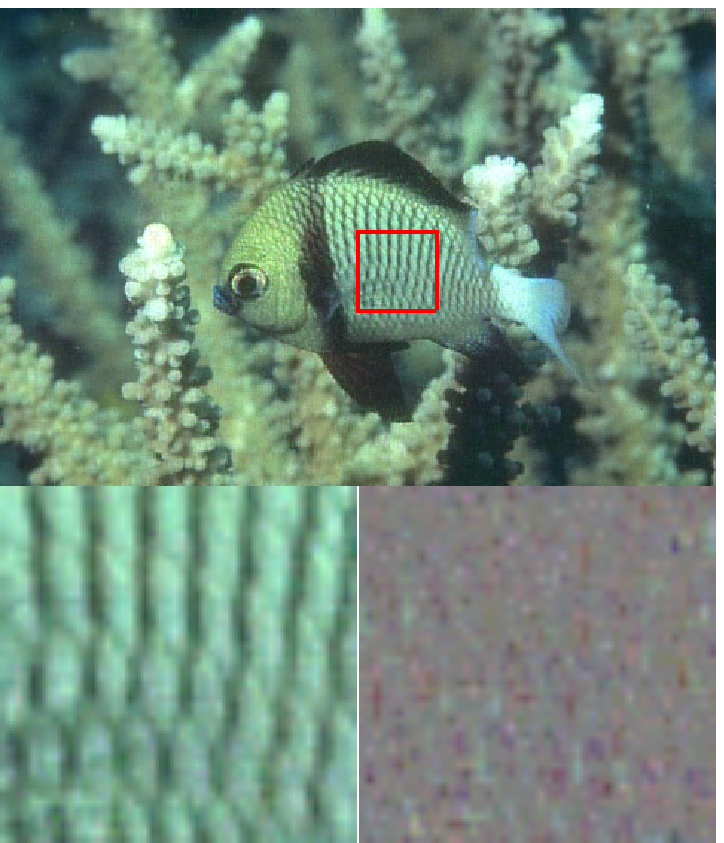}\vspace{0pt}
			\includegraphics[width=\linewidth]{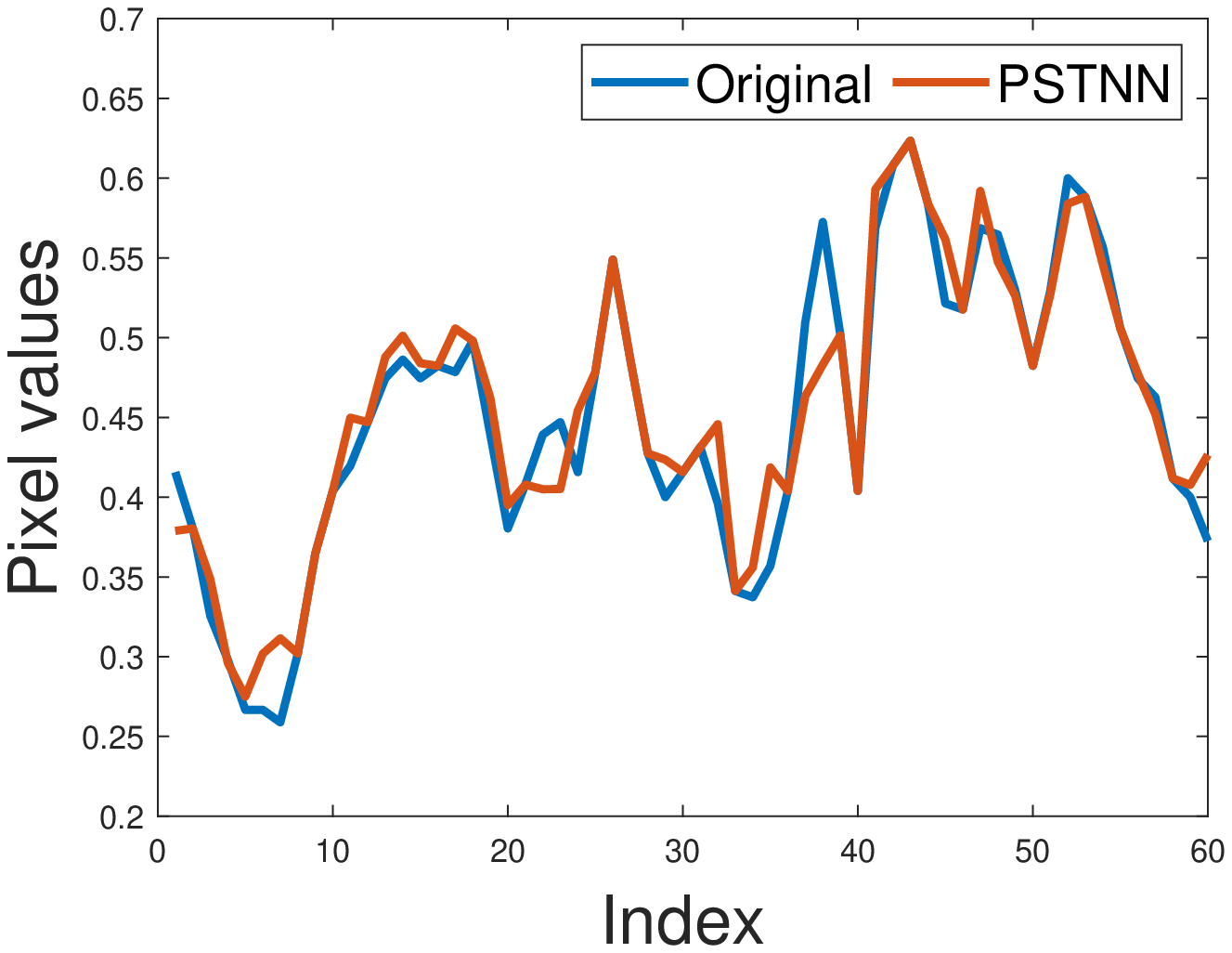}
			\caption{PSTNN}
		\end{subfigure}
		\begin{subfigure}[b]{0.138\linewidth}
			\centering			
			\includegraphics[width=\linewidth]{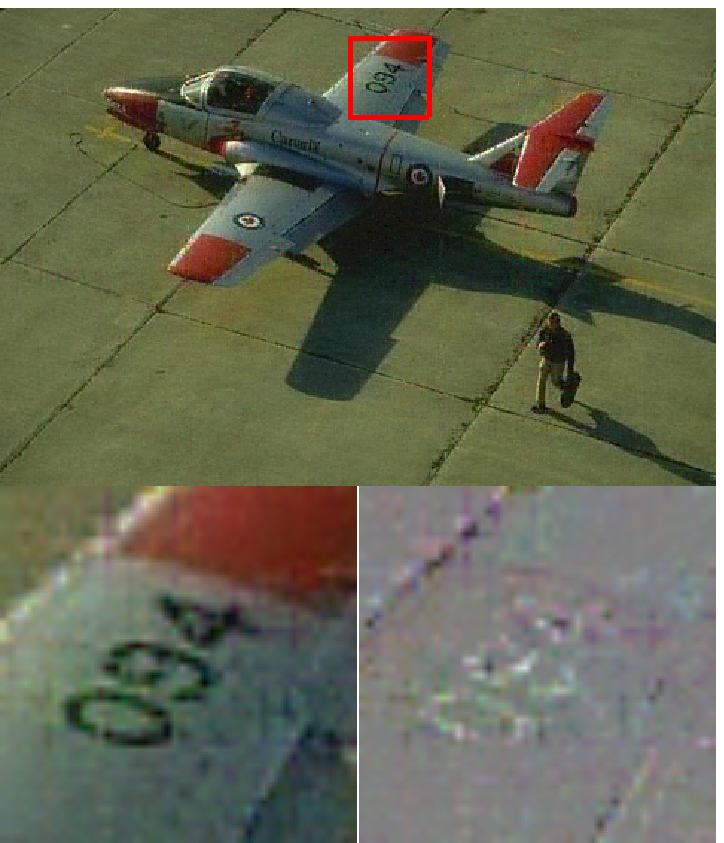}\vspace{0pt}
			\includegraphics[width=\linewidth]{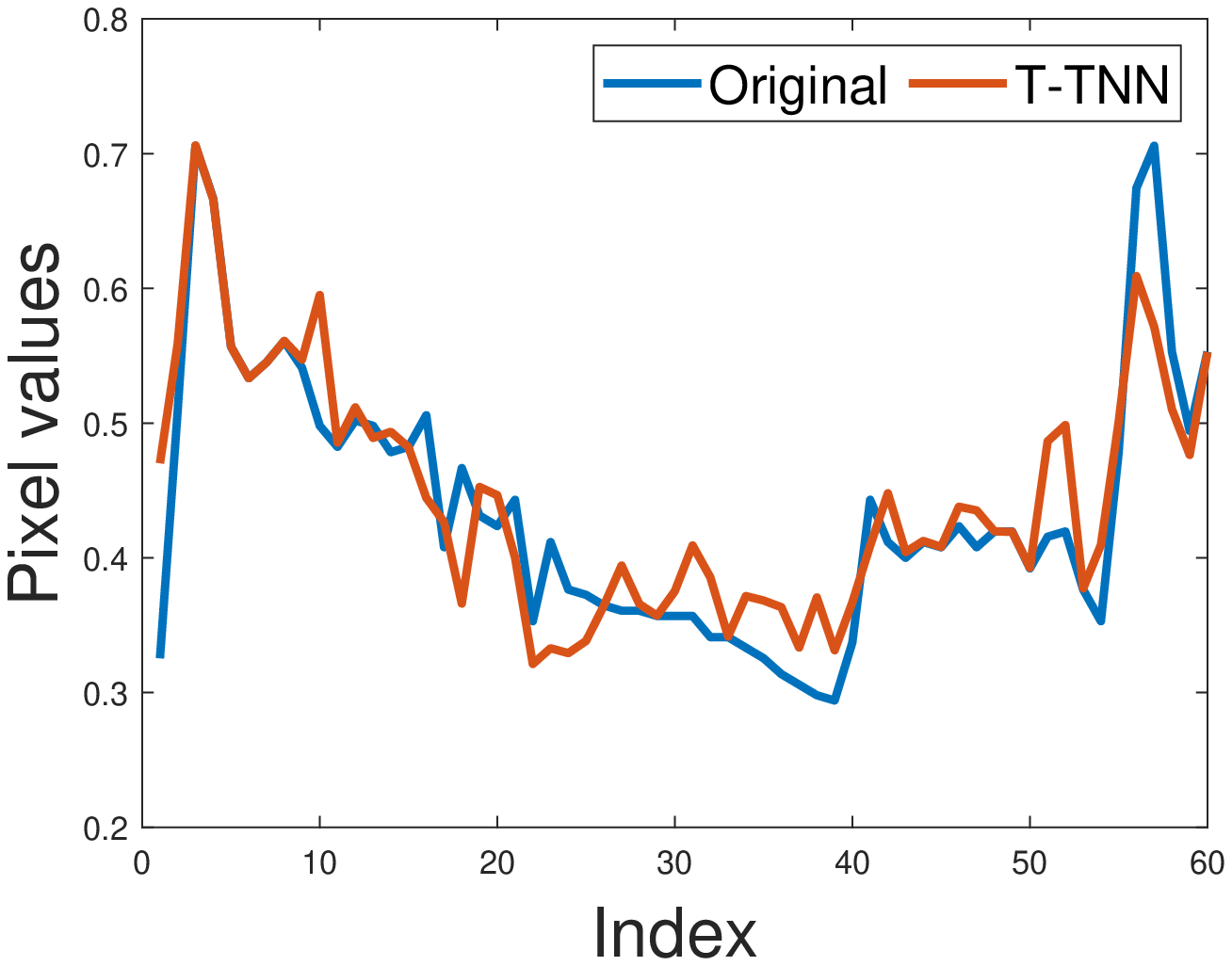}\vspace{0pt}
			\includegraphics[width=\linewidth]{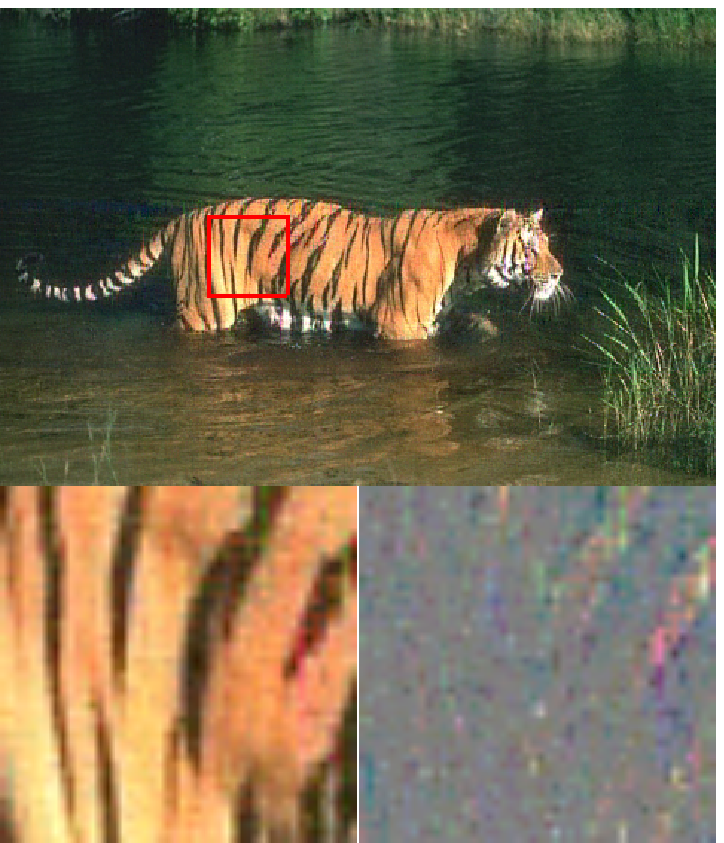}\vspace{0pt}
			\includegraphics[width=\linewidth]{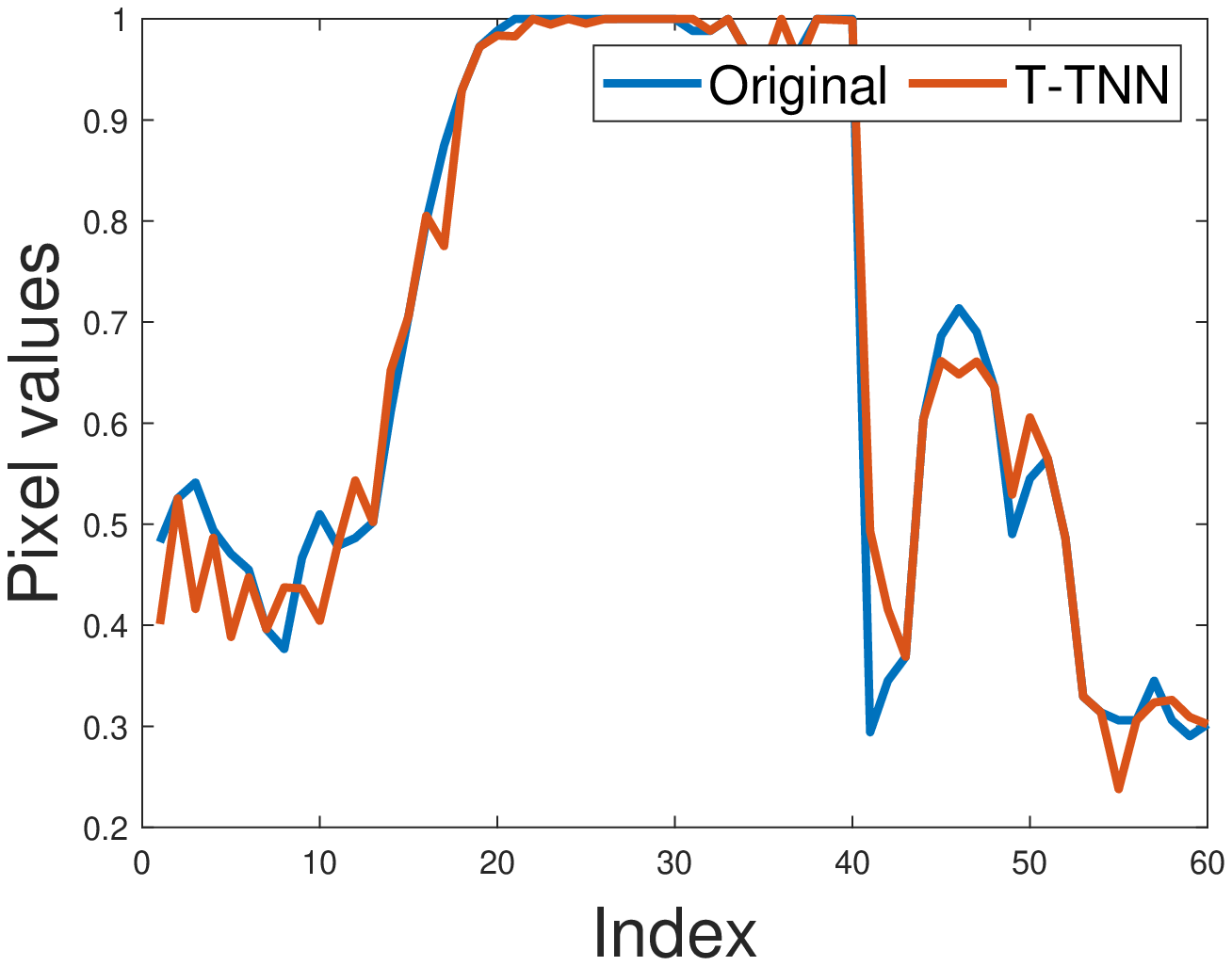}\vspace{0pt}
			\includegraphics[width=\linewidth]{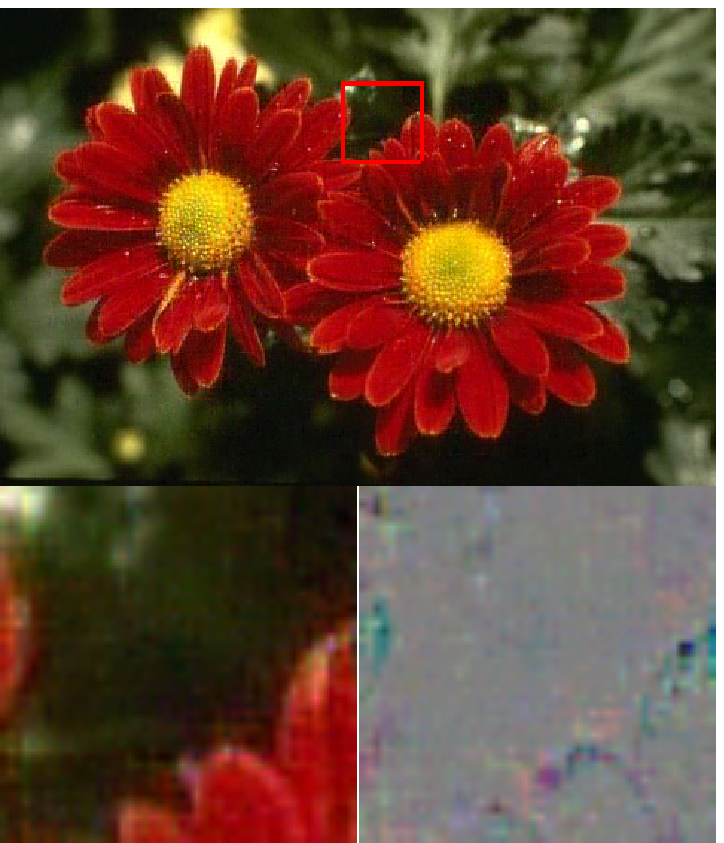}\vspace{0pt}
			\includegraphics[width=\linewidth]{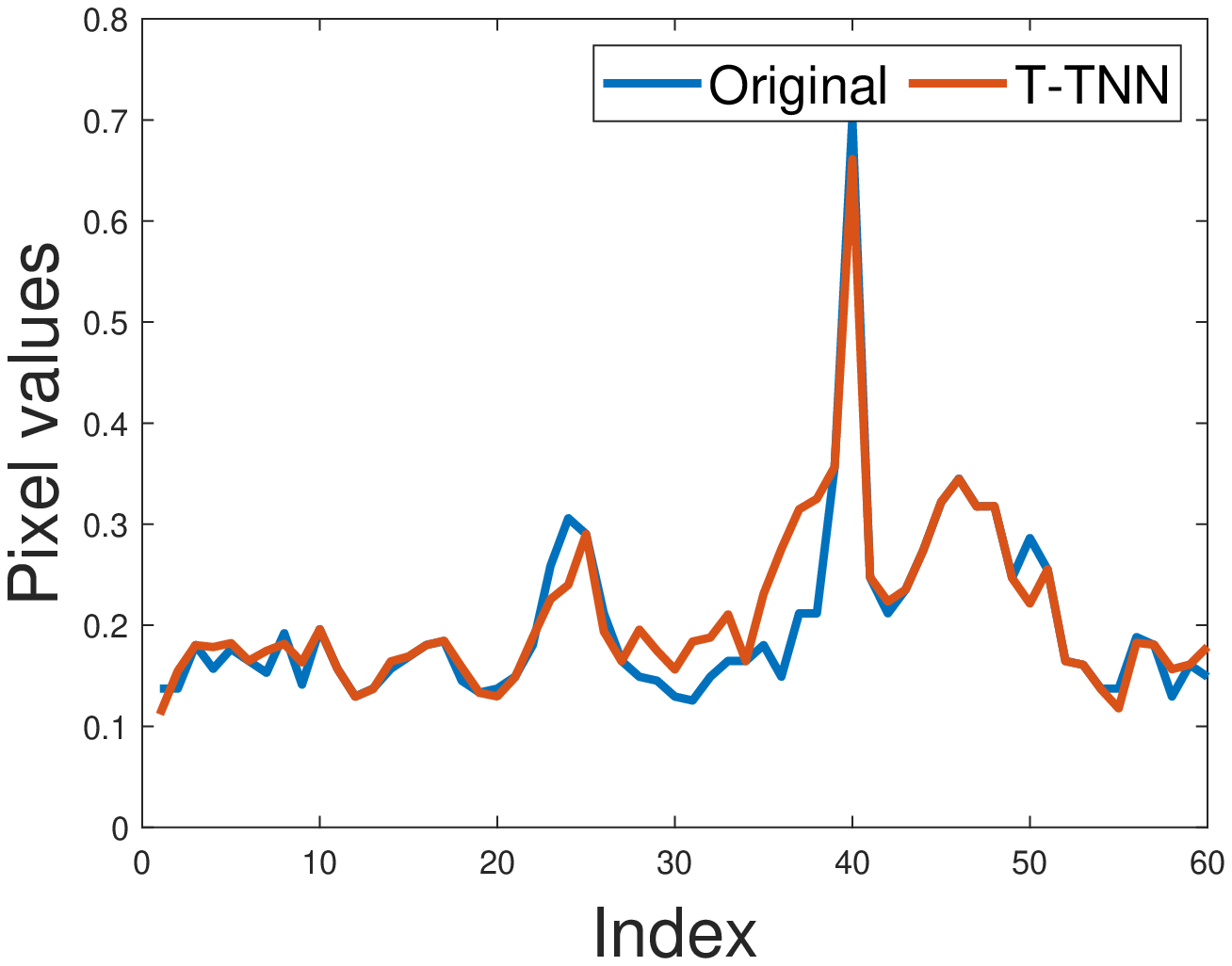}\vspace{0pt}
			\includegraphics[width=\linewidth]{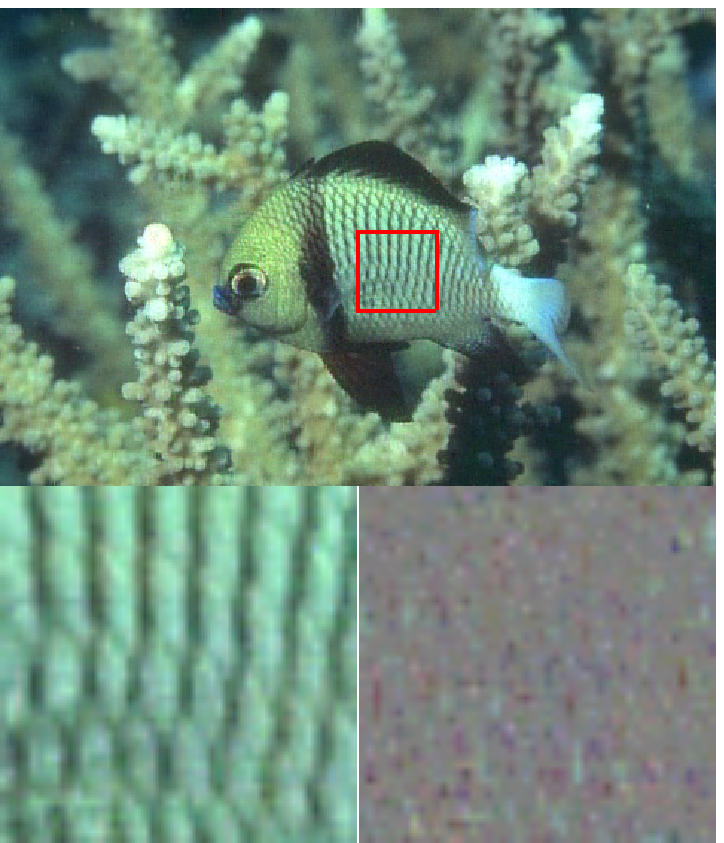}\vspace{0pt}
			\includegraphics[width=\linewidth]{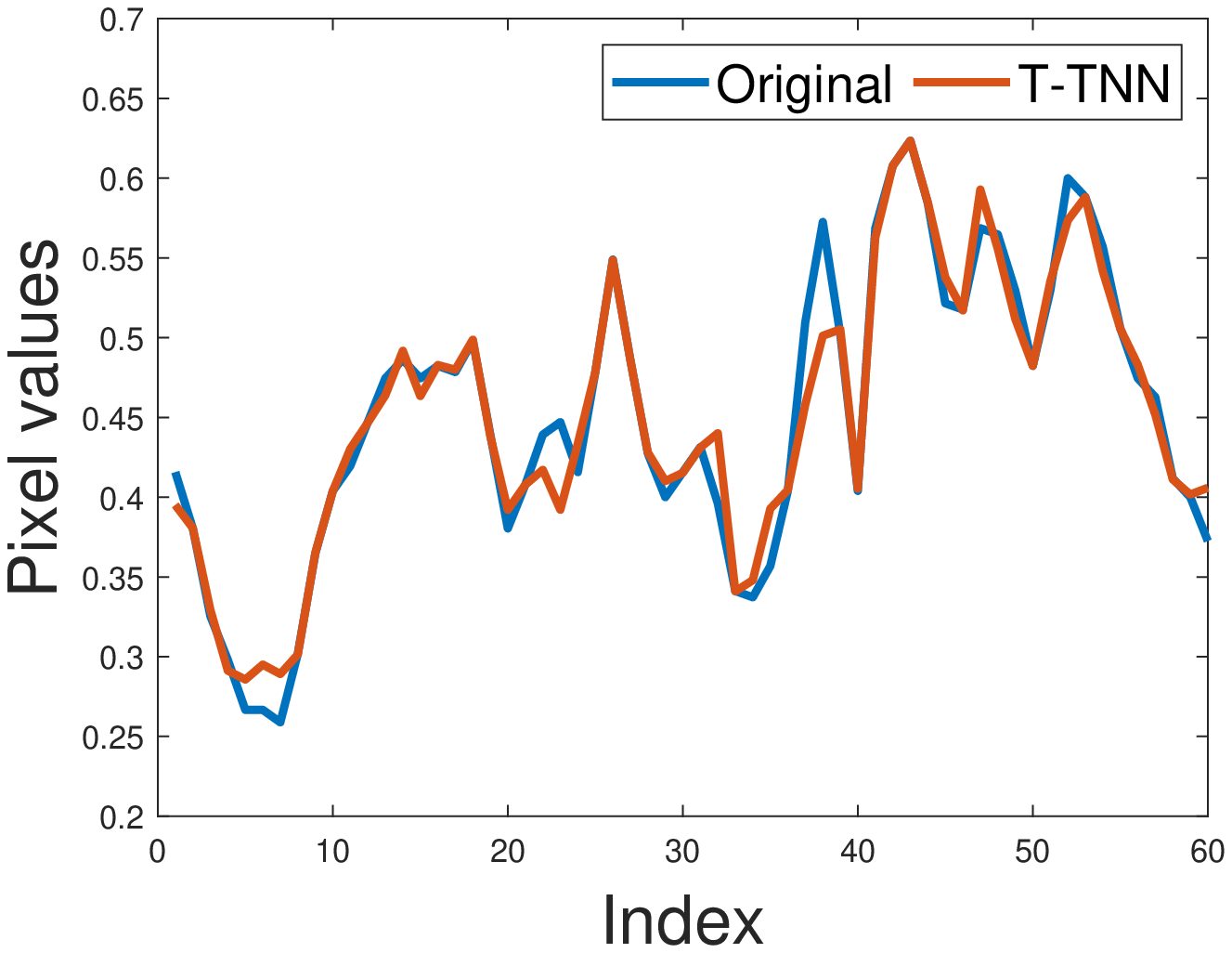}
			\caption{T-TNN}
		\end{subfigure}	
		\begin{subfigure}[b]{0.138\linewidth}
			\centering
			\includegraphics[width=\linewidth]{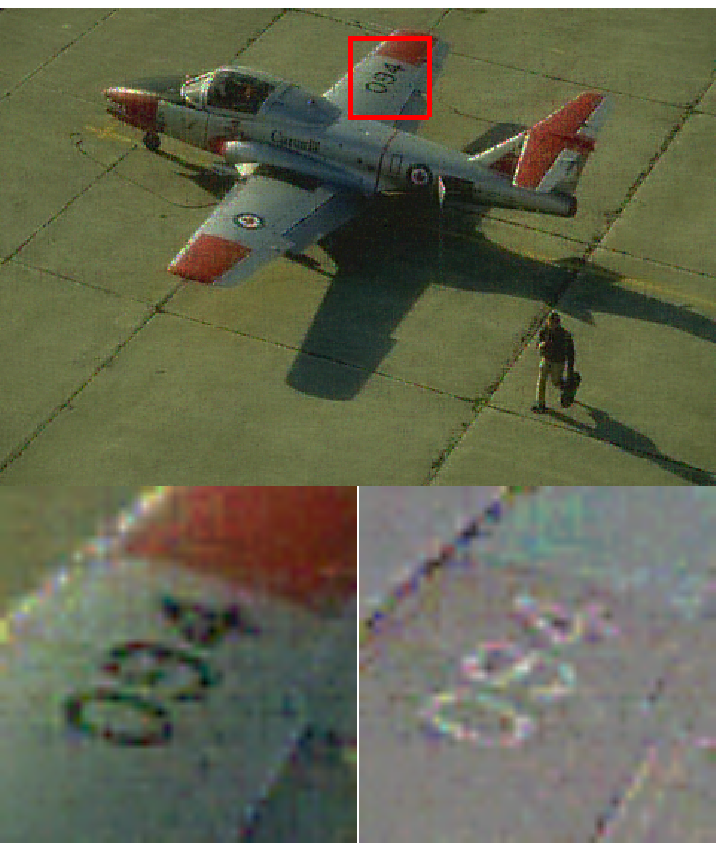}\vspace{0pt}
			\includegraphics[width=\linewidth]{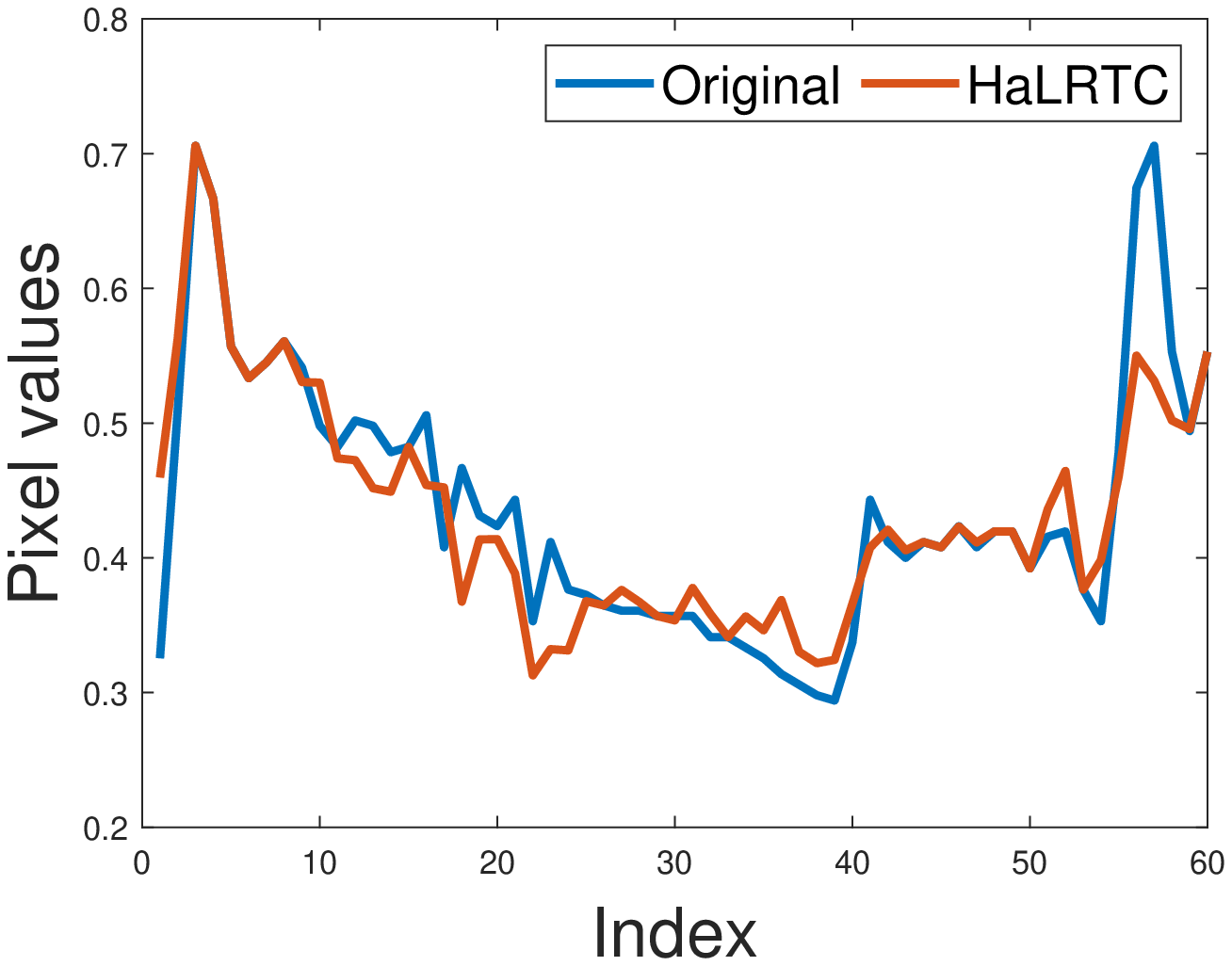}\vspace{0pt}
			\includegraphics[width=\linewidth]{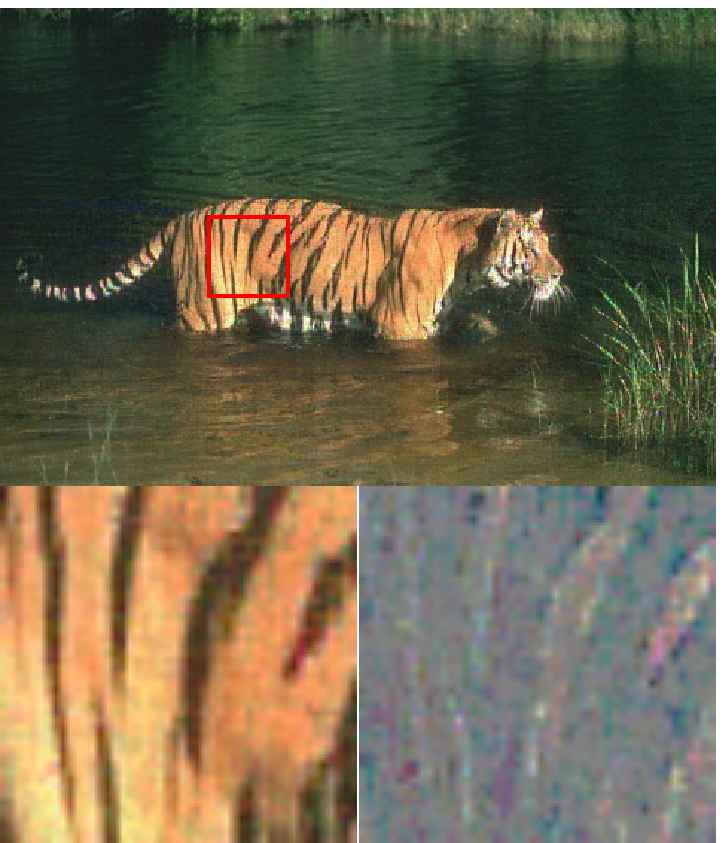}\vspace{0pt}
			\includegraphics[width=\linewidth]{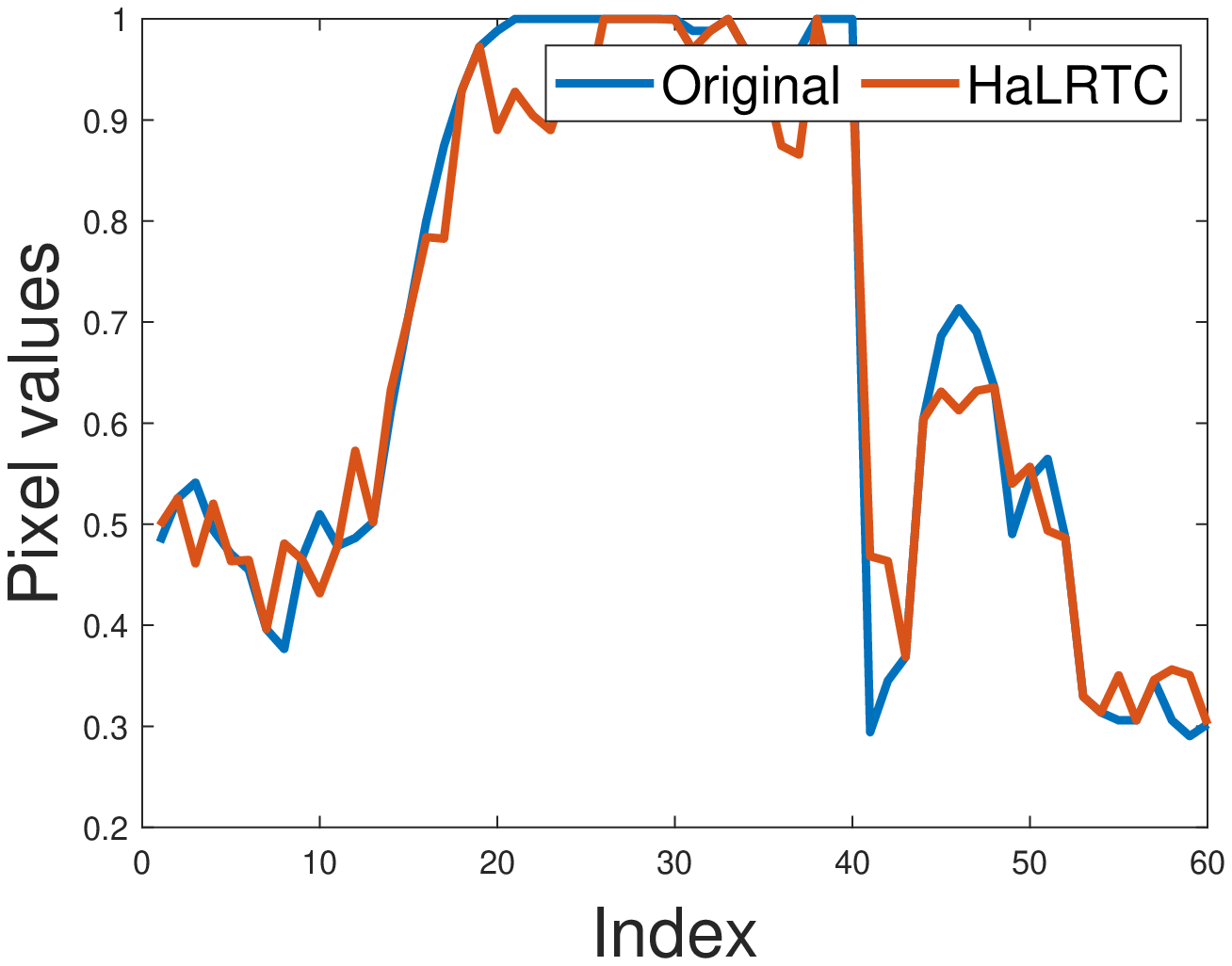}\vspace{0pt}
			\includegraphics[width=\linewidth]{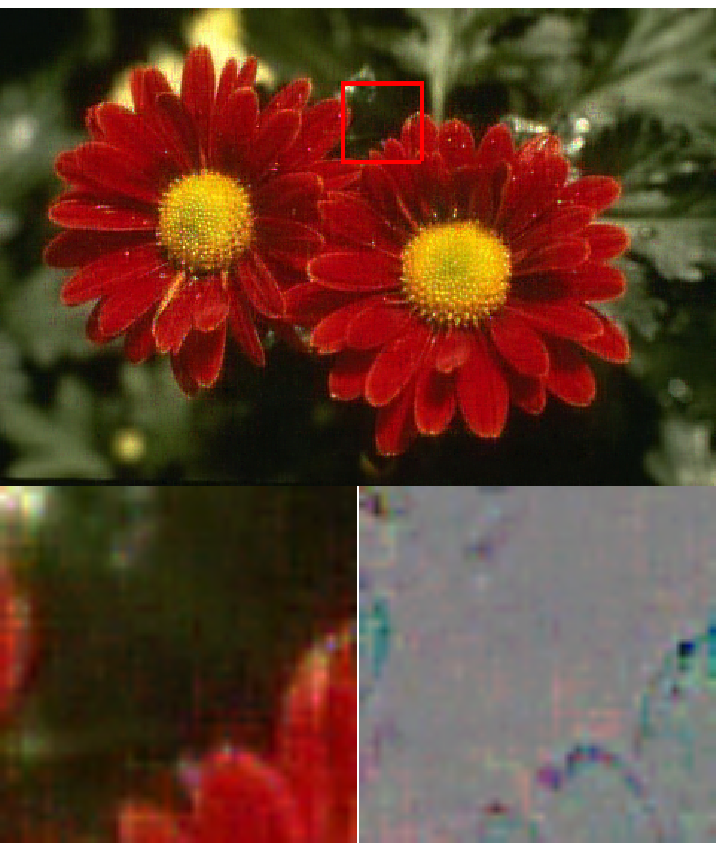}\vspace{0pt}
			\includegraphics[width=\linewidth]{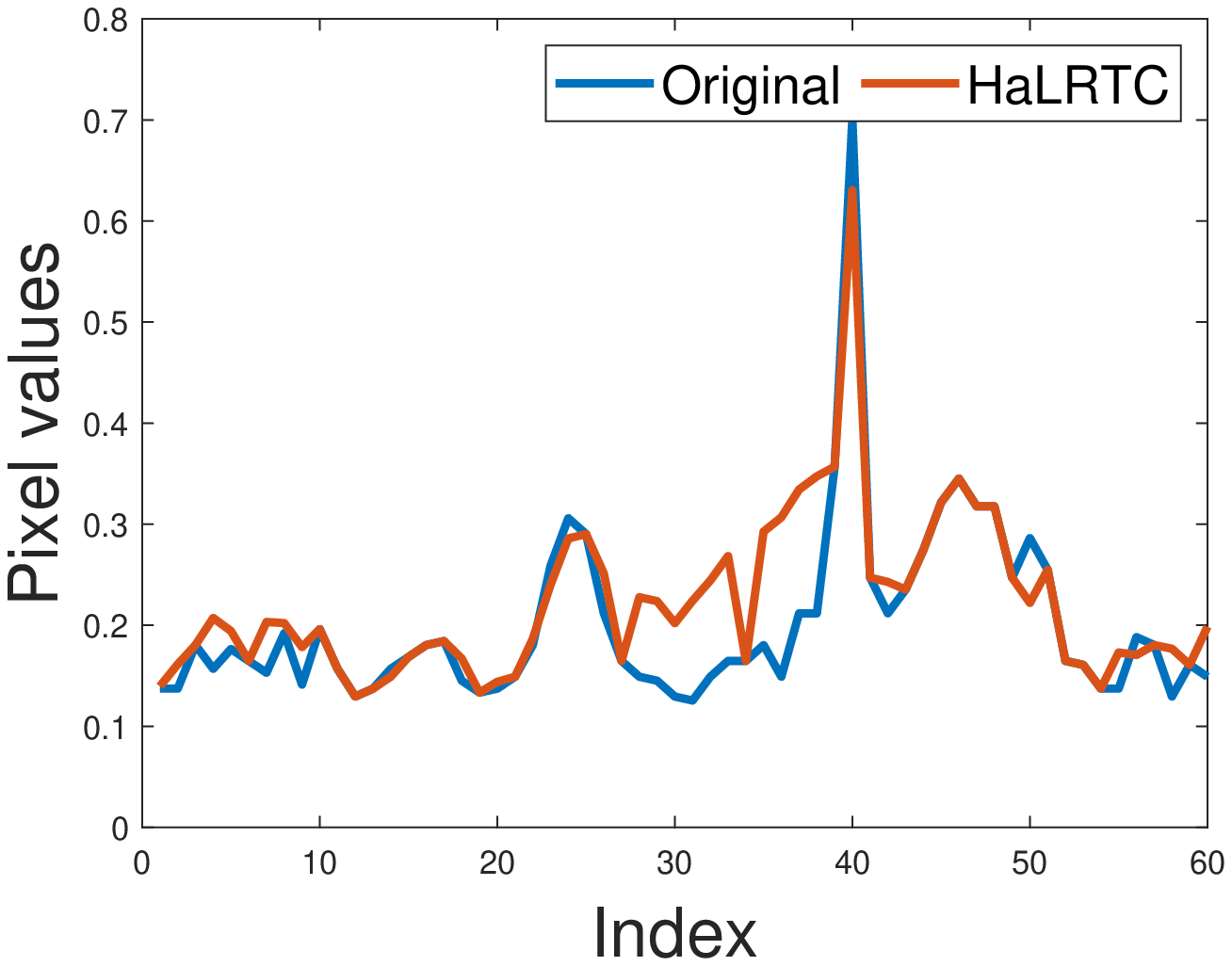}\vspace{0pt}
			\includegraphics[width=\linewidth]{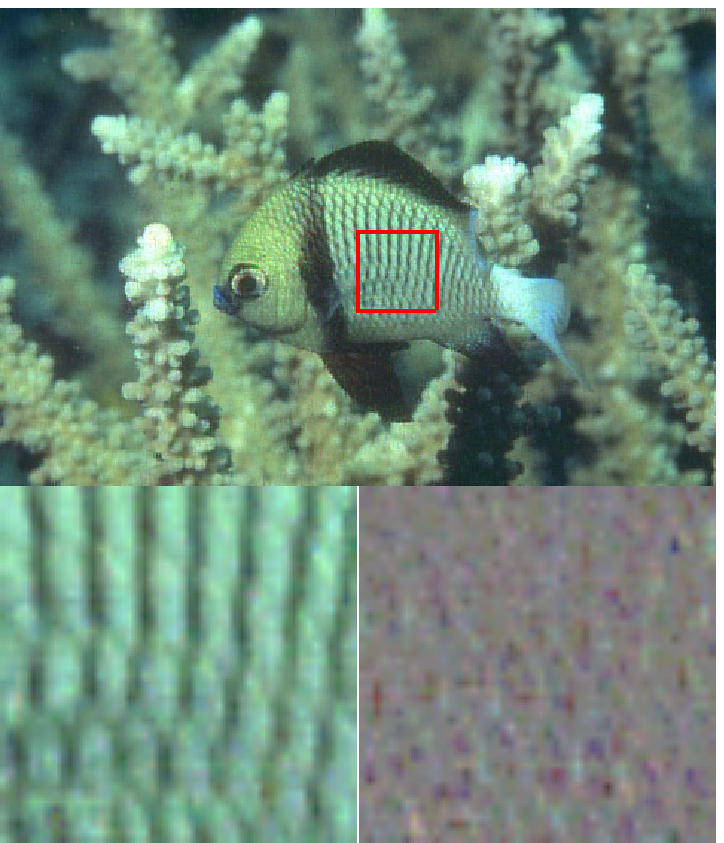}\vspace{0pt}
			\includegraphics[width=\linewidth]{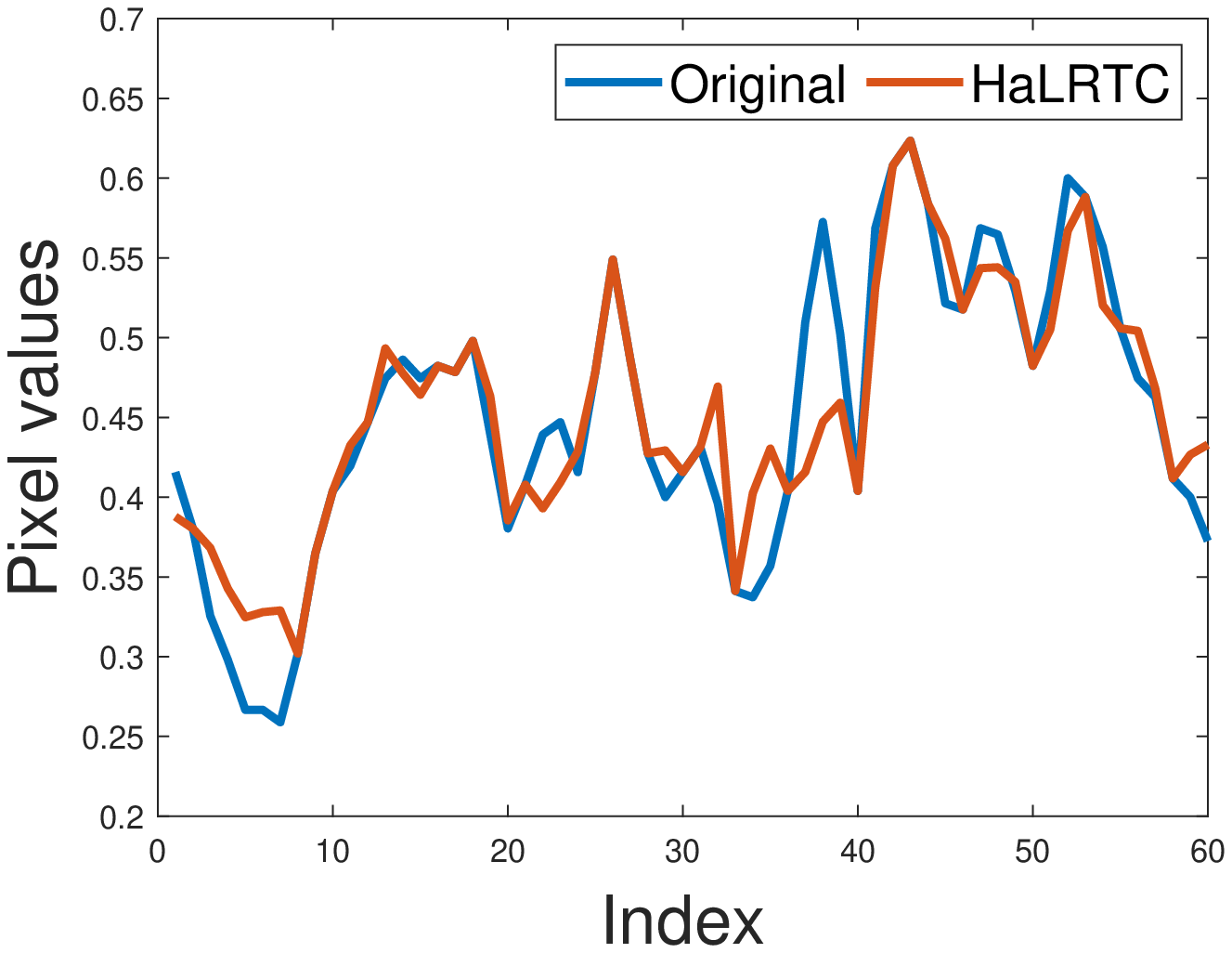}
			\caption{HaLRTC}
		\end{subfigure}					
	\end{subfigure}
	\vfill
	\caption{Examples of color image inpainting with $ SR=40\% $. From top to bottom are  ``Airplane", ``Tiger", ``Flower" and ``Fish". For better visualization, we show the zoom-in region, the corresponding error map (difference from the original) and the corresponding partial residuals of the region.}
	\label{fig:colorimage}
\end{figure*}

\subsection{Texture Image Inpainting}
In this subsection, we use the PSU near-regular texture database\footnote{http://vivid.cse.psu.edu/.} to evaluate our proposed method TNNR for texture image inpainting. Each texture image is resized to $ 300 \times 400 \times 3 $ in order to improve the computation efficiency.
In our test, four texture images are randomly selected from this database, including ``Stone", ``Pattern", ``Leaves" and ``Barbed Wire". 
The data of images are normalized in the range $ \left[0,1\right] $. The sampling rates (SR) are set as $ 30\% $, $ 35\% $ and $ 40\% $. 

Table \ref{tab:Texture} lists PSNR, SSIM, FSIM and the corresponding running times under different methods. For different sampling rates, our TNNR obtains the results with the best quantitative metrics. In particular, TNNR consumes the least running time among methods except the HaLRTC method. Additionally, HaLRTC consumes the least running time, but has poor quantitative metrics. To further demonstrate the recovery performance, samples of texture images recovered by different algorithms are shown in Figure \ref{fig:Texture}. From the recovery results, our method performs better in filling the missing values of the four testing texture images. It can deal with the the edges of patterns better.

\begin{figure*}[htbp]
	\centering
	\begin{subfigure}[b]{1\linewidth}
		\begin{subfigure}[b]{0.138\linewidth}
			\centering
			\includegraphics[width=\linewidth]{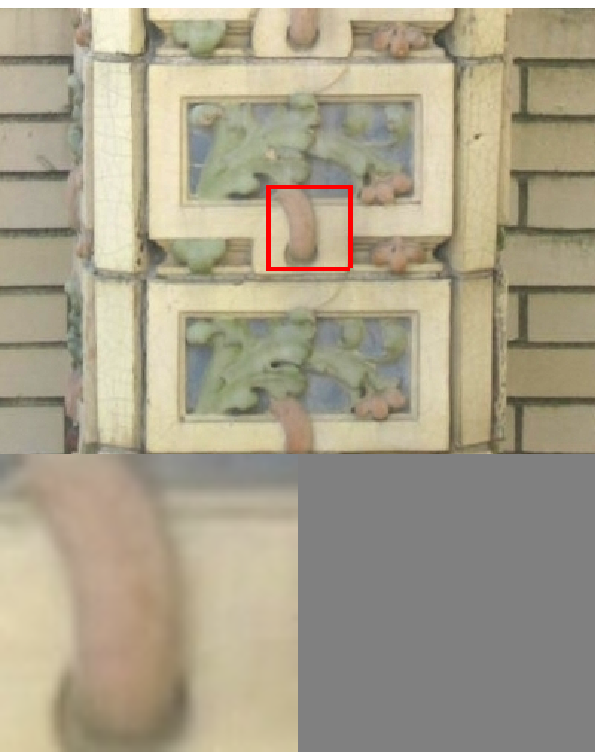}\vspace{0pt}
			\includegraphics[width=\linewidth]{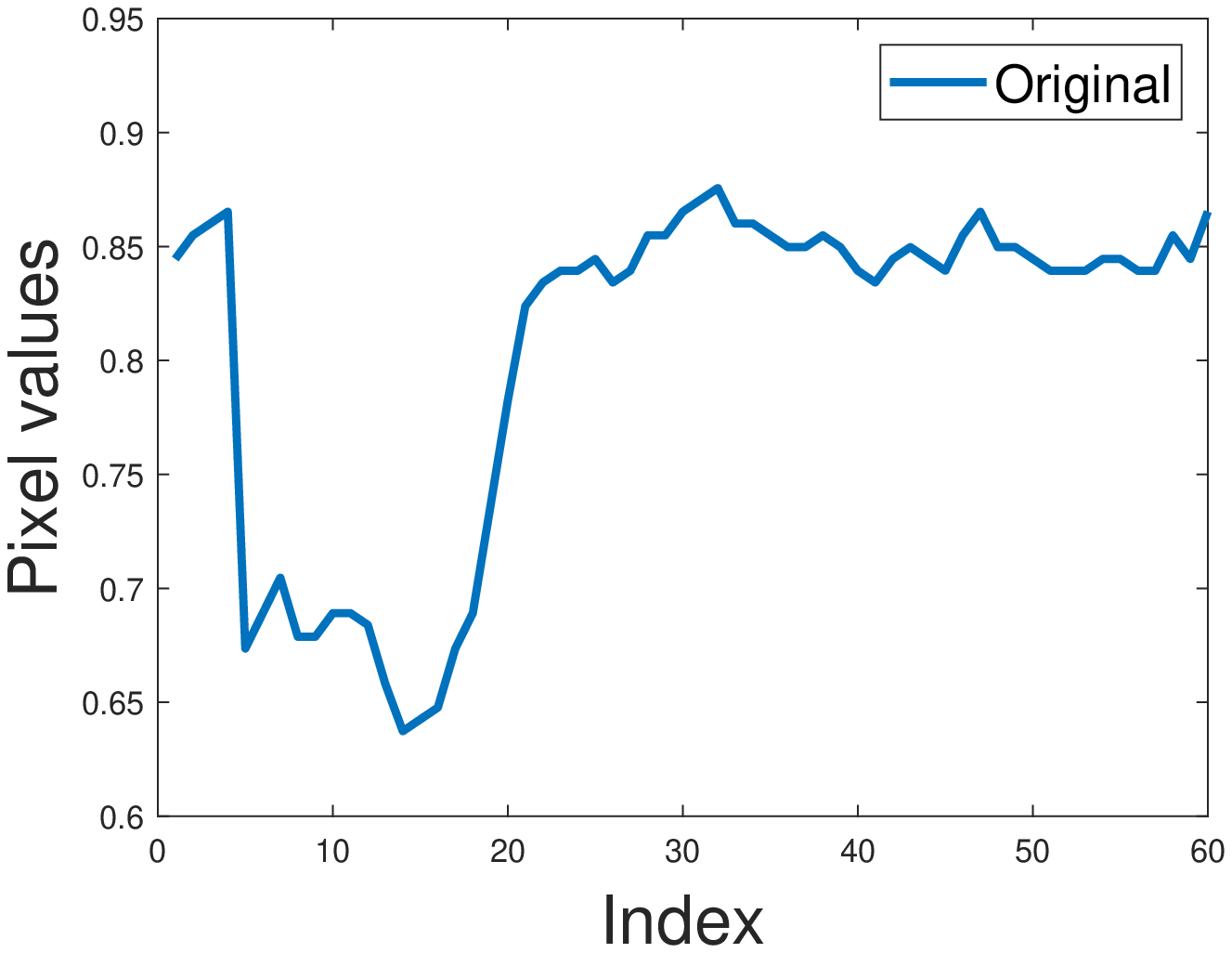}\vspace{0pt}
			\includegraphics[width=\linewidth]{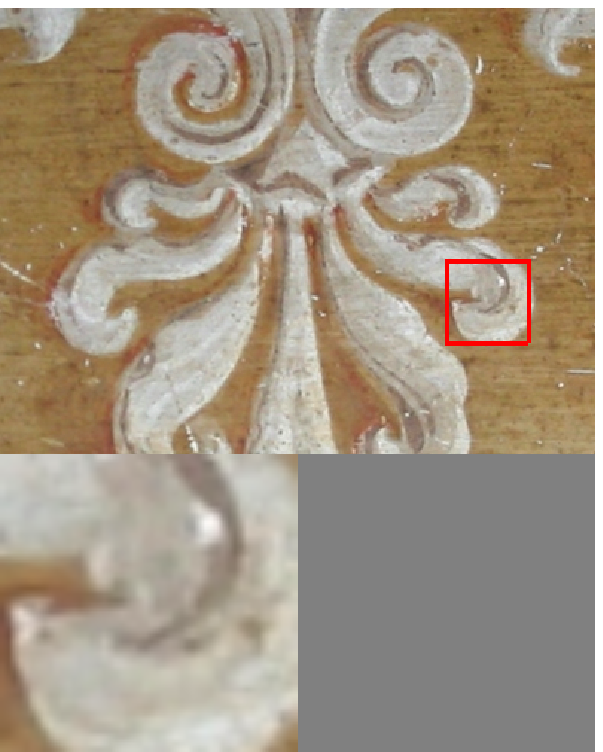}\vspace{0pt}
			\includegraphics[width=\linewidth]{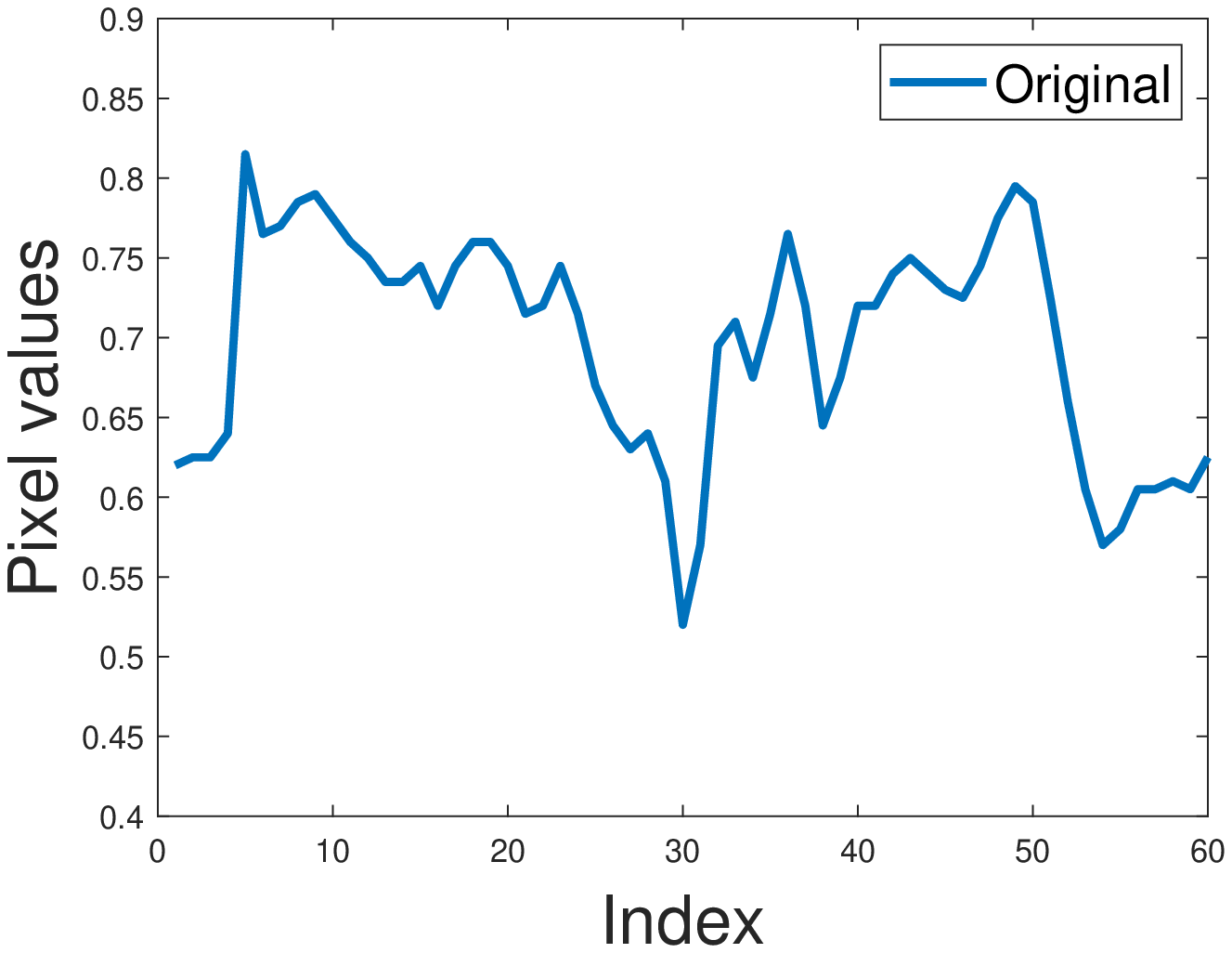}\vspace{0pt}
			\includegraphics[width=\linewidth]{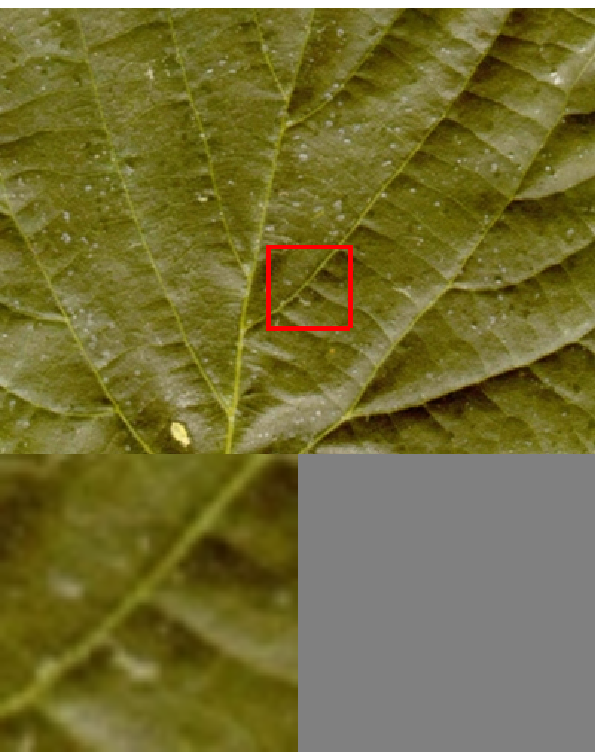}\vspace{0pt}
			\includegraphics[width=\linewidth]{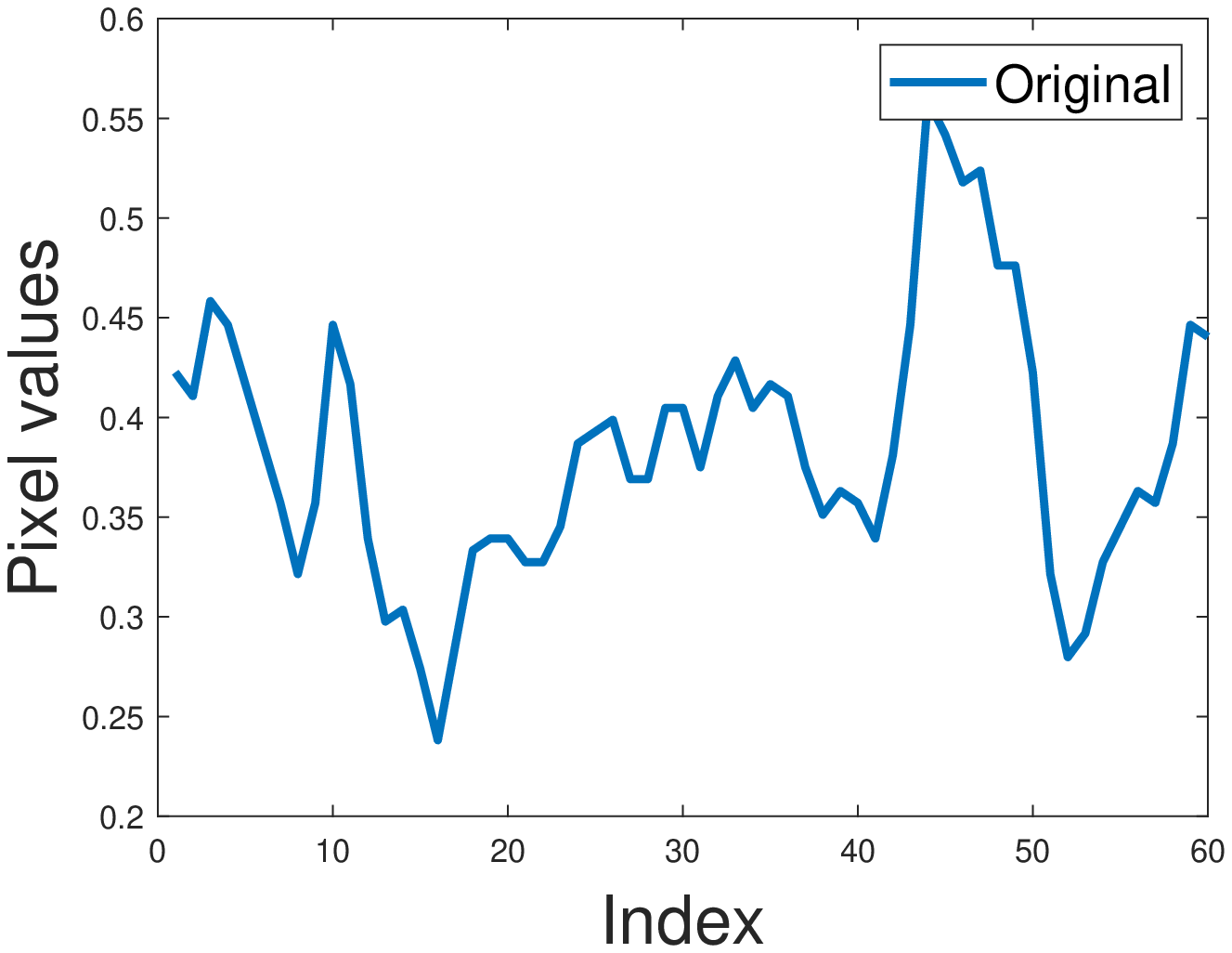}\vspace{0pt}
			\includegraphics[width=\linewidth]{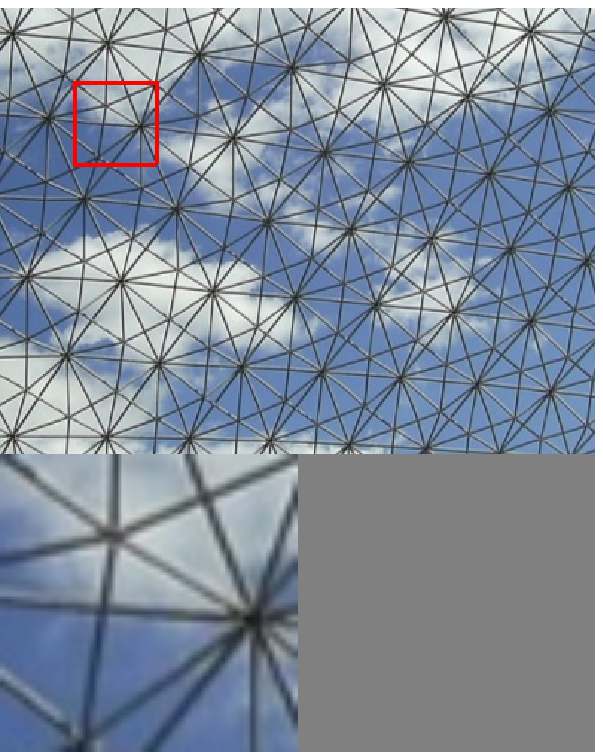}\vspace{0pt}
			\includegraphics[width=\linewidth]{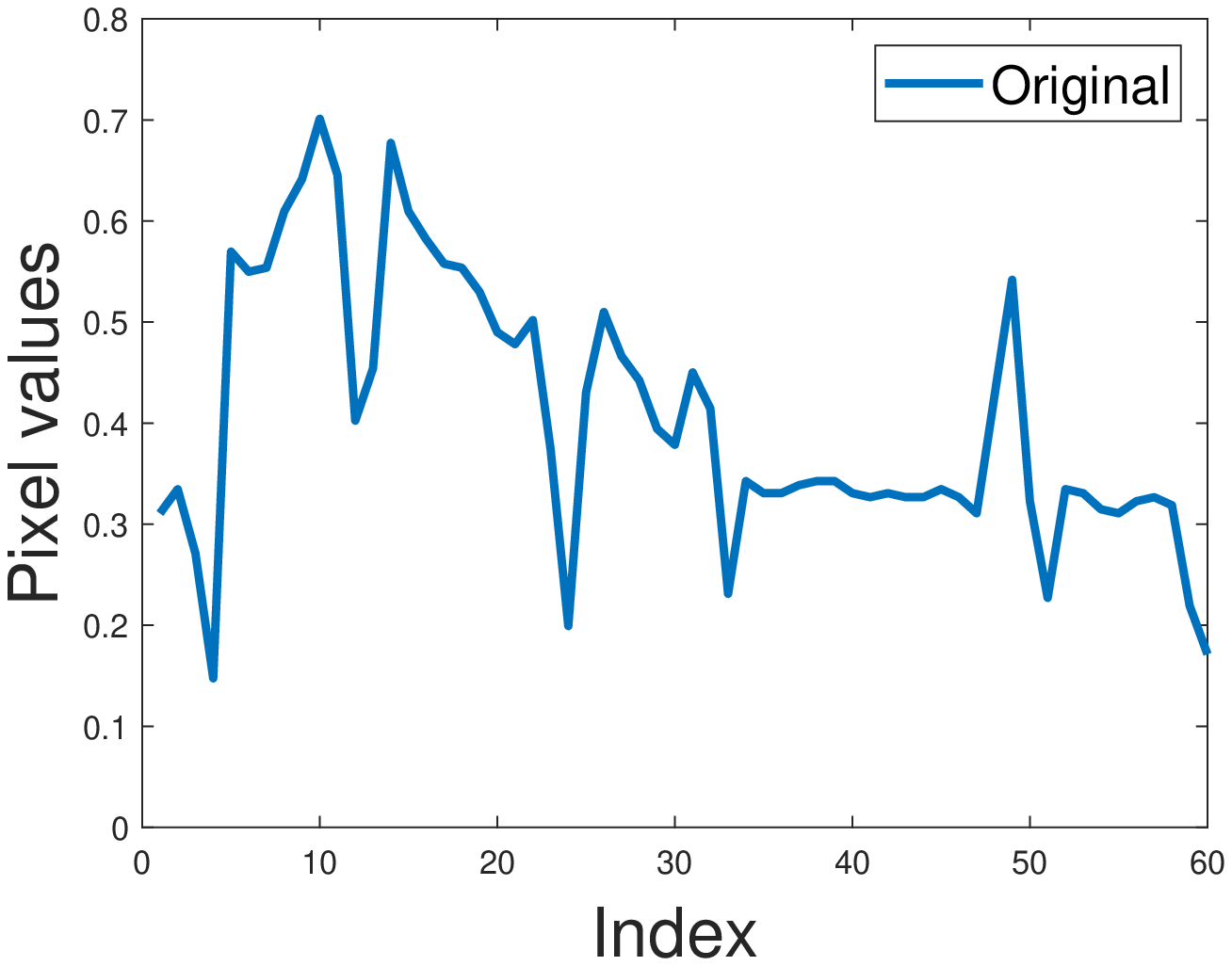}
			\caption{Original}
		\end{subfigure}   	
		\begin{subfigure}[b]{0.138\linewidth}
			\centering
			\includegraphics[width=\linewidth]{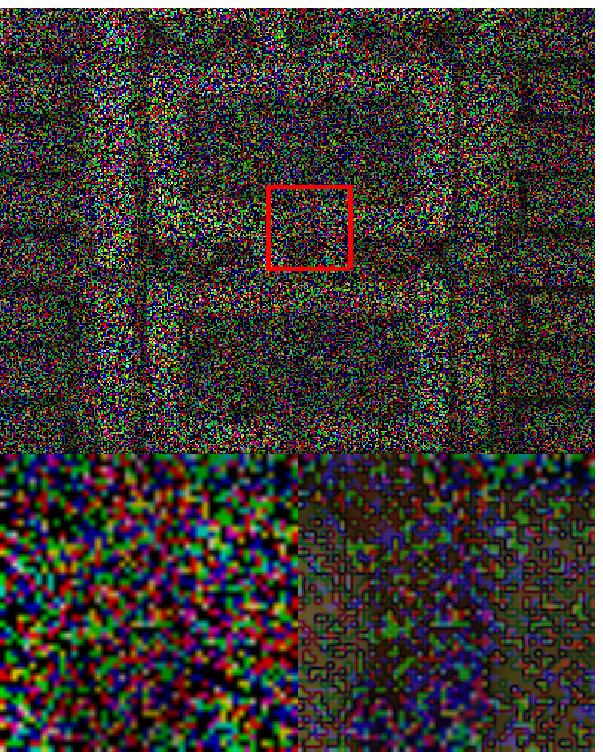}\vspace{0pt}
			\includegraphics[width=\linewidth]{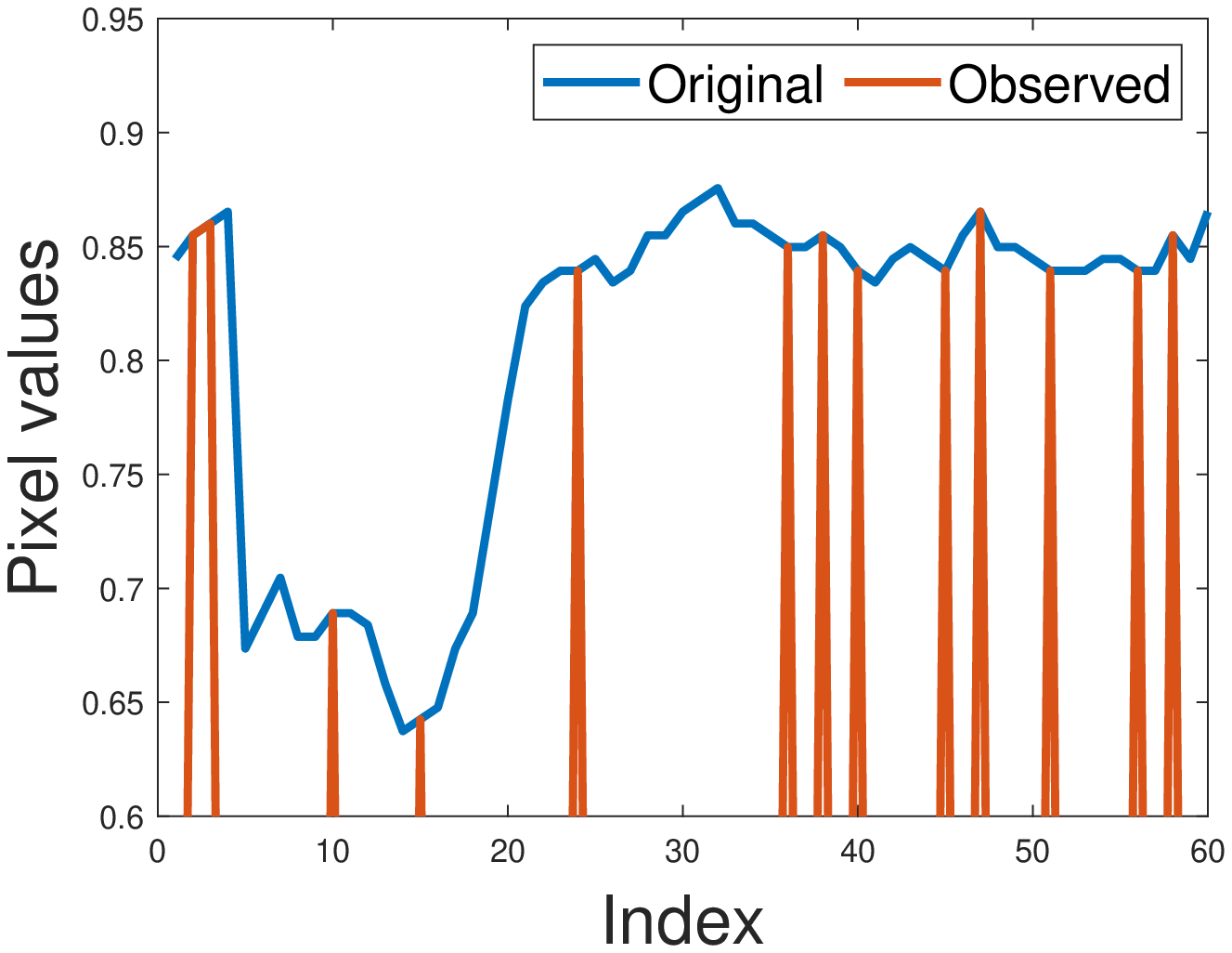}\vspace{0pt}
			\includegraphics[width=\linewidth]{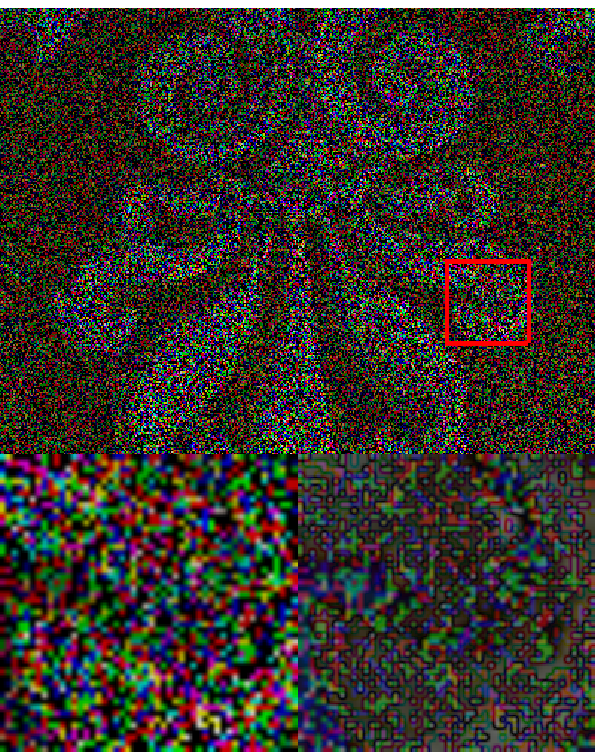}\vspace{0pt}
			\includegraphics[width=\linewidth]{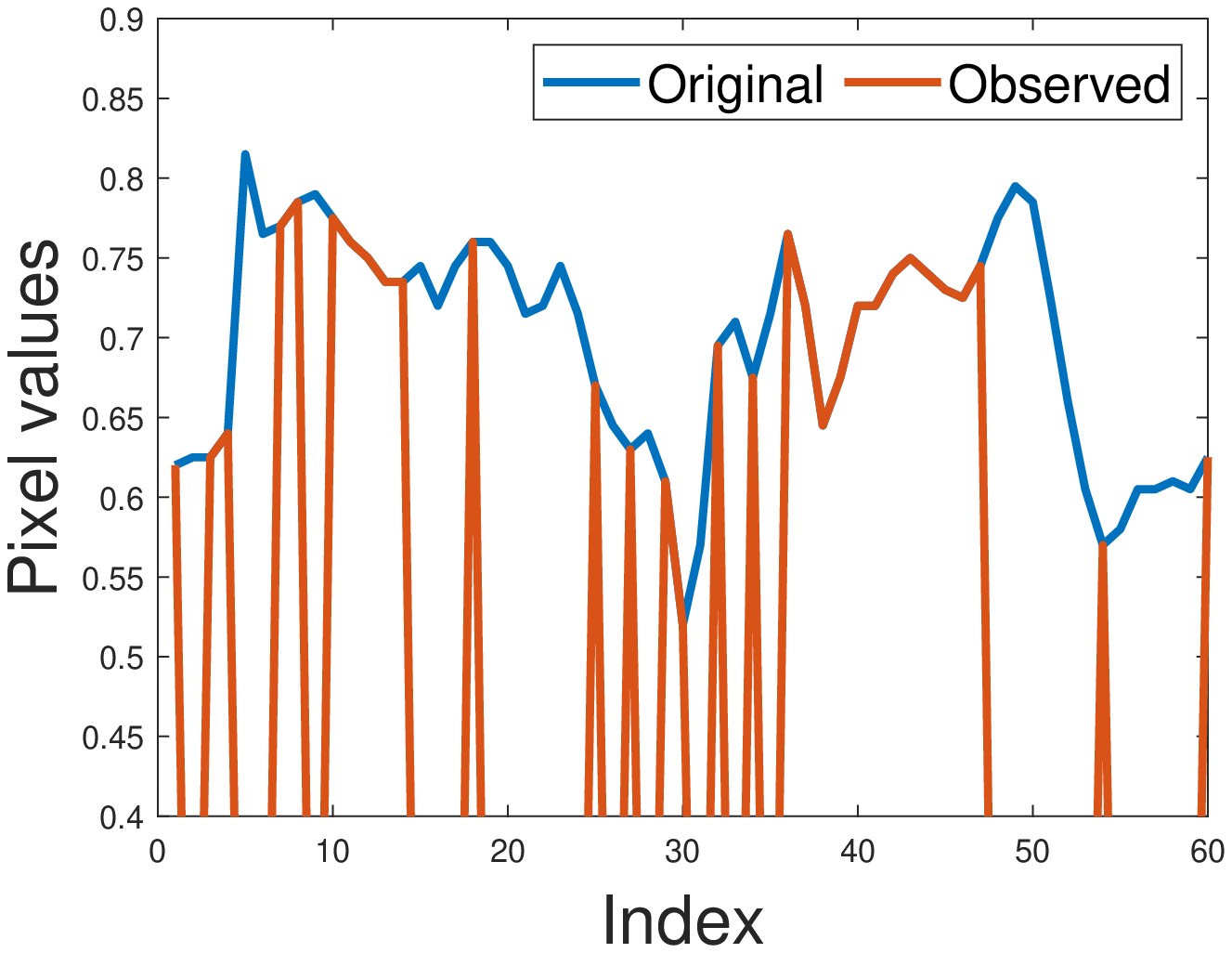}\vspace{0pt}
			\includegraphics[width=\linewidth]{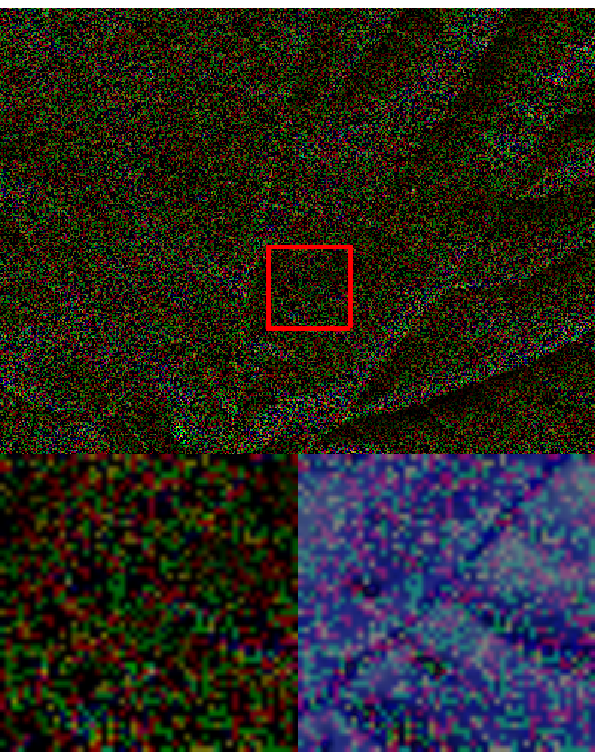}\vspace{0pt}
			\includegraphics[width=\linewidth]{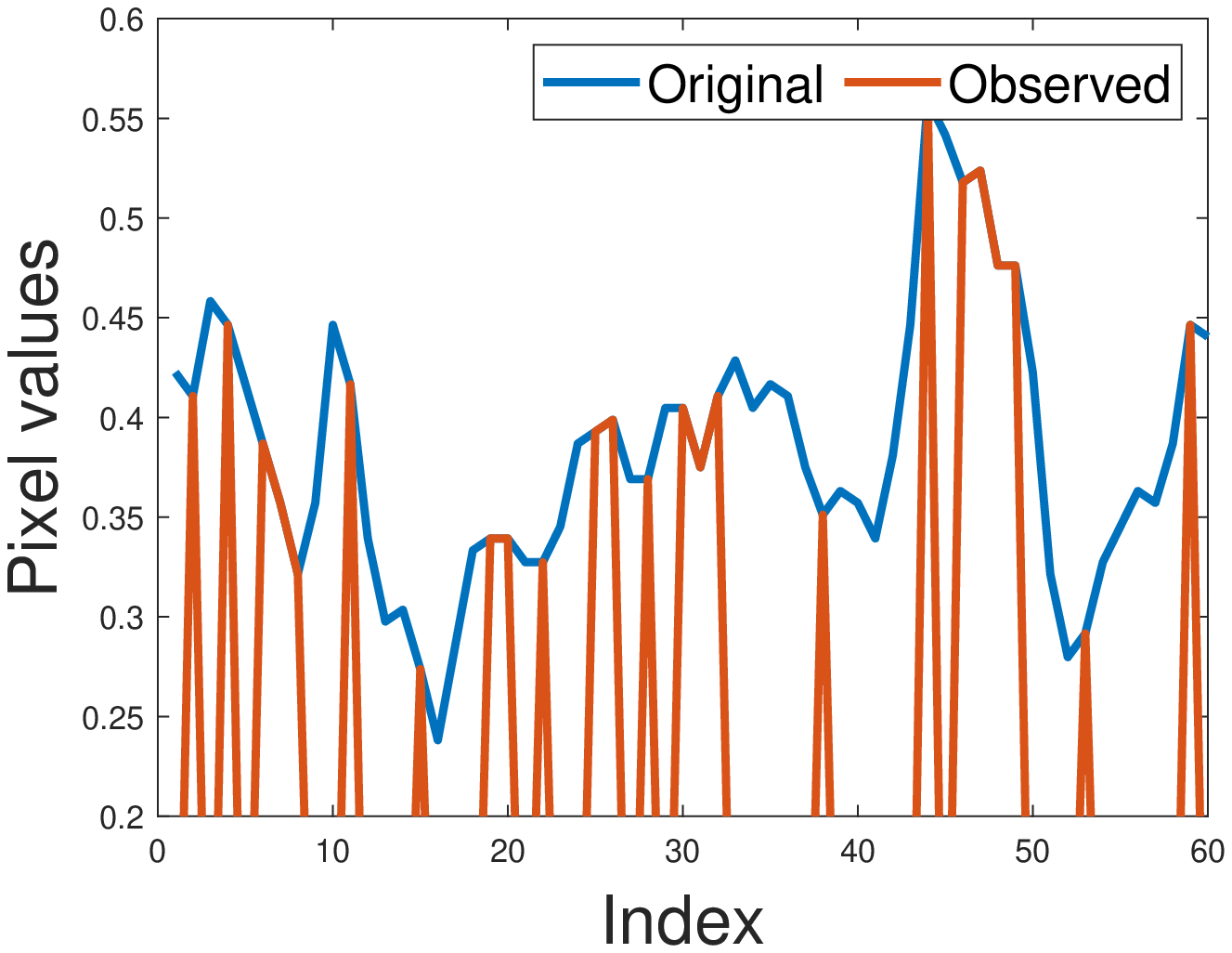}\vspace{0pt}
			\includegraphics[width=\linewidth]{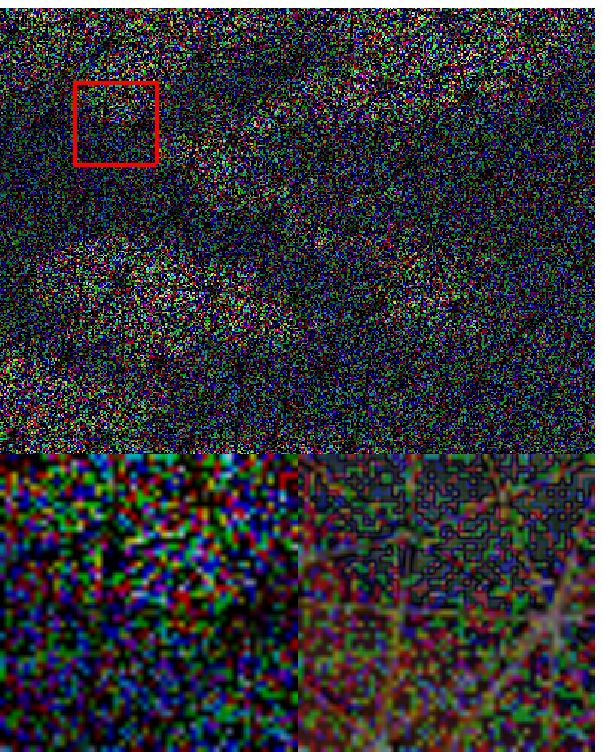}\vspace{0pt}
			\includegraphics[width=\linewidth]{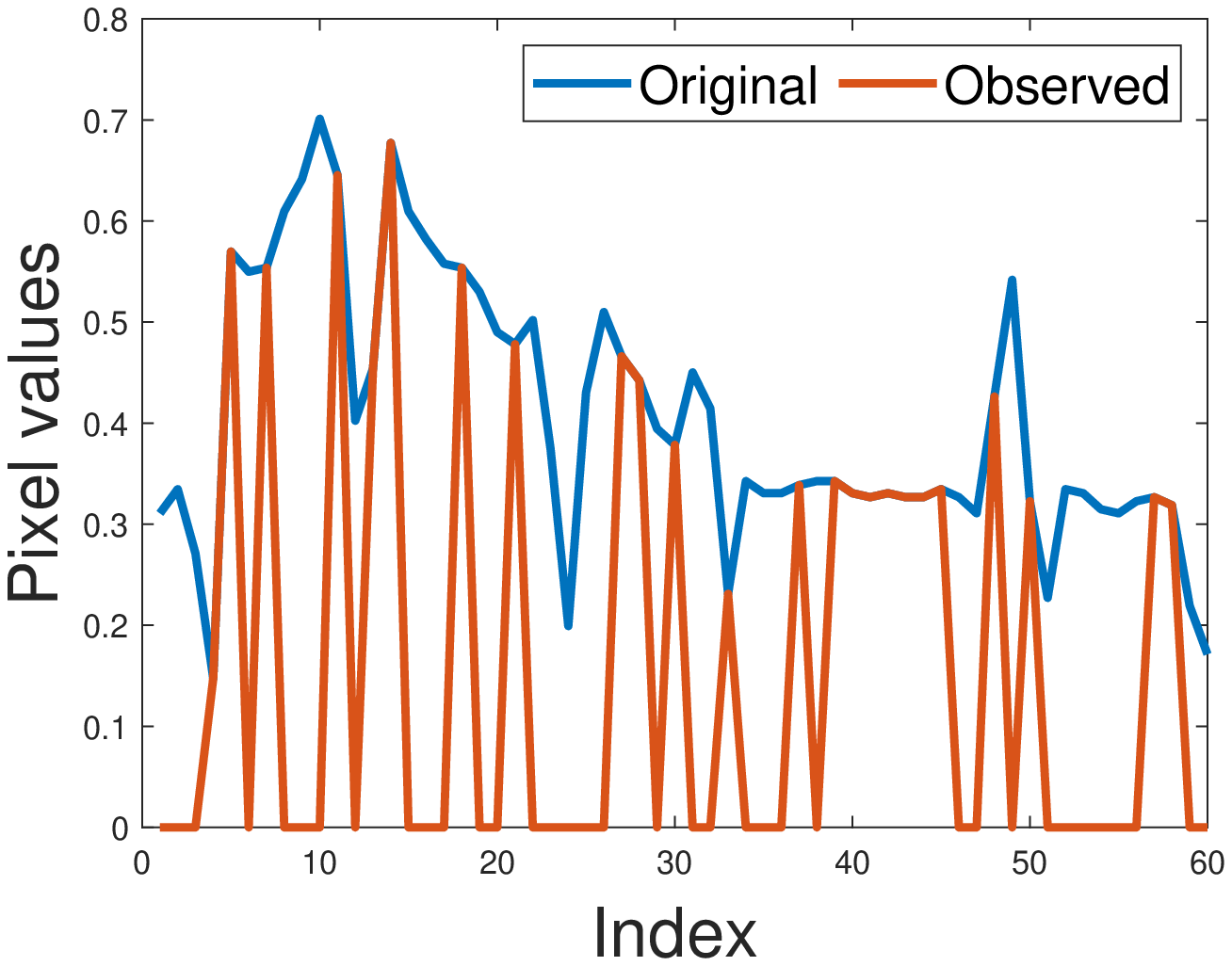}
			\caption{Observed}
		\end{subfigure}
		\begin{subfigure}[b]{0.138\linewidth}
			\centering
			\includegraphics[width=\linewidth]{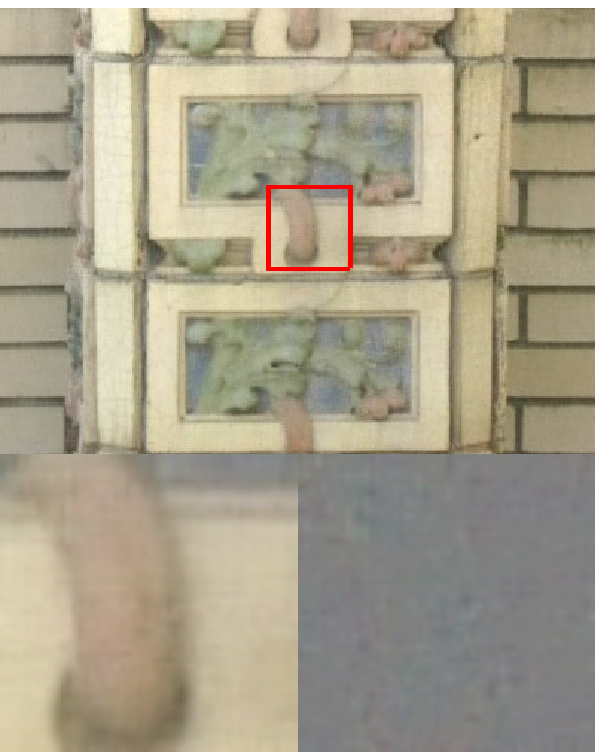}\vspace{0pt}
			\includegraphics[width=\linewidth]{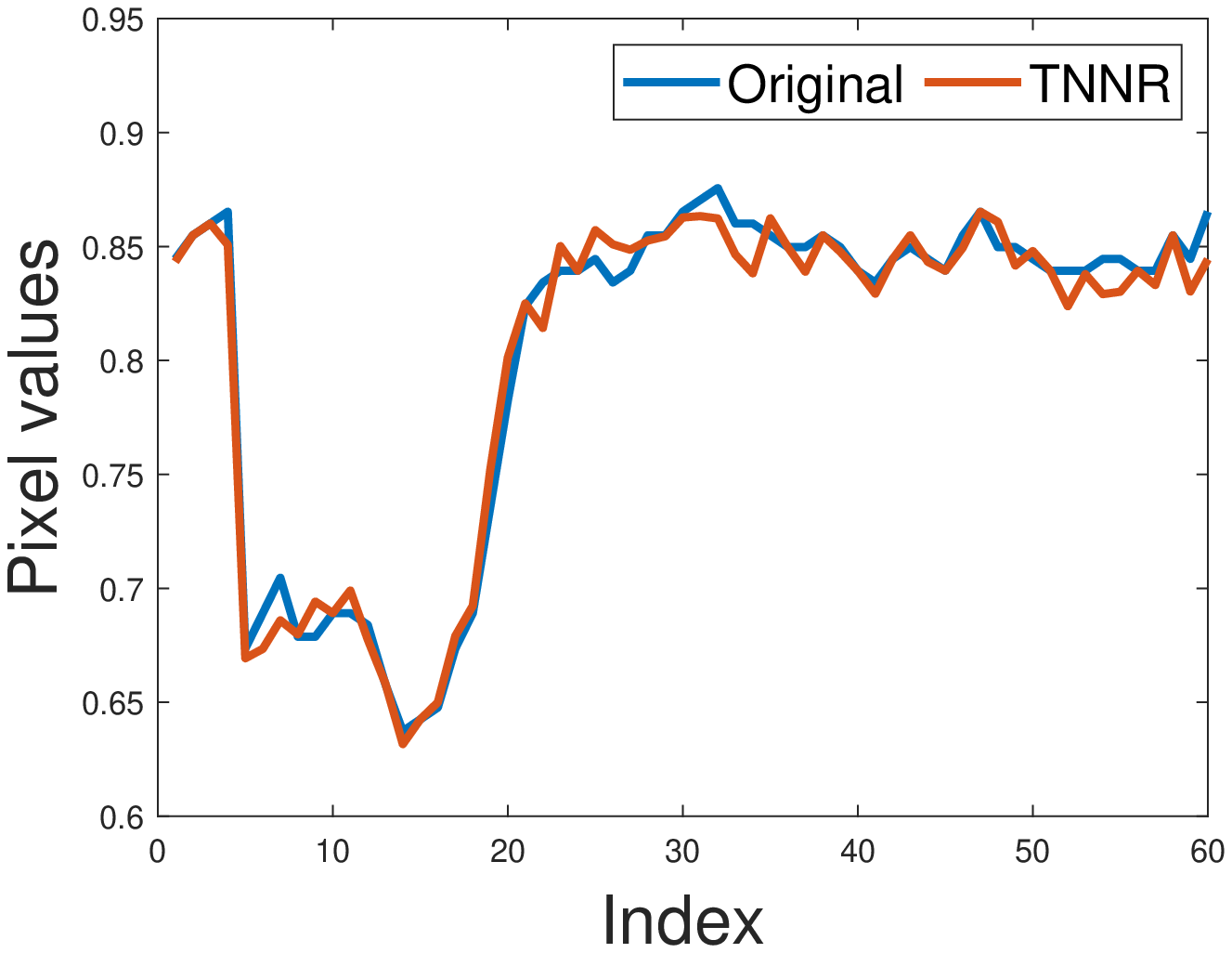}\vspace{0pt}
			\includegraphics[width=\linewidth]{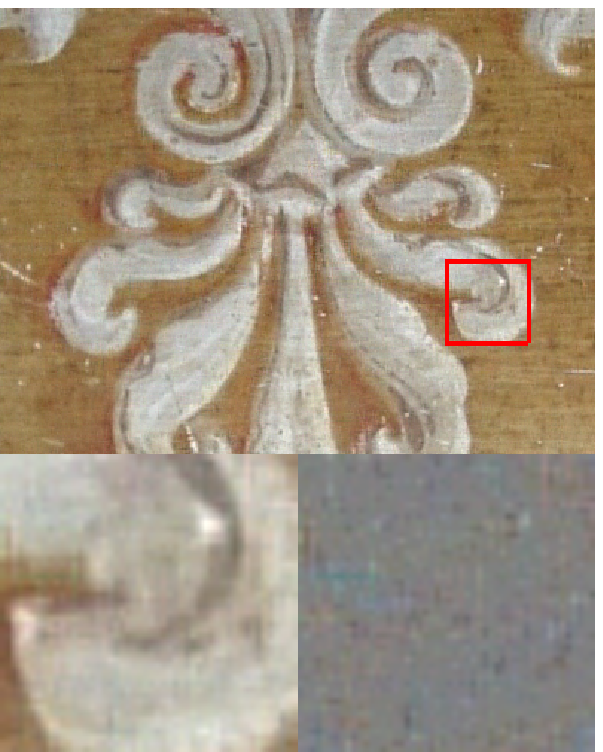}\vspace{0pt}
			\includegraphics[width=\linewidth]{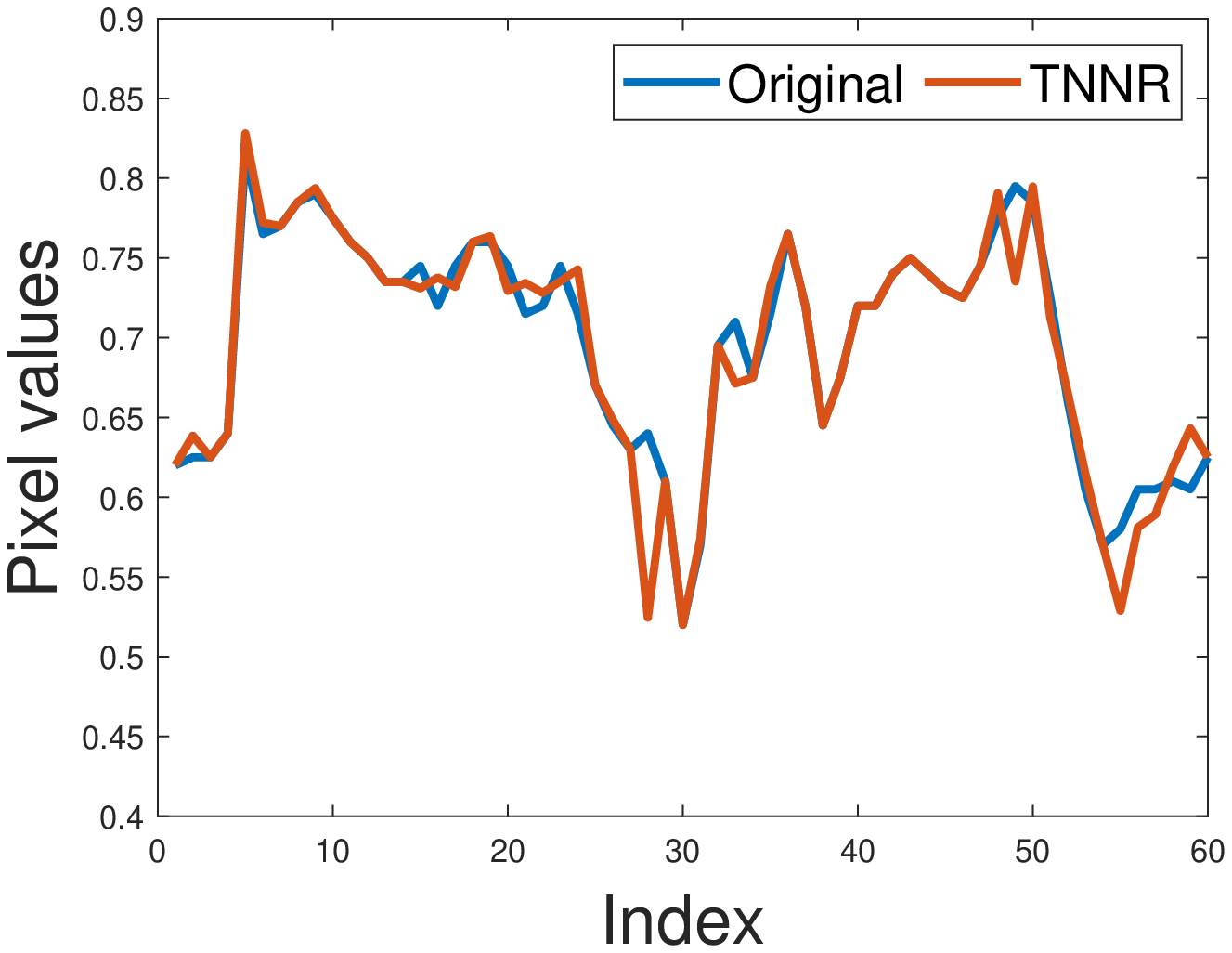}\vspace{0pt}
			\includegraphics[width=\linewidth]{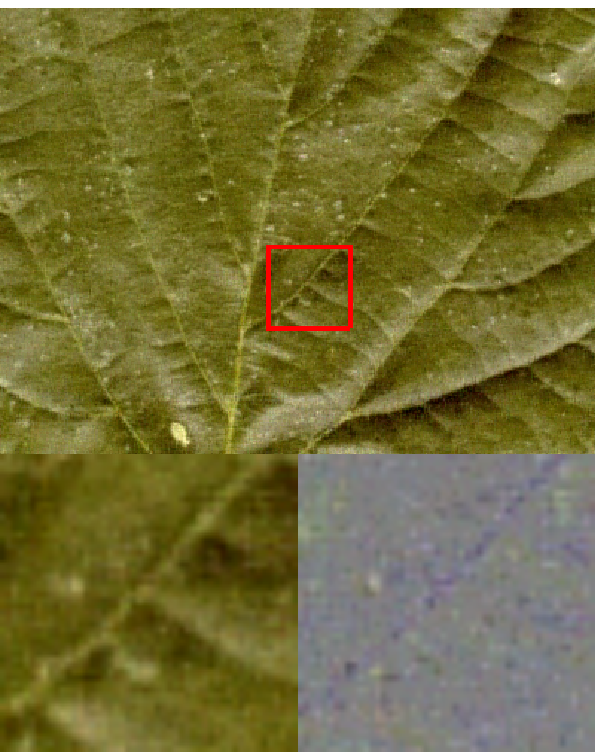}\vspace{0pt}
			\includegraphics[width=\linewidth]{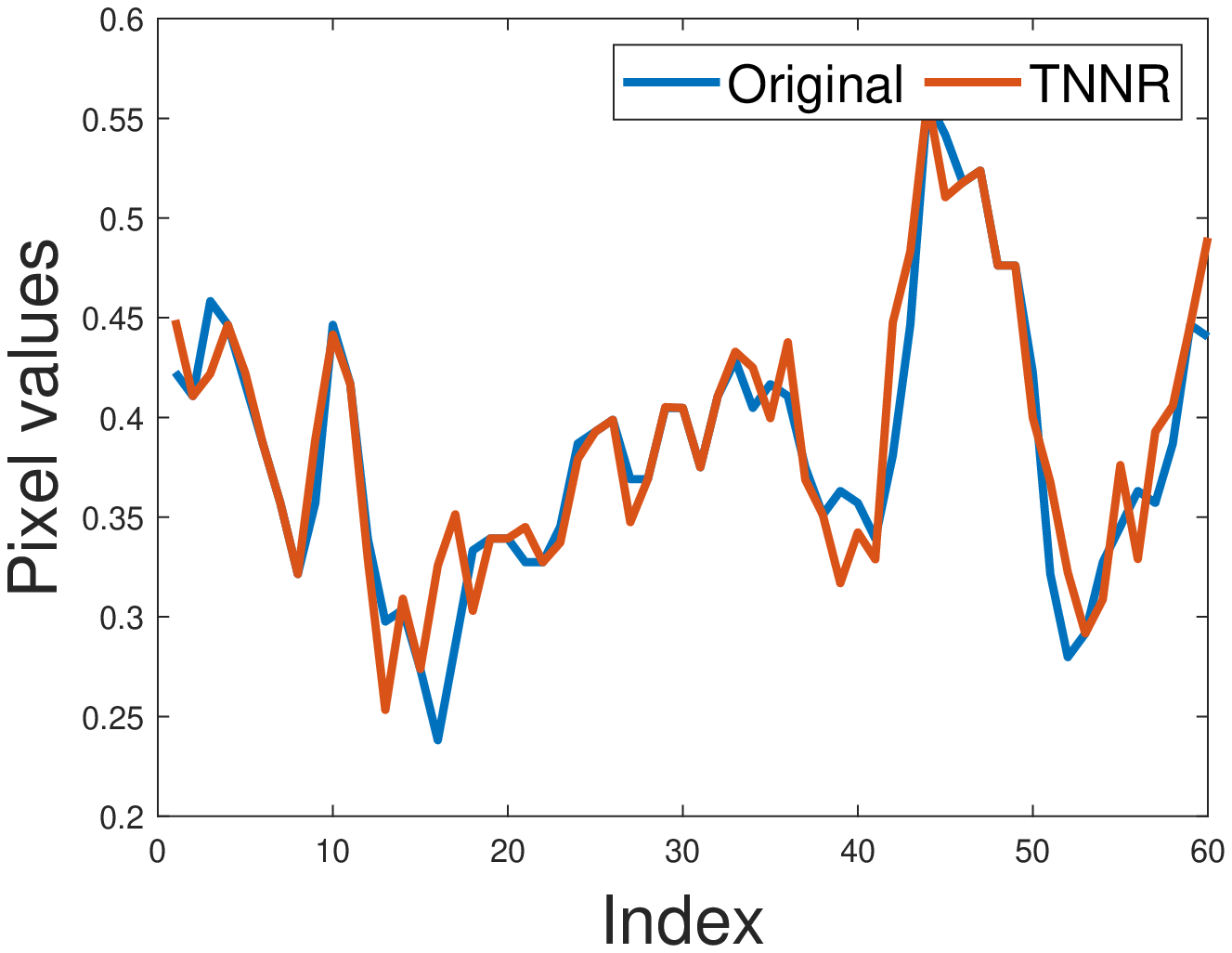}\vspace{0pt}
			\includegraphics[width=\linewidth]{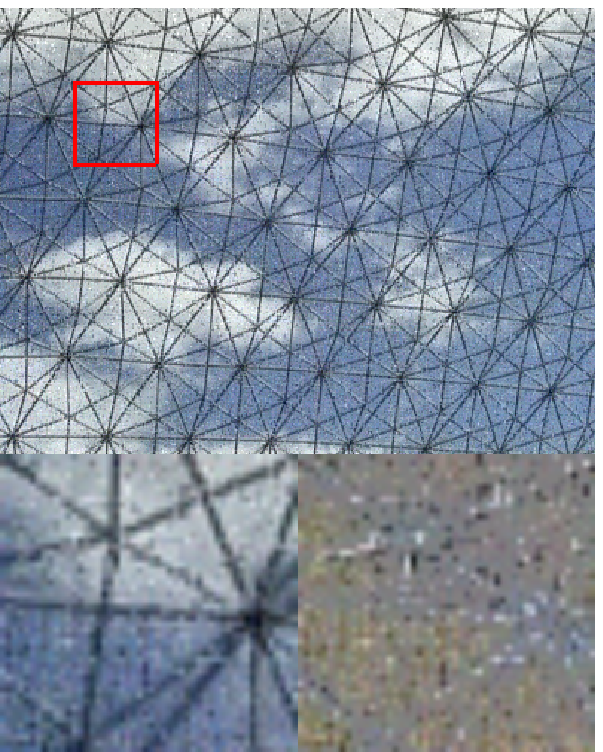}\vspace{0pt}
			\includegraphics[width=\linewidth]{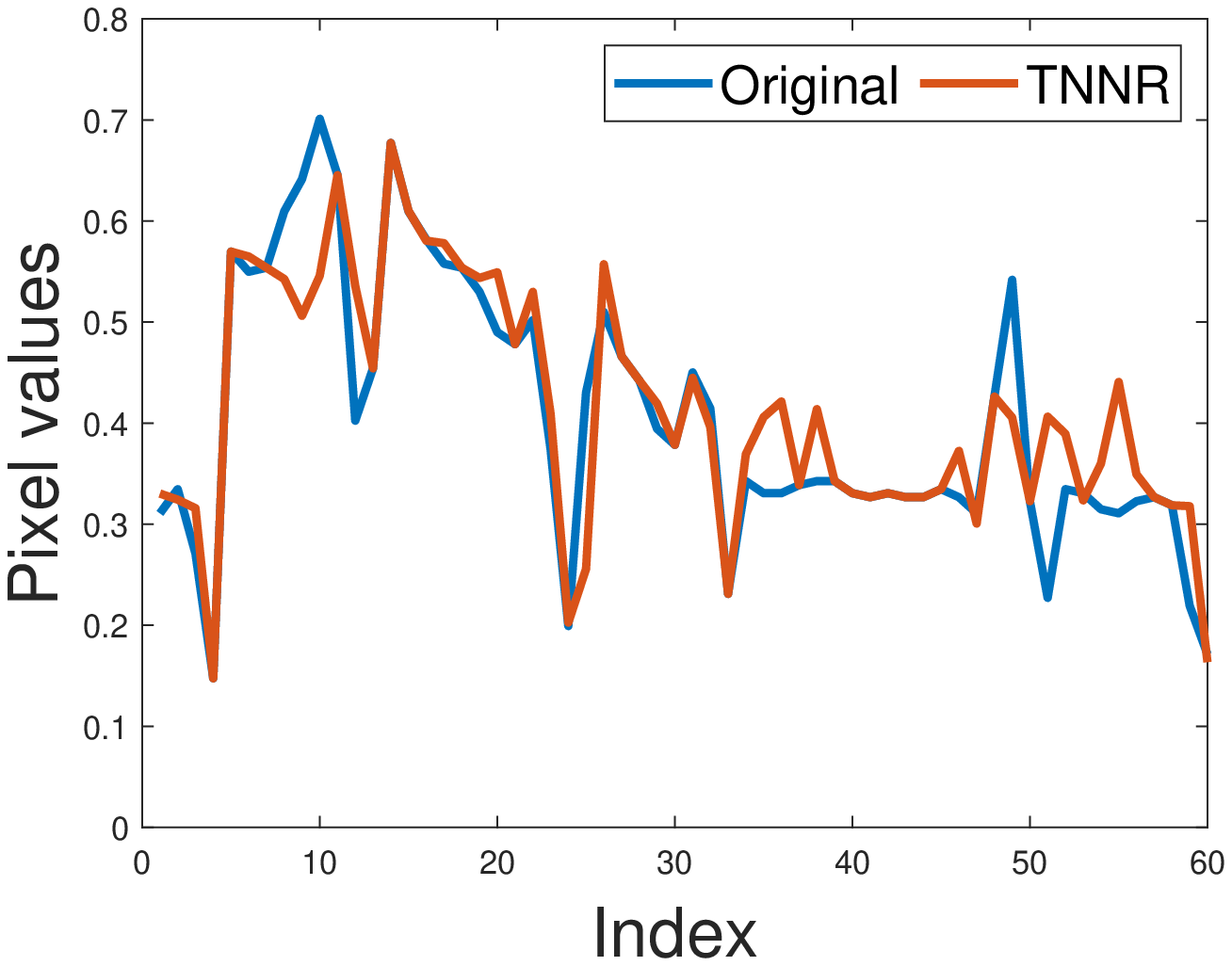}
			\caption{TNNR}
		\end{subfigure}
		\begin{subfigure}[b]{0.138\linewidth}
			\centering		
			\includegraphics[width=\linewidth]{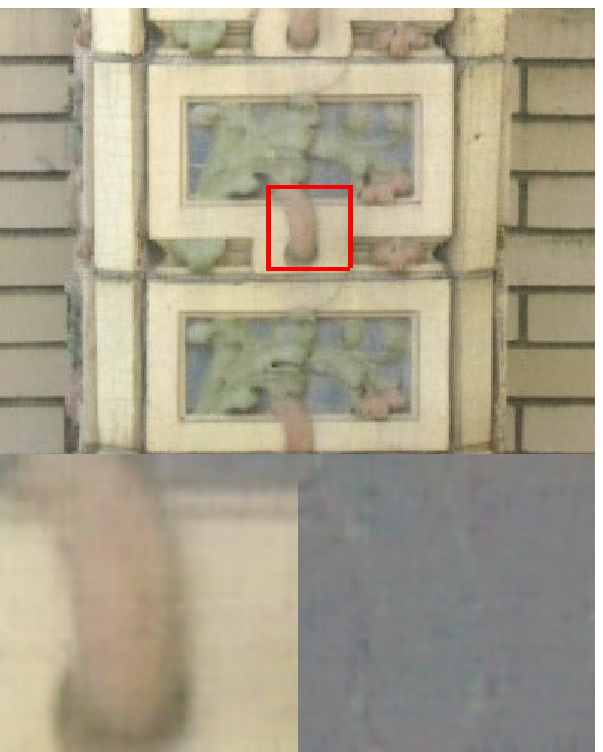}\vspace{0pt}
			\includegraphics[width=\linewidth]{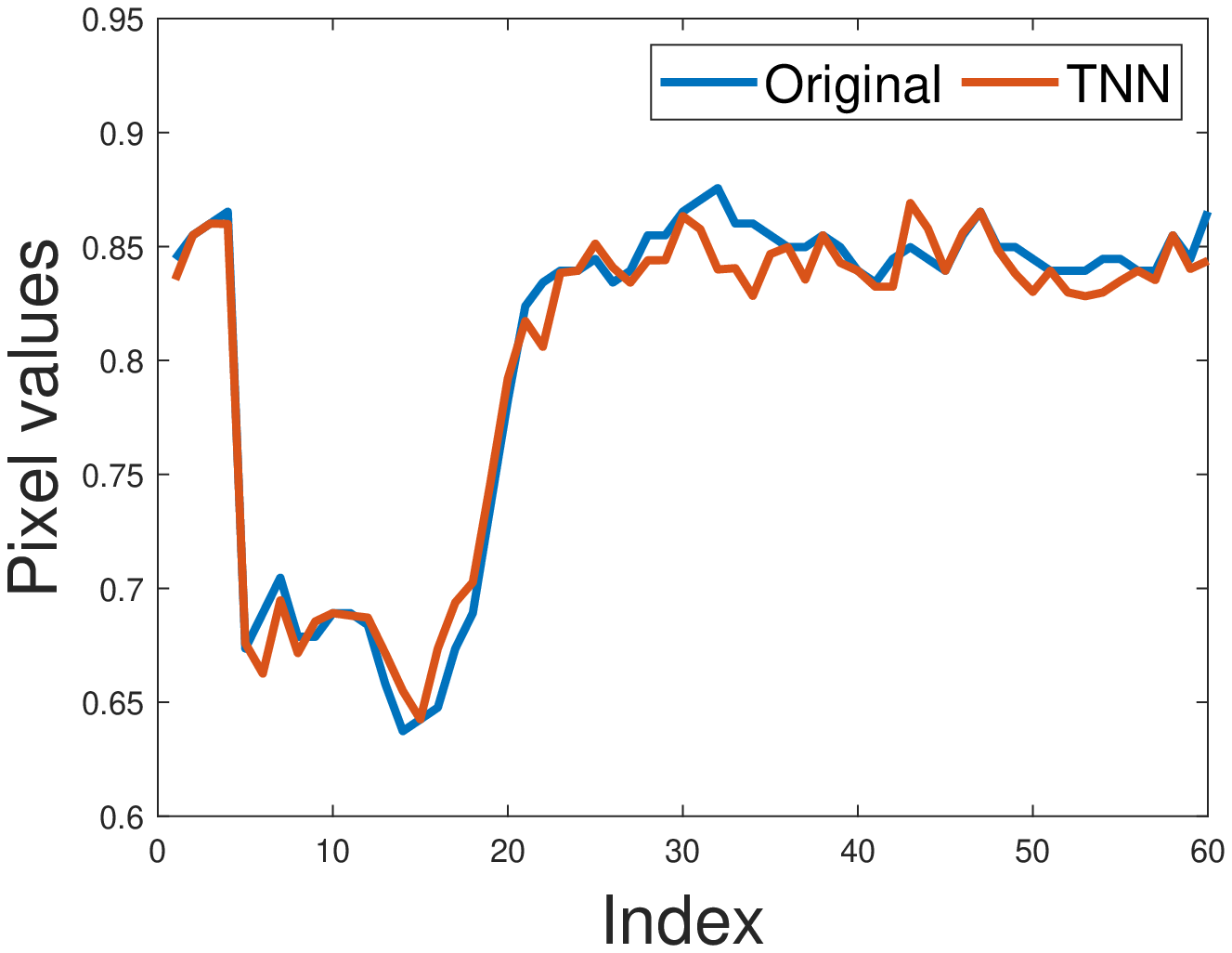}\vspace{0pt}
			\includegraphics[width=\linewidth]{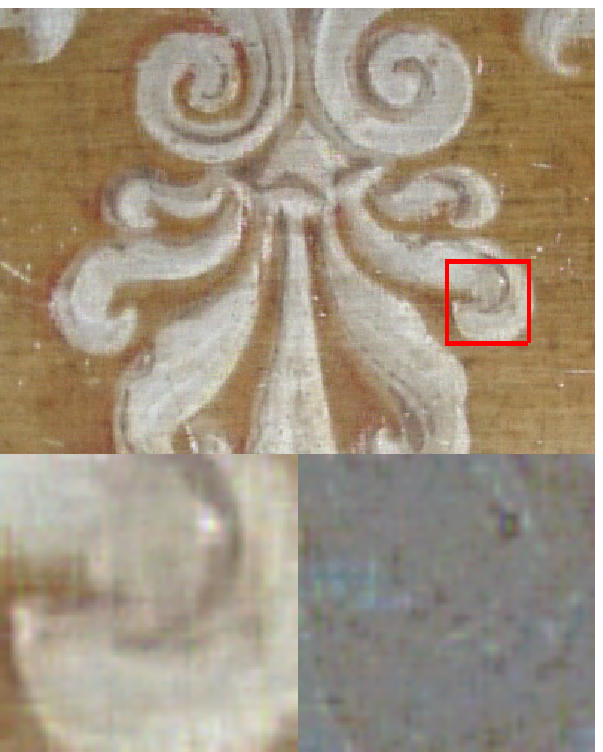}\vspace{0pt}
			\includegraphics[width=\linewidth]{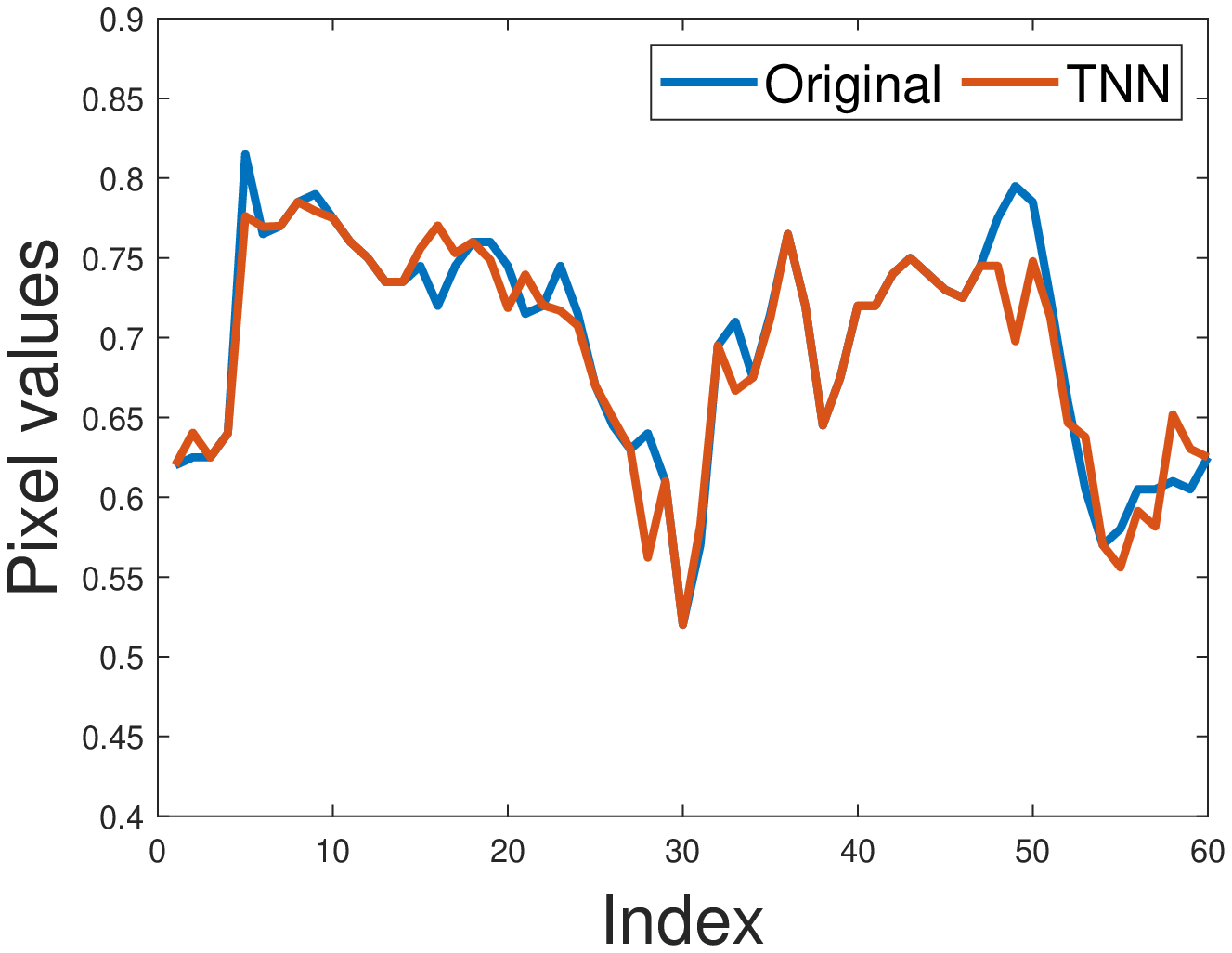}\vspace{0pt}
			\includegraphics[width=\linewidth]{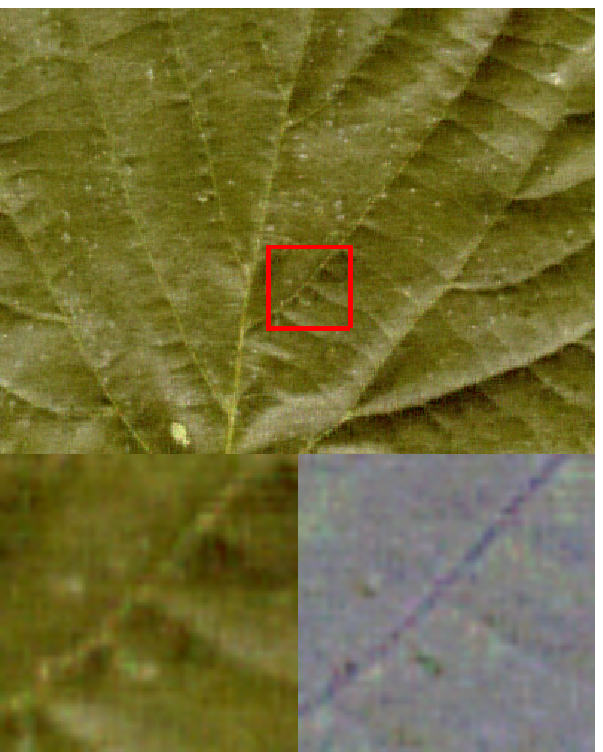}\vspace{0pt}
			\includegraphics[width=\linewidth]{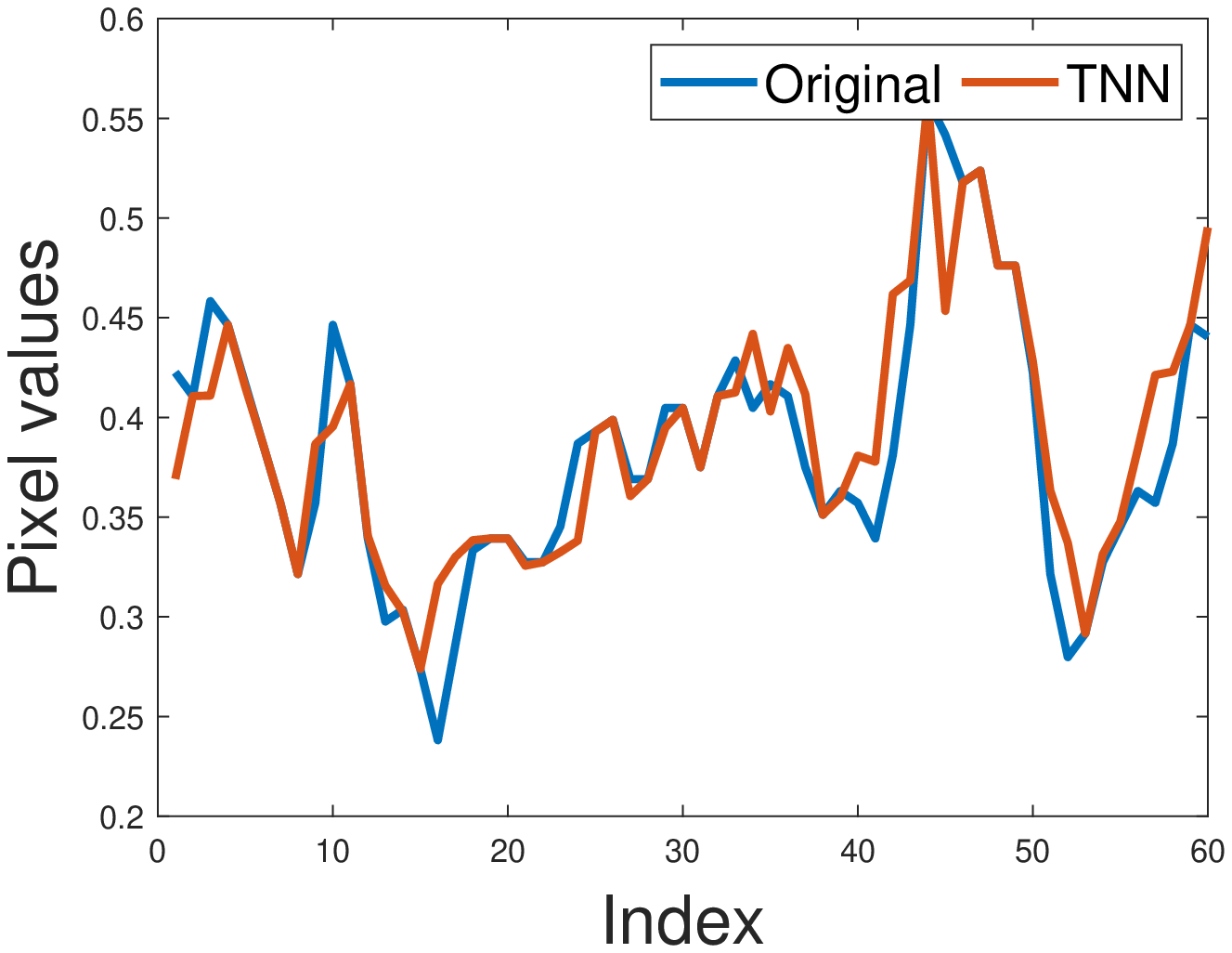}\vspace{0pt}
			\includegraphics[width=\linewidth]{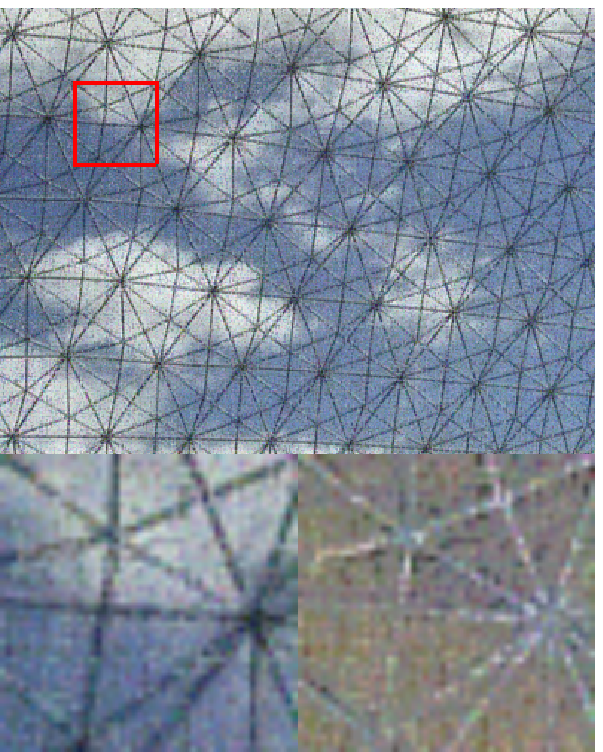}\vspace{0pt}
			\includegraphics[width=\linewidth]{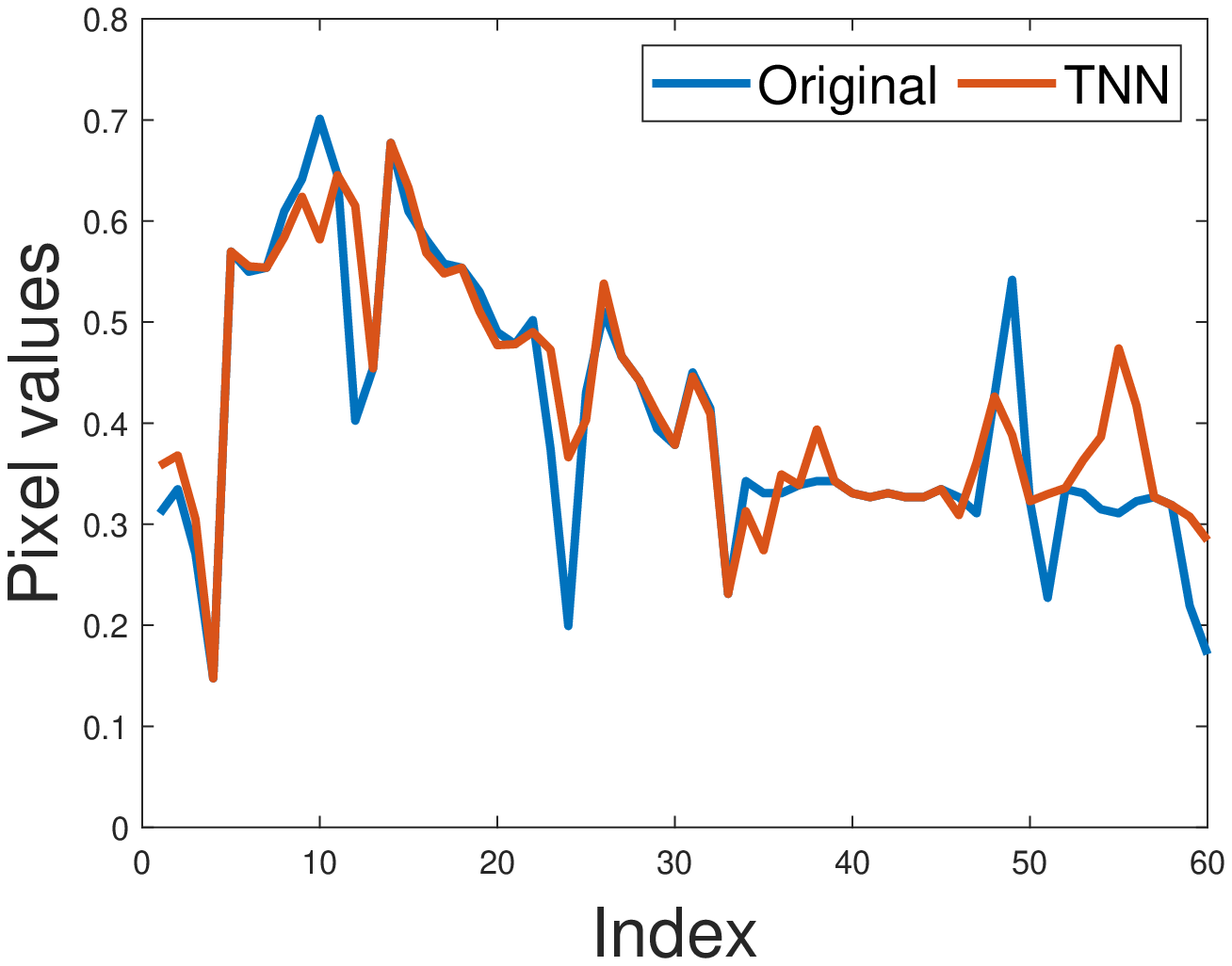}
			\caption{TNN}
		\end{subfigure}
		\begin{subfigure}[b]{0.138\linewidth}
			\centering			
			\includegraphics[width=\linewidth]{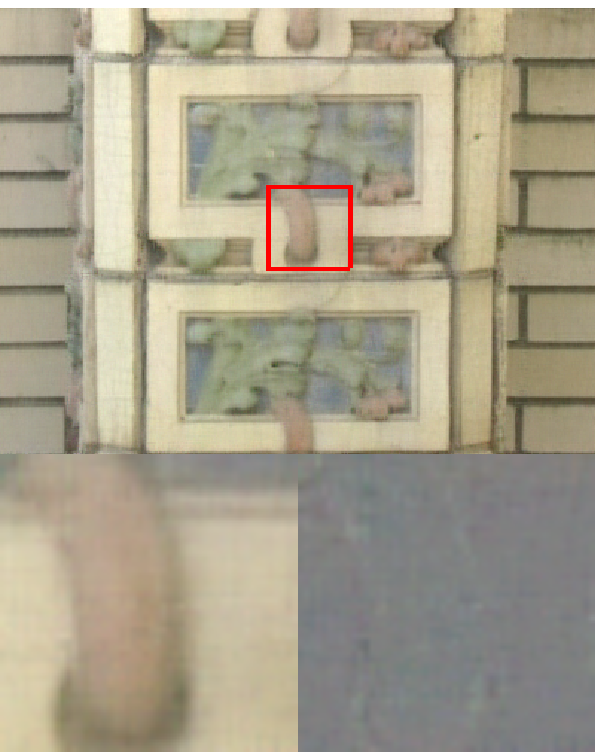}\vspace{0pt}
			\includegraphics[width=\linewidth]{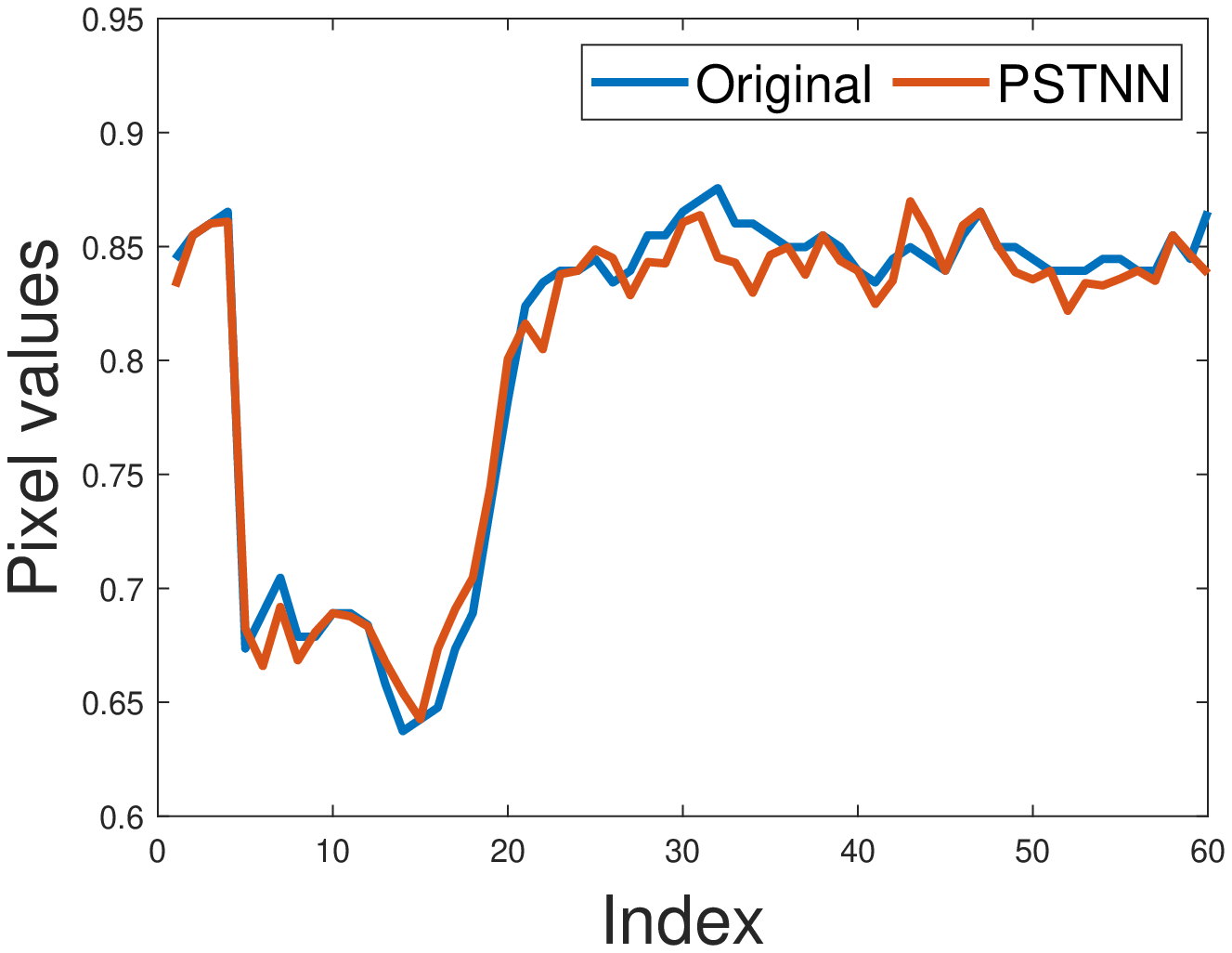}\vspace{0pt}
			\includegraphics[width=\linewidth]{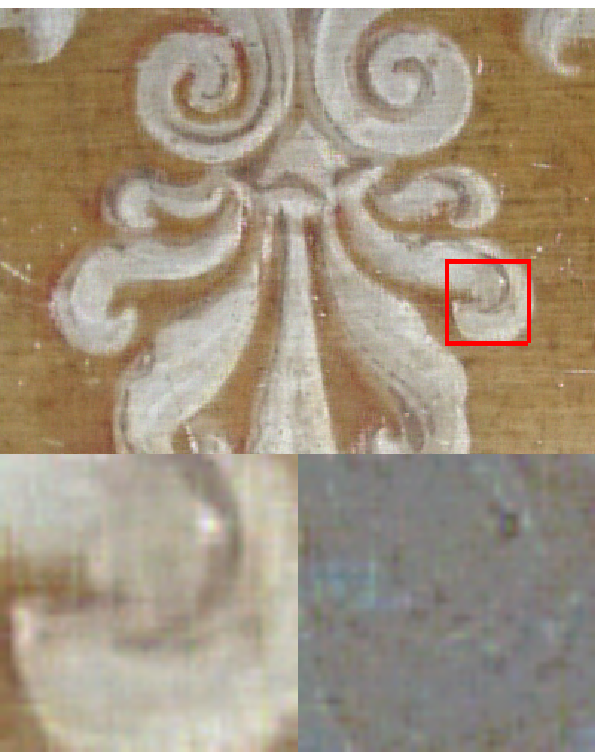}\vspace{0pt}
			\includegraphics[width=\linewidth]{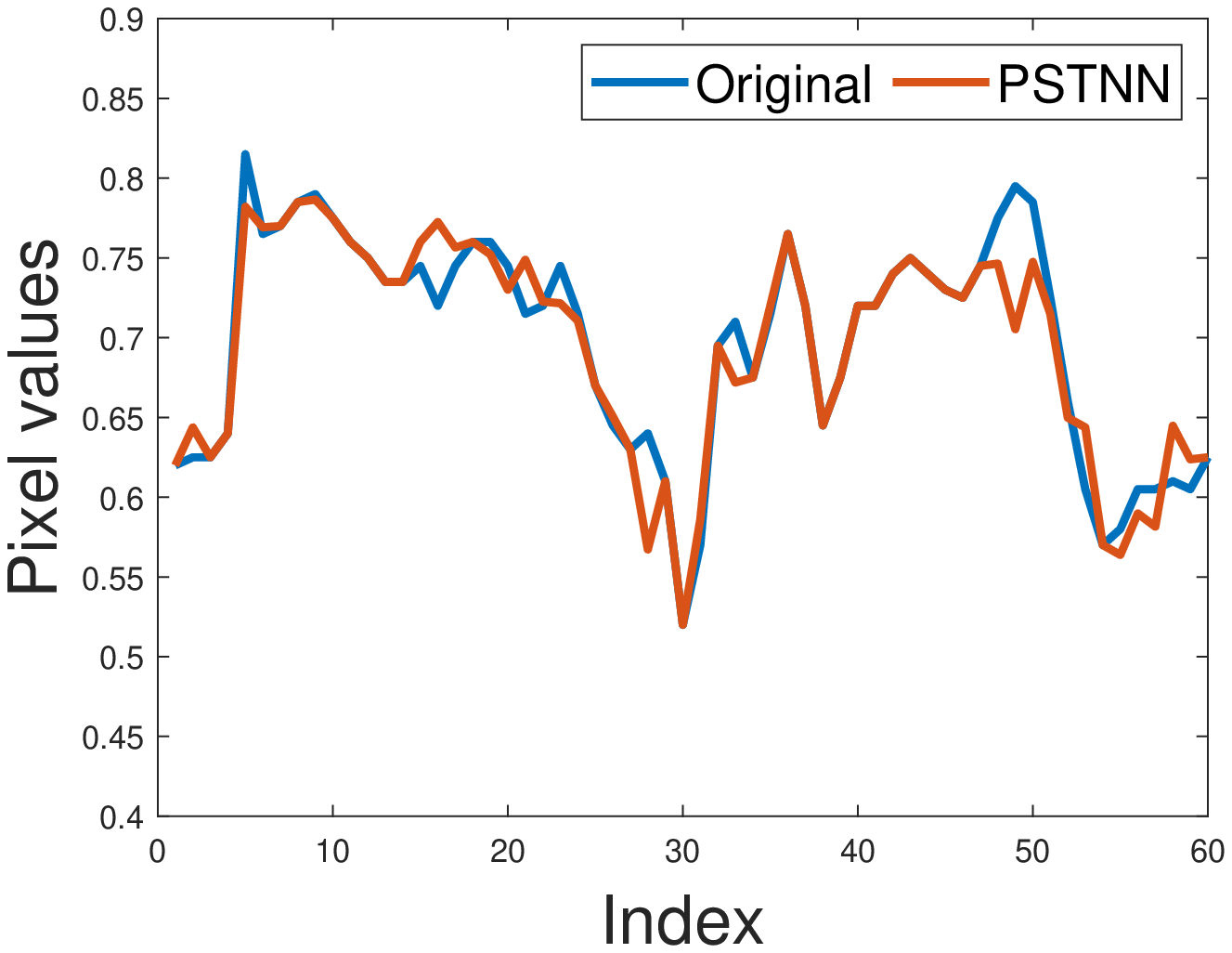}\vspace{0pt}
			\includegraphics[width=\linewidth]{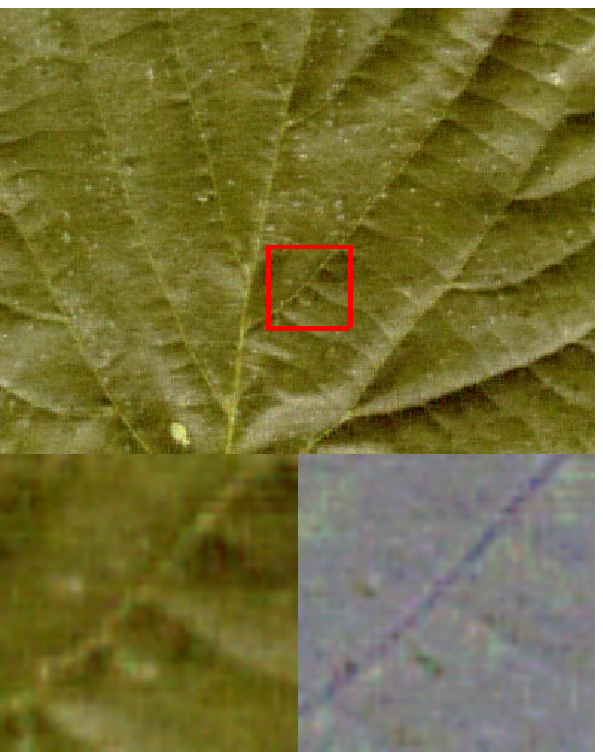}\vspace{0pt}
			\includegraphics[width=\linewidth]{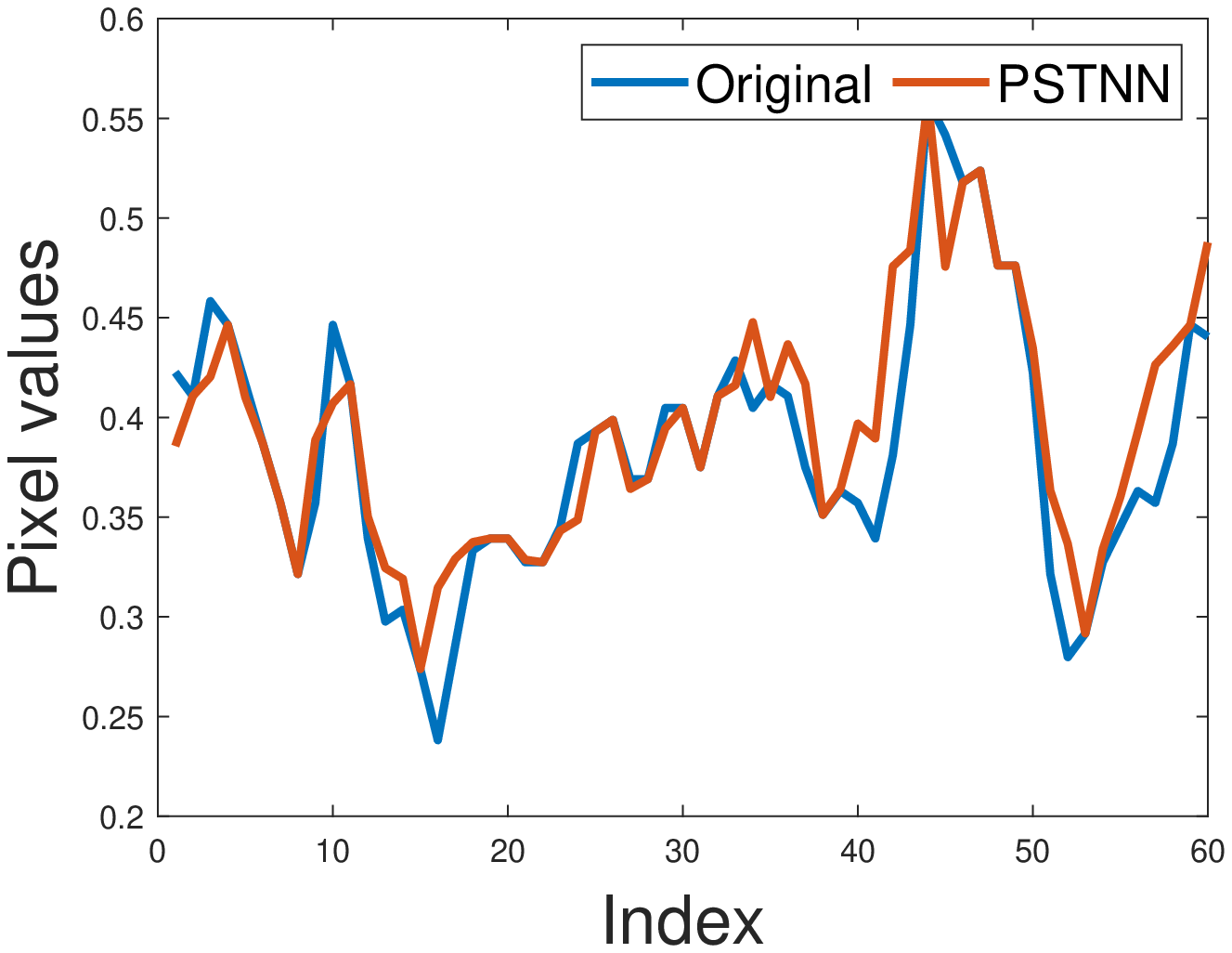}\vspace{0pt}
			\includegraphics[width=\linewidth]{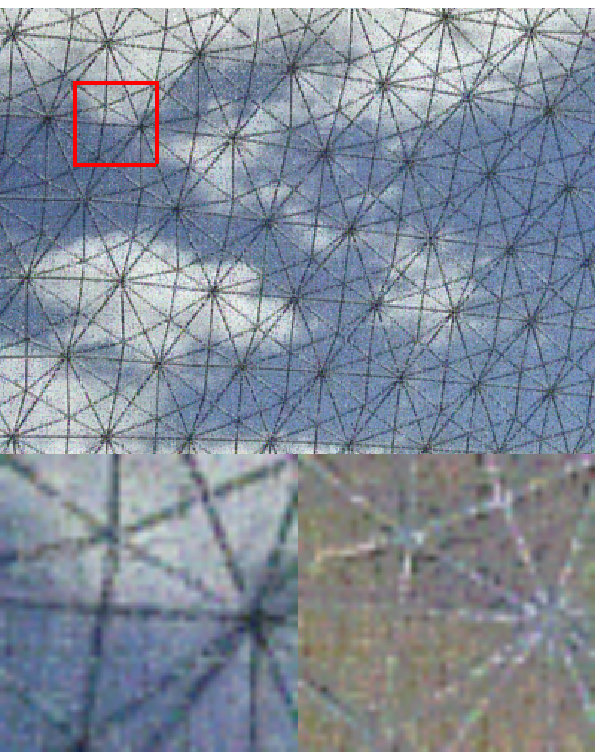}\vspace{0pt}
			\includegraphics[width=\linewidth]{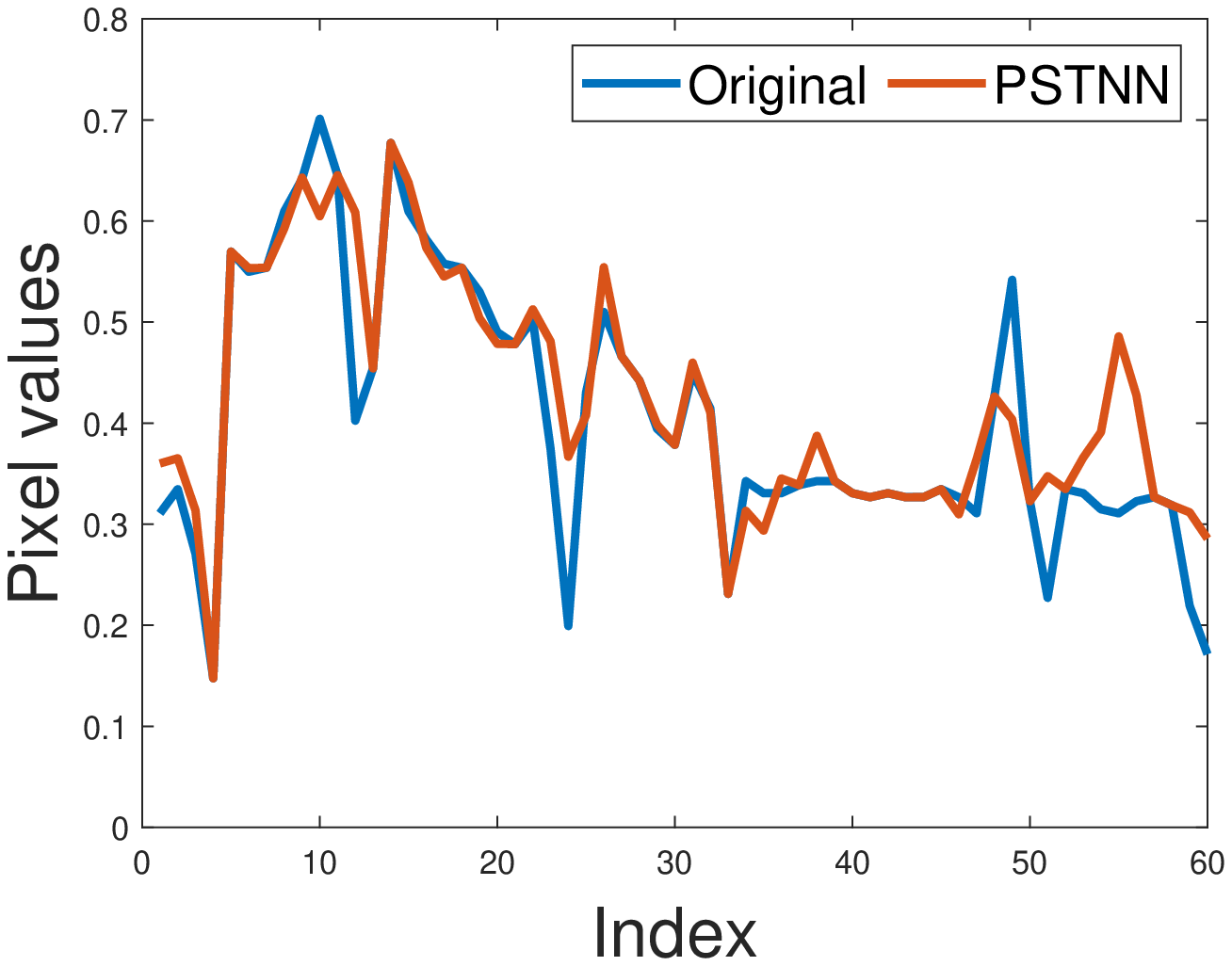}
			\caption{PSTNN}
		\end{subfigure}
		\begin{subfigure}[b]{0.138\linewidth}
			\centering			
			\includegraphics[width=\linewidth]{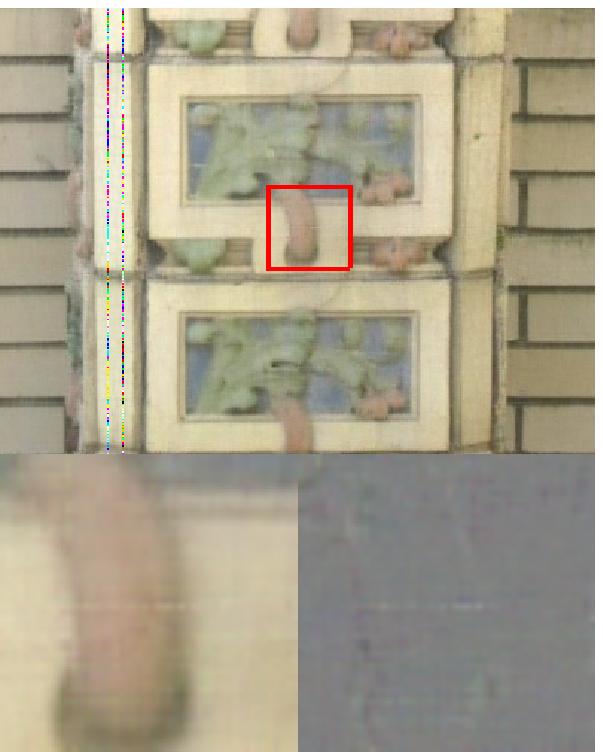}\vspace{0pt}
			\includegraphics[width=\linewidth]{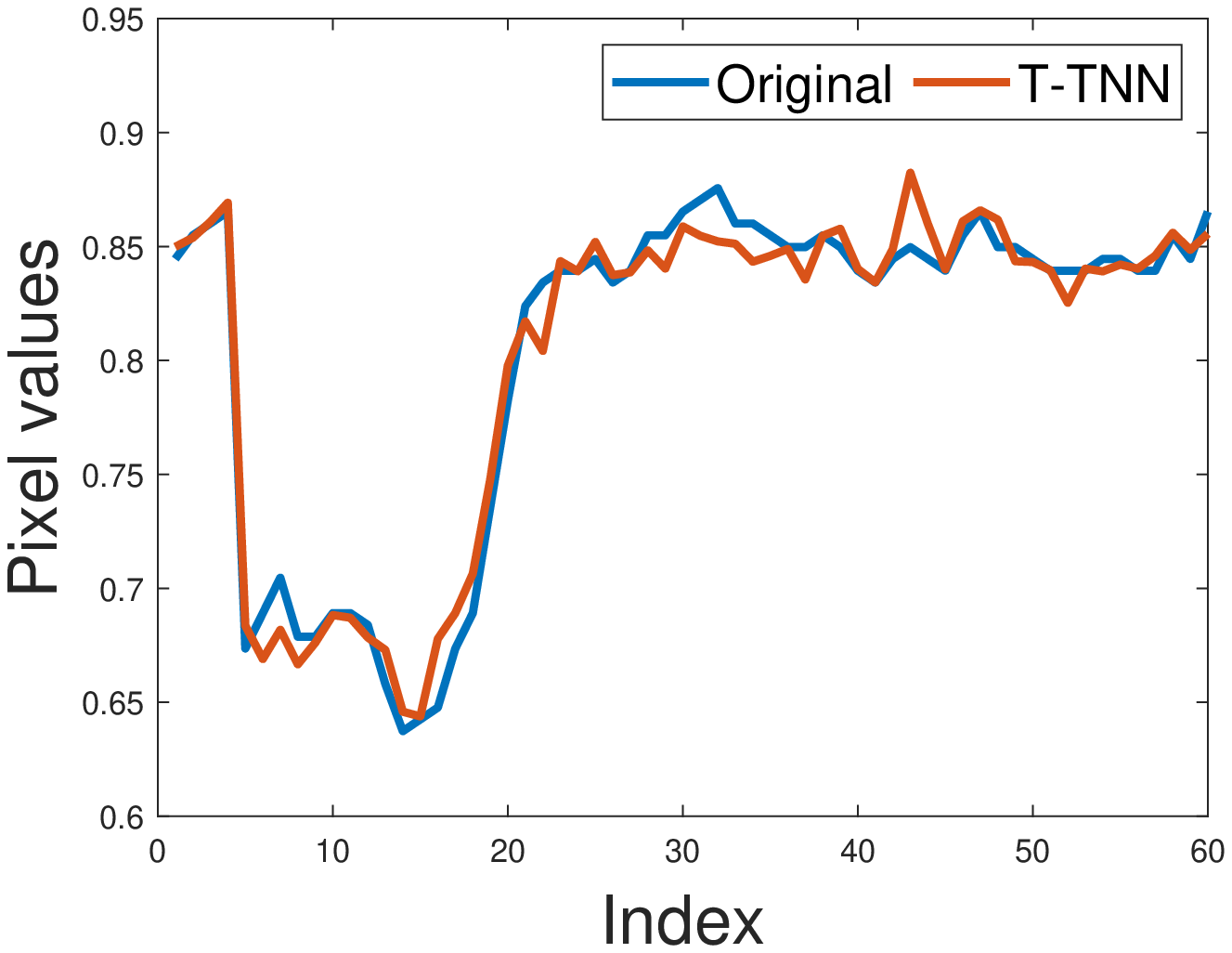}\vspace{0pt}
			\includegraphics[width=\linewidth]{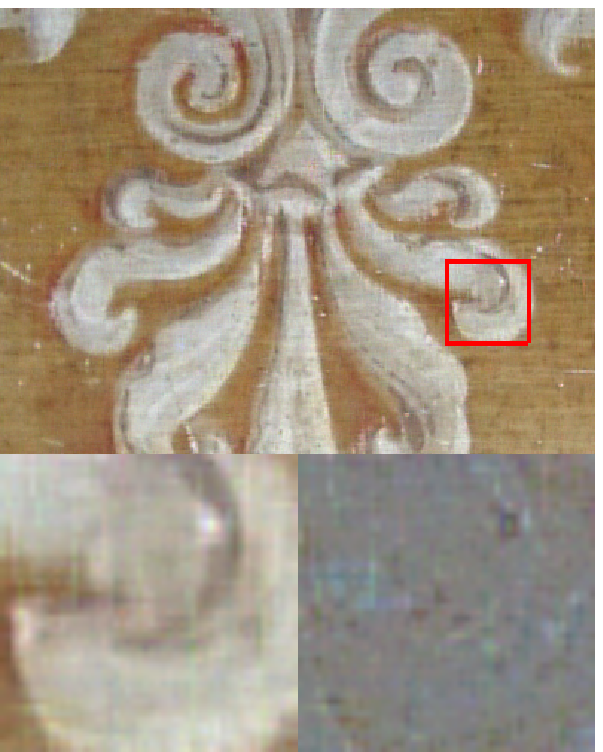}\vspace{0pt}
			\includegraphics[width=\linewidth]{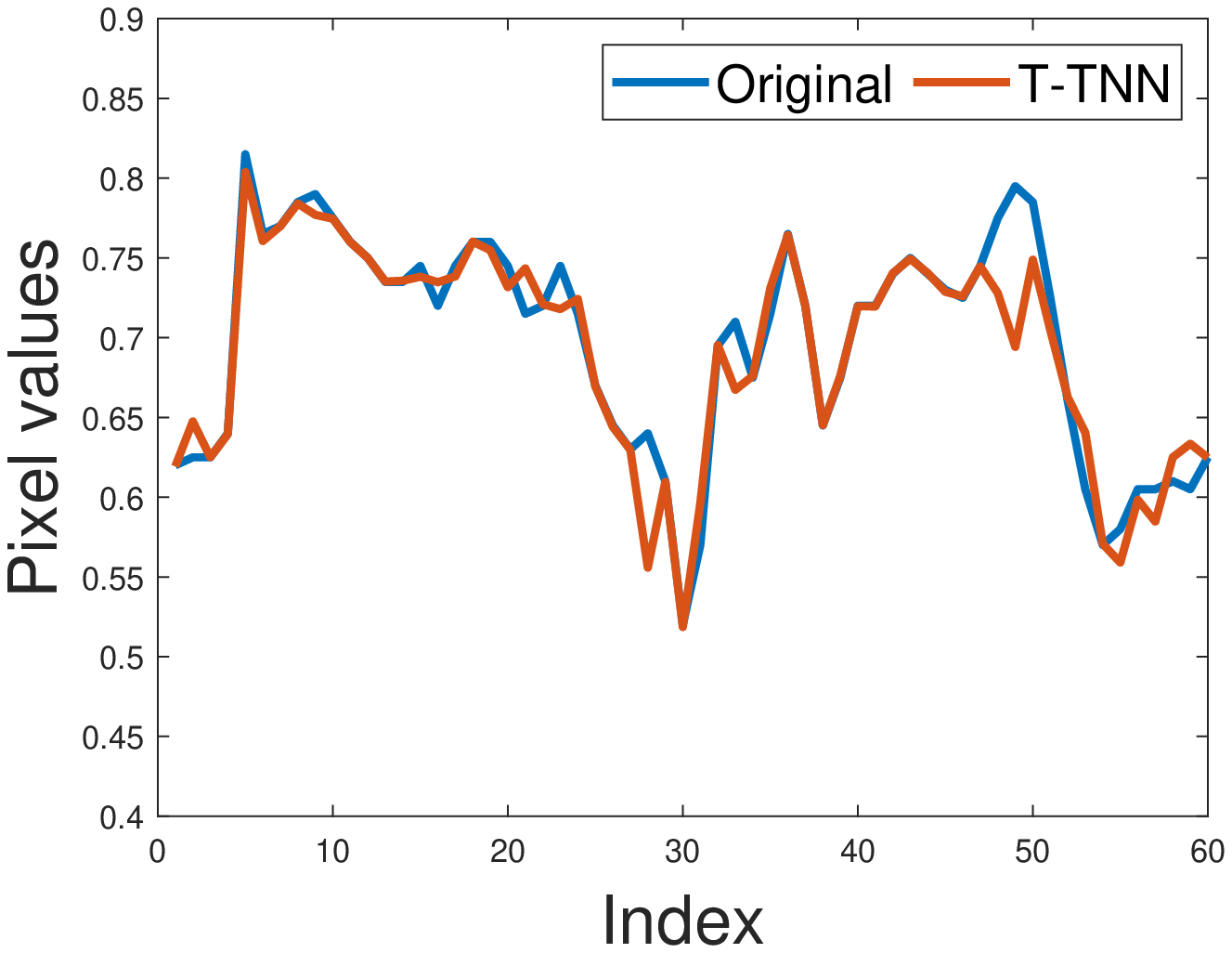}\vspace{0pt}
			\includegraphics[width=\linewidth]{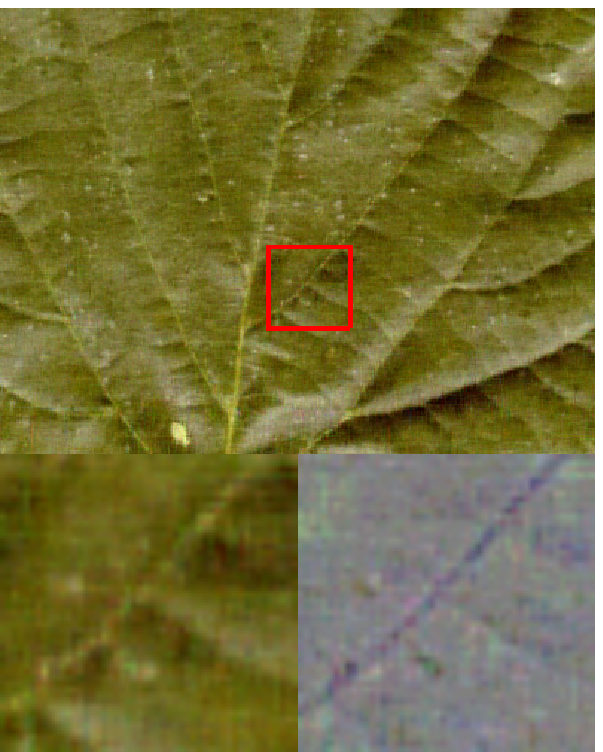}\vspace{0pt}
			\includegraphics[width=\linewidth]{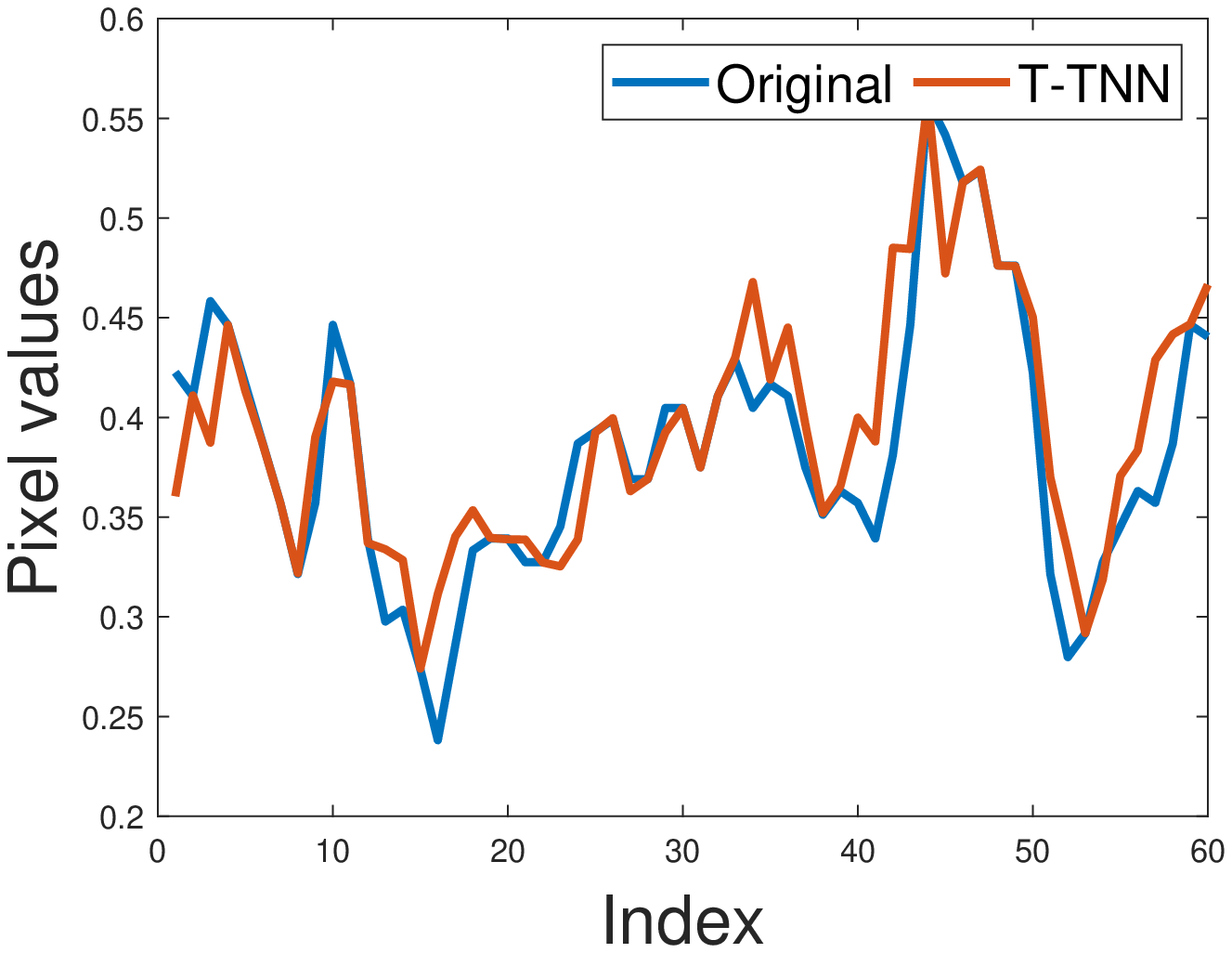}\vspace{0pt}
			\includegraphics[width=\linewidth]{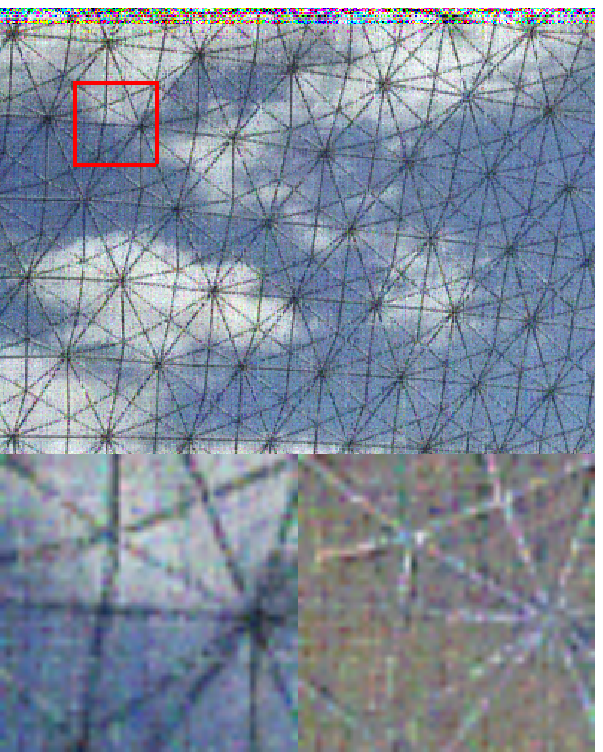}\vspace{0pt}
			\includegraphics[width=\linewidth]{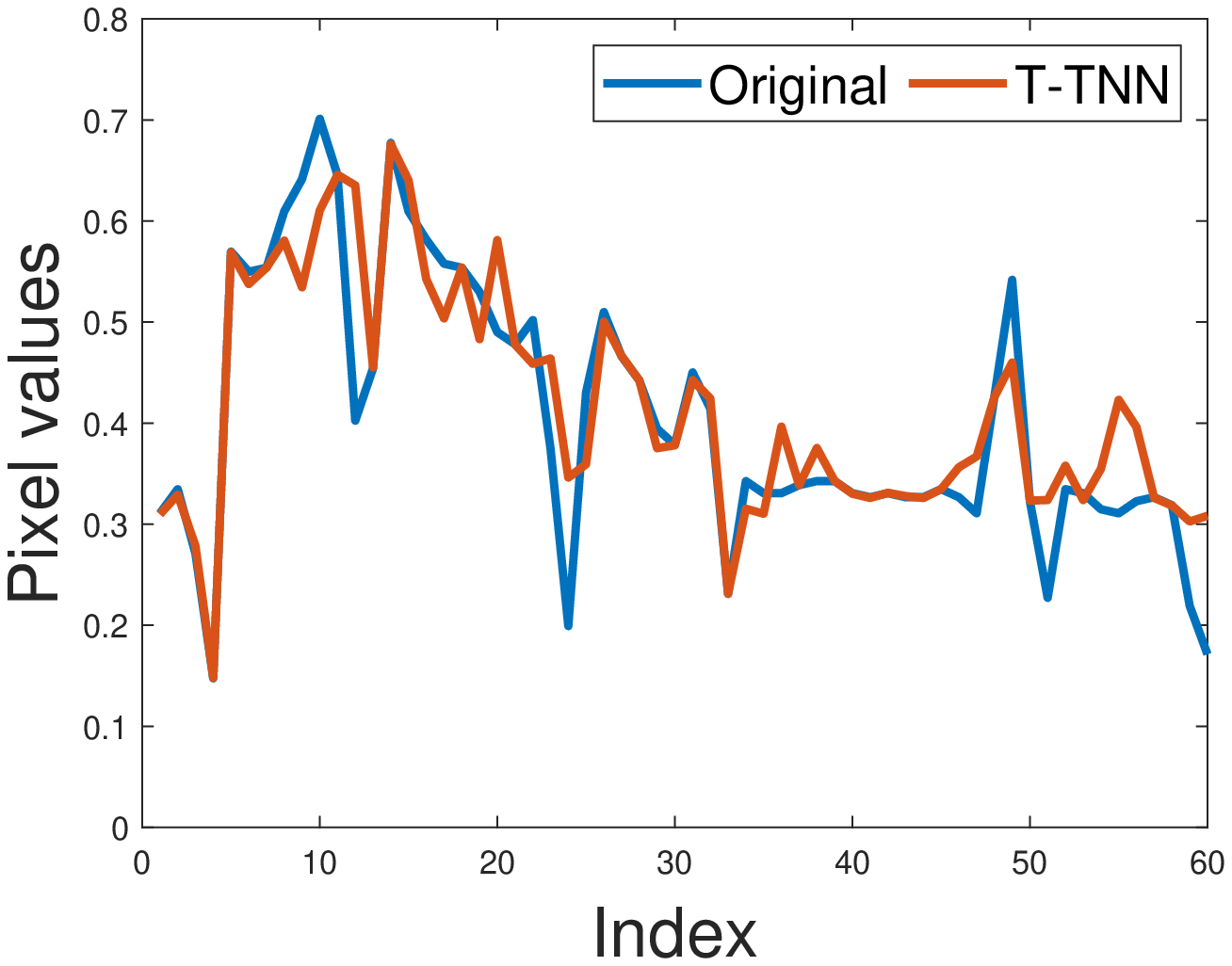}
			\caption{T-TNN}
		\end{subfigure}	
		\begin{subfigure}[b]{0.138\linewidth}
			\centering
			\includegraphics[width=\linewidth]{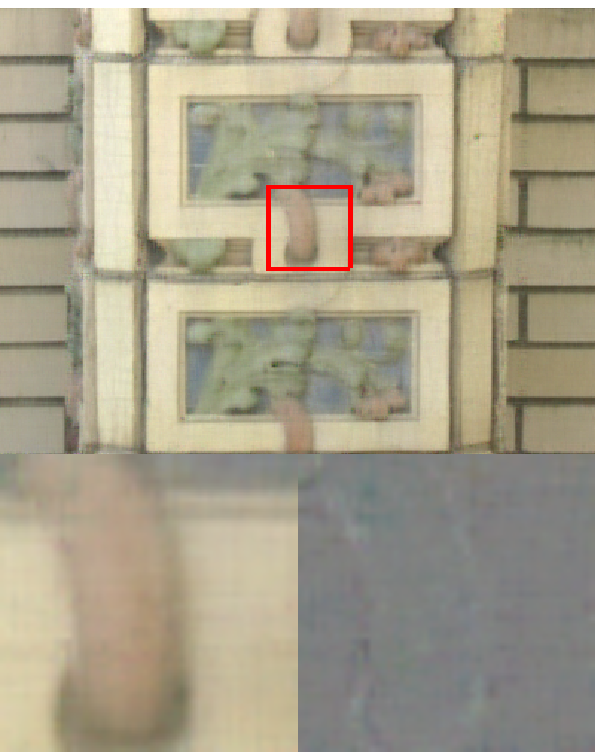}\vspace{0pt}
			\includegraphics[width=\linewidth]{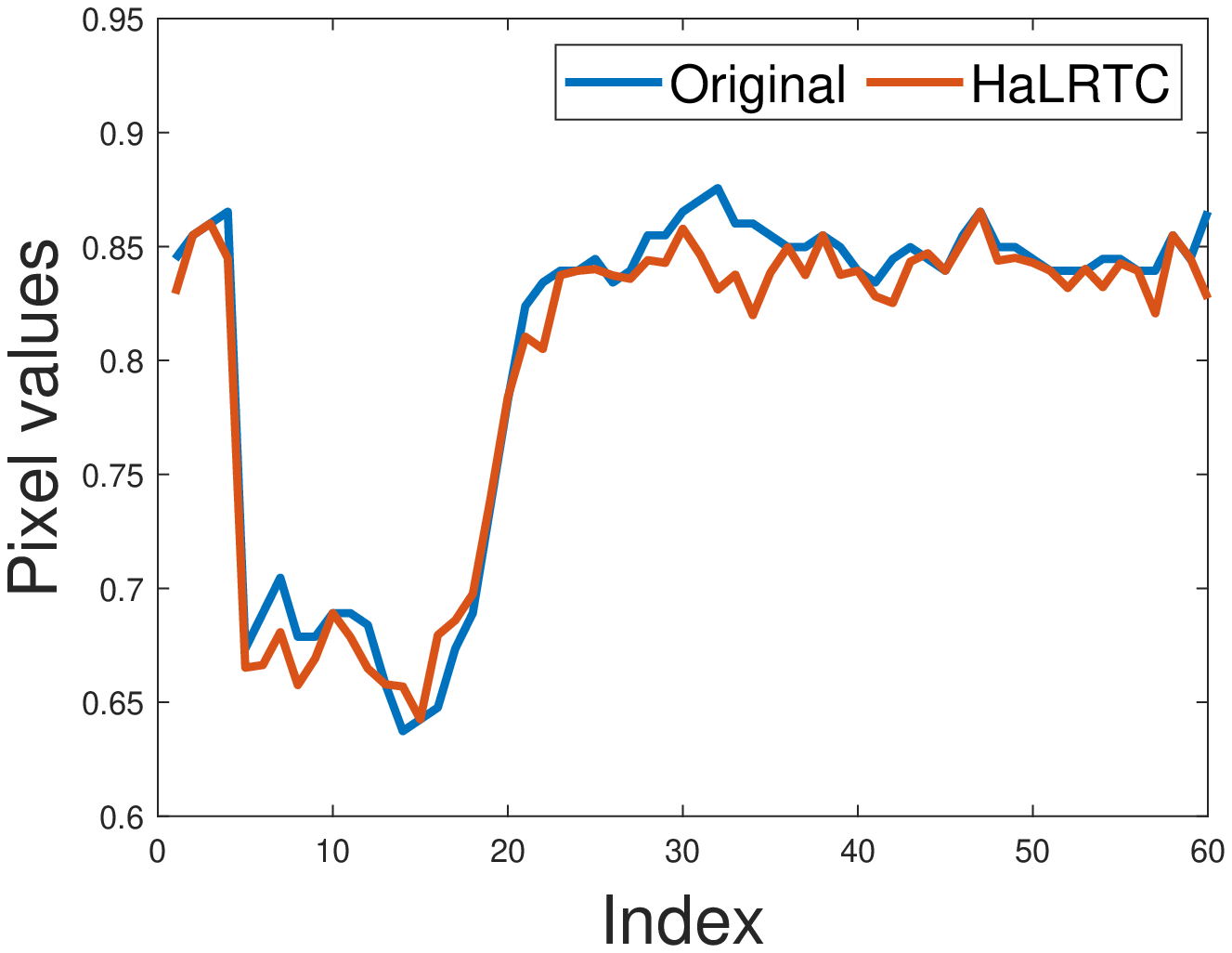}\vspace{0pt}
			\includegraphics[width=\linewidth]{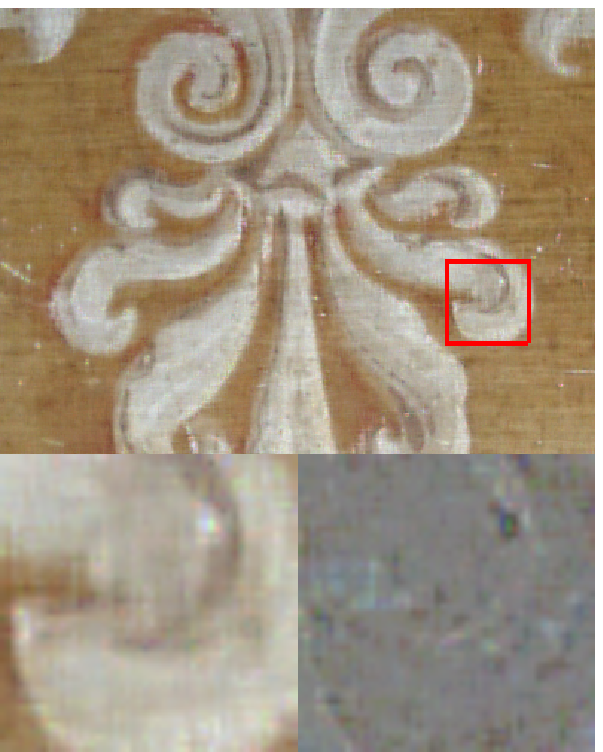}\vspace{0pt}
			\includegraphics[width=\linewidth]{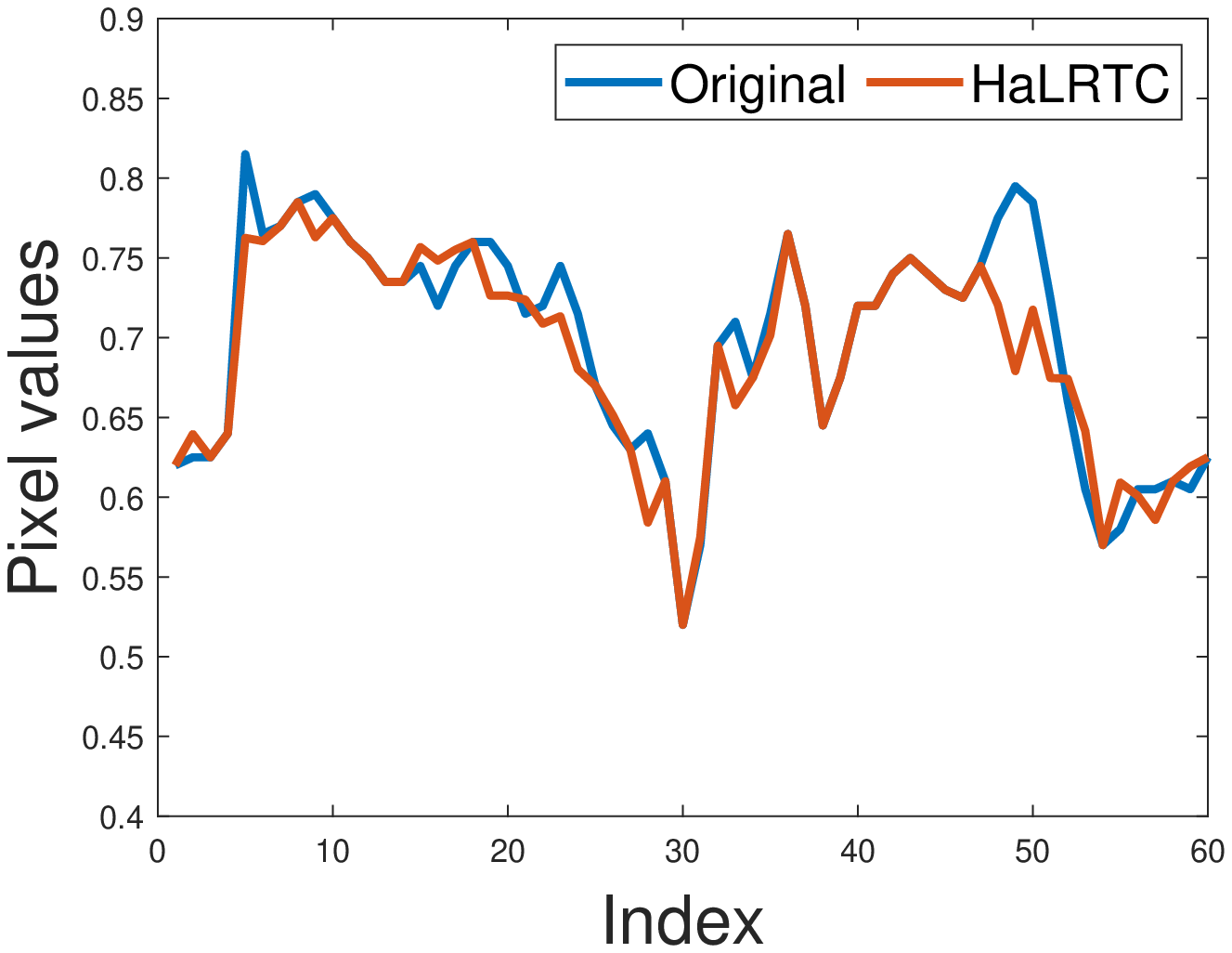}\vspace{0pt}
			\includegraphics[width=\linewidth]{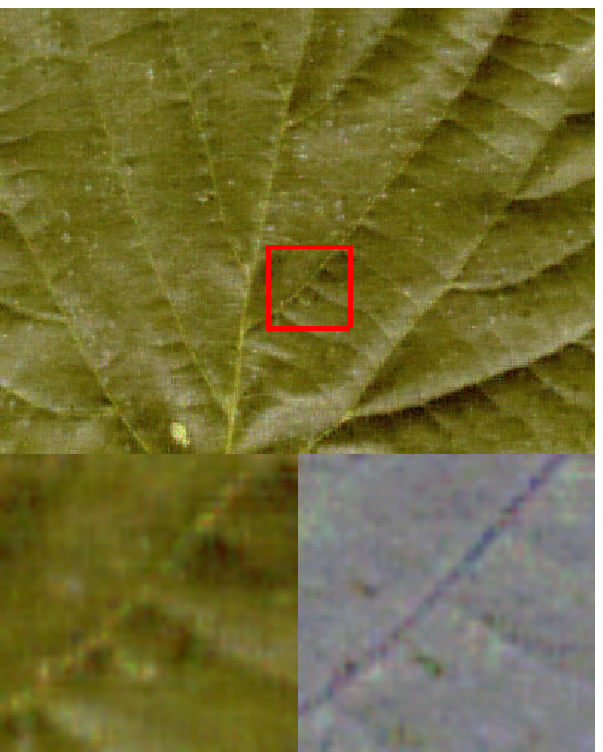}\vspace{0pt}
			\includegraphics[width=\linewidth]{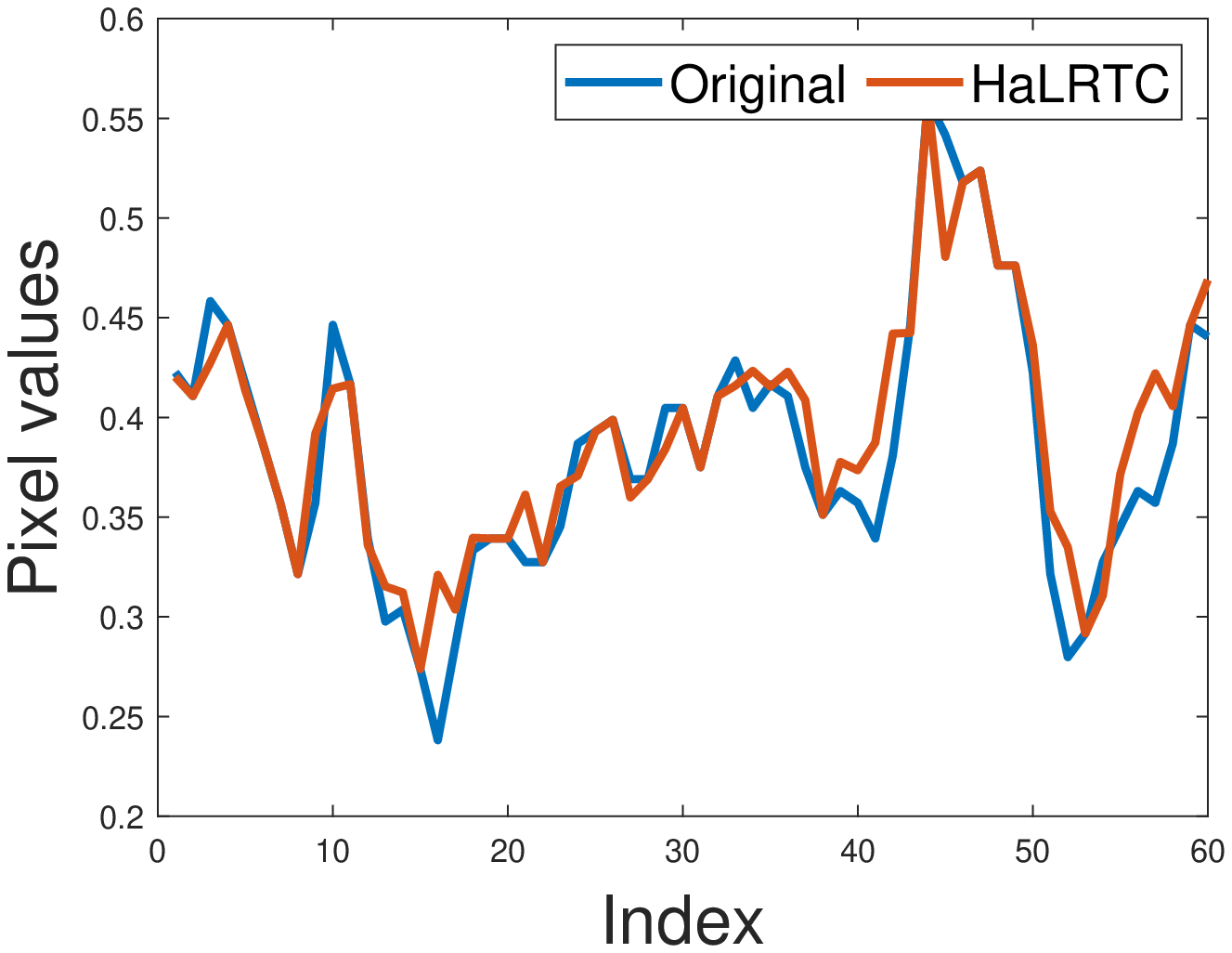}\vspace{0pt}
			\includegraphics[width=\linewidth]{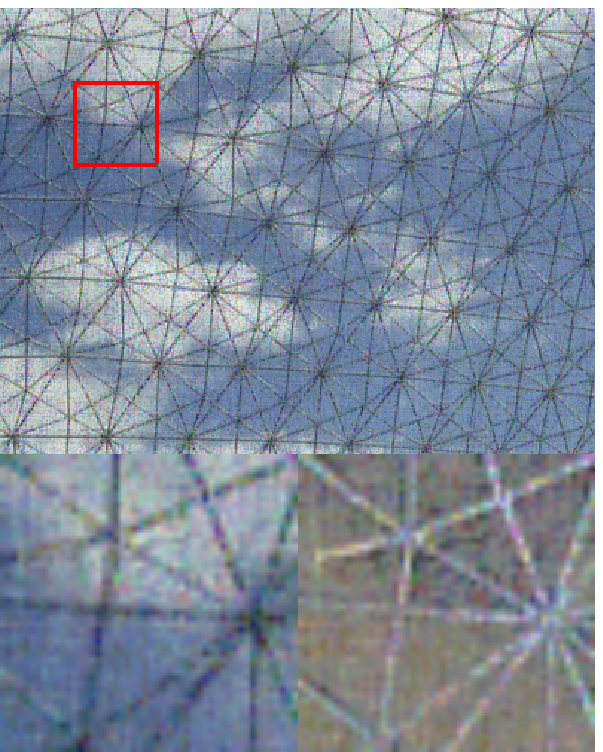}\vspace{0pt}
			\includegraphics[width=\linewidth]{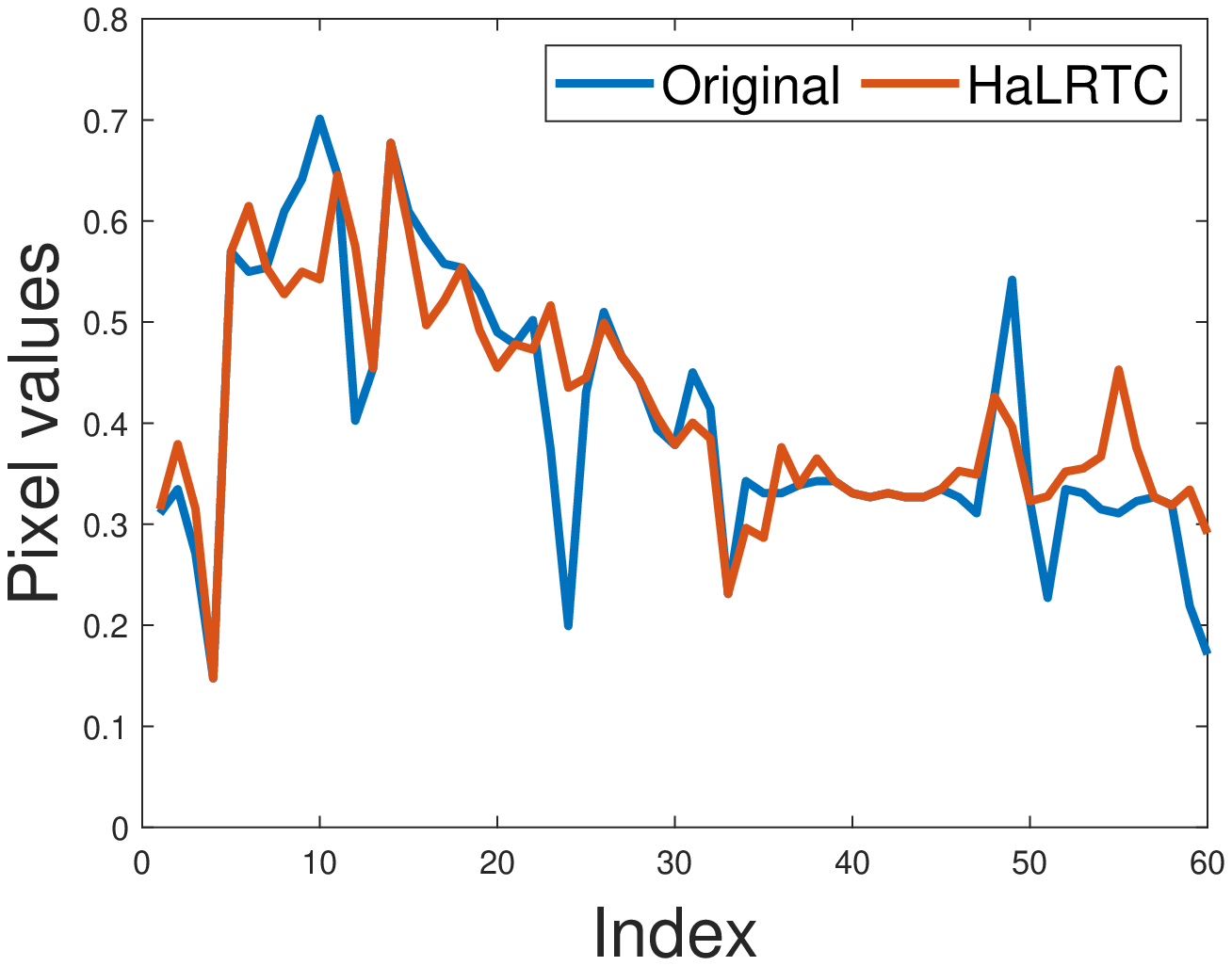}
			\caption{HaLRTC}
		\end{subfigure}					
	\end{subfigure}
	\vfill
	\caption{Examples of texture image inpainting with $ SR=30\% $. From top to bottom are ``Stone", ``Pattern", ``Leaves" and ``Barbed Wire". For better visualization, we show the zoom-in region, the corresponding error map (difference from the original) and the corresponding partial residuals of the region.}
	\label{fig:Texture}
\end{figure*}

% Table generated by Excel2LaTeX from sheet 'Texture'
\begin{table*}[htbp]
	\centering
	\caption{TEXTURE IMAGE INPAINTING PERFORMANCE COMPARISON: PSNR, SSIM, FSIM AND RUNNING TIME. THE BEST AND THE SECOND BEST PERFORMING METHODS IN EACH IMAGE ARE HIGHLIGHTED IN RED AND BOLD, RESPECTIVELY}
	\begin{tabular}{c|c|cccc|cccc|cccc}
		\hline
		\multirow{2}{*}{Texture Image} & \multirow{2}{*}{Methods} & \multicolumn{4}{c|}{$ SR=30\% $}    & \multicolumn{4}{c|}{$ SR=35\% $}   & \multicolumn{4}{c}{$ SR=40\% $} \\
\cline{3-14}          &       & PSNR  & SSIM  & FSIM  & Time  & PSNR  & SSIM  & FSIM  & Time  & PSNR  & SSIM  & FSIM  & Time \\
		\hline
\multirow{5}{*}{Stone} & TNNR  & \textcolor[rgb]{ 1,  0,  0}{37.535 } & \textcolor[rgb]{ 1,  0,  0}{0.949 } & \textcolor[rgb]{ 1,  0,  0}{0.980 } & \textbf{4.423 } & \textcolor[rgb]{ 1,  0,  0}{39.218 } & \textcolor[rgb]{ 1,  0,  0}{0.965 } & \textcolor[rgb]{ 1,  0,  0}{0.986 } & \textbf{3.881 } & \textcolor[rgb]{ 1,  0,  0}{40.605 } & \textcolor[rgb]{ 1,  0,  0}{0.974 } & \textcolor[rgb]{ 1,  0,  0}{0.990 } & \textcolor[rgb]{ 1,  0,  0}{3.608 } \\
& TNN   & 35.728  & 0.937  & 0.974  & 60.565  & 37.419  & 0.954  & 0.982  & 61.704  & 38.939  & 0.966  & 0.987  & 36.808  \\
& PSTNN & \textbf{35.825 } & \textbf{0.939 } & \textbf{0.975 } & 11.579  & 37.408  & 0.955  & 0.982  & 11.677  & 38.854  & 0.966  & 0.987  & 11.739  \\
& T-TNN & 28.026  & 0.905  & 0.966  & 367.973  & \textbf{38.098 } & \textbf{0.959 } & \textbf{0.984 } & 252.761  & \textbf{39.683 } & \textbf{0.970 } & \textbf{0.988 } & 70.191  \\
& HaLRTC & 34.703  & 0.931  & 0.969  & \textcolor[rgb]{ 1,  0,  0}{3.929 } & 36.119  & 0.947  & 0.976  & \textcolor[rgb]{ 1,  0,  0}{3.764 } & 37.440  & 0.959  & 0.982  & \textbf{3.785 } \\
\hline
\multirow{5}{*}{Pattern} & TNNR  & \textcolor[rgb]{ 1,  0,  0}{32.876 } & \textcolor[rgb]{ 1,  0,  0}{0.872 } & \textcolor[rgb]{ 1,  0,  0}{0.952 } & \textcolor[rgb]{ 1,  0,  0}{4.310 } & \textcolor[rgb]{ 1,  0,  0}{35.198 } & \textcolor[rgb]{ 1,  0,  0}{0.917 } & \textcolor[rgb]{ 1,  0,  0}{0.969 } & \textcolor[rgb]{ 1,  0,  0}{4.568 } & \textcolor[rgb]{ 1,  0,  0}{37.373 } & \textcolor[rgb]{ 1,  0,  0}{0.948 } & \textcolor[rgb]{ 1,  0,  0}{0.981 } & \textbf{4.384 } \\
& TNN   & 30.781  & 0.836  & 0.939  & 31.306  & 32.514  & 0.882  & 0.955  & 65.534  & 34.424  & 0.917  & 0.969  & 64.443  \\
& PSTNN & 31.044  & 0.844  & 0.941  & 10.874  & 32.747  & 0.887  & 0.956  & 11.840  & 34.571  & 0.920  & 0.970  & 11.843  \\
& T-TNN & \textbf{31.842 } & \textbf{0.856 } & \textbf{0.946 } & 99.708  & \textbf{33.592 } & \textbf{0.898 } & \textbf{0.961 } & 80.806  & \textbf{35.452 } & \textbf{0.930 } & \textbf{0.974 } & 67.539  \\
& HaLRTC & 30.872  & 0.851  & 0.936  & \textbf{5.226 } & 32.358  & 0.887  & 0.951  & \textbf{5.555 } & 33.956  & 0.916  & 0.965  & \textcolor[rgb]{ 1,  0,  0}{3.460 } \\
\hline
\multirow{5}{*}{Leaves} & TNNR  & \textcolor[rgb]{ 1,  0,  0}{29.517 } & \textcolor[rgb]{ 1,  0,  0}{0.777 } & \textcolor[rgb]{ 1,  0,  0}{0.924 } & \textbf{4.254 } & \textcolor[rgb]{ 1,  0,  0}{31.170 } & \textcolor[rgb]{ 1,  0,  0}{0.836 } & \textcolor[rgb]{ 1,  0,  0}{0.944 } & \textbf{4.432 } & \textcolor[rgb]{ 1,  0,  0}{32.790 } & \textcolor[rgb]{ 1,  0,  0}{0.879 } & \textcolor[rgb]{ 1,  0,  0}{0.960 } & \textbf{4.511 } \\
& TNN   & 28.998  & 0.767  & 0.918  & 30.869  & 30.294  & 0.819  & 0.936  & 61.087  & 31.645  & 0.862  & 0.950  & 54.025  \\
& PSTNN & 29.228  & 0.776  & 0.919  & 5.087  & 30.468  & 0.825  & \textbf{0.937 } & 11.361  & 31.765  & 0.865  & 0.951  & 11.115  \\
& T-TNN & \textbf{29.296 } & \textbf{0.771 } & \textbf{0.918 } & 172.234  & \textbf{30.698 } & \textbf{0.828 } & \textbf{0.937 } & 87.762  & \textbf{32.078 } & \textbf{0.871 } & \textbf{0.952 } & 73.918  \\
& HaLRTC & 28.670  & 0.766  & 0.905  & \textcolor[rgb]{ 1,  0,  0}{3.758 } & 29.821  & 0.812  & 0.923  & \textcolor[rgb]{ 1,  0,  0}{1.980 } & 31.052  & 0.853  & 0.941  & \textcolor[rgb]{ 1,  0,  0}{3.580 } \\
\hline
\multirow{5}{*}{Barbed Wire} & TNNR  & \textcolor[rgb]{ 1,  0,  0}{21.977 } & \textcolor[rgb]{ 1,  0,  0}{0.703 } & \textcolor[rgb]{ 1,  0,  0}{0.862 } & \textbf{4.127 } & \textcolor[rgb]{ 1,  0,  0}{23.622 } & \textcolor[rgb]{ 1,  0,  0}{0.771 } & \textcolor[rgb]{ 1,  0,  0}{0.895 } & \textcolor[rgb]{ 1,  0,  0}{3.927 } & \textcolor[rgb]{ 1,  0,  0}{25.384 } & \textcolor[rgb]{ 1,  0,  0}{0.832 } & \textcolor[rgb]{ 1,  0,  0}{0.922 } & \textbf{4.008 } \\
& TNN   & 21.665  & 0.675  & 0.839  & 48.753  & 23.073  & 0.750  & 0.876  & 62.956  & 24.414  & 0.806  & 0.905  & 63.915  \\
& PSTNN & \textbf{21.815 } & \textbf{0.684 } & \textbf{0.845 } & 4.966  & \textbf{23.212 } & \textbf{0.756 } & \textbf{0.881 } & 11.410  & \textbf{24.538 } & \textbf{0.811 } & \textbf{0.909 } & 11.611  \\
& T-TNN & 18.385  & 0.609  & 0.805  & 306.874  & 20.054  & 0.704  & 0.850  & 369.429  & 22.445  & 0.793  & 0.895  & 263.812  \\
& HaLRTC & 20.328  & 0.600  & 0.785  & \textcolor[rgb]{ 1,  0,  0}{3.488 } & 21.327  & 0.671  & 0.825  & \textbf{4.591 } & 22.304  & 0.728  & 0.857  & \textcolor[rgb]{ 1,  0,  0}{1.988 } \\
\hline
	\end{tabular}%
	\label{tab:Texture}%
\end{table*}%

For further comparison, we recover  images of the deterministically masked images by ``Rectangle" and ``Grid". Clearly, the masked images are no-mean-sampling. 
As shown in Figure \ref{fig:Mask} and Table \ref{tab:Mask}, we display the recovered results by different methods. Clearly, TNNR obtains the most visually satisfying results among the compared methods, and the PSNR, SSIM, FSIM results are higher than the corresponding compared methods in all cases. From time consumption, TNNR is the fastest method. The mask image inpainting results are also consistent with the color and texture image inpainting results and all these demonstrate that TNNR can perform tensor completion better and runs more efficiently.

% Table generated by Excel2LaTeX from sheet 'Mask'
\begin{table}[htbp]
	\centering
	\caption{MASK IMAGE INPAINTING PERFORMANCE COMPARISON: PSNR, SSIM, FSIM AND RUNNING TIME. THE BEST AND THE SECOND BEST PERFORMING METHODS IN EACH IMAGE ARE HIGHLIGHTED IN RED AND BOLD, RESPECTIVELY}
	\begin{tabular}{c|c|cccc}
		\hline
Mask  & Methods & PSNR  & SSIM  & FSIM  & Time \\
\hline
\multirow{5}{*}{Rectangle} & TNNR  & \textcolor[rgb]{ 1,  0,  0}{37.764 } & \textcolor[rgb]{ 1,  0,  0}{0.985 } & \textcolor[rgb]{ 1,  0,  0}{0.992 } & \textcolor[rgb]{ 1,  0,  0}{1.978 } \\
& TNN   & 36.479  & \textbf{0.984 } & \textbf{0.991 } & 34.425  \\
& PSTNN & \textbf{36.756 } & \textbf{0.984 } & \textcolor[rgb]{ 1,  0,  0}{0.992 } & 6.515  \\
& T-TNN & 16.306  & 0.943  & 0.915  & 328.535  \\
& HaLRTC & 27.030  & 0.972  & 0.964  & \textbf{4.999 }\\
\hline
\multirow{5}{*}{Grid} & TNNR  & \textcolor[rgb]{ 1,  0,  0}{28.844 } & \textcolor[rgb]{ 1,  0,  0}{0.945 } & \textcolor[rgb]{ 1,  0,  0}{0.961 } & \textcolor[rgb]{ 1,  0,  0}{1.327 }\\
& TNN   & 27.734  & 0.935  & 0.955  & 42.225  \\
& PSTNN & 27.953  & \textbf{0.938 } & 0.957  & 4.404  \\
& T-TNN & \textbf{28.237 } & \textcolor[rgb]{ 1,  0,  0}{0.945 } & \textbf{0.958 } & 223.260  \\
& HaLRTC & 27.417  & 0.933  & 0.953  & \textbf{1.278 } \\
\hline
	\end{tabular}%
	\label{tab:Mask}%
\end{table}%

\begin{figure*}[htbp]
	\centering
	\begin{subfigure}[b]{1\linewidth}
		\begin{subfigure}[b]{0.138\linewidth}
			\centering
			\includegraphics[width=\linewidth]{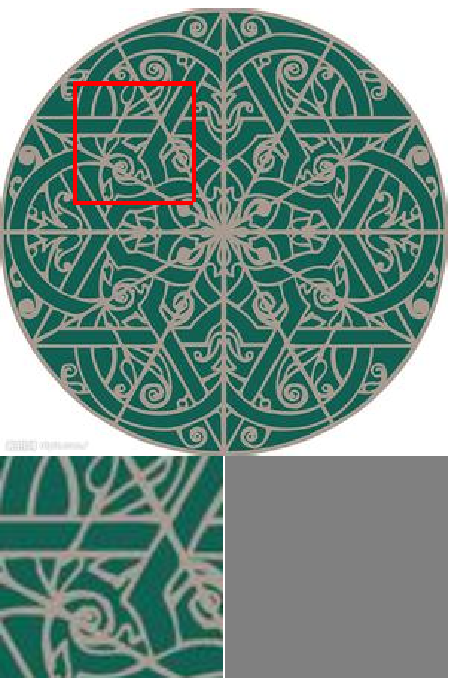}\vspace{0pt}
			\includegraphics[width=\linewidth]{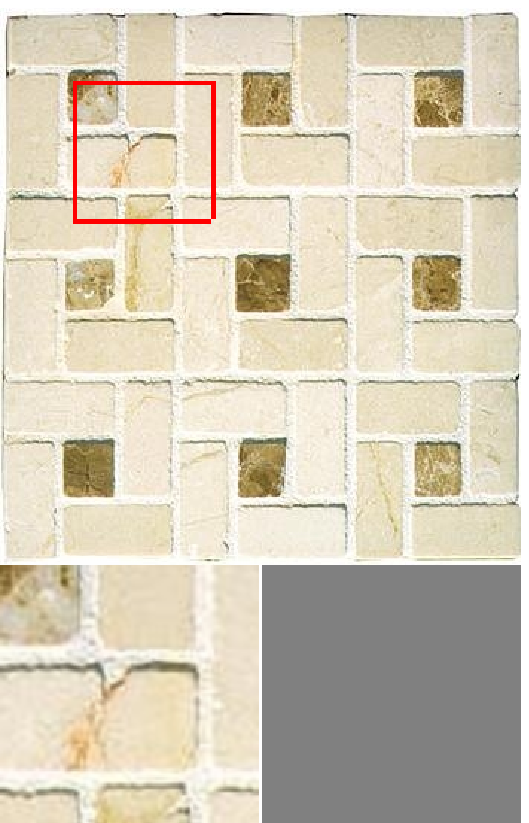}
			\caption{Original}
		\end{subfigure}   	
		\begin{subfigure}[b]{0.138\linewidth}
			\centering
			\includegraphics[width=\linewidth]{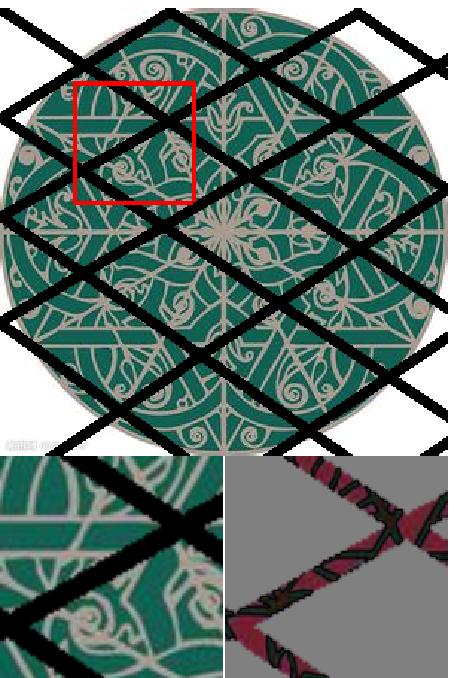}\vspace{0pt}
			\includegraphics[width=\linewidth]{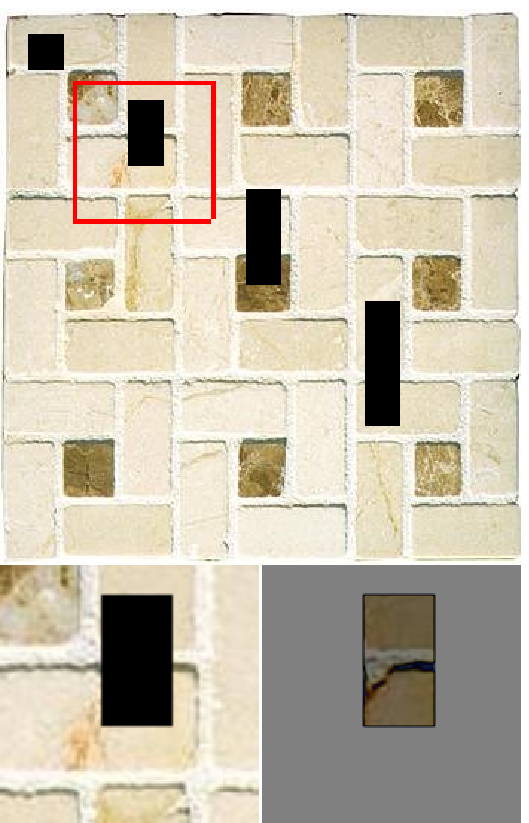}
			\caption{Observed}
		\end{subfigure}
		\begin{subfigure}[b]{0.138\linewidth}
			\centering
			\includegraphics[width=\linewidth]{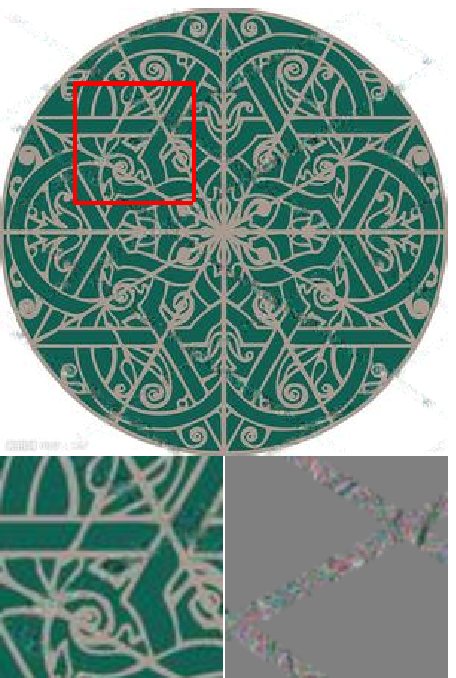}\vspace{0pt}
			\includegraphics[width=\linewidth]{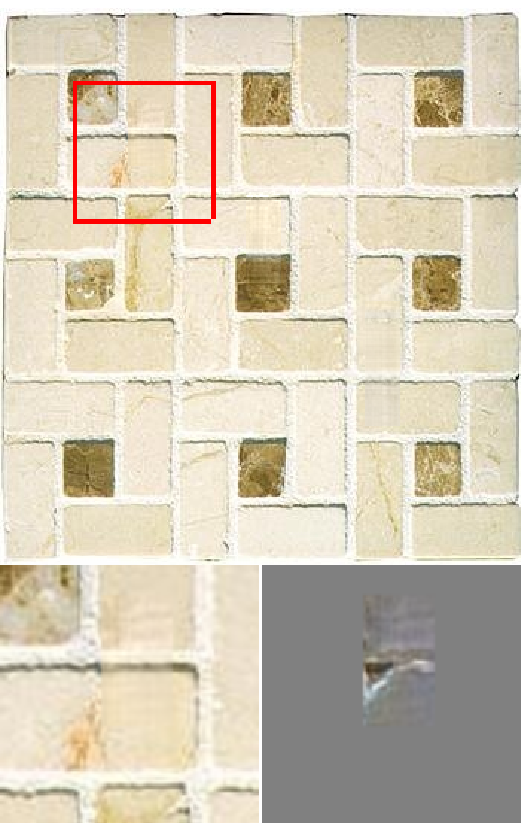}
			\caption{TNNR}
		\end{subfigure}
		\begin{subfigure}[b]{0.138\linewidth}
			\centering		
			\includegraphics[width=\linewidth]{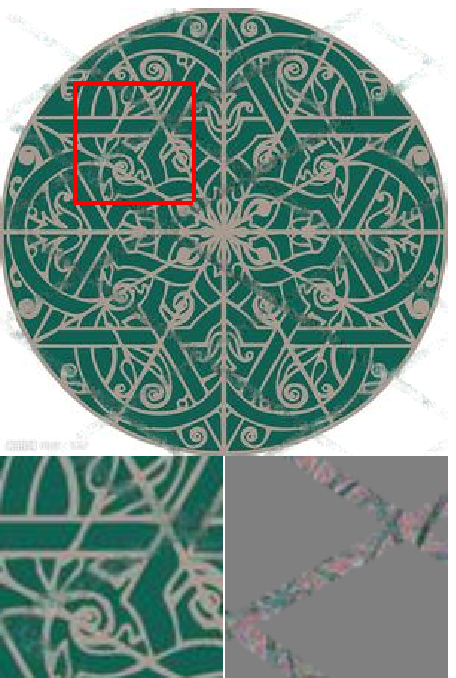}\vspace{0pt}
			\includegraphics[width=\linewidth]{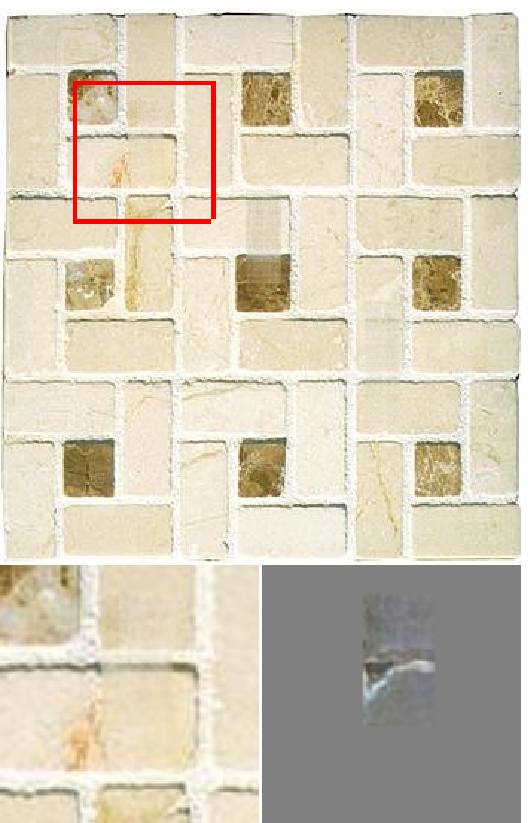}
			\caption{TNN}
		\end{subfigure}
		\begin{subfigure}[b]{0.138\linewidth}
			\centering			
			\includegraphics[width=\linewidth]{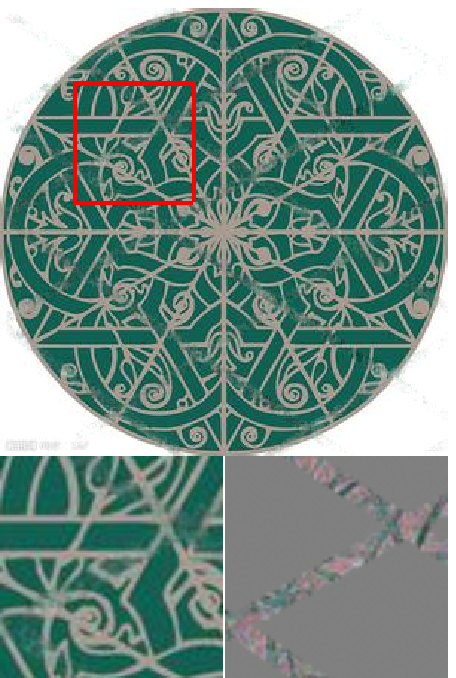}\vspace{0pt}
			\includegraphics[width=\linewidth]{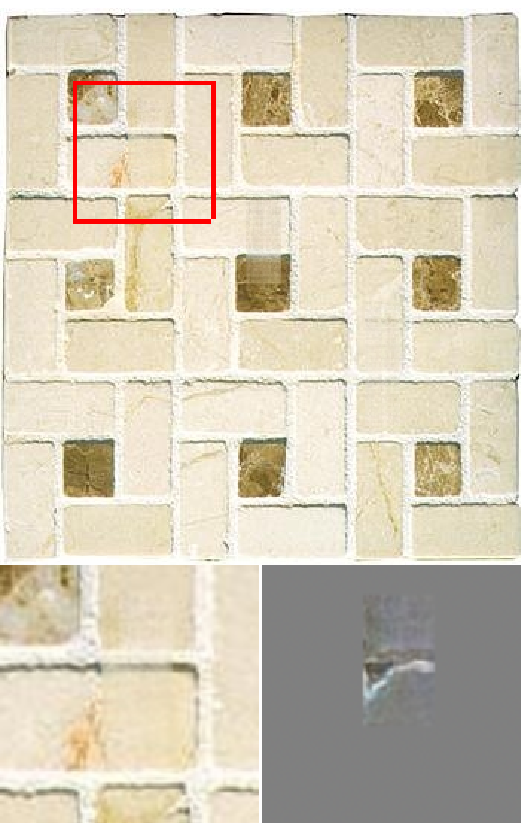}
			\caption{PSTNN}
		\end{subfigure}
		\begin{subfigure}[b]{0.138\linewidth}
			\centering			
			\includegraphics[width=\linewidth]{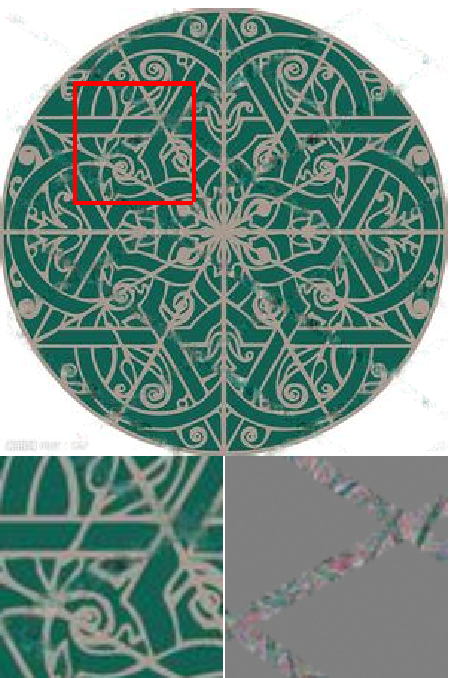}\vspace{0pt}
			\includegraphics[width=\linewidth]{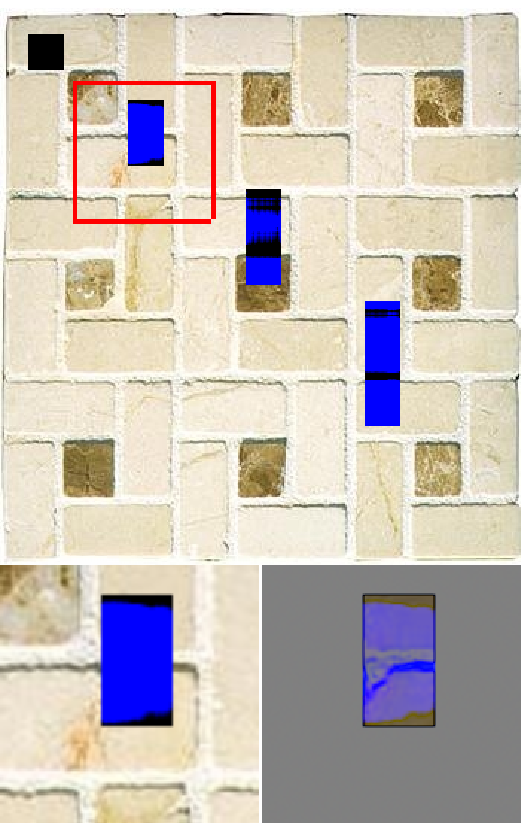}
			\caption{T-TNN}
		\end{subfigure}	
		\begin{subfigure}[b]{0.138\linewidth}
			\centering
			\includegraphics[width=\linewidth]{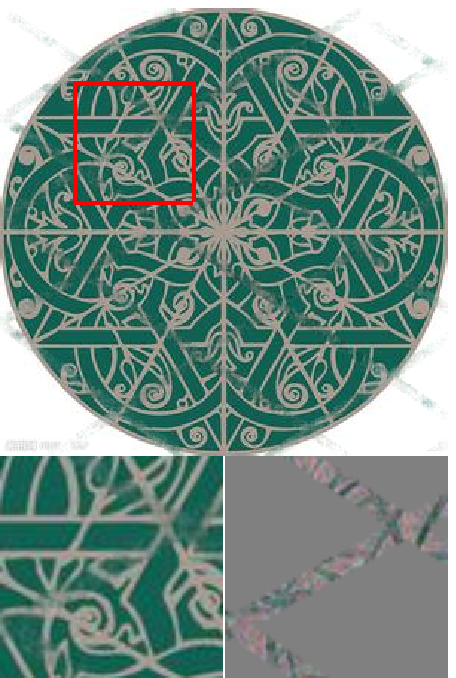}\vspace{0pt}
			\includegraphics[width=\linewidth]{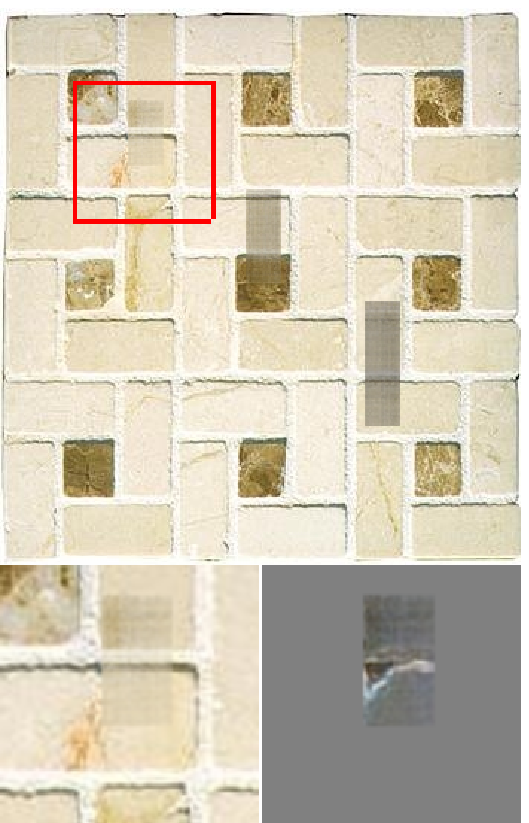}
			\caption{HaLRTC}
		\end{subfigure}					
	\end{subfigure}
	\vfill
	\caption{Examples of mask image inpainting. From top to bottom are ``Grid" and ``Rectangle". For better visualization, we show the zoom-in region and the corresponding error map (difference from the original).}
	\label{fig:Mask}
\end{figure*}

\subsection{Video Inpainting}
We evaluate our method on the widely used YUV Video Sequences\footnote{http://trace.eas.asu.edu/yuv/.}. Each sequence contains at least 150 frames and we use the first 30 frames of each. In the experiments, we test our method and other methods on two videos. The frame sizes of all videos are $ 144 \times 176 $ pixels. The sampling rates (SR) are set as $ 30\% $, $ 35\% $ and $ 40\% $.

As shown in Figure \ref{fig:Video}, each test video is shown at the 8th frame. Based on the results of the two tests, TNNR performs better at filling in the missing values. It is better to deal with the details of the frames. The PSNR, SSIM and FSIM metrics also shows the best results with TNNR, which are consistent with Table \ref{tab:Video}.

\begin{figure*}[htbp]
	\centering
	\begin{subfigure}[b]{1\linewidth}
		\begin{subfigure}[b]{0.138\linewidth}
			\centering
			\includegraphics[width=\linewidth]{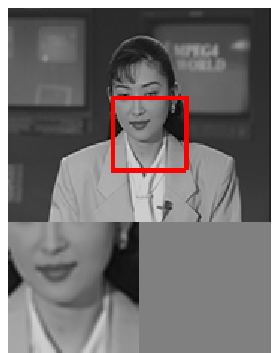}\vspace{0pt}
			\includegraphics[width=\linewidth]{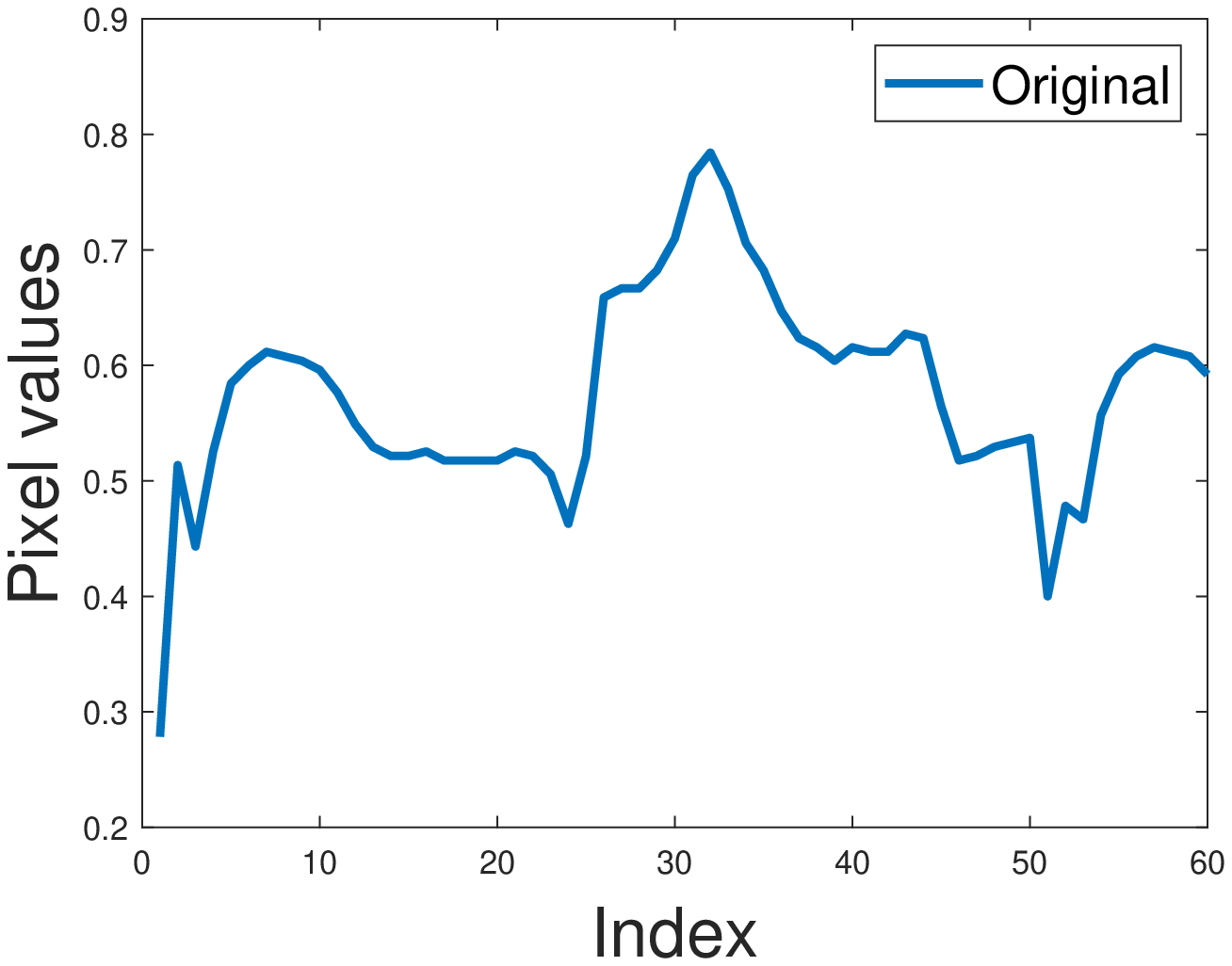}\vspace{0pt}
			\includegraphics[width=\linewidth]{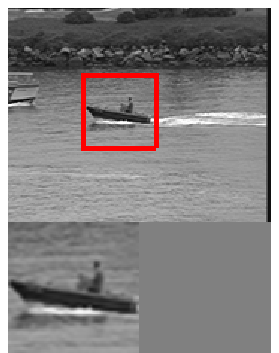}\vspace{0pt}
			\includegraphics[width=\linewidth]{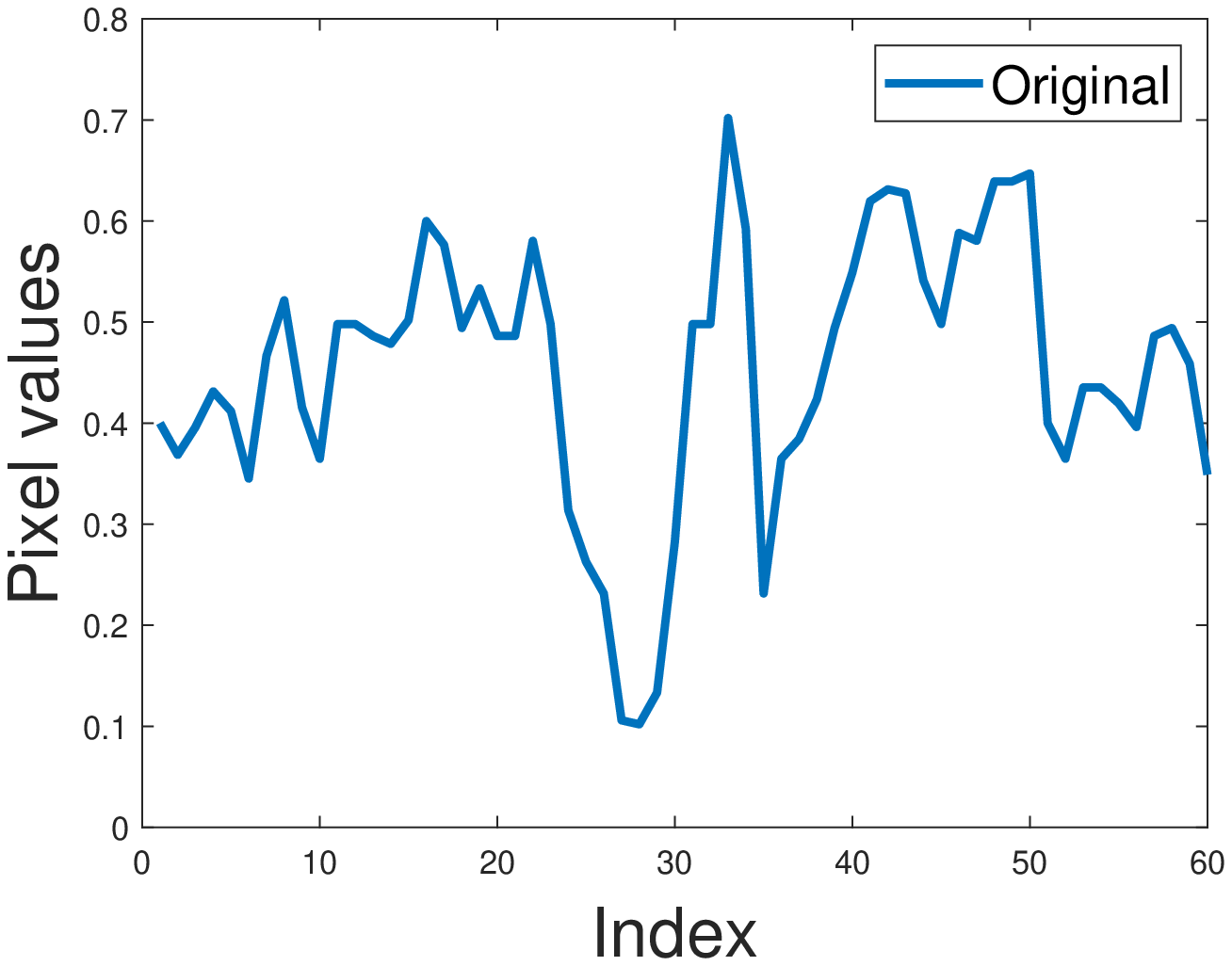}
			\caption{Original}
		\end{subfigure}   	
		\begin{subfigure}[b]{0.138\linewidth}
			\centering
			\includegraphics[width=\linewidth]{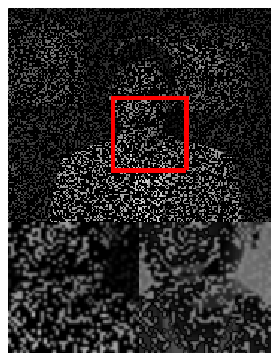}\vspace{0pt}
			\includegraphics[width=\linewidth]{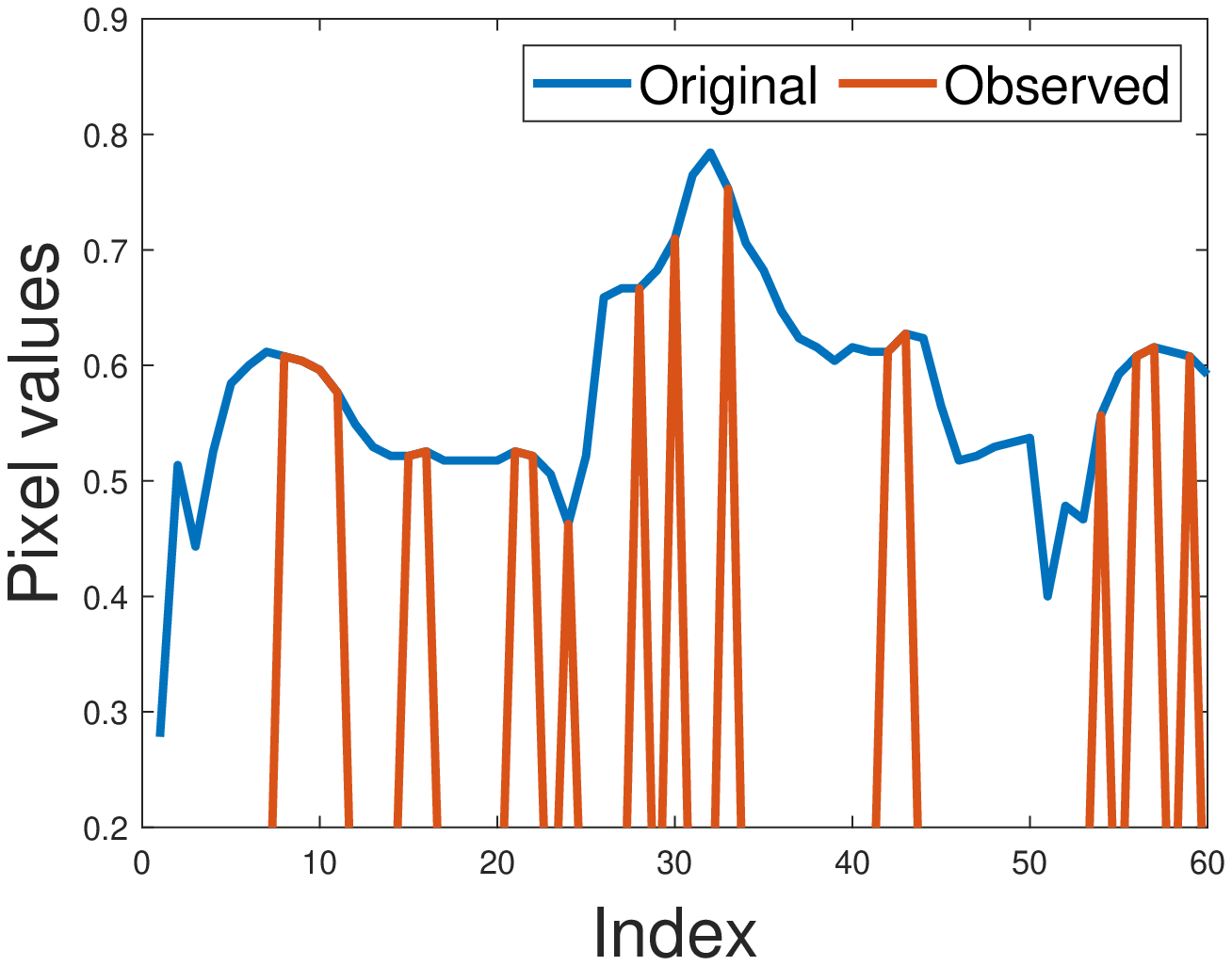}\vspace{0pt}
			\includegraphics[width=\linewidth]{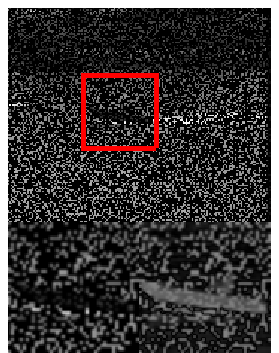}\vspace{0pt}
			\includegraphics[width=\linewidth]{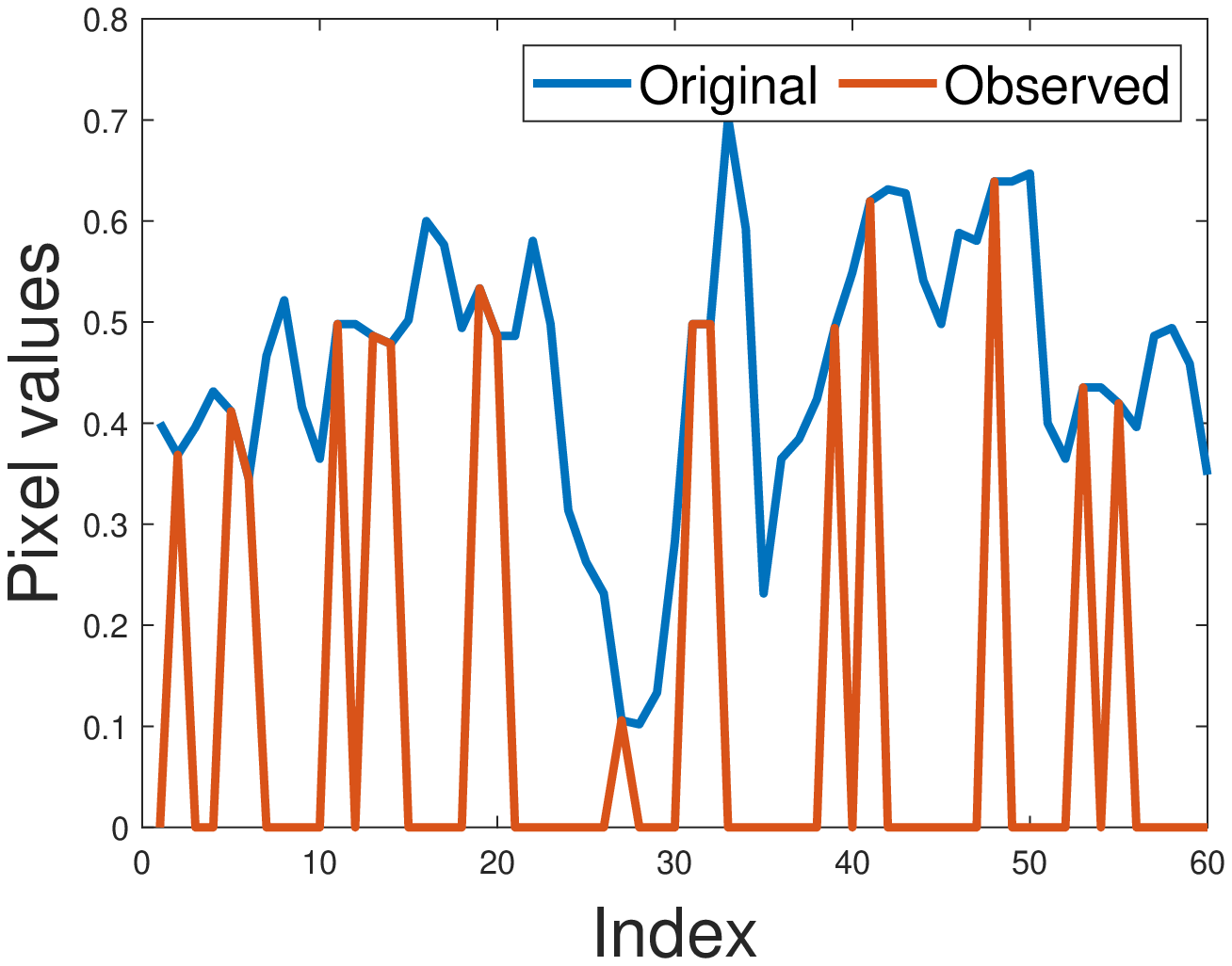}
			\caption{Observed}
		\end{subfigure}
		\begin{subfigure}[b]{0.138\linewidth}
			\centering
			\includegraphics[width=\linewidth]{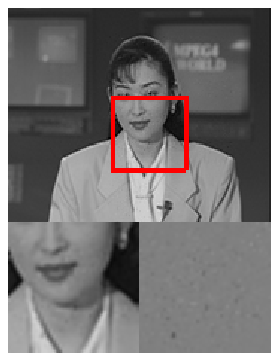}\vspace{0pt}
			\includegraphics[width=\linewidth]{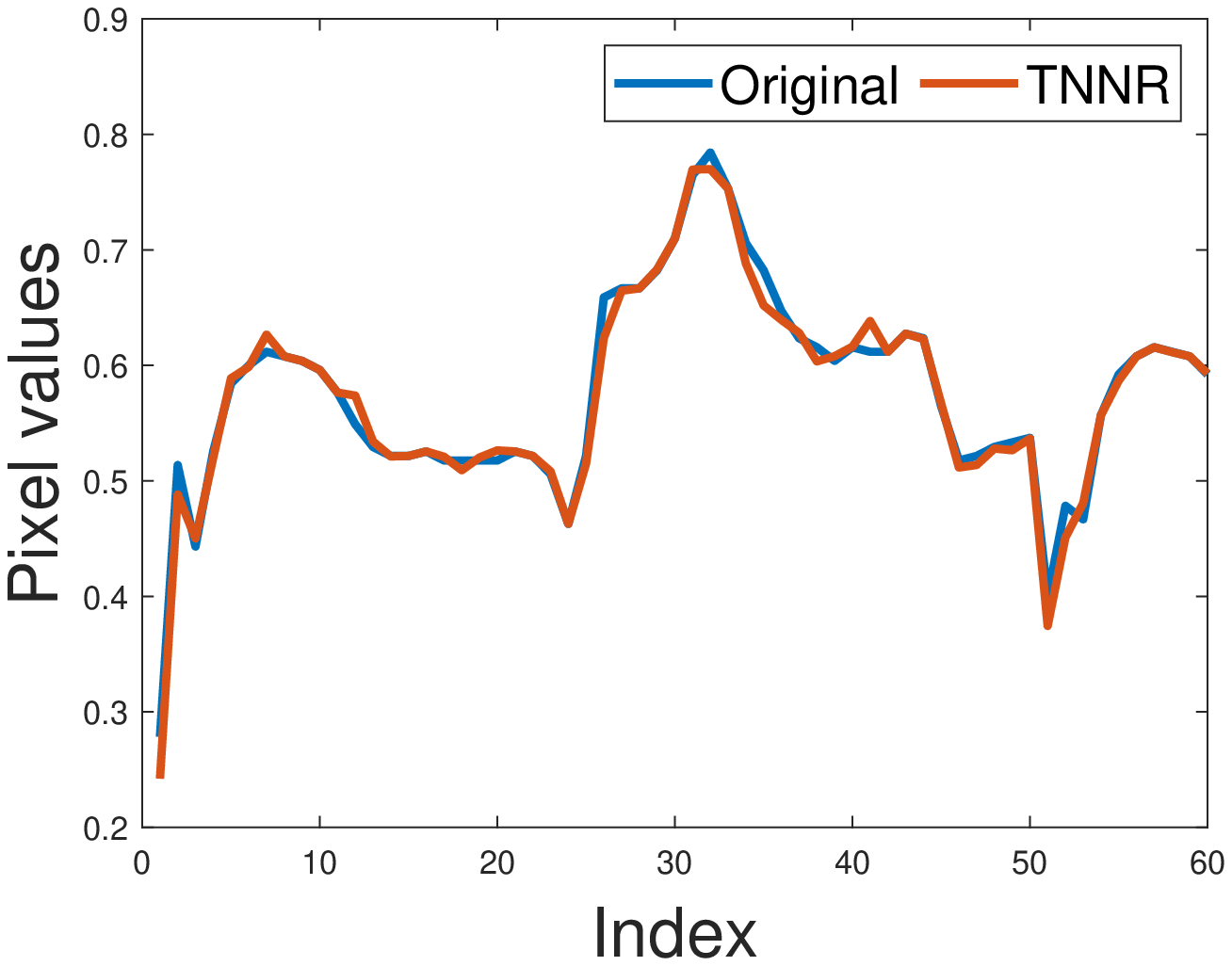}\vspace{0pt}
			\includegraphics[width=\linewidth]{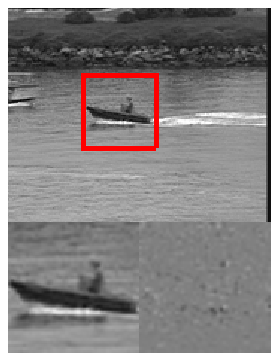}\vspace{0pt}
			\includegraphics[width=\linewidth]{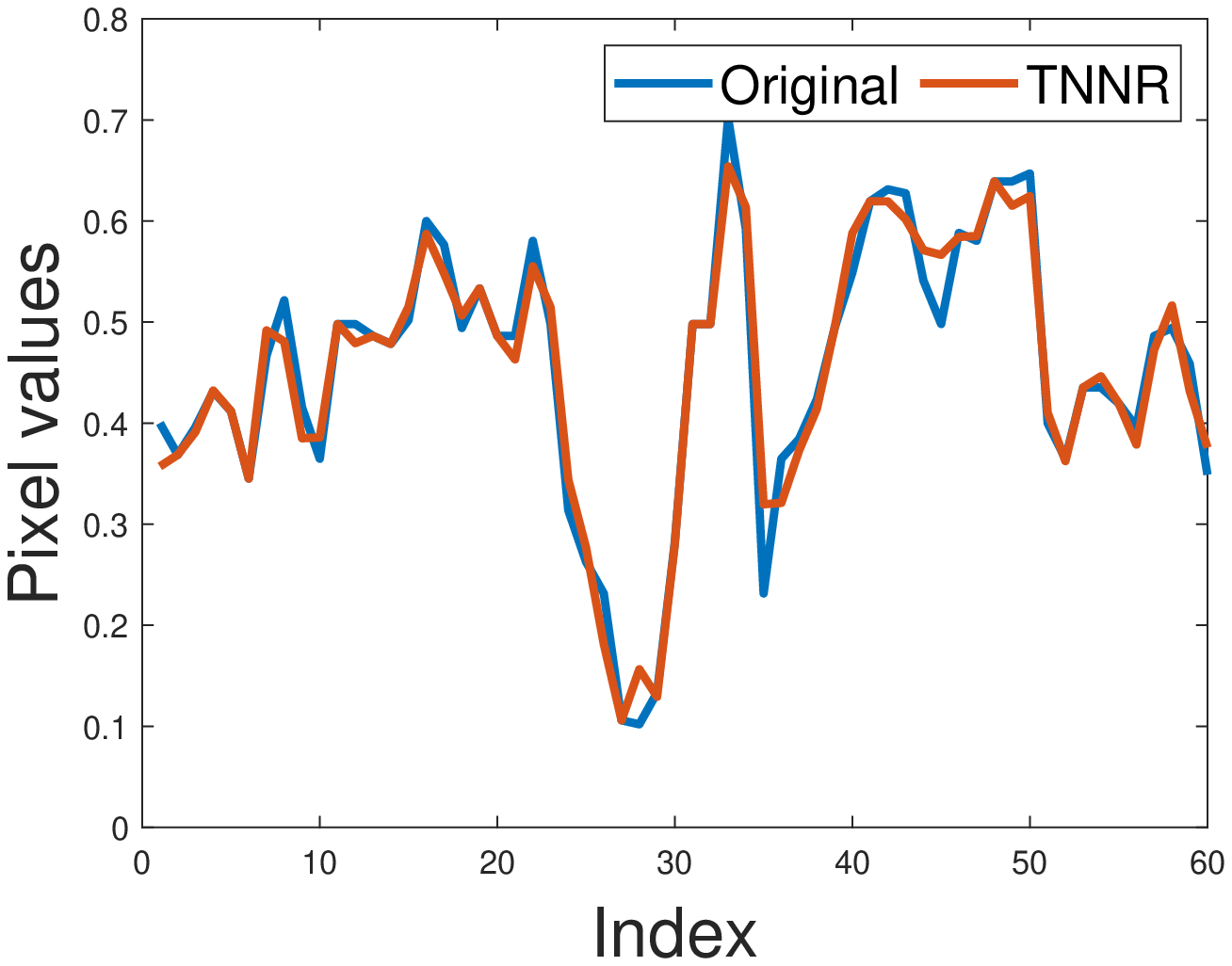}
			\caption{TNNR}
		\end{subfigure}
		\begin{subfigure}[b]{0.138\linewidth}
			\centering		
			\includegraphics[width=\linewidth]{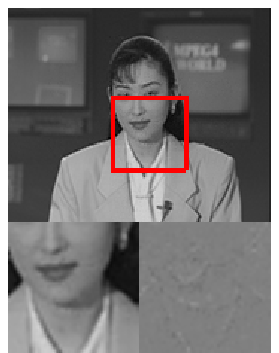}\vspace{0pt}
			\includegraphics[width=\linewidth]{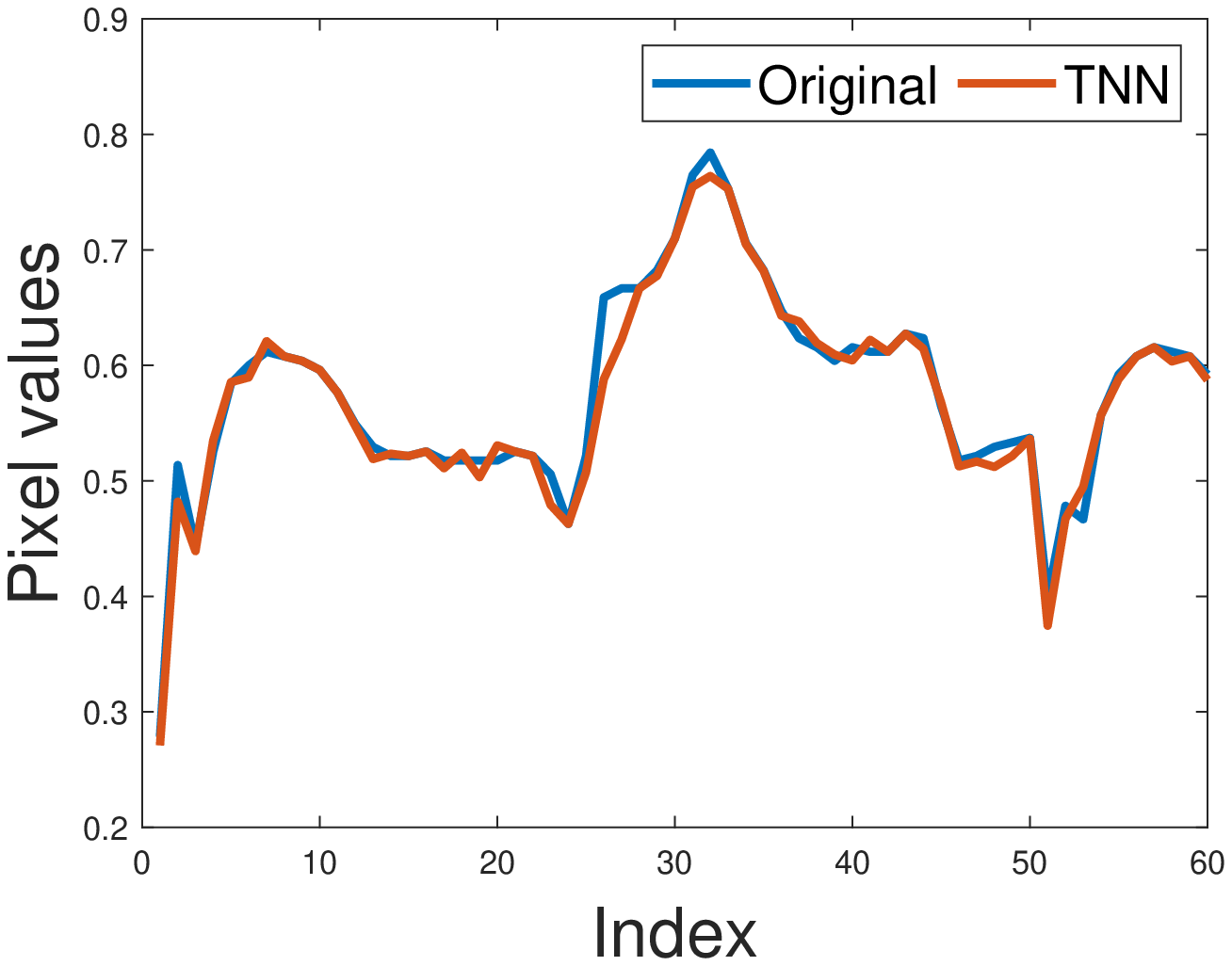}\vspace{0pt}
			\includegraphics[width=\linewidth]{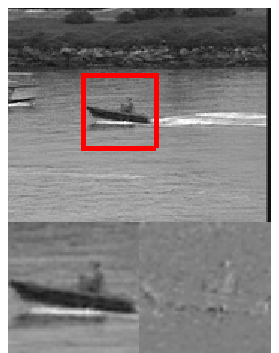}\vspace{0pt}
			\includegraphics[width=\linewidth]{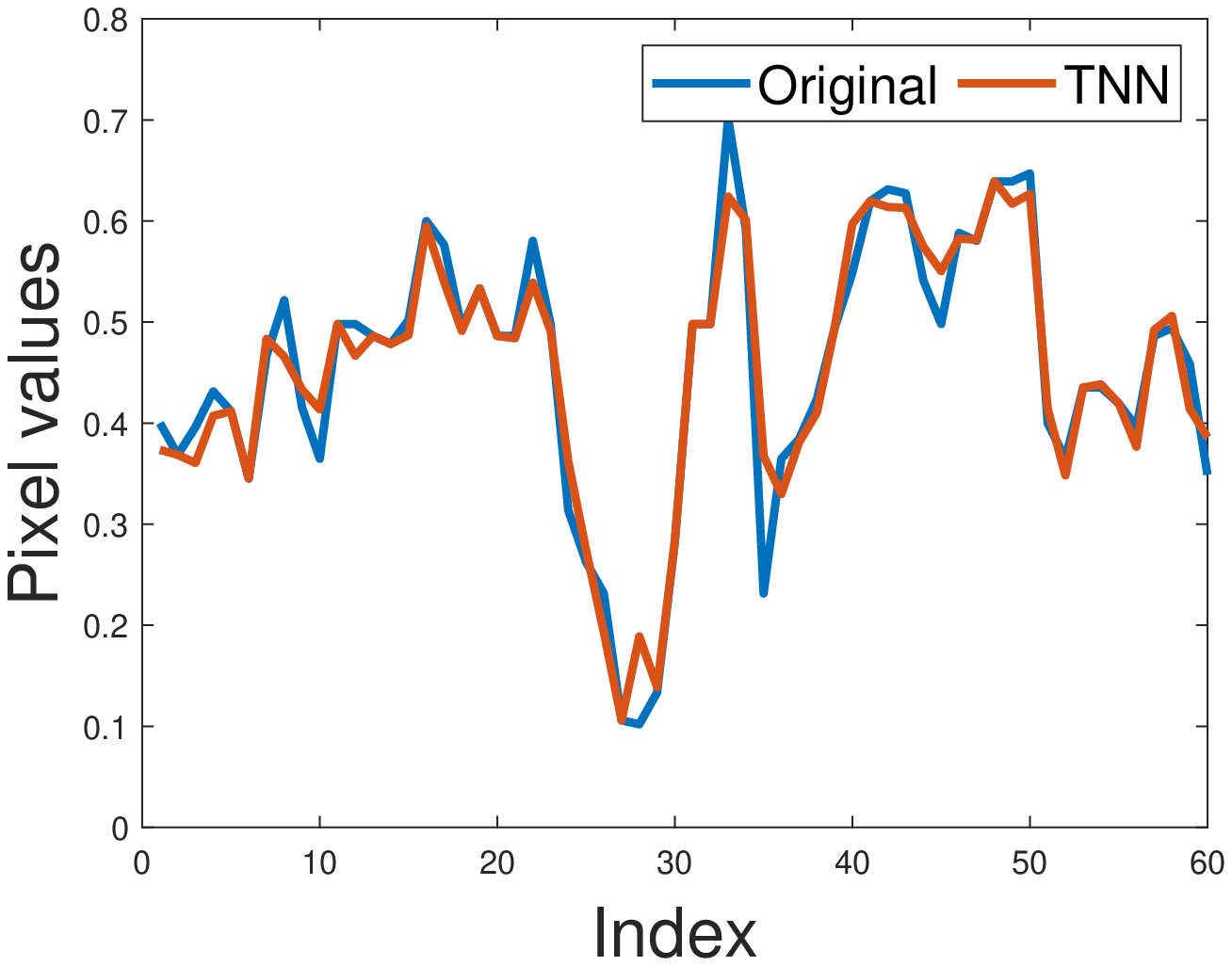}
			\caption{TNN}
		\end{subfigure}
		\begin{subfigure}[b]{0.138\linewidth}
			\centering			
			\includegraphics[width=\linewidth]{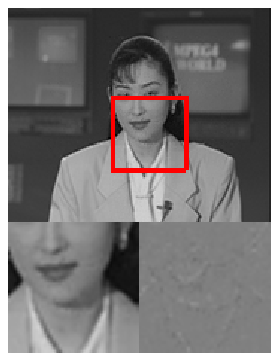}\vspace{0pt}
			\includegraphics[width=\linewidth]{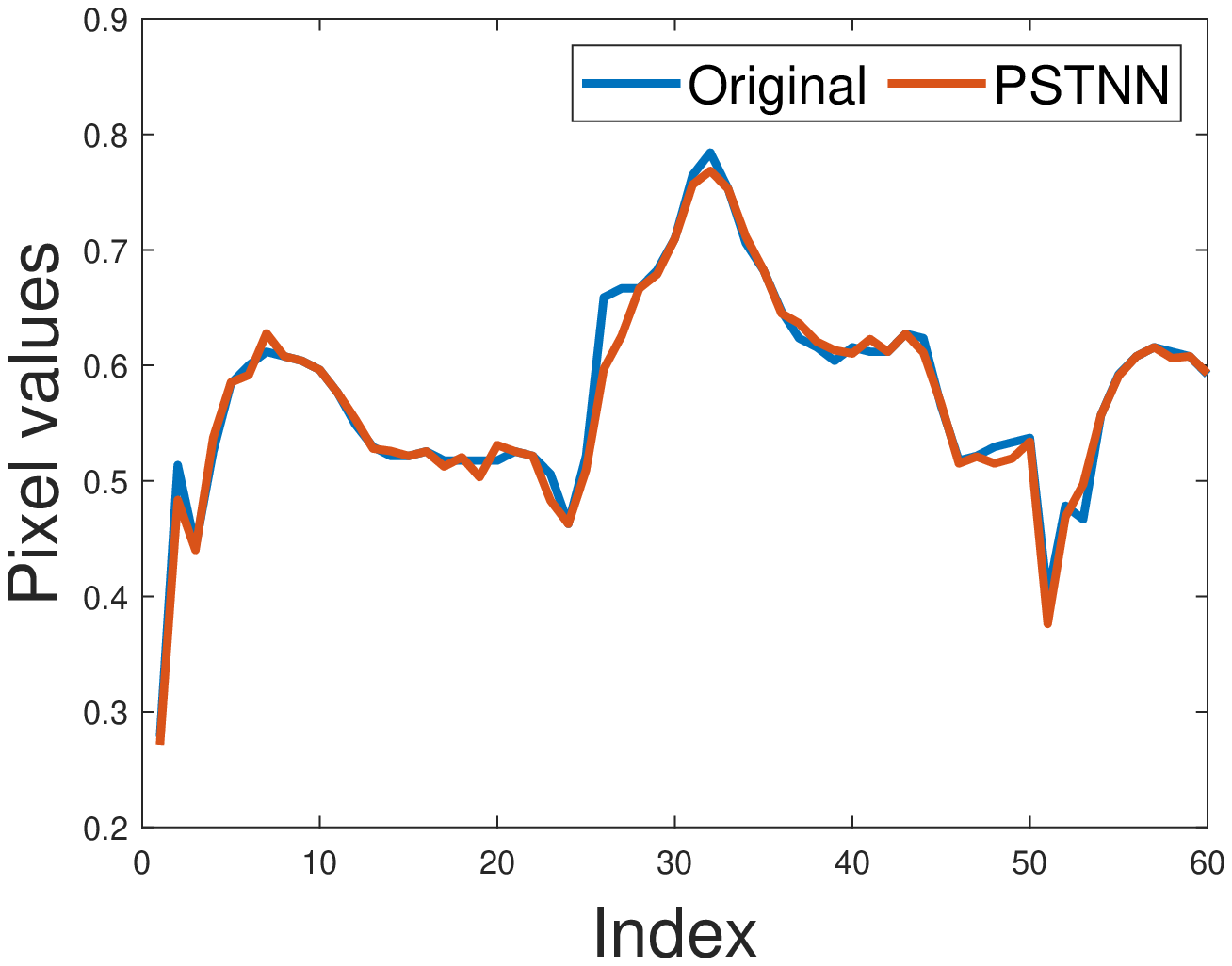}\vspace{0pt}
			\includegraphics[width=\linewidth]{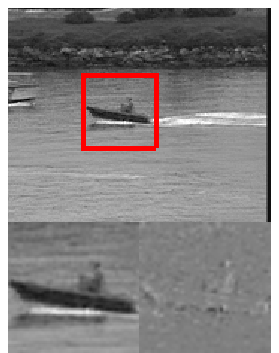}\vspace{0pt}
			\includegraphics[width=\linewidth]{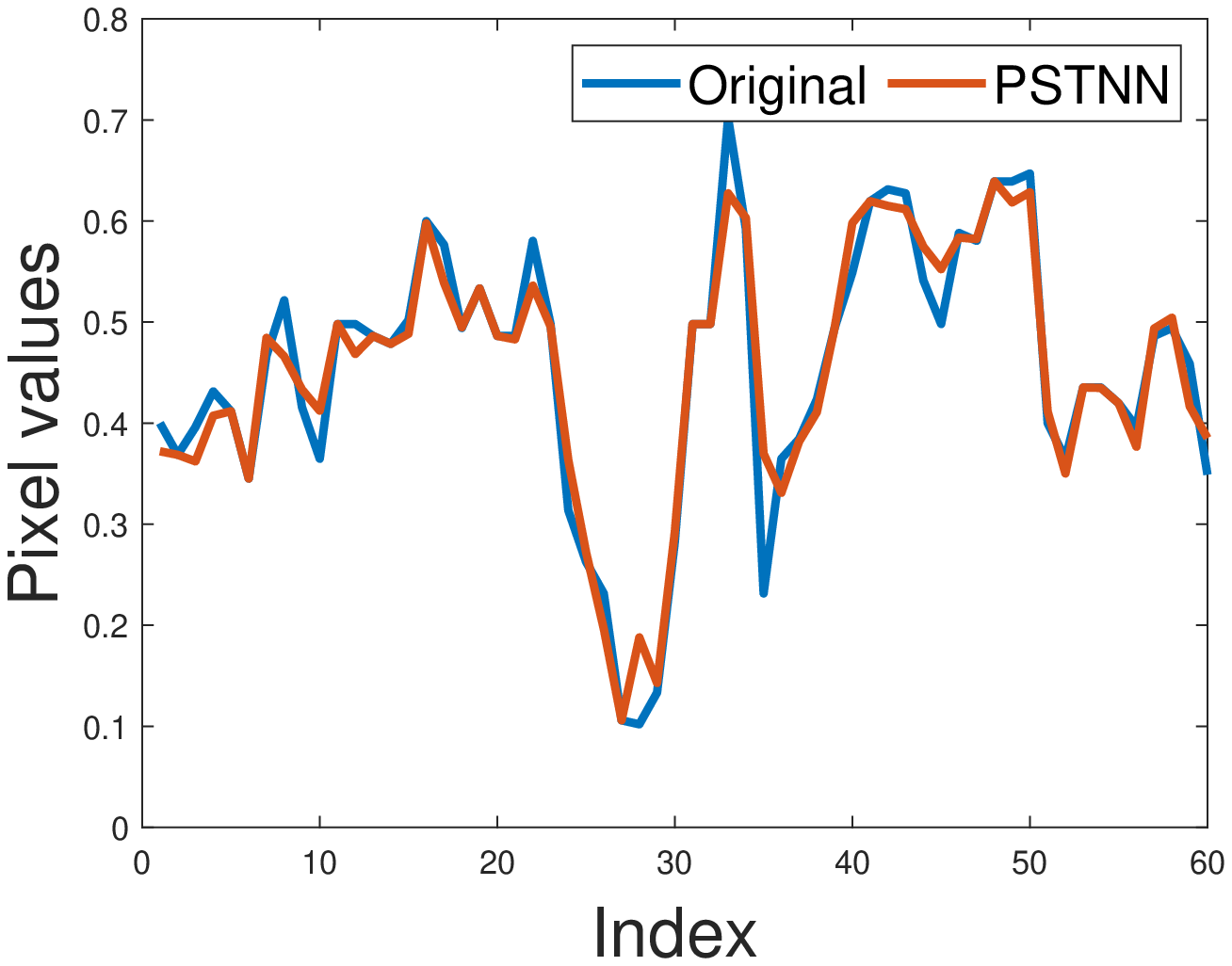}
			\caption{PSTNN}
		\end{subfigure}
		\begin{subfigure}[b]{0.138\linewidth}
			\centering			
			\includegraphics[width=\linewidth]{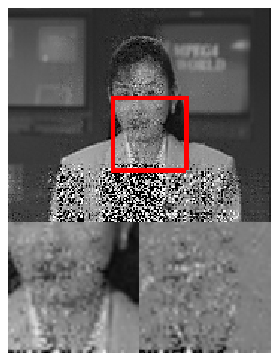}\vspace{0pt}
			\includegraphics[width=\linewidth]{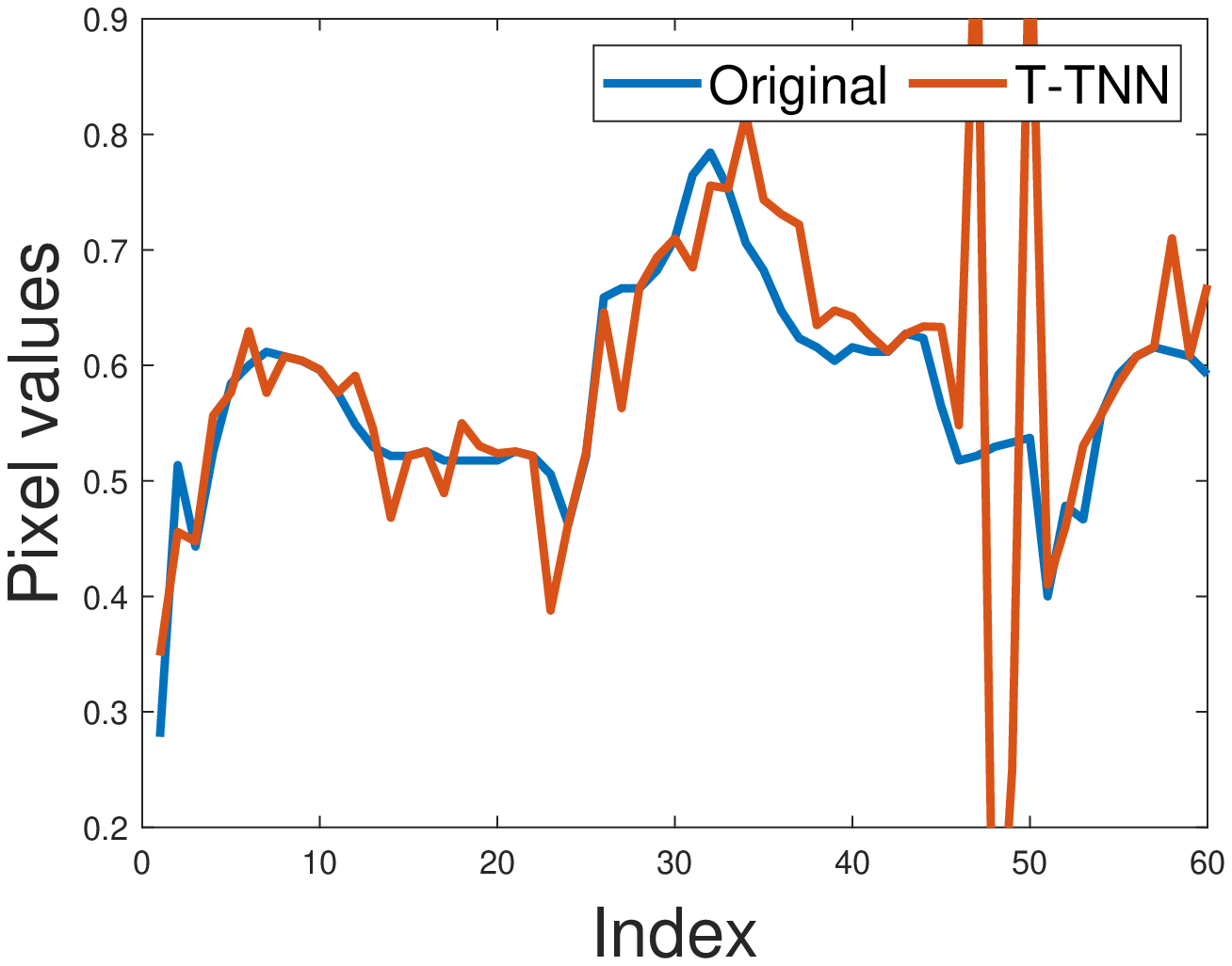}\vspace{0pt}
			\includegraphics[width=\linewidth]{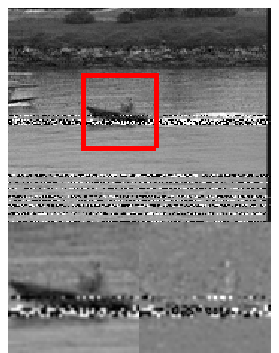}\vspace{0pt}
			\includegraphics[width=\linewidth]{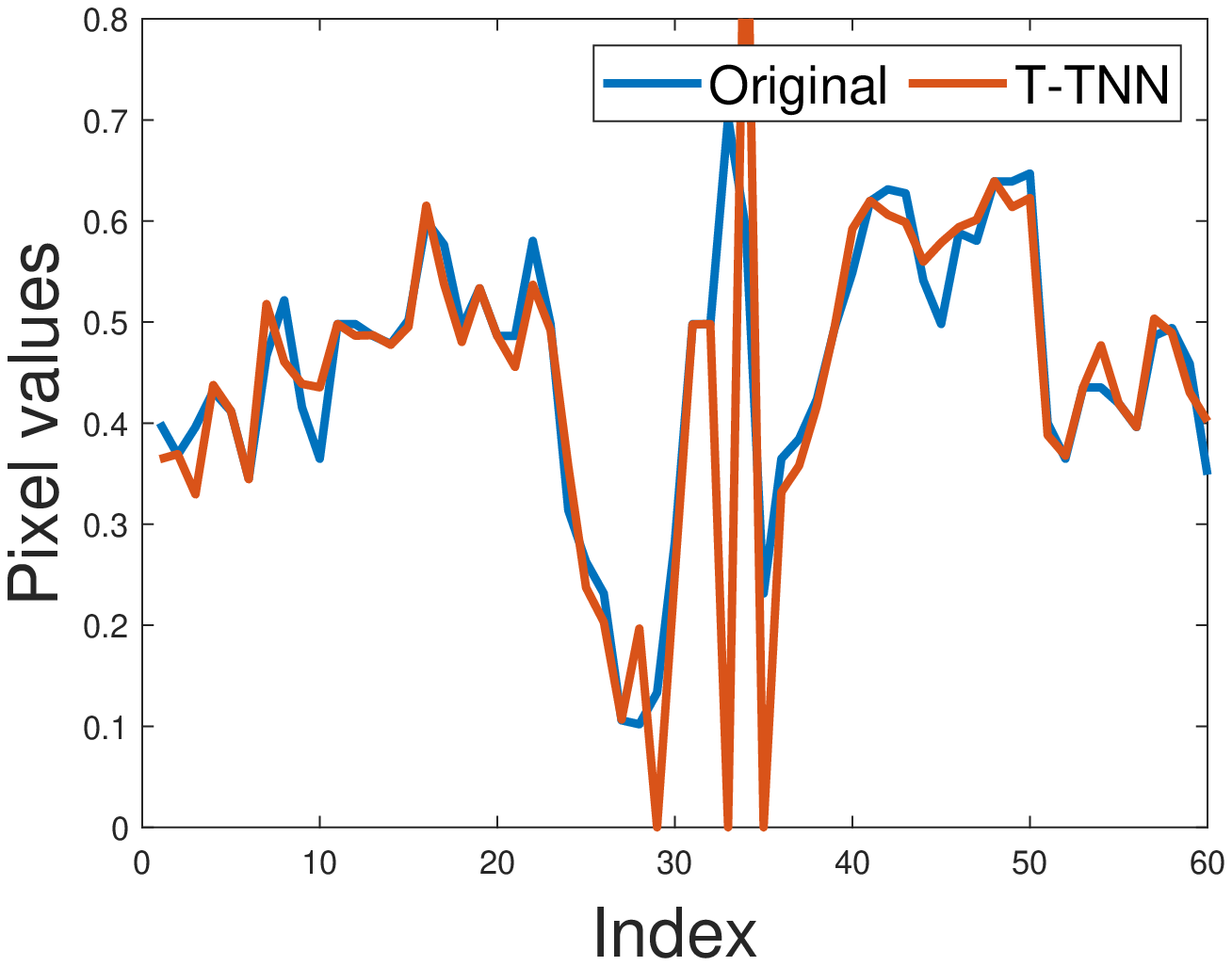}
			\caption{T-TNN}
		\end{subfigure}	
		\begin{subfigure}[b]{0.138\linewidth}
			\centering
			\includegraphics[width=\linewidth]{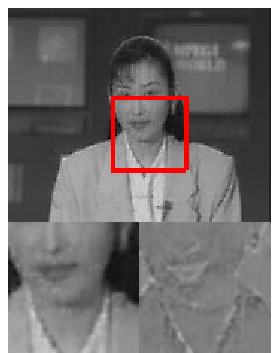}\vspace{0pt}
			\includegraphics[width=\linewidth]{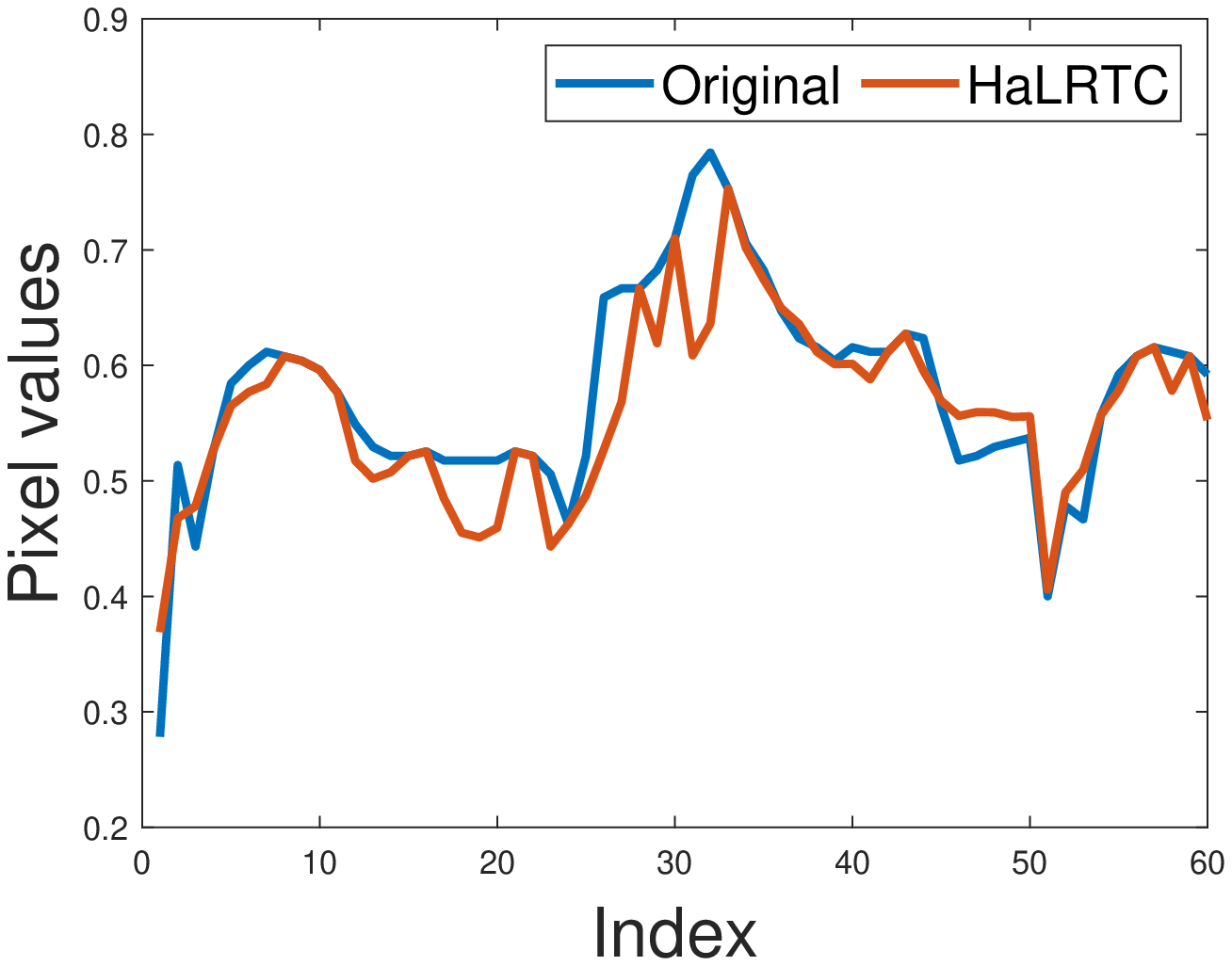}\vspace{0pt}
			\includegraphics[width=\linewidth]{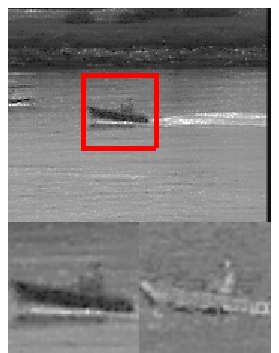}\vspace{0pt}
			\includegraphics[width=\linewidth]{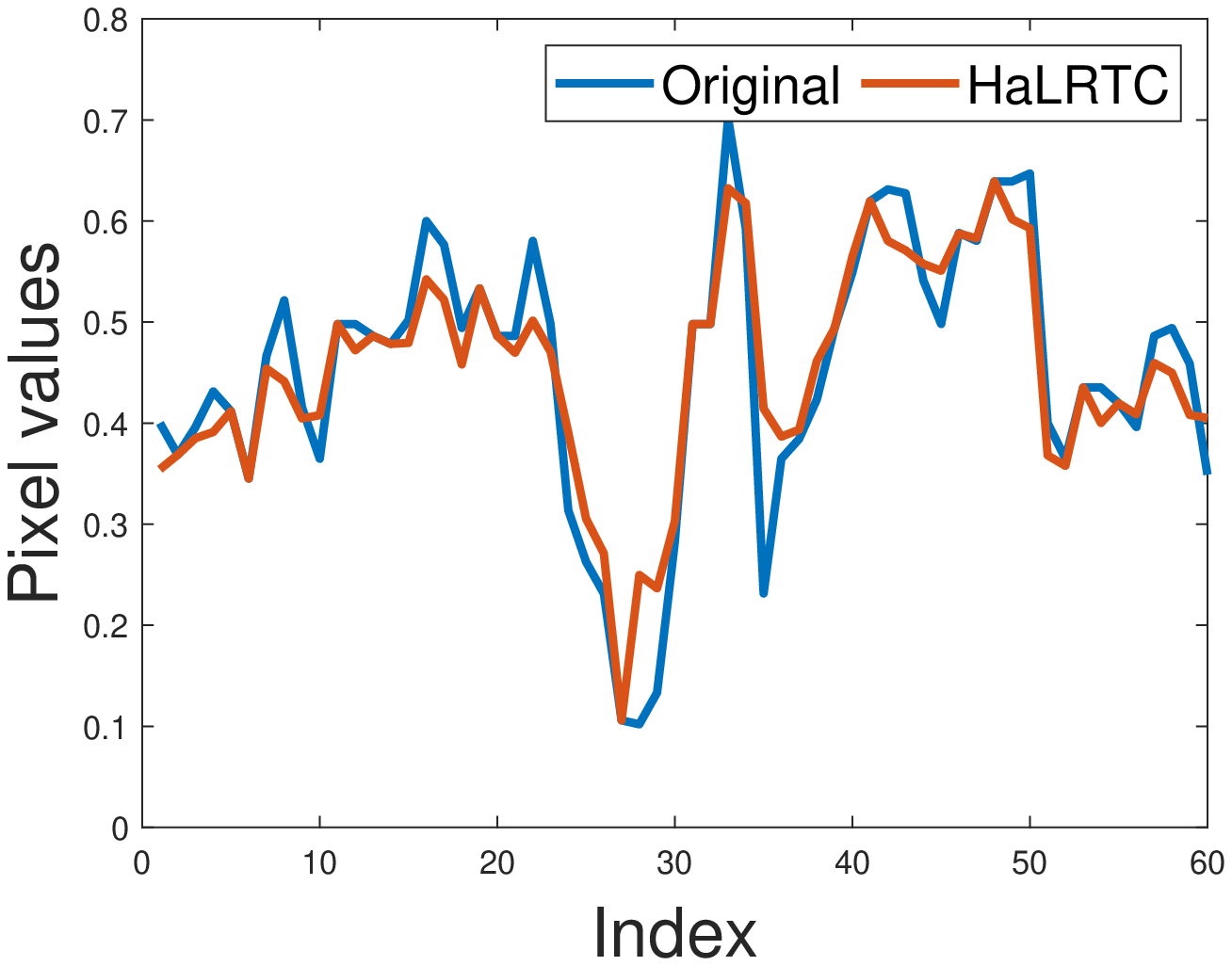}
			\caption{HaLRTC}
		\end{subfigure}					
	\end{subfigure}
	\vfill
	\caption{Examples of video inpainting with $ SR=30\% $. From top to bottom are ``Akiyo" and ``Coastguard". For better visualization, we show the zoom-in region, the corresponding error map (difference from the original) and the corresponding partial residuals of the region.}
	\label{fig:Video}
\end{figure*}

% Table generated by Excel2LaTeX from sheet 'Video'
\begin{table*}[htbp]
	\centering
	\caption{VIDEO INPAINTING PERFORMANCE COMPARISON: PSNR, SSIM, FSIM AND RUNNING TIME. THE BEST AND THE SECOND BEST PERFORMING METHODS IN EACH IMAGE ARE HIGHLIGHTED IN RED AND BOLD, RESPECTIVELY}
	\begin{tabular}{c|c|cccc|cccc|cccc}
		\hline
		\multirow{2}{*}{Video} & \multirow{2}{*}{Methods} & \multicolumn{4}{c|}{$SR=30\%$} & \multicolumn{4}{c|}{$SR=35\%$} & \multicolumn{4}{c}{$SR=40\%$} \\
		\cline{3-14}          &       & PSNR  & SSIM  & FSIM  & Time  & PSNR  & SSIM  & FSIM  & Time  & PSNR  & SSIM  & FSIM  & Time \\
		\hline
		\multirow{5}{*}{Akiyo} & TNNR  & \textcolor[rgb]{ 1,  0,  0}{40.255 } & \textcolor[rgb]{ 1,  0,  0}{0.987 } & \textcolor[rgb]{ 1,  0,  0}{0.989 } & \textbf{53.612 } & \textcolor[rgb]{ 1,  0,  0}{41.423 } & \textcolor[rgb]{ 1,  0,  0}{0.991 } & \textcolor[rgb]{ 1,  0,  0}{0.991 } & 51.388  & \textcolor[rgb]{ 1,  0,  0}{42.747 } & \textcolor[rgb]{ 1,  0,  0}{0.993 } & \textcolor[rgb]{ 1,  0,  0}{0.993 } & 24.766  \\
		& TNN   & 39.129  & 0.985  & \textcolor[rgb]{ 1,  0,  0}{0.989 } & 115.522  & 40.494  & 0.989  & \textcolor[rgb]{ 1,  0,  0}{0.991 } & 122.066  & 41.512  & 0.991  & \textcolor[rgb]{ 1,  0,  0}{0.993 } & 61.382  \\
		& PSTNN & \textbf{39.262 } & \textbf{0.986 } & \textcolor[rgb]{ 1,  0,  0}{0.989 } & 41.657  & \textbf{40.583 } & \textbf{0.990 } & \textcolor[rgb]{ 1,  0,  0}{0.991 } & \textbf{41.010 } & \textbf{41.562 } & \textbf{0.992 } & \textcolor[rgb]{ 1,  0,  0}{0.993 } & \textbf{17.285 } \\
		& T-TNN & 15.797  & 0.696  & 0.839  & 1243.143  & 17.909  & 0.739  & 0.874  & 766.470  & 19.351  & 0.797  & 0.904  & 513.139  \\
		& HaLRTC & 31.214  & 0.924  & \textbf{0.963 } & \textcolor[rgb]{ 1,  0,  0}{1.897 } & 33.076  & 0.949  & \textbf{0.975 } & \textcolor[rgb]{ 1,  0,  0}{1.125 } & 34.811  & 0.964  & \textbf{0.983 } & \textcolor[rgb]{ 1,  0,  0}{1.194 } \\
		\hline
		\multirow{5}{*}{Coastguard} & TNNR  & \textcolor[rgb]{ 1,  0,  0}{31.452 } & \textcolor[rgb]{ 1,  0,  0}{0.900 } & \textcolor[rgb]{ 1,  0,  0}{0.931 } & \textbf{11.747 } & \textcolor[rgb]{ 1,  0,  0}{32.189 } & \textcolor[rgb]{ 1,  0,  0}{0.914 } & \textcolor[rgb]{ 1,  0,  0}{0.944 } & \textbf{18.292 } & \textcolor[rgb]{ 1,  0,  0}{33.315 } & \textcolor[rgb]{ 1,  0,  0}{0.934 } & \textcolor[rgb]{ 1,  0,  0}{0.952 } & \textbf{14.643 } \\
		& TNN   & 30.785  & 0.883  & 0.922  & 83.394  & 31.702  & 0.904  & 0.937  & 120.633  & 32.637  & 0.921  & \textbf{0.946 } & 56.524  \\
		& PSTNN & \textbf{30.851 } & \textbf{0.884 } & \textbf{0.923 } & 36.425  & \textbf{31.744 } & \textbf{0.905 } & \textbf{0.938 } & 34.795  & \textbf{32.663 } & \textbf{0.922 } & \textbf{0.946 } & 17.883  \\
		& T-TNN & 15.870  & 0.602  & 0.750  & 1111.489  & 18.316  & 0.657  & 0.834  & 938.164  & 20.803  & 0.792  & 0.879  & 985.819  \\
		& HaLRTC & 26.734  & 0.779  & 0.888  & \textcolor[rgb]{ 1,  0,  0}{1.467 } & 28.025  & 0.825  & 0.913  & \textcolor[rgb]{ 1,  0,  0}{1.264 } & 29.045  & 0.858  & 0.927  & \textcolor[rgb]{ 1,  0,  0}{1.223 } \\
		\hline
	\end{tabular}%
	\label{tab:Video}%
\end{table*}%

In addition, Figure \ref{fig:Frame} displays the PSNR, SSIM, FSIM values of each frontal slice of video ``Akiyo" and ``Bridge". As observed, in almost all frontal slices, the PSNR, SSIM and FSIM metrics of the proposed TNNR are much higher than those of the other compared methods. 

\begin{figure*}
	\centering
	\includegraphics[width=1\linewidth]{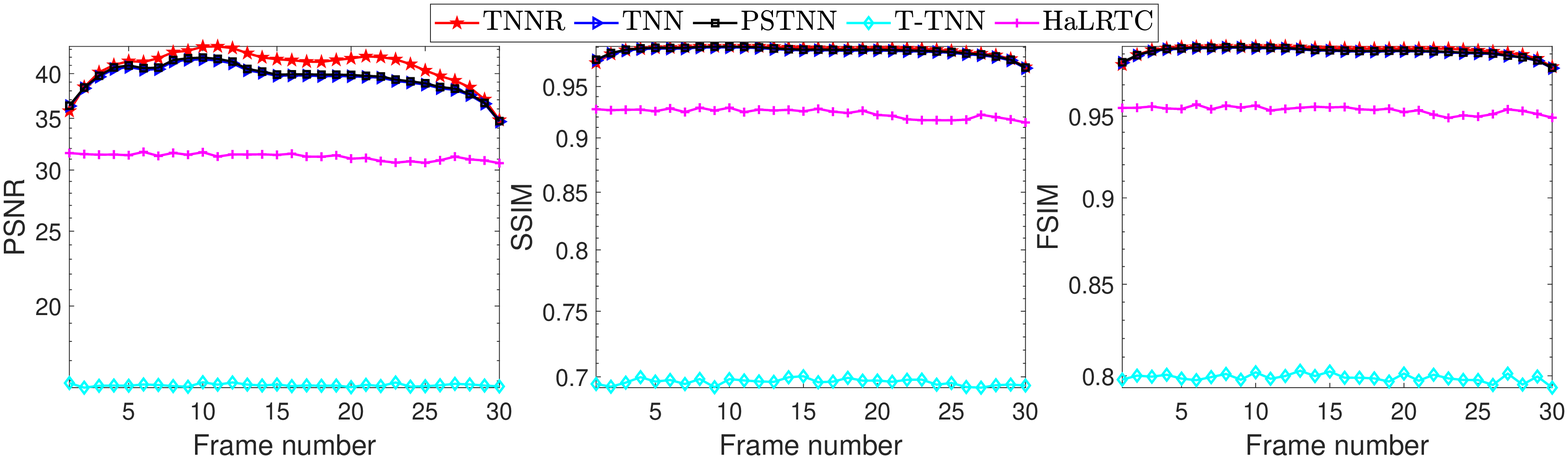}\vspace{0pt}
	\includegraphics[width=1\linewidth]{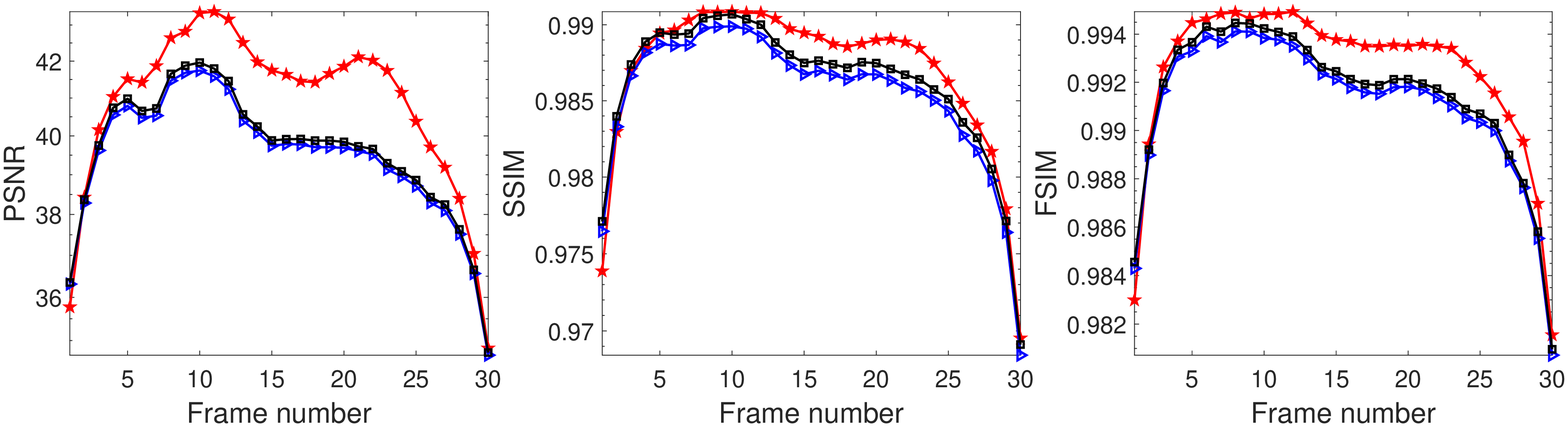}\vspace{0pt}
	\includegraphics[width=1\linewidth]{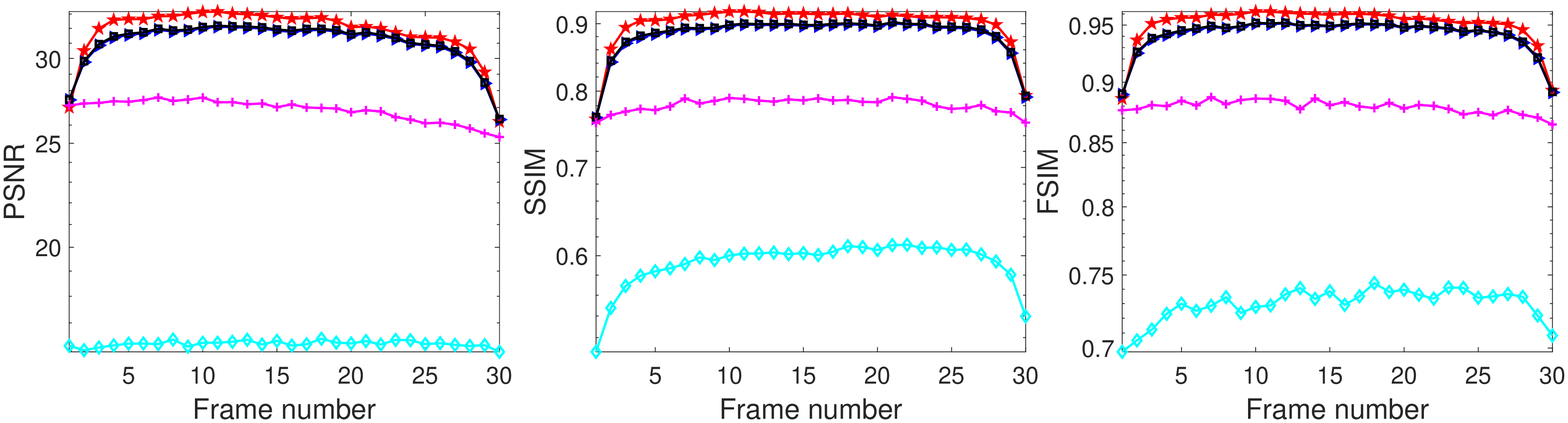}\vspace{0pt}
	\includegraphics[width=1\linewidth]{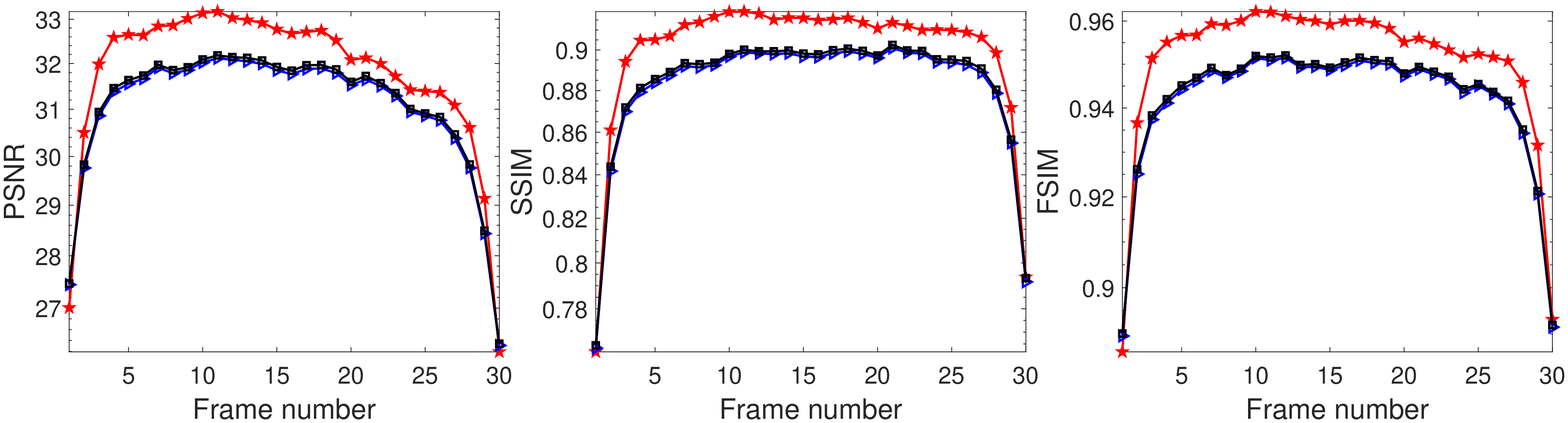}
	\caption{All frontal slices obtained by different methods on the video ``Akiyo" and ``Coastguard" with $ SR=30\% $. For better visualization, we list all frontal slices obtained by TNNR, TNN and PSTNN.}
	\label{fig:Frame}
\end{figure*}

\section{Conclusion}
In this paper, we propose a general and flexible rank relaxation function, named TNNR relaxation function, for efficiently solving the third order tensor recovery problem.
We develop the general inertial smoothing proximal gradient method to accelerate the proposed model. Besides, we prove that our proposed optimization method can converge to a critical point. Combining the KL property with some milder assumptions, we further give its global convergence guarantee. We compare the performance of the proposed method with state-of-the-art methods via numerical experiments on color images, texture images and video. Our method outperforms many state-of-the-art methods quantitatively and visually.

\appendices
\onecolumn

\section{Proof of Theorem \ref{thm:CC}}

Before proving the Theorem \ref{thm:CC}, we first recall the following well-known and fundamental property for a smooth function in the class $C^{1,1}$.

\begin{lemma}\cite{BT09,Ber99}\label{lem:smooth}
	Let function $f(\cdot)$ satisfy Assumption \ref{ass:smooth}. Then, for any $\X_1, \X_2 \in \mathbb{R}^{n_1 \times n_2 \times n_3}$, one has
	\begin{align}\label{eq:smooth}
		f(\X_1) \leq f(\X_2)+\langle\X_1-\X_2, \nabla f(\X_2)\rangle+\frac{L_f}{2}\|\X_1-\X_2\|^{2}.	
	\end{align}
\end{lemma}
Next, we are ready to discuss the convergence by constructing the auxiliary sequence. For any $t \in \mathbb{N}$, define
\begin{equation}\label{H}
	H_{\delta_{t+1}}\left(\X^{t+1}, \X^{t}\right):=F\left(\X^{t+1}\right)+\delta_{t+1} \left\|\X^{t+1}-\X^{t}\right\|^{2},
\end{equation}
where $ \delta_{t+1}=\mu^{t+1}\theta_1^{t+1}/2 $.

Then, we present two lemmas, which will be used later.

\begin{lemma}
	Let function $f(\cdot)$ satisfy Assumption \ref{ass:smooth}. Then, $ g\left(\X\right):=f\left(\X\right)+\frac{L_f}{2}\left\|\X\right\|^2 $ is a convex function.
\end{lemma}
\begin{proof}
For any $\X_1, \X_2 \in \mathbb{R}^{n_1 \times n_2 \times n_3}$, we have 
\begin{equation}
	\begin{aligned}
		\left(\nabla g\left(\X_1\right)-\nabla g\left(\X_2\right)\right)^T\left(\X_1-\X_2\right) &=L_f\left\|\X_1-\X_2\right\|^{2}+\left(\nabla f\left(\X_1\right)-\nabla f\left(\X_2\right)\right)^T\left(\X_1-\X_2\right) \\
		& \geq L_f\left\|\X_1-\X_2\right\|^{2}-\left\|\X_1-\X_2\right\|\left\|\nabla f\left(\X_1\right)-\nabla f\left(\X_2\right)\right\| \geq 0.
	\end{aligned}
\end{equation}
Thus, $ g\left(\X\right) $ is a convex function.
\end{proof}

\begin{lemma}\label{thm:SDC}
	The sequences $\left\{\X^{t}\right\}_{t \in \mathbb{N}}$ generated by TNNR own the following properties:
	\begin{itemize}
		\item[(i)] $ H_{\delta_{t+1}}\left(\X^{t+1}, \X^{t}\right) $ is monotonically decreasing. Indeed,
		\begin{equation}\label{pro:sdc}
			H_{\delta_{t+1}}\left(\X^{t+1}, \X^{t}\right)-H_{\delta_{t}}\left(\X^{t}, \X^{t-1}\right) \le -\frac{\varepsilon L_f}{2}\left\|\X^{t+1}-\X^t\right\|^2;
		\end{equation}
		\item[(ii)] $\lim\limits_{t \rightarrow +\infty}\left(\X^{t}-\X^{t+1}\right)=0$.
	\end{itemize}
\end{lemma}
\begin{proof}
	(i) From \eqref{weight} and Lemma \ref{lem:gradient}, we have
	\begin{equation}\label{pro:weight}
		\begin{aligned}
			\rho\left(g_k^{t+1}\right) \leq \rho\left(g_k^{t}\right)+\alpha_k^{t}\left(g_k^{t+1}-g_k^{t}\right),\quad
			\rho\left(h_i^{k,t+1}\right) \leq \rho\left(h_i^{k,t}\right)+ \beta_i^{k,t}\left(h_i^{k,t+1}-h_i^{k,t}\right), 	
		\end{aligned}
	\end{equation}	
where 
	$$ \sigma_i^{k,t}=\sigma_i\left(\bar X^{(k,t)}\right),\, h_i^{k,t}=\rho\left(\sigma_i^{k,t}\right),\, g_k^{t}=\sum_{i=1}^{r}\rho\left(h_i^{k,t}\right) .$$
Since $ \X^{t+1} $ is a global solution to problem \eqref{X}, we get
	\begin{equation}\label{pro:opt}
		\begin{aligned}
			&\lambda\sum_{k=1}^{n_3}\sum_{i=1}^{r}\alpha_k^t\beta_i^{k,t}\rho\left(\sigma_i^{k,t+1}\right)+\left\langle\nabla f\left(\Z^t\right), \X^{t+1}-\Y^t\right\rangle+\frac{\mu^t}{2}\left\|\X^{t+1}-\Y^t\right\|^{2}	\\\le & \lambda\sum_{k=1}^{n_3}\sum_{i=1}^{r}\alpha_k^t\beta_i^{k,t}\rho\left(\sigma_i^{k,t}\right) +\left\langle\nabla f\left(\Z^t\right), \X^{t}-\Y^t\right\rangle+\frac{\mu^t}{2}\left\|\X^{t}-\Y^t\right\|^{2}.
		\end{aligned}
	\end{equation}
From Lemma \ref{lem:smooth}, we obtain
	\begin{align}\label{eq:smooth-jl}
	f(\X^{t+1}) \leq f(\Z^t)+\langle\X^{t+1}-\Z^t, \nabla f(\Z^t)\rangle+\frac{L_f}{2}\|\X^{t+1}-\Z^t\|^{2}.	
\end{align}
Moreover, by the convexity of $ f\left(\X\right)+\frac{L_f}{2}\left\|\X\right\|^2 $, it holds that
\begin{equation}\label{pro:con}
	f\left(\Z^t\right)+\left\langle \X^{t}-\Z^{t}, \nabla f\left(\Z^{t}\right)\right\rangle \leq f\left(\X^{t}\right)+\frac{L_f}{2}\left\|\X^{t}-\Z^{t}\right\|^2.
\end{equation}
	From \eqref{pro:weight}-\eqref{pro:con}, we have
\begin{equation}\label{pro:dec}
\begin{aligned}
	F\left(\X^{t+1}\right)-F\left(\X^{t}\right)&\le-\frac{\mu^t}{2}\left\|\X^{t+1}-\Y^t\right\|^{2}+\frac{\mu^t}{2}\left\|\X^{t}-\Y^t\right\|^{2}+\frac{L_f}{2}\|\X^{t+1}-\Z^t\|^{2}+\frac{L_f}{2}\left\|\X^{t}-\Z^{t}\right\|^2\\
		&= \mu^{t}\left\langle \Y^{t}-\X^{t}, \X^{t+1}-\X^{t}\right\rangle-\frac{\mu^t}{2}\left\|\X^{t+1}-\X^t\right\|^{2}+\frac{L_f}{2}\|\X^{t+1}-\Z^t\|^{2}+\frac{L_f}{2}\left\|\X^{t}-\Z^{t}\right\|^2.
\end{aligned}
\end{equation}	
Denote $\Delta_{t}:=\X^{t}-\X^{t-1}$, we have
\begin{equation}\label{del}
	\theta_1^t\Delta_{t}=\Y^{t}-\X^{t},\quad \theta_2^t\Delta_{t}=\Z^{t}-\X^{t},\quad \theta_2^t\Delta_{t}-\Delta_{t+1}=\Z^{t}-\X^{t+1},
\end{equation}
combining which with \eqref{pro:dec}, we obtain
\begin{equation}\label{pro:cau}
	\begin{aligned}
		F\left(\X^{t+1}\right)-F\left(\X^{t}\right)&\le\mu^{t}\theta_1^t\left\langle \Delta_{t}, \Delta_{t+1}\right\rangle-\frac{\mu^t}{2}\left\|\Delta_{t+1}\right\|^{2}+\frac{L_f}{2}\|\theta_2^t\Delta_{t}-\Delta_{t+1}\|^{2}+\frac{L_f\left(\theta_2^t\right)^2}{2}\left\|\Delta_{t}\right\|^2\\
		&\le \frac{L_f-\mu^{t}+\mu^{t}\theta_1^t-L_f\theta_2^t}{2}\left\|\Delta_{t+1}\right\|^2+\frac{2L_f\left(\theta_2^t\right)^2+\mu^{t}\theta_1^t-L_f\theta_2^t}{2}\left\|\Delta_{t}\right\|^2,
	\end{aligned}
\end{equation}
where the second inequality comes from $ \mu^{t}\theta_1^t-L_f\theta_2^t\ge 0 $ and Cauchy-Schwartz inequality.

Combining \eqref{pro:cau}, \eqref{H} and Assumption \ref{ass:pra}, one has
$$ H_{\delta_{t+1}}\left(\X^{t+1}, \X^{t}\right)-H_{\delta_{t}}\left(\X^{t}, \X^{t-1}\right) \le -\frac{\varepsilon \mu^t}{2}\left\|\X^{t+1}-\X^t\right\|^2\le -\frac{\varepsilon L_f}{2}\left\|\X^{t+1}-\X^t\right\|^2. $$
Thus $ H_{\delta_{t+1}}\left(\X^{t+1}, \X^{t}\right) $ is monotonically decreasing.
	
(ii) Summing all the inequalities in \eqref{pro:sdc} for $ t\ge 1 $, we get
	\begin{equation*}
		\frac{\varepsilon L_f}{2} \sum_{t=1}^{+\infty}\left\|\X^{t+1}-\X^{t}\right\|^{2} \le H_{\delta_{1}}\left(\X^{1}, \X^{0}\right)-F\left(\X^{+\infty}\right)<+\infty.
	\end{equation*}
	In particular, it implies that $\lim\limits_{t \rightarrow +\infty}\left(\X^{t}-\X^{t+1}\right)=0$.
\end{proof}

\textbf{Now, we prove Theorem \ref{thm:CC}.}

\begin{proof}
(i) It follows from Lemma \ref{thm:SDC}-(i) that
\begin{equation*}
	F\left(\X^{t+1}\right) \le	H_{\delta_{t+1}}\left(\X^{t+1}, \X^{t}\right) \le	H_{\delta_{1}}\left(\X^{1}, \X^{0}\right) < +\infty.
\end{equation*}
Then the boundedness of $ \left\lbrace\X^{t}\right\rbrace $ is obtained based on the Assumption \ref{ass:bounded}. According to the Bolzano-Weirstrass theorem \cite{Fri93}, $ \left\lbrace\X^{t}\right\rbrace $ must have at least one accumulation point.	
	
(ii) By the definition of $ H_{\delta_{t+1}}\left(\X^{t+1}, \X^{t}\right) $ in \eqref{H}, Assumption \ref{ass:bounded} and Theorem \ref{thm:SDC}-(i), we have that the sequence $ \left\lbrace H_{\delta_{t+1}}\left(\X^{t+1}, \X^{t}\right) \right\rbrace  $ is bounded and monotonically nonincreasing. It implies that $ \left\lbrace H_{\delta_{t+1}}\left(\X^{t+1}, \X^{t}\right) \right\rbrace  $ is convergent. Furthermore, since $\lim_{t \rightarrow +\infty}\left(\X^{t}-\X^{t+1}\right)=0$, the sequence $ \left\lbrace F\left(\X^{t}\right)\right\rbrace $ is convergent.

(iii) From the first-order optimality condition of \eqref{X}, we get
	\begin{equation}
	\mathcal{G}^{t_l+1}+\mu^{t_l}\left(\X^{t_l+1}-\Y^{t_l}\right)+\nabla f\left(\Z^{t_l}\right)=0,
\end{equation}
where
$$\mathcal{G}^{t_l+1} \in \partial\left(\lambda\sum_{k=1}^{n_3}\sum_{i=1}^{r}\alpha_k^{t_l}\beta_i^{k,t_l}\rho\left(\sigma_i^{k,t_l+1}\right)\right).$$ 
Thus
\begin{equation}\label{pro:gra}
\mathcal{G}^{t_l+1}+\mu^{t_l}\left(\Delta_{t_l+1}-\theta_1^{t_l}\Delta_{t_l}\right)+\nabla f\left(\X^{t_l}+\theta_2^{t_l}\Delta_{t_l}\right)=0.
\end{equation}
From the fact that $\lim_{t \rightarrow +\infty}\left(\X^{t}-\X^{t+1}\right)=0$ in Theorem \ref{thm:SDC}, we have $\lim_{l \rightarrow +\infty} \X^{t_{l}+1}=\X^\star$. Thus, $\sigma_i^{k,t_l+1} \rightarrow \sigma_i^{k,\star}$
hold as $ l \rightarrow +\infty $. From Assumption \ref{ass:rho}, we can conclude that $ \alpha_k^{t_l}\beta_i^{k,t_l}\rightarrow \alpha_k^{\star}\beta_i^{k,\star}  $. Passing to the limit in \eqref{pro:gra}, from the continuity of $ \nabla f $ and the closedness of $\partial \rho$, we obtain
$$ 0= \partial\left(\lambda\sum_{k=1}^{n_3}\sum_{i=1}^{r}\alpha_k^{\star}\beta_i^{k,\star}\rho\left(\sigma_i^{k,\star}\right)\right) +\nabla f\left(\X^\star\right)\in \partial F\left(\X^\star\right) .$$
Thus, $ \X^\star$ is a stationary point of $ F\left(\X\right) $. From Assumptions \ref{ass:rho} and \ref{ass:smooth}, we know that $ f $ and $ \rho $ are continuous, thus
	\begin{equation*}
	\begin{aligned}
		\lim_{l \rightarrow +\infty}F\left(\X^{t_{l}}\right)
		=\lim_{l \rightarrow +\infty}\left\lbrace\lambda\sum_{k=1}^{n_3}\rho\left(\sum_{i=1}^{r} \rho\left(\rho\left(\sigma_i^{k,t_l}\right)\right)\right)+f\left(\X^{t_l}\right)\right\rbrace
		=F\left(\X^\star\right).	
	\end{aligned}
\end{equation*}	
This completes the proof.
\end{proof}

\section{Proof of Theorem \ref{thm:whole}}
Before we prove Theorem \ref{thm:whole}, we first present four lemmas.

\begin{lemma}\cite{KM11}\label{lem:equ}
Suppose that $\mathcal{A} ,\, \mathcal{B} $ are two arbitrary tensors. Let $\mathcal{F}=\mathcal{A} * \mathcal{B}$, $ \Q $ is an orthogonal tensor. Then, the following properties hold.
\begin{itemize}
	\item[(i)] $\|\mathcal{A}\|^{2}=\frac{1}{n_{3}}\|\bar A\|^{2}$;
	\item[(ii)] $\mathcal{F}=\mathcal{A} * \mathcal{B}$ and $\bar F=\bar A \bar B$ are equivalent to each other;
	\item[(iii)] $ \left\| \A*\Q\right\|=\left\| \A\right\|.  $
\end{itemize}
\end{lemma}
\begin{lemma}
Suppose that $\mathcal{A} \in \mathbb{R}^{n_{1} \times n_{2} \times n_{3}},\, \mathcal{B} \in\mathbb{R}^{n_{2} \times n_{4} \times n_{3}}$ are two arbitrary tensors. Then $ \left\|\A*\B\right\|\le \sqrt{n_3}\left\|\A\right\|\left\|\B\right\| $ and $ \left\|\A*\B\right\|\le \left\|\A\right\|\max_k\left\|\bar B^{(k)}\right\|_2 $.
\end{lemma}
\begin{proof}
	From Lemma \ref{lem:equ}, we have
\begin{equation*}
\left\|\A*\B\right\|=\frac{1}{\sqrt{n_3}}\left\|\bar A\bar B\right\|\le\frac{1}{\sqrt{n_3}}\left\|\bar A\right\|\left\|\bar B\right\|=\sqrt{n_3}\left\|\A\right\|\left\|\B\right\|
\end{equation*}
and 
\begin{equation*}
	\left\|\A*\B\right\|=\frac{1}{\sqrt{n_3}}\left\|\bar A\bar B\right\|\le\frac{1}{\sqrt{n_3}}\left\|\bar A\right\|\left\|\bar B\right\|_2=\left\|\A\right\|\max_k\left\|\bar B^{(k)}\right\|_2.
\end{equation*}

\end{proof}

\begin{lemma}\label{lem:whole}
	Suppose that Assumption \ref{ass:rho}-\ref{ass:pra} hold and $\delta_{t} \equiv \delta$ for all $t \in \mathbb{N}$. Let $\left\{\X^{t}\right\}$ be the sequence generated by TNNR. Consider the function
	\begin{equation*}
		\begin{aligned}
			H: \mathbb{R}^{n_1\times n_2\times n_3} \times \mathbb{R}^{n_1\times n_2\times n_3} \rightarrow(-\infty,+\infty], \quad H(\X, \Y):=F(\X)+\delta\|\X-\Y\|^{2}.
		\end{aligned}
	\end{equation*}
	Then $\left\{H\left(\X^{t+1}, \X^{t}\right)\right\}$ satisfies the following assertions:
	\begin{itemize}
		\item [(H1)] $H\left(\X^{t+1}, \X^{t}\right)+\frac{\varepsilon L_f}{2}\left\|\Delta_{t+1}\right\|^{2} \leq H\left(\X^{t}, \X^{t-1}\right)$ for all $t \geq 1$;
		\item [(H2)] for all $t \in \mathbb{N}$, there exist $\omega^{t+1} \in \partial H\left(\X^{t+1}, \X^{t}\right)$ such that $$\omega^{t+1} \leq \left(L_f+\frac{\mu^0\pi+\left(1+\pi r \right)\pi^2 \tilde{L}L_\rho\sqrt{n_3r}}{\tau}+4\delta \right)\left(\left\|\Delta_{k+1}\right\|+\left\|\Delta_{k}\right\|\right);$$
		\item[(H3)] there exists a subsequence $\left\lbrace\X^{t_l}\right\rbrace_{l \in \mathbb{N}}$ such that $ \lim_{l\rightarrow \infty}\X^{t_{l}}=\X^\star $, and it further holds that $\lim _{l \rightarrow \infty} H\left(\X^{t+1}, \X^{t}\right)=H\left(\X^\star, \X^\star\right)$.
	\end{itemize}	
\end{lemma}
\begin{proof}
(1) The assertion in (H1) can be obtained directly by Theorem \ref{thm:SDC}-(i).

(2) Using the first-order optimality condition of \eqref{X}, we have
\begin{equation*}
	\lambda\partial\left\|\X^{t+1}\right\|_{\rho_{\alpha^t}^{\beta^t}}+\mu^{t}\left(\X^{t+1}-\Y^{t}\right)+\nabla f\left(\Z^{t}\right)=0.
\end{equation*}
Thus,
\begin{equation}\label{bound-1}
	\lambda\U^{t+1}*\W_{t+1}*\left( \V^{t+1}\right)^*+\mu^{t}\left(\X^{t+1}-\Y^{t}\right)+\nabla f\left(\Z^{t}\right)=0,
\end{equation}
where $\X^{t+1}=\U^{t+1}*{\mathcal S^{t+1}}*\left(\V^{t+1}\right)^*$ be the t-SVD of $\X^{t+1}$,
\begin{equation}
	\bar W_{t+1}^{(k)}=\mathscr{D}^i_r\left(\alpha_k^t\beta_i^{k,t}\rho'\left(\sigma_i^{k,t+1}\right)\right),\,\alpha_k^{t} = \rho'\left(\sum_{i=1}^{r} \rho\left(\rho\left(\sigma_i^{k,t}\right)\right)\right),\,
	\beta_i^{k,t}= \rho'\left(\rho\left(\sigma_i^{k,t}\right)\right).
\end{equation}
From the definition of $ H\left(\X, \Y\right) $, we get $ \nabla_{\Y} H\left(\X^{t+1}, \Y^{t+1}\right)=-2 \delta(\X^{t+1}-\Y^{t+1}) $ and
\begin{equation*}
	\lambda\partial\sum_{k=1}^{n_3}\rho\left(\sum_{i=1}^{r} \rho\left(\rho\left(\sigma_i^{k,t+1}\right)\right)\right)+\nabla f(\X^{t+1})+2 \delta(\X^{t+1}-\Y^{t+1}) \in \partial_{\X} H\left(\X^{t+1}, \Y^{t+1}\right).
\end{equation*}
Thus, 
\begin{equation}\label{bound-2}
	\lambda\U^{t+1}*\M_{t+1}*\left( \V^{t+1}\right)^*+\nabla f(\X^{t+1})+2 \delta(\X^{t+1}-\Y^{t+1}) \in \partial_{\X} H\left(\X^{t+1}, \Y^{t+1}\right),
\end{equation}
where $ \bar M_{t+1}^{(k)}=\mathscr{D}^i_r\left(\alpha_k^{t+1}\beta_i^{k,t+1}\rho'\left(\sigma_i^{k,t+1}\right)\right) $. Let $ \D_t $ is a tensor with $ \bar D_{t}^{(k)}=\mathscr{D}^i_r\left(\alpha_k^t\beta_i^{k,t}\right) $. From Assumption \ref{ass:rho}, we get $ \D_t $ is invertible. From \eqref{bound-1} and \eqref{bound-2}, we obtain
\begin{equation*}
\begin{aligned}
	\omega^{t+1}_\X:&=-\U^{t+1}*\D_{t+1}*\D_{t}^{-1}*\left(\U^{t+1}\right)^**\left(\mu^{t}\left(\X^{t+1}-\Y^{t}\right)+\nabla f\left(\Z^{t}\right)\right)+\nabla f(\X^{t+1})+2 \delta(\X^{t+1}-\Y^{t+1})\\ &\in \partial_{\X} H\left(\X^{t+1}, \Y^{t+1}\right). 
\end{aligned}
\end{equation*}
Let $ \omega_\Y^{t+1}:=-2 \delta(\X^{t+1}-\X^{t}) $, we get $\omega^{t+1}:=\left(\omega_\X^{t+1}, \omega_\Y^{t+1}\right) \in \partial H\left(\X^{t+1}, \X^{t}\right)$.

It follows from Theorem \ref{thm:CC}-(i) that $\left\{\X^{t}\right\}$ is bounded. Then the sequence $\left\{\Z^{t}\right\}$ is also bounded by \eqref{X}. Combining that $\nabla f$ is continuous, there exists $\tilde{L}>0$ such that
$$
\left\|\nabla f\left(\Z^{t}\right)\right\| \leq \tilde{L} .
$$
Considering that $\rho^{\prime}$ is positive and continuous, for any $t$, $ k\in[m] $ and $i \in [r]$, and $\left\{\sigma_i^{k,t}\right\}$ is bounded for the boundness of $\left\{\X^{t}\right\}$. Therefore, there exist $\tau,\, \pi>0$ such that
$$
\tau \leq \alpha_k^t\beta_i^{k,t} \leq \pi,\quad \alpha_k^t \leq \pi,\quad \beta_i^{k,t} \leq \pi,\quad
\rho'\left(\sigma_i^{k,t}\right) \leq \pi.
$$
Hence, we have
\begin{equation}\label{bound-3}
	\max_k\left\|\left(\bar D_{t}^{(k)}\right)^{-1}\right\|_2 \leq \max _{i,k} \frac{1}{\alpha_k^t\beta_i^{k,t}} \leq \frac{1}{\tau}
\end{equation}
and
\begin{equation}\label{bound-4}
	\max_k\left\|\bar D_{t+1}^{(k)}\left(\bar D_{t}^{(k)}\right)^{-1}\right\|_2 \leq \max _{i,k} \frac{\alpha_k^{t+1}\beta_i^{k,t+1}}{\alpha_k^t\beta_i^{k,t}} \leq \frac{\pi}{\tau}.
\end{equation}
\iffalse
\begin{equation}\label{bound-4}
	\left\|\D_{t+1}*\D_{t}^{-1}\right\|=\frac{1}{\sqrt{n_3}}\sum_{k=1}^{n_3}\left\|\bar\D_{t+1}^{(k)}\left(\bar\D_{t}^{(k)}\right)^{-1}\right\|\le \frac{m}{\sqrt{n_3}}\sum_{k=1}^{n_3}\left\|\bar\D_{t+1}^{(k)}\left(\bar\D_{t}^{(k)}\right)^{-1}\right\|_2 \leq \frac{m}{\sqrt{n_3}}\sum_{k=1}^{n_3}\max _{i} \frac{\alpha_k^{t+1}\beta_i^{k,t+1}}{\alpha_k^t\beta_i^{k,t}} \leq \frac{m\pi\sqrt{n_3}}{\tau}.
\end{equation} 
\fi
Combining \eqref{bound-3} and \eqref{bound-4}, we get
\begin{equation}\label{bound-5}
	\begin{aligned}
		\left\|\omega^{t+1}\right\| \le& \left\|\omega_\X^{t+1}\right\|+ \left\|\omega_\Y^{t+1}\right\|\\
		\le& \left\|\nabla f(\X^{t+1})-\U^{t+1}*\D_{t+1}*\D_{t}^{-1}*\left(\U^{t+1}\right)^**\nabla f\left(\Z^{t}\right)\right\|+\mu^t\left\|\X^{t+1}-\Y^{t}\right\|\max_k\left\|\bar\D_{t+1}^{(k)}\left(\bar\D_{t}^{(k)}\right)^{-1}\right\|_2		
		+4\delta\left\|\Delta_{t+1}\right\|\\
		\le& \left\|\nabla f\left(\Z^{t}\right)-\U^{t+1}*\D_{t+1}*\D_{t}^{-1}*\left(\U^{t+1}\right)^**\nabla f\left(\Z^{t}\right)\right\|+\left\|\nabla f\left(\X^{t+1}\right)-\nabla f\left(\Z^{t}\right)\right\|+ \frac{\mu^t\pi}{\tau}\left\|\X^{t+1}-\Y^{t}\right\|	\\&+4\delta\left\|\Delta_{t+1}\right\|\\
		 \le&\left\|\U^{t+1}*\D_{t}*\D_{t}^{-1}*\left(\U^{t+1}\right)^**\nabla f\left(\Z^{t}\right)-\U^{t+1}*\D_{t+1}*\D_{t}^{-1}*\left(\U^{t+1}\right)^**\nabla f\left(\Z^{t}\right)\right\|
		 \\&+ L_f\left\|\Delta_{t+1}\right\|+L_f\theta_2^t\left\|\Delta_{t}\right\|+\frac{\mu^t\pi}{\tau}\left\|\Delta_{t+1}\right\|+\frac{\mu^t\theta_1^t\pi}{\tau}\left\|\Delta_{t}\right\|+4\delta\left\|\Delta_{t+1}\right\|\\
		 \le&  \frac{\sqrt{n_3}\tilde{L}}{\tau}\left\|\D_{t+1}-\D_{t} \right\|+\left(L_f\theta_2^t+\frac{\mu^t\theta_1^t\pi}{\tau}\right)\left\|\Delta_{t}\right\|+\left(L_f+\frac{\mu^t\pi}{\tau}+4\delta \right)\left\|\Delta_{t+1}\right\|,
	\end{aligned}
\end{equation}
where the last inequality follows from 
\begin{equation*}
\begin{aligned}
	&\left\|\U^{t+1}*\D_{t}*\D_{t}^{-1}*\left(\U^{t+1}\right)^**\nabla f\left(\Z^{t}\right)-\U^{t+1}*\D_{t+1}*\D_{t}^{-1}*\left(\U^{t+1}\right)^**\nabla f\left(\Z^{t}\right)\right\|\\
	\le &\sqrt{n_3}\left\|\U^{t+1}*\D_{t}*\D_{t}^{-1}*\left(\U^{t+1}\right)^*-\U^{t+1}*\D_{t+1}*\D_{t}^{-1}*\left(\U^{t+1}\right)^*\right\|\left\|\nabla f\left(\Z^{t}\right)\right\|\\
	\le &\sqrt{n_3}\tilde{L}\left\|\D_{t}*\D_{t}^{-1}-\D_{t+1}*\D_{t}^{-1}\right\|\\
	\le &  \sqrt{n_3}\tilde{L}\left\|\D_{t}-\D_{t+1}\right\|\max_k\left\|\left(\bar D_{t}^{(k)}\right)^{-1}\right\|_2\\
	\le &\frac{\sqrt{n_3}\tilde{L}}{\tau}\left\|\D_{t+1}-\D_{t} \right\|.
\end{aligned}
\end{equation*}
It follows from Assumption \ref{ass:rho} that
\begin{equation}
\begin{aligned}
	&\left| \alpha_k^{t+1}\beta_i^{k,t+1}-\alpha_k^t\beta_i^{k,t}\right| \\
	\le&\left| \alpha_k^{t+1}\beta_i^{k,t+1}-\alpha_k^{t+1}\beta_i^{k,t}\right|+\left| \alpha_k^{t+1}\beta_i^{k,t}-\alpha_k^t\beta_i^{k,t}\right|\\
	\le&\pi\left| \beta_i^{k,t+1}-\beta_i^{k,t}\right|+\pi\left| \alpha_k^{t+1}-\alpha_k^t\right|\\
	\le&\pi L_\rho\left| \rho\left(\sigma_i^{k,t+1}\right)-\rho\left(\sigma_i^{k,t}\right)\right|+\pi L_\rho\left| \sum_{i=1}^{r} \rho\left(\rho\left(\sigma_i^{k,t+1}\right)\right)-\sum_{i=1}^{r} \rho\left(\rho\left(\sigma_i^{k,t}\right)\right)\right|\\
	\le&\pi^2 L_\rho\left| \sigma_i^{k,t+1}-\sigma_i^{k,t}\right|+\pi^2 L_\rho\sum_{i=1}^{r}\left|  \rho\left(\sigma_i^{k,t+1}\right)- \rho\left(\sigma_i^{k,t}\right)\right|\\
	\le & \pi^2 L_\rho\left| \sigma_i^{k,t+1}-\sigma_i^{k,t}\right|+\pi^3 L_\rho\sum_{i=1}^{r}\left|  \sigma_i^{k,t+1}- \sigma_i^{k,t}\right|\\
	\le & \pi^2 L_\rho\left\| \bar X^{(k,t+1)}-\bar X^{(k,t)}\right\|+\pi^3 L_\rho\sum_{i=1}^{r}\left\| \bar X^{(k,t+1)}-\bar X^{(k,t)}\right\|\\
	=&\left(1+\pi r \right)\pi^2 L_\rho\left\| \bar X^{(k,t+1)}-\bar X^{(k,t)}\right\|,
\end{aligned}
\end{equation}
where the last inequality follows from \cite[Theorem 3.3.16]{Mer92}. Thus,
\begin{equation}\label{bound-6}
	\sqrt{n_3}\left\|\D_{t+1}-\D_{t} \right\|=\sum_{k=1}^{n_3}\left\|\bar D_{t+1}^{(k)}-\bar D_{t}^{(k)} \right\|\le\left(1+\pi r \right)\pi^2 L_\rho\sqrt{r}\sum_{k=1}^{n_3}\left\| \bar X^{(k,t+1)}-\bar X^{(k,t)}\right\|=\left(1+\pi r \right)\pi^2 L_\rho\sqrt{n_3r}\left\| \Delta_{t+1}\right\|.
\end{equation}
Combining \eqref{bound-5} and \eqref{bound-6}, we obtain
\begin{equation}\label{}
\begin{aligned}
	\left\|\omega^{t+1}\right\|&\le  \left(L_f\theta_2^t+\frac{\mu^t\theta_1^t\pi}{\tau}\right)\left\|\Delta_{t}\right\|+\left(L_f+\frac{\mu^t\pi+\left(1+\pi r \right)\pi^2 \tilde{L}L_\rho \sqrt{n_3r}}{\tau}+4\delta \right)\left\|\Delta_{t+1}\right\|\\
	&\le\left(\frac{L_f}{2}+\frac{\mu^0\pi}{2\tau}\right)\left\|\Delta_{t}\right\|+\left(L_f+\frac{\mu^0\pi+\left(1+\pi r \right)\pi^2 \tilde{L}L_\rho\sqrt{n_3r}}{\tau}+4\delta \right)\left\|\Delta_{t+1}\right\|\\
	&\le \left(L_f+\frac{\mu^0\pi+\left(1+\pi r \right)\pi^2 \tilde{L}L_\rho\sqrt{n_3r}}{\tau}+4\delta \right)\left(\left\|\Delta_{t}\right\|+\left\|\Delta_{t+1}\right\|\right). 
\end{aligned}
\end{equation}

(3) The proof is similar to Theorem \ref{thm:CC}-(ii) and hence we omit it here.

\end{proof}

\textbf{Now, we prove Theorem \ref{thm:whole}.}

\begin{proof}
From Lemma \ref{lem:whole}, we obtain the conclusion according to \cite[Theorem 2]{WL19}. The desired result is obtained.	
\end{proof}


\ifCLASSOPTIONcaptionsoff
  \newpage
\fi



%\bibliographystyle{IEEEtran}
\twocolumn
\bibliographystyle{IEEEtran}
\bibliography{TNNR}

\iffalse

\fi


%\vfill


\end{document}